\documentclass[12pt]{article}

\usepackage[colorlinks, linkcolor=blue,citecolor=blue]{hyperref}           
\usepackage{color}

\usepackage{graphicx,subfigure,amsmath,amssymb,amsfonts,bm,epsfig,epsf,url,dsfont}
\usepackage{amsthm,mathrsfs}
\usepackage{tikz}
\usepackage{multirow}
\usepackage{bbm}      
\usepackage{booktabs}
\usepackage{cases}
\usepackage{fullpage,enumitem}
\usepackage[small,bf]{caption}
\usepackage[top=1in,bottom=1in,left=1in,right=1in]{geometry}
\usepackage{fancybox}
\usepackage{hyperref}
\usepackage{algorithm}
\usepackage{verbatim}
\usepackage{algorithmic}
\usepackage{amsthm}
\usepackage{mathtools}

\usepackage{scalerel,stackengine}
\stackMath
\newcommand\reallywidehat[1]{%
\savestack{\tmpbox}{\stretchto{%
0  \scaleto{%
    \scalerel*[\widthof{\ensuremath{#1}}]{\kern-.6pt\bigwedge\kern-.6pt}%
    {\rule[-\textheight/2]{1ex}{\textheight}}%WIDTH-LIMITED BIG WEDGE
  }{\textheight}% 
}{0.5ex}}%
\stackon[1pt]{#1}{\tmpbox}%
}
\newenvironment{nscenter}
 {\parskip=0pt\par\nopagebreak\centering}
 {\par\noindent\ignorespacesafterend}
\newcommand{\floor}[1]{\left \lfloor #1 \right \rfloor}
\newcommand{\ceil}[1]{\left \lceil #1 \right \rceil}
\newcommand{\bE}{\mathbb{E}}

\newcommand{\cT}{\mathcal{T}}

\newtheorem{definition}{Definition}
\newtheorem{example}{Example}

\newtheorem{theorem}{Theorem}

\newtheorem{conjecture}{Conjecture}
\newtheorem{lemma}{Lemma}
\newtheorem{prop}{Proposition}

\newtheorem{rmk}{Remark}

\newenvironment{fminipage}%
  {\begin{Sbox}\begin{minipage}}%
  {\end{minipage}\end{Sbox}\fbox{\TheSbox}}

\makeatletter
\newcommand*{\rom}[1]{\expandafter\@slowromancap\romannumeral #1@}
\makeatother
\newtheorem{remark}{Remark}
\newtheorem{cor}{Corollary} 
\newtheorem{obs}{Observation}

\newcommand{\Ind}{\mathbbm{1}}

\newcommand{\vcd}{\mathsf{{vcd}}}

\newcommand{\vc}{\tau}
\newcommand{\smu}{\zeta}
\newcommand{\sd}{\mu}

\newcommand{\abs}[1]{\left|#1\right|}
 
\newcommand{\N}{\mathbb{N}}

\newcommand{\E}{\mathbb{E}}

\def\P{{\mathbb P}}

\newcommand{\indep}{\perp \!\!\! \perp}

\newcommand {\pr} {\mathbb{P}}

\newcommand{\calA}{{\cal A}}

\newcommand{\calE}{{\cal E}}

\newcommand{\calG}{{\cal G}}
\newcommand{\calH}{{\cal H}}

\newcommand{\calJ}{{\cal J}}
\newcommand{\calK}{{\cal K}}

\newcommand{\calM}{{\cal M}}
\newcommand{\calN}{{\cal N}}

\newcommand{\calP}{{\cal P}}
\newcommand{\calQ}{{\cal Q}}

\newcommand{\calS}{{\cal S}}

\DeclarePairedDelimiterX{\set}[1]{\{}{\}}{\setargs{#1}}
\DeclarePairedDelimiterX{\cond}[1]{[}{]}{\setargs{#1}}
\NewDocumentCommand{\setargs}{>{\SplitArgument{1}{;}}m}
{\setargsaux#1}
\NewDocumentCommand{\setargsaux}{mm}
{\IfNoValueTF{#2}{#1} {#1\,\delimsize|\,\mathopen{}#2}}%{#1\:;\:#2}

\newcommand{\be}{\begin{equation}}
\newcommand{\ee}{\end{equation}}
\newcommand{\beqna}{\begin{eqnarray}}
\newcommand{\eeqna}{\end{eqnarray}}

\newcommand{\p}[1]{\left(#1\right)}
\newcommand{\pp}[1]{\left[#1\right]}
\newcommand{\ppp}[1]{\left\{#1\right\}}
\newcommand{\norm}[1]{\left\|#1\right\|}
\newcommand{\innerP}[1]{\left\langle#1\right\rangle}

\usepackage{xspace}

\newcommand{\s}[1]{\mathsf{#1}}

\makeatletter
\def\thanks#1{\protected@xdef\@thanks{\@thanks
        \protect\footnotetext{#1}}}
\makeatother

\makeatletter
\renewcommand{\paragraph}{%
  \@startsection{paragraph}{4}%
  {\z@}{1.25ex \@plus 1ex \@minus .2ex}{-1em}%
  {\normalfont\normalsize\bfseries}%
}
\makeatother

\allowdisplaybreaks

\begin{document}
\title{Detecting Arbitrary Planted Subgraphs in Random Graphs}
%\title{Statistical and Computational Limits of Detecting Arbitrary Planted Subgraphs in Random Graphs}
\author{Dor Elimelech~~~~~~~~~~~~~Wasim Huleihel\thanks{D. Elimelech and W. Huleihel are with the School of Electrical Engineering and Computer Engineering, at Tel Aviv University, {T}el {A}viv 6997801, Israel (e-mails:  \texttt{dorelimelech@tauex.tau.ac.il, wasimh@tauex.tau.ac.il}). This work is supported by the ISRAEL SCIENCE FOUNDATION (grant No. 1734/21).}}

\maketitle
\begin{abstract}
    
The problems of detecting and recovering planted structures/subgraphs in Erd\H{o}s-R\'{e}nyi random graphs, have received significant attention over the past three decades, leading to many exciting results and mathematical techniques. However, prior work has largely focused on specific ad hoc planted structures and inferential settings, while a general theory has remained elusive. In this paper, we bridge this gap by investigating the detection of an \emph{arbitrary} planted subgraph $\Gamma = \Gamma_n$ in an Erd\H{o}s-R\'{e}nyi random graph $\mathcal{G}(n, q_n)$, where the edge probability within $\Gamma$ is $p_n$. We examine both the statistical and computational aspects of this problem and establish the following results. In the dense regime, where the edge probabilities $p_n$ and $q_n$ are fixed, we tightly characterize the information-theoretic and computational thresholds for detecting $\Gamma$, and provide conditions under which a computational-statistical gap arises. Most notably, these thresholds depend on $\Gamma$ only through its number of edges, maximum degree, and maximum subgraph density. Our lower and upper bounds are general and apply to any value of $p_n$ and $q_n$ as functions of $n$. Accordingly, we also analyze the sparse regime where $q_n = \Theta(n^{-\alpha})$ and $p_n-q_n =\Theta(q_n)$, with $\alpha\in[0,2]$, as well as the critical regime where $p_n=1-o(1)$ and $q_n = \Theta(n^{-\alpha})$, both of which have been widely studied, for specific choices of $\Gamma$. For these regimes, we show that our bounds are tight for all planted subgraphs investigated in the literature thus far\textemdash{}and many more. Finally, we identify conditions under which detection undergoes sharp phase transition, where the boundaries at which algorithms succeed or fail shift abruptly as a function of $q_n$. 
\end{abstract}

\section{Introduction}\label{sec:intro}

\sloppy

The study of structures in networks is a central problem at the intersection of graph theory, computer science, and statistics, with widespread applications in fields such as social and biological sciences. A key aspect of this research involves identifying communities--groups of nodes with many internal connections and relatively few links to the rest of the network. While much of the existing work has focused on assigning nodes to specific communities, an equally fundamental challenge is determining whether a small, well-connected group exists within a large random graph. This problem, introduced by \cite{arias2014community}, has practical implications for event detection and cluster monitoring, as well as theoretical significance in understanding the statistical and computational limits of community detection \cite{chen2016statistical}.

Machine learning problems inherently involve two key aspects: \emph{statistical} and \emph{computational}. The statistical aspect determines the accuracy of inference tasks, while the computational aspect examines the efficiency of solving them. Traditionally, these aspects have been studied separately, with information theory and statistics providing the foundation for understanding statistical limits. However, computational feasibility has often been overlooked, despite its growing importance as modern datasets continue to expand. Recent research, e.g., \cite{berthet2013complexity,ma2015computational,cai2015computational,chen2016statistical,hopkins2017bayesian,Hopkins18,gamarnik2020lowdegree,barak2016nearly,Lenka16,Lesieur_2015,hajek2015computational,bandeira2018notes,brennan18a,brennan19,brennan20a}, and many references therein, has revealed a fundamental gap between the data required for statistically optimal methods and the limits of computationally efficient algorithms, specifically in problems with a planted combinatorial structure. This \emph{statistical-to-computational gap} suggests that while a task may be statistically feasible, no known efficient algorithm can achieve optimal performance.

In this paper, we investigate the inferential problem of detecting the presence of an \emph{arbitrary} subgraph planted in otherwise Erd\H{o}s-R\'{e}nyi random graph. This is formulated as follows. Let $n\in\mathbb{N}$, $q=q_n\in(0,1)$, $p=p_n\in(0,1)$, and $\Gamma=\Gamma_n$ be an arbitrary undirected sequence of graphs; $\Gamma$ is referred to as the planted subgraph. We consider the following detection problem. Under the null hypothesis, the observed graph $\s{G}$ is an instance from the Erd\H{o}s-R\'{e}nyi random distribution $\calG(n,q_n)$, with edge density $q_n$. Under the alternative hypothesis, we first draw a uniform random copy of $\Gamma_n$ in the complete graph, and then include each edge of $\Gamma_n$ in $\s{G}$ with probability $p_n$, while other edges are drawn with probability $q_n$; this is referred to as the \emph{union ensemble} (as planting can be thought of as taking the union of $\calG(n,q_n)$ with (some) edges of $\Gamma_n$).  
%the observed graph $\s{G}$ is the union of $\calG(n,q_n)$ with a uniform random copy of $\Gamma_n$ in the complete graph, but whose edges are kept with probability $p_n$ \Dor{This is slightly inaccurate - the union model is indeed equivalent to our model, however the $p_n$ of the union model is not the $p_n$ we use in our results.}. 
The last few decades have observed a wide range of research on specific planted subgraphs. Perhaps the most well-known canonical example is the planted clique problem \cite{jerrum1992large}, which involves distinguishing between $\calG(n,1/2)$, and randomly selected $k$-clique embedded within $\calG(n,1/2)$. Several other planted subgraph models have been explored in this research area, including the planted dense subgraph problem \cite{arias2014community,arias2015detecting,hajek2015computational}, where an instance of $\calG(n,p_n)$is embedded in $\calG(n,q_n)$, the planted tree model \cite{massoulie19a}, the planted Hamiltonian cycle problem \cite{Bagaria20}, the planted matching problem \cite{10.1214/20-AAP1660}, and the planted bipartite problem \cite{HuleihelBip}, just to name a few.

Although, at a high level, the techniques used to analyze different planted subgraph models share common principles, their detailed mathematical treatment is often intricate and highly problem-specific. In fact, not only do the analytical bounds typically vary from one problem to another, but different planted subgraphs also exhibit fundamentally distinct statistical and computational behaviors. For instance, many planted structures, such as the planted clique problem, exhibit rich and complex behavior, with distinct phases where detection is either statistically impossible, computationally hard, or computationally easy. In striking contrast, certain structures, such as paths or stars, do not exhibit a computationally hard phase \cite{massoulie19a}. Furthermore, in some cases, such as trees planted in sparse graphs with $q=\Theta(n^{-1})$ detection undergoes a sharp phase transition, where the boundaries at which algorithms succeed or fail shift abruptly as a function of $q$. This behavior contrasts with that of many other subgraphs, where the transition is gradual.
%It is important to understand what causes such phases. 

Attempts to establish a unified framework for detecting arbitrary planted subgraphs can be found in the literature. Such an attempt already appeared in \cite{addario2010combinatorial}, albeit for a slightly different model. However, for general structures, the lower and upper bounds provided therein are loose. More recently, \cite{Huleihel2022} proposed the general testing problem described above, along with a variant in which the planted structure appears as an induced subgraph. The latter was examined from both information-theoretic and computational perspectives in the dense regime, where $p=1$ and $q=\Theta(1)$. In \cite{pmlr-v247-yu24a}, it was shown that, in the dense regime, the optimal constant-degree polynomial is always obtained by counting stars. Finally, \cite{pmlr-v195-mossel23a} investigated the recovery variant, which involves \emph{estimating} the planted subgraph, establishing tight results for certain families of subgraphs, again primarily in the dense regime. Despite these efforts, the fundamental task of characterizing the statistical and computational limits for detecting general subgraphs in random graphs\textemdash{}to the best of our knowledge\textemdash{}remains largely open.

The quest for a fundamental theory for the general planted subgraph setting constitutes the main impetus behind this paper. Specifically, motivated by the unresolved challenges outlined above, we pose the following questions:

\vspace{0.2cm}
\centerline{\noindent\fbox{\parbox{0.9\textwidth}{
\begin{nscenter}
     \emph{What properties of $\Gamma$ govern the statistical and computational limits of detection?}

      \emph{For which planted structures does a statistical-to-computational gap exist?}
          
     \emph{Under what conditions do sharp phase transitions take place?}
\end{nscenter}
}}}
\vspace{0.2cm}

\paragraph{Main contributions (informal).} In this paper, we address the questions above in various settings, as detailed below. We begin our investigation with perhaps the most canonical setting of this problem: the dense regime, where the edge probabilities $p_n$ and $q_n$ are fixed and independent of $n$. For this setting, we tightly characterize the information-theoretic thresholds for detecting an arbitrary planted subgraph $\Gamma$. Specifically, let $\mu(\Gamma_n)$, $|e(\Gamma_n)|$, and $d_{\max}(\Gamma_n)$ denote the maximum subgraph density (see, Definition~\ref{def:maxsubden}), the number of edges, and the maximum degree in $\Gamma_n$, respectively. Then, roughly speaking, we show that the statistical barrier is given by

\vspace{-0.4cm}
\begin{center}
    \fbox{
        \parbox[t][0.9cm][c]{0.51\textwidth}{ % Adjust height here (e.g., 2.5cm)
            \[
            \chi^2(p_n||q_n) \approx \frac{\log n}{\mu(\Gamma_n)} \wedge \frac{n^2}{|e(\Gamma_n)|^2} \wedge \frac{n\log n}{d_{\max}^2\left(\Gamma_n\right)}
            \]
        }
    }
\end{center}
%
%\begin{center}
 %   \fbox{
 %       \parbox{0.51\textwidth}{ 
 %           \[
 %           \chi^2(p_n \parallel q_n) \approx \frac{\log n}{\mu(\Gamma_n)} \wedge \frac{n^2}{|e(\Gamma_n)|^2} \wedge \frac{n \log n}{d_{\max}^2\left(\Gamma_n\right)}
  %          \]
  %      }
  %  }
%\end{center}

\noindent where $\chi^2(p_n||q_n)$ is the Chi-square distance between $p_n$ and $q_n$. To wit, if $\chi^2(p_n||q_n)$ is ``larger" than the term on the right-hand side (r.h.s.), then detection is possible; otherwise, detection is information-theoretically impossible. The algorithms that achieve this barrier rely on brute-force search for the subgraph that achieves the maximum subgraph density, counting the total number of edges, and evaluating the maximum degree in the observed graph. The more interesting and challenging part is the proof of the lower bound. Specifically, it is rather folklore that to rule out the possibility of detection, it suffices to upper bound the second moment of the likelihood, and that this second moment is given by the moment-generating function of the intersection between two random copies of $\Gamma$ in the complete graph. That is, the optimal risk $\s{R}_n^\star$ satisfies,
\begin{align}
    \s{R}_n^\star\geq1-\frac{1}{2}\sqrt{\bE\pp{(1+\chi^2(p_n||q_n))^{|e(\Gamma\cap\Gamma')|}}-1}.
\end{align}
This is, more or less, the starting point for the analysis of arguably any planted structure considered in the literature. The challenge here is that the distribution of $|e(\Gamma\cap\Gamma')|$ is intricate and strongly depends on $\Gamma$. For instance, when $\Gamma$ is a $k$-clique, we have $|e(\Gamma\cap\Gamma')|\stackrel{\s{d}}{=}\binom{H}{2}$, where $H\sim\s{Hypergeometric}(n,k,k)$. On the other hand, if $\Gamma$ is a $k$-path, the distribution of $|e(\Gamma\cap\Gamma')|$ is significantly more complex, implicit, and dominated by a certain Markov chain \cite{massoulie19a}. Consequently, even for these two relatively simple structures, the analytical techniques required are quite different. Attempts such as those in \cite{addario2010combinatorial} to bound the above using standard probabilistic arguments either result in loose bounds or impose restrictions on the family of subgraphs. Furthermore, it is well understood that the subgraph detection problem is closely related to the study of sharp thresholds for the appearance of specific subgraphs in $\calG(n,q_n)$, as explored in \cite{erdos_renyi_60}, as well as upper tail bounds for subgraph counts in random graphs—a subject that has received significant attention over the past four decades, leading to strong results (e.g., \cite{Alon1981OnTN,Janson,Friedgut1998OnTN,Jaikumar,Wojciech} and references therein). While these results provide useful bounds on the optimal risk, we demonstrate that even the best existing bounds fall short in capturing the statistical barrier described above. 

We also investigate the computational aspects of our problem and tightly characterize the computational barrier using the framework of low-degree polynomials. Specifically, as expected, for some choices of $\Gamma$, there exists a gap between the statistical limits we derive and the performance of the efficient algorithms we construct (i.e., the count and maximum-degree tests). We conjecture that this gap is, in fact, inherent\textemdash{}namely, below the computational barrier, no polynomial-time algorithms exist. To provide evidence for this conjecture, we follow a recent line of work (e.g., \cite{hopkins2017bayesian,Hopkins18,Kunisky19,Cherapanamjeri20,gamarnik2020lowdegree}) and demonstrate that the class of low-degree polynomials fails to solve the detection problem in this conjectured hard regime. Roughly speaking, we show that such a gap may exist only if $\Gamma$ has \emph{super-logarithmic density}, i.e., $\mu(\Gamma) = \Omega(\log|v(\Gamma)|)$. Moreover, in such a case, all $O(\log n)$-degree polynomials fail, provided that,

%\vspace{0.2cm}
\centerline{\noindent\fbox{\parbox[t][0.6cm][c]{0.35\textwidth}{
\begin{align*}
        |e(\Gamma_n)|\vee d_{\max}^2(\Gamma_n)\ll n
\end{align*}}}}
\vspace{0.2cm} 

\noindent This is why, for example, there is a hard phase in the planted clique problem but not in the planted path problem. In the former, $\mu(\Gamma_n) = \frac{|v(\Gamma_n)|-1}{2}=\Omega(\log|v(\Gamma_n)|)$, whereas in the latter $\mu(\Gamma_n) = \frac{|v(\Gamma_n)|-1}{|v(\Gamma_n)|}=o(\log|v(\Gamma_n)|)$. The results above completely resolve the statistical and computational limits of the subgraph detection problem in the dense regime. While our main focus is on this regime, our bounds are, in fact, general and apply to any values of $p_n$ and $q_n$ as functions of $n$. Although these bounds are not tight in general, they become tight for all planted subgraphs studied in the literature thus far. For concreteness, we pay special attention to the sparse regime, where $\chi^2(p||q) = \Theta(n^{-\alpha})$ for $\alpha\in[0,2]$, and the critical regime, where $\chi^2(p||q) = \Theta(n^{\alpha})$ for $\alpha\in[0,2]$, and characterize the impossible, hard, and easy regimes as a function of the polynomial growth/decay of each of the relevant parameters (e.g., number of vertices, edge, etc.). Finally, in the critical regime, we show that detection undergoes a sharp phase transition as a function of $q$.

The rest of this paper is organized as follows. In Section~\ref{sec:problem}, we introduce the problem setup and provide some necessary preliminaries. Section~\ref{sec:mainresults} presents our main results for the various regimes, including the dense regime, the sparse regime, and the critical regime. Section~\ref{sec:upperbound} and Section~\ref{sec:statlowerbound} are devoted to the the derivation of our upper and lower bounds, respectively. Finally, in Section~\ref{sec:complowerbound} we establish our computational lower bounds.
\subsection{Related work}\label{subsec:relatedwork}
This work is part of an expanding field that explores planted combinatorial structures in random graphs and matrices from both statistical and computational perspectives. Below, we mention previous research on this topic. However, we would like to emphasize that this is by no means an exhaustive overview, and many other related studies exist in the literature.

\paragraph{Planted subgraphs and matrices.} 

The planted clique problem was first introduced in \cite{alon1998finding}, where a spectral algorithm was shown to successfully recover the planted clique when $k = \Omega(\sqrt{n})$. Since then, various approaches have been developed for solving this problem, including approximate message passing, semidefinite programming, nuclear norm minimization, and other combinatorial methods \cite{feige2000finding}; \cite{mcsherry2001spectral}; \cite{feige2010finding}; \cite{ames2011nuclear}; \cite{dekel2014finding}; \cite{deshpande2015finding}; \cite{chen2016statistical}. Notably, all these algorithms require $k = \Omega(\sqrt{n})$, which has led to the widely held planted clique conjecture, suggesting that no polynomial-time algorithm can recover the planted clique when $k = o(\sqrt{n})$. Beyond planted cliques, related problems have been explored. Several works \cite{coja2015independent}; \cite{gamarnik2014limits}; \cite{rahman2017local} analyze greedy and local algorithms for detecting independent sets in Erd\H{o}s-R\'{e}nyi-type random graph models when $q = \Theta(n^{-1})$. Additionally, \cite{feige2005finding} introduces a spectral method for recovering a planted independent set under similar conditions, while \cite{coja2003finding} establishes that polynomial-time recovery is possible when $q = \Theta(n^{-\alpha})$ for $\alpha \in (0,1)$, provided $q \gg \frac{n}{k^2}$.

The planted dense subgraph detection problem has been extensively studied in works such as \cite{arias2014community}; \cite{butucea2013detection}; \cite{verzelen2015community}; \cite{Hajek2015}, with more general formulations explored in \cite{chen2016statistical}; \cite{hajek2016information}; \cite{montanari2015finding}; \cite{candogan2018finding}. A key result by \cite{hajek2015computational} establishes a reduction from the planted clique problem for the regime $p = c q$ (with some constant $c > 1$) and $q = \Theta(n^{-\alpha})$, where $p$ and $q$ denote community and background edge densities, respectively. This result was further strengthened in \cite{brennan18a}, showing that it holds for all $p > q$ when $p - q = O(q)$. When $p = \omega(q)$, the problem transitions into the planted dense subgraph regime analyzed in \cite{bhaskara2010detecting}. Other variations of planted structure detection include the planted tree model \cite{massoulie19a}, the planted Hamiltonian cycle problem \cite{Bagaria20}, the planted matching problem \cite{10.1214/20-AAP1660}, the planted bipartite problem \cite{HuleihelBip}, and the detection of planted dense cycles \cite{ChengWein, pmlr-v195-mao23a}. These papers reference a broad range of additional related works, further expanding the study of planted structures in random graphs and matrices.

A particularly relevant topic in this context is community detection in the stochastic block model, which has been the subject of extensive research (see \cite{abbe2017community} for a survey). Recent results indicate that the two-community stochastic block model does not exhibit computational-statistical gaps for partial and exact recovery when the edge density scales as $\Theta(n^{-1})$ \cite{mossel2012stochastic}; \cite{mossel2013proof}; \cite{abbe2016exact}; \cite{abbe2017community}. Another well-studied problem is Gaussian biclustering, which has been analyzed in both detection and recovery settings. The detection problem is examined in \cite{butucea2013detection}; \cite{ma2015computational}; \cite{montanari2015limitation}, while recovery approaches are considered in \cite{shabalin2009finding}; \cite{kolar2011minimax}; \cite{balakrishnan2011statistical}; \cite{cai2015computational}; \cite{chen2016statistical}; \cite{hajek2016information}; \cite{brennan19}; \cite{10447320,10387722}. Additionally, a significant body of research has investigated the spectral properties of the spiked Wigner model \cite{peche2006largest}; \cite{feral2007largest}; \cite{capitaine2009largest}, with spectral algorithms and information-theoretic bounds developed in \cite{montanari2015limitation}; \cite{perry2016statistical}; \cite{perry2016optimality}; \cite{banks2018information}; \cite{hopkins2017power}.

\paragraph{General planting.} Several studies have examined the detection and recovery of general planted subgraphs and matrices in random graphs. In \cite{addario2010combinatorial}, the authors formulated a hypothesis testing problem where, given an observed realization of an $n$-dimensional Gaussian vector, one must determine whether the vector originates from a standard normal distribution or if a subset of components, belonging to a predefined class, has a nonzero mean. Using probabilistic techniques, they derived general upper and lower bounds for detection. While these bounds are loose in the general case, they are tight for certain special cases. The methods introduced in this work have since been widely adopted in subsequent research. More recently, \cite{Huleihel2022} introduced two models for planting general subgraphs in random graphs: the union model, which we study in this paper, and an alternative approach where the planted structure appears as an induced subgraph. The latter was analyzed in the dense regime ($p=1$ and $q=\Theta(1)$) from both an information-theoretic and a computational perspective. However, these results are not directly comparable to ours, as the union and induced subgraph models exhibit fundamentally different behaviors. The computational limits of the union planting model were further investigated in \cite{pmlr-v247-yu24a}, in the dense regime where $p=1$ and $q$ remains fixed. It was shown that in this setting, the optimal \emph{constant}-degree polynomial is always obtained by counting stars. In contrast, our work examines both statistical and computational limits while allowing for polynomials of degree $O(\log n)$. Finally, \cite{pmlr-v195-mossel23a} studied the recovery problem, which involves estimating the planted subgraph rather than merely detecting its presence. The authors established upper and lower bounds for recovery, which are tight for specific families of subgraphs, again focusing primarily on the dense regime.

\paragraph{Average-case complexity.} Over the past decade, significant progress has been made in developing a rigorous understanding of the fundamental limits of efficient algorithms. Recent works (e.g., \cite{berthet2013complexity,ma2015computational,cai2015computational,krauthgamer2015semidefinite,hajek2015computational,chen2016statistical,wang2016average,wang2016statistical,gao2017sparse,brennan18a,brennan19,wu2018statistical,brennan20a,hopkins2017bayesian,Hopkins18,Kunisky19,Cherapanamjeri20,gamarnik2020lowdegree,barak2016nearly,deshpande2015improved,meka2015sum,TengyuWig15,kothari2017sum,hopkins2016integrality,raghavendra2019highdimensional,hopkins2017power,mohanty2019lifting,Feldman17,feldman2018complexity,Diakonikolas17,DiakonikolasKong19,Lenka16,Lesieur_2015,Lesieur_2016,Krzakala10318,Ricci_Tersenghi_2019,bandeira2018notes,schramm2020computational,pmlr-v134-brennan21a,Wein2021Independent,Abhishek23}) have uncovered a striking phenomenon common to many high-dimensional problems with a planted structure: a fundamental gap exists between the amount of data required by computationally efficient algorithms and the data needed for statistically optimal procedures. Rigorous evidence supporting this computational-statistical gap has been established through various approaches, which can be broadly classified into the following categories:  
\begin{enumerate}[leftmargin=*]
    \item \textbf{Failure under certain computational models:} This approach demonstrates that even powerful classes of computationally efficient algorithms fail in the conjectured computationally hard regime. Examples include:  
    \begin{itemize}[leftmargin=*]
        \item \textbf{Low-degree polynomials} are a powerful tool for analyzing the computational complexity of high-dimensional inference problems. They provide a rigorous way to understand the gap between statistical and computational limits, e.g., \cite{hopkins2017bayesian,Hopkins18,Kunisky19,Cherapanamjeri20,gamarnik2020lowdegree}. 
        \item \textbf{Sum-of-squares hierarchy} is a powerful framework for designing and analyzing algorithms for combinatorial and optimization problems. It is based on semi-definite programming (SDP) relaxations and provides tight approximations to polynomial optimization problems by considering higher-degree polynomial constraints, e.g.. \cite{barak2016nearly,deshpande2015improved,meka2015sum,TengyuWig15,kothari2017sum,hopkins2016integrality,raghavendra2019highdimensional,hopkins2017power,mohanty2019lifting}.
        \item \textbf{Statistical query algorithms} are a class of computational models to study learning problems where the algorithm does not access individual data points but instead queries expectations of functions over the data distribution. This framework has become a fundamental tool for analyzing the computational complexity of high-dimensional inference problems, e.g., \cite{Feldman17,feldman2018complexity,Diakonikolas17,DiakonikolasKong19,pmlr-v134-brennan21a}.
        \item \textbf{Message-passing algorithms} are a class of iterative methods used in high-dimensional statistical inference, combinatorial optimization, etc. They are also studied in the context of average-case complexity because they often achieve optimal performance in structured random models but fail in regimes conjectured to be computationally hard, e.g., \cite{Lenka16,Lesieur_2015,Lesieur_2016,Krzakala10318,Ricci_Tersenghi_2019,bandeira2018notes}.  
    \end{itemize}  
    \item \textbf{Average-case reductions:} This method establishes hardness by reducing a problem to another problem conjectured to be computationally difficult, such as the planted clique problem, e.g., \cite{berthet2013complexity,ma2015computational,cai2015computational,chen2016statistical,hajek2015computational,wang2016average,wang2016statistical,gao2017sparse,brennan18a,brennan19,wu2018statistical,brennan20a}.  
\end{enumerate}

\subsection{Notation}\label{subsec:notation}

In this paper, we adopt the following notational conventions. We denote the size of any finite set $\calS$ by $|\calS|$. For $n\in\mathbb{N}$ we let $[n] = \{1,\ldots,n\}$, and $\binom{\calS}{n}\triangleq\{\calA\subseteq\calS:|\calA|=n\}$. For a subset $\calS\subseteq\mathbb{R}$, let $\Ind\pp{\calS}$ denote the indicator function of the set $\calS$. For $a,b\in\mathbb{R}$, we let $a\vee b\triangleq\max\{a,b\}$ and $a\wedge b\triangleq\min\{a,b\}$. We denote by $\s{Bern}(p)$ and $\s{Binomial}(n,p)$ the Bernoulli and binomial distributions with $n$ trials and success probability $p$, respectively. We denote by $\s{Hypergeometric}(n,k,m)$ the Hypergeometric distribution with parameters $(n,k,m)$. Given a finite or measurable set $\mathcal{X}$, let $\s{Unif}[\mathcal{X}]$ denote the uniform distribution on $\mathcal{X}$. For two random variables $\s{X}$ and $\s{Y}$, we write $\s{X}\indep\s{Y}$ if $\s{X}$ and $\s{Y}$ are statistically independent. For probability measures $\mathbb{P}$ and $\mathbb{Q}$, let $d_{\s{TV}}(\mathbb{P},\mathbb{Q})=\frac{1}{2}\int |\mathrm{d}\mathbb{P}-\mathrm{d}\mathbb{Q}|$, $\chi^2(\mathbb{P}||\mathbb{Q}) = \int\frac{(\mathrm{d}\mathbb{P}-\mathrm{d}\mathbb{Q})^2}{\mathrm{d}\mathbb{Q}}$, and $d_{\s{KL}}(\mathbb{P}||\mathbb{Q}) = \bE_{\mathbb{P}}\log\frac{\mathrm{d}\mathbb{P}}{\mathrm{d}\mathbb{Q}}$, denote the total variation distance, the $\chi^2$-divergence, and the Kullback-Leibler (KL) divergence, respectively. In particular, we use $\chi^2(p||q)$ to denote the $\chi^2$-distance between two Bernoulli distributions with parameters $p$ and $q$; it is straightforward to check that $\chi^2(p||q) = \frac{(p-q)^2}{q(1-q)}$. We will frequently use standard asymptotic notations $O,o,\Omega,\omega,\Theta$. %For two positive sequences $\{a_n\}$ and $\{b_n\}$, we write $a_n = O(b_n)$ if $a_n\leq Cb_n$, for some absolute constant $C$ and for all $n$; $a_n = \Omega(b_n)$, if $b_n = O(a_n)$; $a_n = \Theta(b_n)$, if $a_n = O(b_n)$ and $a_n = \Omega(b_n)$, $a_n = o(b_n)$ or $b_n = \omega(a_n)$, if $a_n/b_n\to0$, as $n\to\infty$. 
The notation $\ll$ refers to polynomially less than in $n$, namely, $a_n\ll b_n$  if $\liminf_{n\to\infty}\log_n a_n<\liminf_{n\to\infty}\log_n b_n$, e.g., $n\ll n^2$, but $n\not\ll n\log_2 n$. Throughout the paper, $\s{C}$ refers to any constant independent of the parameters of the problem at hand and will be reused for different constants. 

A graph $\s{G}$ is a pair $(v(\s{G}),e(\s{G}))$, where $v(\s{G})$ is the vertex set and $e(\s{G}) \subseteq\binom{v(\s{G})}{2}$ is the edge set; $|v(\s{G})|$ and $|e(\s{G})|=|\s{G}|$ denote the sizes of thereof.  We say $\s{H}\subseteq\s{G}$ is a subgraph of $\s{G}$ if $v(\s{H})\subseteq v(\s{G})$ and $e(\s{H})\subseteq e(\s{G})$. We treat subgraphs $\s{H}$ containing no isolated vertices as sets of edges, and vice versa; thus, we sometimes use $|\s{H}|$ to denote $|e(\s{H})|$. In general, we denote subsets of edges with capital letters $H\subseteq E$, and the corresponding induced graph by $\s{H}$. The complete graph on $n$ vertices is denoted by $\calK_n$. For vertices, we also use capital letters $U\subseteq V$, and the induced graph as $G_U$.
An isomorphism of graphs $\s{G}_1$ and $\s{G}_2$ is a bijection $f:v(\s{G}_1)\to v(\s{G}_2)$ such that any two vertices $u,v\in\s{G}_1$ are adjacent in $\s{G}_1$ if and only if $f(u)$ and $f(v)$ are adjacent in $\s{G}_2$. We use $\s{G}_1\cong\s{G}_2$ to denote that $\s{G}_1$ and $\s{G}_2$ are isomorphic. An automorphism of $\s{G}$ is a graph isomorphism from $\s{G}$ to itself. The set of automorphisms of a graph $\s{G}$, i.e., the automorphism group, is denoted by $\s{Aut}(\s{G})$. If $|v(\s{G})|\leq n$, we let $\calS_{\s{G}}$ be the set of isomorphic copies of $\s{G}$ in $\calK_n$.
 A simple combinatorial counting argument reveals that $|\calS_{\s{G}}| = \binom{n}{|v(\s{G})|}  \frac{|v(\s{G})|!}{\vert\mathsf{Aut}{\left(\s{G}\right)}\vert}$ (see, e.g., \cite[Lemma 5.1]{FriezeKaronski2015}).

%%%%%%%%%%%%%%%%%%%%%%%%%%%%%%%%%%%%%%%%%%%%%%%%%%%%%%%%%%%

\section{Problem Setup and Preliminaries}\label{sec:problem}

In this section we describe the setting we plan to study alongside several important preliminaries. Let $\Gamma=(\Gamma_n)_{n\in\N}$ be a sequence of graphs such that for all $ n\in\mathbb{N}$, $\Gamma_n=(v(\Gamma_n),e(\Gamma_n))$ is an \emph{arbitrary} undirected graph without isolated vertices, such that $|v(\Gamma_n)|\leq n$. %Throughout the paper $v(\Gamma_n)$ and $e(\Gamma_n)$ denote the set of vertices and edges of $\Gamma_n$, respectively; $|v(\Gamma_n)|\leq n$ and $|e(\Gamma_n)|=|\Gamma_n|$ denote the sizes of thereof. 
We let $\calS_{\Gamma_n}$ be the set of isomorphic copies of $\Gamma_n$ in $\calK_n$.
% a simple combinatorial counting argument reveals that $|\calS_{\Gamma_n}| = \binom{n}{|v(\Gamma_n)|}  \frac{|v(\Gamma_n)|!}{\vert\mathsf{Aut}{\left(\Gamma_n\right)}\vert}$ (see, e.g., \cite[Lemma 5.1]{FriezeKaronski2015})
We shall refer to $\Gamma_n$ as the \emph{planted/hidden} structure. 

Our detection problem can be phrased as the following simple hypothesis testing problem. Under the null hypothesis $\calH_0$, the observed graph $\s{G}$ is an instance from the Erd\H{o}s-R\'{e}nyi random distribution $\calG(n,q_n)$, with edge density $0<q_n<1$. Under the alternative hypothesis $\calH_1$, the observed graph $\s{G}$ is the union of an Erd\H{o}s-R\'{e}nyi random graph with a uniform random copy of $\Gamma_n$, whose edges are kept with probability $p_n$ such that $q_n<p_n\leq1$. To wit, under $\calH_1$ the graph $\s{G}$ on $n$ vertices is constructed as follows: we first draw $\Gamma_n\sim\s{Unif}(\calS_{\Gamma_n})$. Then, each edge of $\Gamma_n$ is included in $\s{G}$ with probability $p_n$, while the leftover edges in $\s{G}$ are included with probability $q_n$. In short, we have the following hypothesis testing problem:
\begin{align}
\calH_0: \s{G} \sim \calG(n,q_n) \quad \s{vs.} \quad \calH_1 : \s{G} \sim \calG_{\Gamma_n}(n,p_n,q_n),\label{eqn:super_hypo}   
\end{align}
where $\calG_{\Gamma_n}(n,p_n,q_n)$ denotes the ensemble of planted graphs, as defined above. We study the above framework in the asymptotic regime where $n\to \infty$. Observing $\s{G}$, the goal is to design a test/algorithm $\phi(\s{G}) \in \{0, 1\}$ that distinguishes between $\calH_0$ and $\calH_1$. Specifically, the average $\mathsf{Type}$ $\mathsf{I}$+$\mathsf{II}$ risk of a
test $\phi$ is defined as
$\s{R}_n (\phi) = \pr_{\calH_0}(\phi(\s{G}) = 1)+ \pr_{\calH_1}(\phi(\s{G}) = 0)$, and the optimal risk is defined as $\s{R}_n^\star\triangleq\inf_{\phi_n}\s{R}_n (\phi_n)$. We consider the following types of detection guarantees.
\begin{definition}[Strong and weak detection]\label{def:detection}
Let $\pr_{\calH_0}$ and $\pr_{\calH_1}$ be the distributions of $\s{G}$ under the null and alternative hypotheses, respectively. A possibly randomized sequence of tests $\phi_n(\s{G}) \in \{0, 1\}$ achieves strong detection if $\limsup_{n\to\infty}\s{R}_n (\phi_n)=0$, and weak detection if $\limsup_{n\to\infty}\s{R}_n (\phi_n)<1$. Conversely, we say that strong detection is impossible if $\liminf_{n\to\infty}\s{R}_n^\star>0$, and weak detection is impossible if $\lim_{n\to\infty}\s{R}_n^\star=1$.
\end{definition}
Our results will be expressed in terms of the following graph theoretic measures. We let $d_{\max}\left(\Gamma_n\right)$ denote the maximum degree in $\Gamma_n$, and we define the density of $\Gamma_n$ as $\eta(\Gamma_n)\triangleq |e(\Gamma_n)|/|v(\Gamma_n)|$.  Finally, we recall the definition of the \emph{maximum subgraph density}.
\begin{definition}[Maximum subgraph density \cite{bollobas_2001}]\label{def:maxsubden}
Let $\s{G}$ be an undirected graph. The maximum subgraph density of $\s{G}$ is
\begin{align}
\sd(\s{G})\triangleq\max\ppp{\eta(\s{H}):\s{H}\subseteq\s{G},\s{H}\neq\emptyset}.\label{eqn:maxDensity}
\end{align}
\end{definition}
\begin{comment}
The following two graph-theoretic measures play important roles in our results and analysis.
\begin{definition}[Vertex cover]
Let $\s{G}=(\s{V},\s{E})$ be an undirected graph. A set $\s{U}\subseteq \s{V}$ is called a vertex cover of $\s{G}$ if any edge in $\s{E}$ has a vertex in $\s{U}$. We define the vertex cover number $\vc(\s{G})$ of $\s{G}$ as the minimal cardinality vertex cover set in $\s{G}$. 
\end{definition}
\begin{definition}\label{def:coverdegreebalanced}
    A sequence of graphs $\Gamma=(\Gamma_n)_n$ is called \textit{vertex cover-degree balanced}, or simply $\vcd$-balanced, if 
    \begin{align}
        \lim_{n\to\infty} \frac{\log\p{\vc(\Gamma)\cdot d_{\max}(\Gamma)}}{\log|e(\Gamma)|}=1.
    \end{align}
\end{definition}
\end{comment}
%We sometimes label all $\Gamma_n$ copies in $\calS_{\Gamma_n}$ as $\Gamma_{1}, \Gamma_{2}, \ldots, \Gamma_{|\calS_{\Gamma_n}|}$. 
Throughout this paper, we sometime relegate the notational dependency of various parameters on the index sequence $n$ implicitly, e.g., we denote the sequence of planted graphs as $\Gamma = (\Gamma_n)_n$, the pair of sequences of edges densities by $p=(p_n)_n$ and $q=(q_n)_n$, etc. Finally, $\chi^2(p||q) = \frac{(p-q)^2}{q(1-q)}$ denotes the $\chi^2$ divergence  between two Bernoulli distributions with parameters $p$ and $q$.
%We mention here that as $\Gamma=(\Gamma_n)_n$ is a sequence of graphs, the graph parameters $(k,\mu,d,\vc)$ are sequences indexed with $n$, as well. Nevertheless we omit the index $n$ from our notation.  

\section{Main Results}\label{sec:mainresults}
 In this section, we will introduce the main results of our paper. As mentioned, we will focus on three main regimes: the dense regime, the sparse regime and the critical regime. It should be emphasized that the statistical lower and upper bounds (presented in Section~\ref{sec:statlowerbound} and Section~\ref{sec:upperbound}) and the computational lower bounds (presented in Section~\ref{sec:complowerbound}) are  general, and  extends beyond this three main regimes for any sequence of graphs $\Gamma$, and for any values of $p$ and $q$. 
 
\subsection{Dense regime}

\paragraph{Statistical limits.} We start with the following result which establishes the statistical limits in the dense regime.

\begin{theorem}[Dense regime]\label{th:StatLimitDense}
Fix a sequence of subgraphs $\Gamma=(\Gamma_n)_n$ and consider the detection problem in \eqref{eqn:super_hypo}. Assume that $\chi^2(p||q)=\Theta(1)$. %Then, the following hold.
\begin{enumerate}
    \item \underline{Lower bounds}: If $\mu(\Gamma_n)\geq \alpha_n\cdot \log|v(\Gamma_n)|$, for some $\alpha_n=\Omega(1)$, then there exists a constant $\underline{C}>0$ such that weak detection is impossible if, 
    \begin{align}
        \mu(\Gamma_n)\leq \underline{C}\cdot \log n.\label{eq:condDenseStat}
    \end{align}
    If $\mu(\Gamma_n)=o(\log|v(\Gamma_n)|)$, then for every $\varepsilon>0$, weak detection is impossible if,
    \begin{align}
        |e(\Gamma_n)|\vee d^2_{\max}(\Gamma_n)\leq n^{1-\varepsilon}.
    \end{align}
    \item \underline{Upper bounds}: There exists $\overline{C}>0$ such that strong detection is possible for every $\varepsilon>0$ if,
    \begin{align}
        \mu(\Gamma_n)\geq \overline{C} \cdot\log n\quad\s{or}\quad|e(\Gamma_n)|\geq n^{1+\varepsilon}\quad\s{or}\quad d_{\max}^2(\Gamma_n)\geq n^{1+\varepsilon}.\label{eq:detectiont}
    \end{align}
    %\begin{align}
    %    \mu(\Gamma_n)\geq \overline{C} \cdot\log n\quad\s{or}\quad|e(\Gamma_n)|\vee d_{\max}^2(\Gamma_n)\geq n^{1+\varepsilon}.\label{eq:detectiont}
    %\end{align}
\end{enumerate} 
\begin{comment}
\begin{enumerate}
    \item If $\mu(\Gamma_n)\geq \alpha_n\cdot \log|v(\Gamma_n)|$, for some $\alpha_n$, then there exist constants $\overline{C}$ and $\underline{C}$ such that weak detection is impossible if, \begin{align}
        \mu(\Gamma_n)\leq \underline{C}\cdot \log n,\label{eq:condDenseStat}
    \end{align} while strong detection is possible if,
    \begin{align}
        \mu(\Gamma_n)\geq \overline{C} \cdot\log n.
    \end{align}
    \item If $\mu(\Gamma_n)=o(\log|v(\Gamma_n)|)$, then for every $\varepsilon>0$, weak detection is impossible if,
    \begin{align}
        |e(\Gamma_n)|\vee d^2_{\max}(\Gamma_n)\leq n^{1-\varepsilon},
    \end{align}
     while strong detection is possible if,
    \begin{align}
        |e(\Gamma_n)|\vee d_{\max}^2(\Gamma_n)\geq n^{1+\varepsilon}.\label{eq:detectiont}
    \end{align}
    \item Under the low-degree polynomial conjecture, detection in polynomial time is impossible if,
\begin{align}
     \max\p{|e(\Gamma_n)|,d^2_{\max}(\Gamma_n)}\leq n^{1-\varepsilon}.\label{eq:CondSaprse}
\end{align}
\end{enumerate} 
\end{comment}
\end{theorem}
A few important remarks are in order. The constants $\overline{C},\underline{C}$ depend on $(p_n,q_n,\alpha_n)$; to keep the exposition simple we have made this dependency implicit, but explicit formulae can be found in the appendices. In many cases these constants are sharp. The algorithms achieving the upper bounds in \eqref{eq:detectiont}, from left to right, rely on brute-force search for the planted structure (i.e., scan test), counting the total number of edges, and evaluating the maximum degree in the observed graph; for more details, we refer the reader to Section~\ref{sec:upperbound}. The latter two tests exhibit polynomial-time computational complexity and are therefore efficient. The scan test, however, is inefficient, requiring exponential time. Accordingly, Theorem~\ref{th:StatLimitDense} shows that for graphs $\Gamma$ with sub-logarithmic density, namely, $\mu(\Gamma_n)=o(\log|v(\Gamma_n)|)$, the detection problem is either statistically impossible or solvable in polynomial-time, and no hard phase occurs. This is not the case for graphs $\Gamma$ with super-logarithmic density, namely, $\mu(\Gamma_n)=\Omega(\log|v(\Gamma_n)|)$. Indeed, we conjecture that in the region where the scan test succeeds, but the count and degree tests fail, no polynomial-time algorithm exist. 

\paragraph{Computational limits.} As an evidence for the above claim, we use the framework of low-degree polynomials (see, e.g., \cite{Hopkins18,Dmitriy19}). We continue with a brief background on the low-degree polynomial (LDP) method, and refer the reader to Section~\ref{sec:complowerbound}, for a more detailed exposition. The LDP framework hings on the hypothesis that all polynomial-time algorithms for solving detection problems are captured/represented by low-degree polynomials. To date, there is increasing and compelling evidence supporting this conjecture. The concepts described below were developed through a fundamental sequence of works in the sum-of-squares optimization literature \cite{barak2016nearly,Hopkins18,hopkins2017bayesian,hopkins2017power}. 

We begin by outlining the fundamentals of of the LDP framework, adhering to the notations and definitions established in \cite{Hopkins18,Dmitriy19}. Recall that any distribution $\pr_{\calH_0}$
defines an inner product of measurable functions $f,g:\Omega_n\to\mathbb{R}$ given by $\left\langle f,g \right\rangle_{\calH_0} = \bE_{\calH_0}[f(\s{G})g(\s{G})]$, along with a norm defined as $\norm{f}_{\calH_0} = \left\langle f,f \right\rangle_{\calH_0}^{1/2}$. Also, recall that the space $L^2(\calH_0)$ represents the Hilbert space of function $f$ with  $\norm{f}_{\calH_0}<\infty$, equipped with the above inner product. The core idea of the LDP method is to identify the ``low-degree" polynomial that distinguishes $\pr_{\calH_0}$ from $\pr_{\calH_1}$ best in the $L^2$ sense. To define this mathematically, let $V_{n, \leq\s{D}}\subset L^2(\calH_0)$ denote the subspace of polynomials of degree at most $\s{D}\in\mathbb{N}$. Then, the \emph{$ \s{D}$-low-degree likelihood-ratio} $\s{L}_{n,\leq  \s{D}}$ is defined as the projection of the likelihood $\s{L}_{n}$ onto $V_{n, \leq\s{D}}$, w.r.t. $\left\langle \cdot,\cdot \right\rangle_{\calH_0}$. Now, recall that the likelihood-ratio is the optimal test to distinguish $\pr_{\calH_0}$ from $\pr_{\calH_1}$, in the $L^2$ sense. As it turns out, the $\s{D}$-low-degree likelihood-ratio shares a similar property \cite{hopkins2017bayesian,hopkins2017power,Dmitriy19}.
%Let us describe the basics of the LDP framework. We follow the notations and definition in \cite{Hopkins18,Dmitriy19}. Recall that any distribution $\pr_{\calH_0}$ induces an inner product of measurable functions $f,g:\Omega_n\to\mathbb{R}$ given by $\left\langle f,g \right\rangle_{\calH_0} = \bE_{\calH_0}[f(\s{G})g(\s{G})]$, and norm $\norm{f}_{\calH_0} = \left\langle f,f \right\rangle_{\calH_0}^{1/2}$. Also, recall that $L^2(\calH_0)$ be the Hilbert space of functions $f$ with $\norm{f}_{\calH_0}<\infty$, and endowed with the inner product above. The idea behind the LDP framework is to find the ``low-degree" polynomial that best distinguishes $\pr_{\calH_0}$ from $\pr_{\calH_1}$ in the $L^2$ sense. Let $V_{n, \leq\s{D}}\subset L^2(\calH_0)$ denote the linear subspace of polynomials of degree at most $ d\in\mathbb{N}$. Then, the \emph{$ d$-low-degree likelihood-ratio} $\s{L}_{n,\leq  d}$ is the projection of a function $\s{L}_{n}$ to $\s{L}_{n, \leq\s{D}}$, where the projection is w.r.t. the inner product $\left\langle \cdot,\cdot \right\rangle_{\calH_0}$. As is well known, the likelihood-ratio is the optimal test to distinguish $\pr_{\calH_0}$ from $\pr_{\calH_1}$, in the $L^2$ sense. Accordingly, it can be shown that over $\s{L}_{n, \leq\s{D}}$, the function $\s{L}_{n, \leq\s{D}}$ exhibits the same property \cite{hopkins2017bayesian,hopkins2017power,Dmitriy19}. 
\begin{lemma}[Optimally of $\s{L}_{n, \leq\s{D}}$ {\cite{hopkins2017bayesian,hopkins2017power,Dmitriy19}}]\label{lem:Dmitriymain}
Consider the following optimization problem:
\begin{equation}
\begin{aligned}
\mathrm{max}
\;\bE_{\calH_1}f(\s{G})
\quad\mathrm{s.t.}
\quad\bE_{\calH_0}f^2(\s{G}) = 1,\; f\in V_{n, \leq\s{D}},
\end{aligned}\label{eqn:optimizationProblemmain}
\end{equation}
Then, the unique solution $f^\star$ for \eqref{eqn:optimizationProblemmain} is the $\s{D}$-low degree likelihood-ratio $f^\star = \s{L}_{n, \leq\s{D}}/\norm{\s{L}_{n, \leq\s{D}}}_{\calH_0}$, and the value of the optimization problem is $\norm{\s{L}_{n, \leq\s{D}}}_{\calH_0}$. 
\end{lemma}
A key characteristic of the likelihood-ratio is that when $\norm{\s{L}_n}_{\calH_0}=O(1)$, then $\pr_{\calH_0}$ and $\pr_{\calH_1}$ are statistically indistinguishable (or, strong detection is impossible). The conjecture below extends this principle to the computational realm. In a nutshell, it suggests that polynomials of degree $\approx\log n$ can effectively represent/cover all polynomial-time algorithms. The statement below is inspired by  \cite{Hopkins18,hopkins2017bayesian,hopkins2017power}, and \cite[Conj. 2.2.4]{Hopkins18}. Here, we present the informal version which appears in \cite[Conj. 1.16]{Dmitriy19}, while more formal statements can be found in, e.g., \cite[Conj. 2.2.4]{Hopkins18} and \cite[Sec. 4]{Dmitriy19}.
\begin{conjecture}[Low-degree conj., informal]\label{conj:1main}
Given a sequence of probability measures $\pr_{\calH_0}$ and $\pr_{\calH_1}$, if there exists $\epsilon>0$ and $\s{D}\geq (\log n)^{1+\epsilon}$, such that $\norm{\s{L}_{n, \leq\s{D}}}_{\calH_0}$ remains bounded as $n\to\infty$, then there is no polynomial-time algorithm that distinguishes $\pr_{\calH_0}$ and $\pr_{\calH_1}$.
\end{conjecture}
We are in a position to state our main result.
\begin{theorem}[LDP lower bound]\label{th:CompLimitDense}
Fix a sequence of subgraphs $\Gamma=(\Gamma_n)_n$. If $\mu(\Gamma_n)=\Omega(\log|v(\Gamma_n)|)$, then under the low-degree polynomial conjecture (see, Conjecture~\ref{conj:1main} in Section~\ref{sec:complowerbound}), for every $\varepsilon>0$, strong detection in polynomial-time is impossible whenever,
\begin{align}
     |e(\Gamma_n)|\vee d^2_{\max}(\Gamma_n)\leq n^{1-\varepsilon}.\label{eq:CondSaprse}
\end{align}
\end{theorem}
%Also, $\varepsilon$ in \eqref{eq:detectiont} can be replaced with any $o(1)$ function.
%We remark that the constants $\underline{C}_{p,q,\alpha}$ and $\overline{C}_{p,q,\alpha}$ can be written implicitly. In fact, in the case of a planted clique, these constants provide tight bounds. We also mention that the threshold for strong detection stated in \eqref{eq:detectiont} is achieved by the count test and the maximum degree tests which runs in polynomial time. It is therefore immediate from Theorem~\ref{th:StatLimitDense} that if $\Gamma$ has sub-logarithmic density, no statistical-computational gap is evident. On the other hand, when $\Gamma$ has super-logarithmic density, we show in Section~\ref{sec:complowerbound} that a statistical-computational gap is evident, as the condition in \eqref{eq:condDenseStat} strictly dominates the condition in \eqref{eq:CondSaprse}. 
We conclude this subsection by stating our main findings above.
\begin{enumerate}[leftmargin=*]
    \item For $\Gamma's$ with sub-logarithmic density $\mu(\Gamma_n)=o(\log|v(\Gamma_n)|)$, there are two complementary regimes:
    \begin{itemize}[leftmargin=*]
        \item \emph{The impossible regime:} No test can detect the planted subgraph regardless of the computational complexity.
        \item \emph{The easy/simple regime:} The subgraph can be detected in linear time with high probability by counting the total number of edges, or evaluating the maximum degree in the observed graph.
    \end{itemize}
    \item For $\Gamma's$ with super-logarithmic density $\mu(\Gamma_n)=\Omega(\log|v(\Gamma_n)|)$, there are three complementary regimes:
    \begin{itemize}[leftmargin=*]
        \item \emph{The impossible regime:} No test can detect the planted subgraph regardless of the computational complexity.
        \item \emph{The hard regime:} While strong detection can be achieved by thresholding the maximum number of edges among all maximum subgraph density graphs, no polynomial-time solver exists in this regime (assuming the LDP conjecture).
        \item \emph{The easy/simple regime:} The subgraph can be detected in linear time with high probability by counting the total number of edges, or evaluating the maximum degree in the observed graph.
    \end{itemize}
\end{enumerate}

\subsection{Other regimes}

As mentioned in the introduction, our techniques and results apply to any sequences $p = (p_n)_n$ and $q = (q_n)_n$ of edge probabilities. While general bounds can be found in Section~\ref{sec:statlowerbound}, we focus here on the following two regimes to keep the exposition simple, as they have received the most attention in the literature: the \emph{sparse regime}, where $\chi^2(p||q)=\Theta(n^{-\alpha})$, for $0\leq \alpha\leq 2$, and the \emph{critical regime}, where $p=1-o(1)$ and $q = \Theta(n^{-\alpha})$, for $0\leq \alpha\leq 2$ (in particular, $\chi^2(p||q)=\Theta(n^{\alpha})$). 

\paragraph{Sparse regime.} In this regime, we will mostly be concerned with the polynomial growth/decay of each of the parameters in the problem. Accordingly, let $\Gamma=(\Gamma_n) $ be a sequence of graphs such that $|v(\Gamma)|=\Theta(n^\beta)$, for $0<\beta<1$, and 
\begin{align}
    \lim_{n\to\infty} \frac{\log|e(\Gamma)|}{\log|v(\Gamma)|}=\epsilon \quad \lim_{n\to\infty} \frac{\log d_{\max}(\Gamma)}{\log|v(\Gamma)|}=\delta  \quad \lim_{n\to\infty} \frac{\log\mu(\Gamma)}{\log|v(\Gamma)|}=\smu.
\end{align}
We have the following result.

\begin{theorem}[Sparse regime]\label{th:StatCompLimitSparse}
Consider the detection problem in \eqref{eqn:super_hypo} and assume that $\chi^2(p||q)=\Theta(n^{-\alpha})$, for $0\leq \alpha\leq 2$.
\begin{enumerate}
    \item Weak detection is impossible if,
\begin{align}
    \beta<\begin{cases}
          \frac{\alpha}{\smu}\wedge\frac{1+\alpha}{2\delta+\smu}\wedge\frac{2+\alpha}{2\epsilon} & 0\leq \smu<1\\
          \frac{\alpha}{\smu}\wedge\frac{1+\alpha}{2\delta}\wedge\frac{2+\alpha}{2\epsilon} & \smu=1,
    \end{cases}
\end{align}
where $\alpha/0\triangleq \infty$, while strong detection is possible if,
\begin{align}
    \beta>\frac{\alpha}{\smu}\wedge\frac{1+\alpha}{2\delta}\wedge\frac{2+\alpha}{2\epsilon}.
\end{align}
\item  Under the low-degree polynomial conjecture (see, Conjecture~\ref{conj:1} in Section~\ref{sec:complowerbound}), strong detection in polynomial-time is impossible if,
\begin{align}
    \beta<\frac{1+\alpha}{2\delta}\wedge\frac{2+\alpha}{2\epsilon}.\label{eqn:compLDPSparseMain}
\end{align}
\end{enumerate}
\end{theorem}
We observe that our statistical bounds are not tight in general, as $\frac{1 + \alpha}{2\delta + \smu} < \frac{1 + \alpha}{2\delta}$, for $\smu>0$. However, in the extreme cases of \emph{sub-polynomial density} (i.e., $\smu = 0$) and \emph{maximal-polynomial density} (i.e., $\smu = 1$), our statistical bounds coincide. We prove in Section~\ref{subsec:SparseRegimeProofs} that this is also the case whenever $\epsilon > 2\delta + \smu/2$, in which case the maximum degree test is redundant (and the scan and count tests dominate). The computational lower bound in \eqref{eqn:compLDPSparseMain}, on the other hand, is tight and complements the performance of the count and maximum degree tests.
 %(which are efficient), achieves detection provided that 
%\begin{align}
%    \beta>\min\p{ \frac{1+\alpha}{2\delta},  \frac{2+\alpha}{2\epsilon}}.
%\end{align}
Furthermore, Theorem~\ref{th:StatCompLimitSparse} shows that a statistical-computational gap emerge iff $\Gamma$ has positive polynomial density, i.e., $\smu>0$.

\paragraph{Critical regime.} Assume that $p=1-o(1)$ and $q=\Theta(n^{-\alpha})$, where $0< \alpha<2$ is fixed, and let $\Gamma=(\Gamma_n)$ be a sequence of graphs. As we prove in Section~\ref{sec:CriticalRegime}, strong detection is possible provided that $\mu>\frac{1}{\alpha}$. Henceforth, we will only consider sequences whose maximal density is bounded by $\frac{1}{\alpha}$. Surprisingly, even when $\mu(\Gamma)$ is assumed to be bounded, it seems that a rich and complicated phenomena appear to take place. Specifically, it is shown in Section~\ref{sec:CriticalRegime} that the  behavior of the statistical limits vary dramatically between three main regimes: $\mu(\Gamma)>1$, $\mu(\Gamma)=1-o(1)$, and $\mu(\Gamma)\leq 1-\delta$, for a fixed $\delta>0$. Due to space limitation, we will focus on the case $\mu(\Gamma)=1-o(1)$, where we observe a sharp phase transition phenomena whenever $\alpha=1$. We remark that the reduction argument to the balanced case (used in the dense and sparse regimes) do not generalize to the critical regime (see Section~\ref{sec:CriticalRegime} for a detailed explanation). We will therefore focus on the $\vcd$-balanced scenario (which captures all specific cases studied in the literature thus far). We also remark in the case where $\Gamma$ has bounded degree (and a sharp phase transition occurs), the balanceness assumption is satisfied automatically. Our results for the regime $\mu(\Gamma)=1-o(1)$ are summarized as follows:

\begin{theorem}\label{thm:phaseTansitionsCrit1}
Let $\Gamma=(\Gamma_n)_n$ be a $\vcd$-balanced sequence of  graphs such that $1-o(1)\leq \mu(\Gamma_n)<1$ for some $o(1)$ function.
\begin{enumerate}[leftmargin=*]
    \item $\underline{0<\alpha<1}$: For every $\varepsilon>0$, weak detection is impossible if,
    \begin{align}
        |e(\Gamma)|\leq n^{1-\frac{\alpha}{2}-\varepsilon} \quad \text{and} \quad d_{\max}(\Gamma)\leq n^{\frac{1-\alpha}{2}-\varepsilon},\label{eq:CriticalCond1m}
    \end{align}
   while strong detection is possible if 
    \begin{align}
        |e(\Gamma)|\geq n^{1-\frac{\alpha}{2}+\varepsilon} \quad \text{or} \quad d_{\max}(\Gamma)\geq n^{\frac{1-\alpha}{2}+\varepsilon}.\label{eq:CriticalCond1.1m}
    \end{align}

    \item $\underline{\alpha=1}$: Assume that $q=\frac{\sigma}{n}$, for some $\sigma>0$.
    \begin{enumerate}
        \item \underline{Polynomial maximum degree:} If $d_{\max}(\Gamma)=\Omega\p{|v(\Gamma)|^\beta}$, for $0<\beta\leq 1$, then, for any $\varepsilon>0$, weak detection is impossible if, 
        \begin{align}
            |e(\Gamma)|\leq n^{\frac{1}{2}-\varepsilon} \quad \text{and} \quad  d_{\max}^{\frac{1}{\beta}}(\Gamma)\cdot \log(d_{\max}(\Gamma))\leq \frac{1-\varepsilon}{2} \log n, \label{eq:CriticalCond2m}
        \end{align}
        while strong detection is possible if, 
        \begin{align}
             |e(\Gamma)|\geq n^{\frac{1}{2}+\varepsilon} \quad \text{or} \quad  d_{\max}(\Gamma)\geq (16+\varepsilon)\log n.
        \end{align}
        \item \underline{Bounded maximum degree:} If $d_{\max}(\Gamma)=O(1)$, then, there exists $\overline{\sigma}_d$ and $\underline{\sigma}_d$, depending on $d_{\max}(\Gamma)$ such that 
        \begin{enumerate}
            \item \underline{If $\sigma > \overline{\sigma}_d$:}
             weak detection is impossible for all $\varepsilon>0$
        \begin{align}
          |e(\Gamma)|\leq n^{\frac{1}{2}-\varepsilon},
        \end{align}
        while strong detection is possible if 
        \begin{align}
           |e(\Gamma)|\geq n^{\frac{1}{2}+\varepsilon}.
        \end{align}
        
            \item \underline{If $ \sigma <\underline{\sigma}_d$} and $\mu(\Gamma)\geq 1-|v(\Gamma)|^{-\beta}$, for $0<\beta\leq 1$,
            then weak detection is impossible if, 
        \begin{align}
         |v(\Gamma)|\leq \frac{\log\p{\frac{e d^2}{\sigma}}}{1+\varepsilon}\cdot \log n,\label{eq:CriticalCond4.1m}
        \end{align}
        while strong detection is possible if, 
        \begin{align}
            |v(\Gamma)|\geq (1+\varepsilon)\cdot\p{\log n}^{\frac{1}{\beta}}.\label{eq:CriticalCond4.2m}
        \end{align}
        \end{enumerate}
    \end{enumerate}
\end{enumerate}
\end{theorem}

To put the results of Theorem~\ref{thm:phaseTansitionsCrit1} perspective, we consider the results of \cite{massoulie19a} concerning the case where $\Gamma$ is a regular tree, and $q=\sigma/n$. The results of \cite{massoulie19a} focuses on three main examples: a star (with increasing degrees), a $D$-regular tree and a path. In the polynomial maximum degree regime, which includes the case of a planted star, our bounds suggest that the statistical limits are determined by the number of edges (as compared to $\sqrt{n}$), and the maximum degree (as compared to poly-logarithmic function of $n$). For the family of bounded degree graphs, which includes the cases of paths and regular trees, the above results shows that the sharp phase transition phenomena observed in 
\cite{massoulie19a} are in fact general. Indeed, 
for $\sigma<\overline{\sigma}_d$, the barrier for detection is determined by comparing $|v(\Gamma)|$ to a poly-logarithmic function of $n$, while if $\sigma>\overline{\sigma}_d$, the barrier for detection is governed by the count test, namely, $|v(\Gamma)|=\omega(\sqrt{n})$.

\section{Upper Bounds}\label{sec:upperbound} 

In this section, we show algorithmic upper bounds for the detection problem in \eqref{eqn:super_hypo} using three simple test statistics. Below, we let $\Gamma_{\max}$ be a subgraph that achieves the maximum in the definition of $\mu{\left(\Gamma\right)}$, and then $\calS_{\Gamma_{\max}}$ is the set of all possible copies of $\Gamma_{\max}$ in $\calK_n$. Given the adjacency matrix $\s{A}_n\in\{0,1\}^{n\times n}$, define,
\begin{align}
    T_{\s{count}}(\s{A}_n)&\triangleq \sum_{i<j}\s{A}_{ij},\label{eqn:countstat}\\
    T_{\s{deg}}(\s{A}_n)&\triangleq \max_{i\in[n]}\sum_{j\in[n]} \s{A}_{ij},\label{eqn:degstat}\\
    T_{\s{scan}}(\s{A}_n)&\triangleq \max_{\bar{\Gamma} \in {\calS_{\Gamma_{\max}}}}\sum_{(i,j)\in\bar{\Gamma}}\s{A}_{ij}.\label{eqn:scanstat}
\end{align}
\sloppy
Accordingly, the corresponding tests are $\phi_{\s{count}}\triangleq\Ind\ppp{T_{\s{count}}(\s{A}_n)\geq \tau_{\mathsf{count}}}$, $\phi_{\s{deg}}\triangleq\Ind\ppp{T_{\s{deg}}(\s{A}_n)\geq \tau_{\mathsf{deg}}}$, and $\phi_{\s{scan}}\triangleq\Ind\ppp{T_{\s{scan}}(\s{A}_n)\geq \tau_{\mathsf{scan}}}$, where $\tau_{\mathsf{count}},\tau_{\mathsf{deg}},\tau_{\mathsf{scan}}\in\mathbb{R}_+$, are thresholds that will be specified later. These tests---or variants thereof\footnote{Most notably, the maximization in the scan statistics in \eqref{eqn:scanstat} is taken over all copies of $\Gamma_{\max}$ (the densest subgraph of $\Gamma$), rather than over all copies of $\Gamma$ (the actual planted subgraph), as was done for all structures studied in the literature. Somewhat surprisingly, however, the latter approach is suboptimal for general structures.}---are well-established and have been applied to the detection of specific planted subgraphs studied in the literature, as well as in various related contexts, such as in \cite{kolar2011minimax, butucea2013detection, ma2015computational, arias2014community, Brennan2018, Hajek2015, Huleihel2022, pmlr-v99-brennan19a, HuleihelBip}. Additionally, the count and degree tests are computationally efficient, with polynomial time complexity on the order of $O(n^2)$. In contrast, the scan test has exponential computational complexity, making it inefficient. Specifically, the search space in \eqref{eqn:scanstat} becomes at least quasi-polynomial when $v(\Gamma_n) = \omega(1)$. The following result establishes sufficient conditions under which the risk associated with each of these tests remains small. 

\begin{theorem}[Algorithmic upper bounds]\label{thm:upperBoundAlgo}
    Consider the detection problem in \eqref{eqn:super_hypo}, and the statistics in \eqref{eqn:countstat}--\eqref{eqn:scanstat}. We have the following set of results.
    \begin{enumerate}
        \item \underline{Count test}: Let $\tau_{\s{count}}\triangleq\binom{n}{2}q + |e{\left(\Gamma\right)}| \frac{p-q}{2}$. Then, $\s{R}_n(\phi_{\s{count}})\to 0$, provided that,\footnote{Note that $(p-q)\cdot|e(\Gamma)|\to\infty$ is essential, as otherwise, detection is statistically impossible (see, Section~\ref{app:simpleLowerBound} for more details).} 
\begin{align}
\lim_{n\to\infty}\pp{\frac{\chi^2(p||q)|e(\Gamma)|^2}{n^2}\wedge(p-q)\cdot|e(\Gamma)|}=\infty.\label{eqn:countCondUpper}
\end{align}
        \item \underline{Degree test}: Let $\tau_{\s{deg}}=\left(n-1\right)q + d_{\max}\left(\Gamma\right) \frac{p-q}{2}$. Then, $\s{R}_n(\phi_{\s{deg}})\to 0$, provided that,
        \begin{align}
            \liminf_{n\to\infty}\pp{\frac{d^2_{\max}(\Gamma)\chi^2(p||q)}{n\log n}\wedge\frac{d_{\max}(\Gamma)(p-q)}{\log n}}>16.\label{eq:MaxDegGeneral}
        \end{align}
        \item \underline{Scan test}: Let $\tau_{\s{scan}}= \kappa\cdot |e(\Gamma_{\max})|$, where $\kappa\in(q,p)$. Then, $\s{R}_n(\phi_{\s{scan}})\to 0$, provided that $p\cdot |e(\Gamma_{\max})|\to\infty$,\footnote{Note that the condition $p\cdot |e(\Gamma_{\max})|\to\infty$ is essential, as otherwise, with positive probability the planted subgraph does not contain any edge.} and,
\begin{align}
    \liminf_{n\to\infty}\frac{\mu(\Gamma)d_{\s{KL}}(p||q)}{\log n}>1.
\end{align}
    % \item \underline{Counting connected components test}: Let $\tau_{\s{cc}} = $ 
    \end{enumerate}
\end{theorem}

\begin{proof}[Proof of Theorem~\ref{thm:upperBoundAlgo}] To prove the theorem, we will upper bound the risk associated with each one of the tests we proposed, starting with the count test.

\paragraph{Count test.} We start with the Type-I error probability. Under $\calH_0$, we note that $T_{\s{count}}(\s{A}_n)\sim\s{Binomial}(\binom{n}{2},q)$. By Chebyshev's inequality,
\begin{align}
    \pr_{\calH_0}\pp{\phi_{\s{count}}(\s{A}_n)=1} &\leq \pr_{\calH_0}\pp{\abs{T_{\s{count}}(\s{A}_n)-\bE_{\calH_0}[T_{\s{count}}(\s{A}_n)]}\geq |e{\left(\Gamma\right)}| \frac{p-q}{2}}\\
    &\leq \frac{4\binom{n}{2}q(1-q)}{|e{\left(\Gamma\right)}|^2(p-q)^2}\\
    & \leq 2\frac{n^2}{|e{\left(\Gamma\right)}|^2\chi^2(p||q)}.\label{eqn:CountTypeI}
\end{align}
Under $\calH_1$, conditioned on the random draw of $\Gamma$, we note that $T_{\s{count}}(\s{A}_n)$ is merely the
independent sum of $\mathsf{Binomial}\p{|e(\Gamma)|,p}$ and $\mathsf{Binomial}\p{\binom{n}{2}-|e(\Gamma)|,q}$. Applying Chebyshev's inequality once again,
\begin{align}
    \pr_{\calH_1}\pp{\phi_{\s{count}}(\s{A}_n)=0} &= \bE_{\Gamma}\pr_{\calH_1\vert\Gamma}\pp{\phi_{\s{count}}(\s{A}_n)=0}\\
    &= \bE_{\Gamma}\pr_{\calH_1\vert\Gamma}\pp{T_{\s{count}}(\s{A}_n)-\bE_{\calH_1}[T_{\s{count}}(\s{A}_n)]\leq-|e{\left(\Gamma\right)}| \frac{p-q}{2}}\\
    &\leq \bE_{\Gamma}\pr_{\calH_1\vert\Gamma}\pp{\abs{T_{\s{count}}(\s{A}_n)-\bE_{\calH_1}[T_{\s{count}}(\s{A}_n)]}\geq|e{\left(\Gamma\right)}| \frac{p-q}{2}}\\
    &\leq \frac{4\cdot\pp{|e(\Gamma)|[p(1-p)-q(1-q)]+\binom{n}{2}q(1-q)}}{|e{\left(\Gamma\right)}|^2(p-q)^2}\\
    &\leq \frac{4\cdot\pp{|e(\Gamma)|(p-q)+\binom{n}{2}q(1-q)}}{|e{\left(\Gamma\right)}|^2(p-q)^2}\\
    &\leq\frac{4}{|e(\Gamma)|(p-q)}+2\frac{n^2}{|e{\left(\Gamma\right)}|^2\chi^2(p||q)}.\label{eqn:CountTypeII}
\end{align}
Therefore, based on \eqref{eqn:CountTypeI} and \eqref{eqn:CountTypeII}, we have $\s{R}_n(\phi_{\s{count}})\to0$, as $n\to\infty$, if \eqref{eqn:countCondUpper} holds.

\paragraph{Degree test.} Let $\s{W}_i(\s{A}_n)\triangleq \sum_{j\in[n]} A_{ij}$, for $i\in[n]$. Then, under $\calH_0$, it is clear that $\s{W}_i(\s{A}_n)\sim \mathsf{Binomial}{\left(n-1,q\right)}$. Therefore, by the union bound and Bernstein's inequality, 
\begin{align}
\pr_{\calH_0}\p{\phi_{\s{deg}}(\s{A})=1} &=\pr_{\calH_0}\p{\max_{i\in[n]}\s{W}_i(\s{A}_n)\geq \tau_{\s{deg}}}\\
    &\leq n \cdot \exp{\left(-
    \frac{d^2_{\max}\left(\Gamma\right) \cdot {\left(p-q\right)}^2/4}{2\left(n-1\right)q(1-q)+d_{\max}\left(\Gamma\right)\cdot{\left(p-q\right)}/3}\right)}    \\
    &\leq \exp{\left(\log{n}-
    \frac{d^2_{\max}\left(\Gamma\right) \cdot {\left(p-q\right)}^2/8}{\left(n-1\right)q(1-q)+d_{\max}\left(\Gamma\right)\cdot{\left(p-q\right)}}\right)}.\label{eqn:degreeTypeI}
\end{align}
Under $\calH_1$, by definition there is at least one row $i^{\star}$ such that $\sum_{j\in[n]} \s{A}_{i^{\star} j}$ is distributed as the independent sum of $\mathsf{Binomial}{\left(d_{\max}\left(\Gamma\right), p\right)}$ and $\mathsf{Binomial}{\left(n-1-d_{\max}\left(\Gamma\right), q\right)}$.  Therefore, by the multiplicative Chernoff's bound, we get,
\begin{align}
    \pr_{\calH_1}\p{\phi_{\s{deg}}(\s{A})=0} &=\pr_{\calH_1}\p{\max_{i\in[n]}\s{W}_i(\s{A}_n)< \tau_{\s{deg}}}\\
    &=\sum_{\Gamma_0\in\calS_{\Gamma}}\pr_{\calH_1}\p{\left.\max_{i\in[n]}\s{W}_i(\s{A}_n)< \tau_{\s{deg}} \right| \Gamma=\Gamma_0}\cdot \pr\p{\Gamma=\Gamma_0}\\
    &\leq \sum_{\Gamma_0\in\calS_{\Gamma}}\pr_{\calH_1} \p{\left.\s{W}_{i^\star(\Gamma_0)}(\s{A}_n)< \tau_{\s{deg}} \right| \Gamma=\Gamma_0}\cdot \pr\p{\Gamma=\Gamma_0}\\
    &\leq \sum_{\Gamma_0\in\calS_{\Gamma}}\exp{\left(-
    \frac{ d^2_{\max}\left(\Gamma\right) \cdot {\left(p-q\right)}^2/4}{2\left(n-1\right)q+2\cdot d_{\max}\left(\Gamma\right)\cdot{\left(p-q\right)}}\right)}\pr\p{\Gamma=\Gamma_0}\\
    &\leq \exp{\left(-
    \frac{ d^2_{\max}\left(\Gamma\right) \cdot {\left(p-q\right)}^2/8}{\left(n-1\right)q+ d_{\max}\left(\Gamma\right)\cdot{\left(p-q\right)}}\right)}.\label{eqn:degreeTypeII}
\end{align}
Examining \eqref{eqn:degreeTypeI} and \eqref{eqn:degreeTypeII}, we see that the risk is dominated by the former, and as so, vanishes whenever \eqref{eqn:degreeTypeI} converges to zero, as $n\to\infty$. A sufficient condition for this is that,
\begin{align}
    \liminf_{n\to\infty}\pp{\frac{d^2_{\max}(\Gamma)\chi^2(p||q)}{n\log n}\wedge\frac{d_{\max}(\Gamma)(p-q)}{\log n}}>16.\label{eqn:Low16}
\end{align}
\begin{rmk}
By optimizing the threshold $\tau_{\s{deg}}$, the factor at the right-hand-side of \eqref{eqn:Low16} can be reduced to $2$.
\end{rmk}

\paragraph{Scan test.} As before, we start with the analysis of the Type-I error probability. For any $\bar{\Gamma}\in\calS_{\Gamma_{\max}}$, let $\s{W}(\bar{\Gamma})\triangleq\sum_{(i,j)\in\bar{\Gamma}}\s{A}_{ij}$. Then, under $\calH_0$, we have $\s{W}(\bar{\Gamma})\sim \mathsf{Binomial}(|e(\bar{\Gamma}|), q)$. Thus, by the union bound and classical tail probabilities \cite{arratia1989tutorial},
\begin{align}
\pr_{\calH_0}\p{\phi_{\s{scan}}(\s{A})=1} &= \pr_{\calH_0}\pp{\max_{\bar{\Gamma}\in\calS_{\Gamma_{\max}}}\s{W}(\bar{\Gamma})\geq\kappa\cdot |e(\Gamma_{\max})|}\\
&\leq \sum_{\bar{\Gamma}\in\calS_{\Gamma_{\max}}}\pr_{\calH_0}\pp{\s{W}(\bar{\Gamma})\geq\kappa\cdot |e(\Gamma_{\max})|}    \\
    &\leq \binom{n}{|v(\Gamma_{\max})|}\cdot\frac{|v(\Gamma_{\max})| !}{\vert\mathsf{Aut}{\left(\bar{\Gamma}\right)}\vert}\cdot e^{-|e(\Gamma_{\max})|\cdot d_{\s{KL}}(\kappa||q)} \\
    &\leq \exp\p{|v(\Gamma_{\max})| \cdot \log{n}-|e(\Gamma_{\max})|\cdot d_{\s{KL}}(\kappa||q)},\\
    &=\exp\p{-|v(\Gamma_{\max})| \cdot \log{n}\pp{\frac{\mu(\Gamma)d_{\s{KL}}(\kappa||q)}{\log n}-1}},\\
    &=n^{-|v(\Gamma_{\max})|\cdot\pp{\frac{\mu(\Gamma)d_{\s{KL}}(\kappa||q)}{\log n}-1}},
\end{align}
where in the above we assume that $\kappa\geq q$. Thus, as long as $|e(\Gamma_{\max})|\cdot d_{\s{KL}}(\kappa||q)-|v(\Gamma_{\max})| \cdot \log{n}\to\infty$, or, equivalently, 
\begin{align}
    \liminf_{n\to\infty}\frac{\mu(\Gamma)d_{\s{KL}}(\kappa||q)}{\log n}>1,
\end{align}
the Type-I error probability will converge to zero. Next, under $\calH_1$, there exists a subgraph $\bar{\Gamma}^\star\in\calS_{\Gamma_{\max}}$, such that $\sum_{(i,j)\in\bar{\Gamma}^\star}\s{A}_{ij} \sim \mathsf{Binomial}{\left(|e{\left(\bar{\Gamma}\right)|}, p\right)}$. Therefore, by Chebyshev's inequality, for $\kappa<p$,
\begin{align}
\pr_{\calH_1}\p{\phi_{\s{scan}}(\s{A})=0} &= \sum_{\Gamma_0\in\calS_{\Gamma}}\pr_{\calH_1}\pp{\left.\max_{\bar{\Gamma}\in\calS_{\Gamma_{\max}}}\s{W}(\bar{\Gamma})<\kappa\cdot |e(\Gamma_{\max})| \right| \Gamma=\Gamma_0}\cdot \pr\p{\Gamma=\Gamma_0}\\
&\leq \sum_{\Gamma_0\in\calS_{\Gamma}}\pr_{\calH_1}\pp{\left.\s{W}(\bar{\Gamma}^\star(\Gamma_0))<\kappa\cdot |e(\Gamma_{\max})|\right| \Gamma=\Gamma_0}\cdot \pr\p{\Gamma=\Gamma_0}\\
&\leq \sum_{\Gamma_0\in\calS_{\Gamma}}\frac{|e(\Gamma_{\max})|p(1-p)}{(p-\kappa)^2|e(\Gamma_{\max})|^2}\cdot  \pr\p{\Gamma=\Gamma_0}   \\
    & = \frac{p(1-p)}{(p-\kappa)^2|e(\Gamma_{\max})|}.
\end{align} 
Consider the convex combination $\kappa = \epsilon q+(1-\epsilon)p$. Note that for any $\epsilon\in(0,1]$ we have $q\leq\kappa<p$, as required above. For this choice,
\begin{align}
\pr_{\calH_1}\p{\phi_{\s{scan}}(\s{A})=0} &\leq \frac{p(1-p)}{\epsilon^2(p-q)^2|e(\Gamma_{\max})|}.
\end{align} 
Thus, our conditions for the Type-I and II error probabilities to converge to zero are
\begin{align}
    \liminf_{n\to\infty}\frac{\mu(\Gamma)d_{\s{KL}}(\kappa||q)}{\log n}&>1,\label{eqn:cond1Scan}\\
    \frac{p(1-p)}{\epsilon^2(p-q)^2|e(\Gamma_{\max})|}&\to0.\label{eqn:cond2Scan}
\end{align}
These conditions can be simplified easily as follows. Recall that we assume that $p>q$, and consider the case where $\frac{p}{q}>c>1$. In this case, the condition in \eqref{eqn:cond2Scan} simplifies to,
\begin{align}
    \frac{(1-p)}{\epsilon^2p(p/q-1)^2|e(\Gamma_{\max})|}\leq \frac{1}{\epsilon^2p(c-1)^2|e(\Gamma_{\max})|},
\end{align}
which converges to zero provided that $p\cdot |e(\Gamma_{\max})|=\omega(1)$, for any $\epsilon>0$. Taking $\epsilon$ sufficiently small, the scan test is successful provided that
\begin{align}
    \liminf_{n\to\infty}\frac{\mu(\Gamma)d_{\s{KL}}(p||q)}{\log n}>1.
\end{align}
Now, consider the case where $\frac{p}{q}$ converges to unity. In this case, we have,
\begin{align}
    d_{\s{KL}}(p||q) &= p\log \frac{p}{q}+(1-p)\log\frac{1-p}{1-q}\\
    & = p\cdot\log \p{1+\frac{p}{q}-1}+(1-p)\log\pp{1-\frac{p-q}{1-q}}\\
    & = p\cdot\pp{\frac{p}{q}-1+O\p{\frac{p}{q}-1}^2}-(1-p)\cdot\pp{\frac{p-q}{1-q}+O\p{\frac{p-q}{1-q}}^2}\\
    & = \frac{(p-q)^2}{q(1-q)}+O\p{\frac{(p-q)^3}{q^2}}\\
    & = \chi^2(p||q)+O\p{\frac{(p-q)^3}{q^2}}.
\end{align}
In the regime $\frac{p}{q}\to1$, it also holds that $\frac{\kappa}{q}\to1$, and thus, \eqref{eqn:cond1Scan} boils down to,
\begin{align}
    \mu(\Gamma)\chi^2(\kappa||q)>\log n,
\end{align}
as $n\to\infty$, which in light of the fact that $\mu(\Gamma)<|e(\Gamma_{\max})|$, implies that \eqref{eqn:cond2Scan} is satisfied. Therefore, only \eqref{eqn:cond1Scan} prevails, which concludes the proof.
    
\end{proof}

\section{Information-Theoretic Lower Bounds}\label{sec:statlowerbound}
This section is dedicated to establishing and proving general information-theoretic (or statistical) lower bounds on the risk of the detection problem stated in Section~\ref{sec:problem}. For the sake of clarity and to transparently present the main ideas and steps of our proof, we begin with an overview of the key stages involved in our strategy. The missing details and proofs of the statements below are provided in Sections~\ref{sec:Polynomial} and \ref{subsec:denseRegime}. While our derivations are essentially independent of the specific scaling of $p$ and $q$ with $n$, to keep the exposition simple, we first consider the dense regime where $\lambda^2\triangleq\chi^2(p||q)=\Theta(1)$ in Subsection~\ref{subsec:denseRegime}, and then generalize our techniques and results in Subsection~\ref{subsec:otherRegimes}. 

\subsection{Lower bounds proof outline}\label{subsec:outline}

In this subsection, we outline the main steps and ideas behind our proofs. As mentioned above, to simplify the exposition, we focus our discussion on the dense regime, with the understanding that our arguments can be generalized to any value of $p$ and $q$. Generally speaking, our proofs hinges on the following four main steps.

\paragraph{1. Analysis via polynomial decomposition.} Our analysis begins with a well known argument, which suggests that the possibility of detection can be ruled out by showing that the second moment of the likelihood ratio, i.e., $\E_{\calH_0}[\s{L}^2(\s{G})]$, is bounded. It is a standard result that this second moment is given by,
\begin{align}
             \E_{\calH_0}\pp{\s{L}^2(\s{G})}=\E_{\Gamma\sim \s{Unif}(\calS_\Gamma)}\pp{(1+\lambda^2)^{|e(\Gamma\cap \Gamma')|}}.\label{eq:SecondMomentExpressionFolk}
         \end{align}
where $\Gamma'$ is a fixed (arbitrary) copy of $\Gamma$ in $\calK_n$, the probability is taken w.r.t. to a random copy $\Gamma\sim \s{Unif}(\calS_\Gamma)$, $\lambda^2$ denotes $\chi^2(p||q)$, and $|e(\Gamma\cap\Gamma')|$ is the number of edges in the intersection of $\Gamma$ and $\Gamma'$. The expectation term in the r.h.s. of \eqref{eq:SecondMomentExpressionFolk} is often a starting point in the analysis of virtually any planted structure considered in the literature. As it turns out, the following equivalent representation, which results from projecting the likelihood function onto its orthogonal polynomial components, proves to be quite useful,
\begin{align}
             \E_{\calH_0}\pp{\s{L}(\s{G})^2}=\sum_{\s{H}\subseteq \Gamma'}\lambda^{2|\s{H}|}\cdot\P_{\Gamma\sim \s{Unif}(\calS_\Gamma)}\pp{\s{H}\subseteq \Gamma},\label{eq:SecondMomentExpressioniNT}
         \end{align}
where the summation is taken over subgraphs that do not contain isolated vertices. To best of our knowledge, the expression on the r.h.s. of \eqref{eq:SecondMomentExpressioniNT} has not been previously used to study the second moment and is crucial for establishing our results.% In what follows, we carefully analyze the sum of ~\eqref{eq:SecondMomentExpressioniNT} and derive our upper bounds the second moment. 

\paragraph{2. Vertex covers.} In light of \eqref{eq:SecondMomentExpressioniNT}, to upper bound the second moment, it suffices to bound $\P_{\Gamma}[\s{H}\subseteq \Gamma]$, for any $\s{H}\subseteq\Gamma'$. It straightforward to show that $\P_{\Gamma}[\s{H}\subseteq \Gamma]=\frac{\calN(\s{H},\Gamma)}{|\calS_{\s{H}}|}$, where $\calN(\s{H},\Gamma)$ denotes the number of times $\s{H}$ appears as a
subgraph of $\Gamma$ (see, Definition~\ref{def:CountsSubg}). Thus, a natural approach is to bound this combinatorial quantity directly. As it turns out, this quantity is closely related to the study of upper tails for subgraph counts in random graphs, a topic that has received significant attention over the past four decades, with several existing strong upper bounds, e.g., \cite{Alon1981OnTN,Janson,Friedgut1998OnTN,Jaikumar,Wojciech}, and many reference therein. However, these bounds, even the best-known ones, fail to capture the correct statistical barrier when applied on \eqref{eq:SecondMomentExpressioniNT} (for more details, see Appendix~\ref{app:JansonType}). 

Interestingly, the probability in \eqref{eq:SecondMomentExpressioniNT} can be rather easily upper bounded in terms of four quantities: $|v(\s{H})|$, $d_{\max}(\Gamma)$, the number of connected components of $\s{H}$, which we denote by $m(\s{H})$, and the vertex cover number $\vc(\Gamma)$, i.e., the minimal size of a vertex cover of $\Gamma$ (see, Definition~\ref{def:coverNumber}). Specifically, as we show in Lemma~\ref{lem:probRandomSubgraph1},
\begin{align}
    \P_{\Gamma}\pp{\s{H}\subseteq\Gamma}\leq \frac{[2\vc(\Gamma)]^{m(\s{H})} [d_{\max}(\Gamma)]^{|v(\s{H})|-m(\s{H})}}{(n-|v(\Gamma)|)^{|v(\s{H})|}}\triangleq \vartheta(m(\s{H}),v(\s{H})).\label{eqn:upperBouondOnProb}
\end{align}
Accordingly, when \eqref{eqn:upperBouondOnProb} is applied on \eqref{eq:SecondMomentExpressioniNT}, it becomes evident that the summands depend on $\s{H}$ only through the values of $(|v(\s{H})|,|\s{H}|,m(\s{H}))$. Thus, if we define $\calS_{m,\ell,j}$ as the set of all subgraphs of $\Gamma$ with exactly $m$ connected components, $\ell$ vertices, and $j$ edges, then, loosely speaking,
\begin{align}
    \E_{\calH_0}\pp{\s{L}(\s{G})^2}\leq\sum_{(m,\ell,j)}\lambda^{2j}\vartheta(m,\ell)|\calS_{m,\ell,j}|.
\end{align}
Using certain powerful combinatorial arguments, we were able to derive an upper bound on $|\calS_{m,\ell,j}|$, which ultimately leads to the following impossibility result.
\begin{theorem}\label{th:lowerVCD}
    Let $\Gamma=(\Gamma_n)_n$ be a sequence of graphs. Then $\E_{\calH_0}[\s{L}(\s{G})^2]=O(1)$ (and therefore strong detection is impossible) if
\begin{align}
        \frac{ \p{1+\chi^2(p||q)}^{\mu(\Gamma) }\cdot\max\p{\vc(\Gamma)d_{\max}(\Gamma),d^2_{\max}(\Gamma)}}{n-|v(\Gamma)|}\leq C, \label{eq:LowerMainCond1}
            \end{align}
for some $C>0$.
         \end{theorem}
Replacing the constant $C$ with any $o(1)$ function in \eqref{eq:LowerMainCond1} would also imply the impossibility of weak detection. It is important to emphasize that, as will be seen later, Theorem~\ref{th:lowerVCD} is a special case (tailored to the dense regime) of a more general result (see Theorem~\ref{th:lowerVCFull}), which aims to capture a variety of radically different regimes, in terms of the subgraph edge density and edge probabilities. 

\paragraph{3. Reduction to ``balanced graphs".} Comparing the lower bound in Theorem~\ref{th:lowerVCD} with our upper bounds in Theorem~\ref{thm:upperBoundAlgo}, we observe that they depend on different quantities and, as a result, do not match. Specifically, the vertex cover number $\vc(\Gamma)$ plays a significant role in the lower bound, while the upper bounds do not account for it. An almost immediate observation here is that for the special family of \textit{vertex cover-degree balanced} graphs, for which $\vc(\Gamma)\cdot d_{\max}(\Gamma)\approx |e(\Gamma)|$ (for a formal definition, see Definition~\ref{def:coverdegreebalanced}), our bounds are in fact tight! To demonstrate this, consider the case where $\Gamma$ has sub-logarithmic density, i.e., $\mu(\Gamma)=o\p{\log|v(\Gamma)|}$, in which case the term $\max\p{\vc(\Gamma)d_{\max}(\Gamma),d_{\max}(\Gamma)^2}$ dominates in \eqref{eq:LowerMainCond1}. Therefore, detection is impossible if,
\begin{align}
    \max\p{|e(\Gamma)|^{1+o(1)},d_{\max}(\Gamma)^{2+o(1)}}=\max\p{\vc(\Gamma)d_{\max}(\Gamma),d_{\max}(\Gamma)^2}^{1+o(1)}\leq C\cdot n,
\end{align}
for some constant $C>0$. On the other hand, Theorem~\ref{thm:upperBoundAlgo} shows that the count and maximum degree tests achieve strong detection when
\begin{align}
    \max\p{|e(\Gamma)|,d_{\max}(\Gamma)^2}= \omega( n^{1+o(1)}).
\end{align}
Now, if $\Gamma$ has a super-logarithmic density, i.e., $\mu(\Gamma)=\Omega\p{\log|v(\Gamma)|}$, then $(1+\lambda^2)^{\mu(\Gamma)}$ becomes the dominant term in \eqref{eq:LowerMainCond1}. This implies that detection is impossible if,
\begin{align}
    \mu(\Gamma)\leq \underline{C}\cdot \log n,
\end{align}
while possible, using the scan test, provided that,
\begin{align}
       \mu(\Gamma)\geq \overline{C}\cdot \log n,
\end{align}
for some constants $\overline{C},\underline{C}>0$. Although the family of $\vcd$-balanced graphs is quite diverse and includes virtually all examples studied in the literature (e.g., cliques, complete bi-partite graphs, regular trees, and perfect matching), there are also fairly simple examples of graphs that are not $\vcd$-balanced, and for which the bound in Theorem~\ref{th:lowerVCD} is loose.
\begin{example}\label{ex:StartUnbalanced}
Consider the sequence $\Gamma=(\Gamma_n)_n$, where $\Gamma_n$ is composed of $k=k_n$ disjoint stars with degree $\floor{k^{1/4}}$ and an additional single star with degree $\floor{k^{3/4}}$. Then,
    \begin{align}
        \vc(\Gamma)=k+1, \quad d_{\max}(\Gamma)=\floor{k^{3/4}},\quad |e(\Gamma)|=k\cdot \floor{k^{1/4}}+\floor{k^{3/4}
        }.
    \end{align}
    In particular, if $k=\omega(1)$, then $\Gamma$ is not $\vcd$-balanced because,
    \begin{align}
       \lim_{n\to\infty} \frac{\log\p{ \vc(\Gamma)d_{\max}(\Gamma)}}{\log|e(\Gamma)|}=\frac{7}{5}>1,
    \end{align}
    which implies that there is a gap between the lower bound in Theorem~\ref{th:lowerVCD} and the upper bounds in Theorem~\ref{thm:upperBoundAlgo}. 
\end{example}
As it turns out, however, much more can be proved. To wit, note that the graphs in the example above can be decomposed into two disjoint $\vcd$-balanced graphs $(\Gamma_1, \Gamma_2)$, where $\Gamma_1$ consists of \emph{the} $k$ stars, each with degree $\floor{k^{1/4}}$, and $\Gamma_2$ is the remaining star with degree $\floor{k^{3/4}}$. Since $\Gamma$ has a sub-logarithmic density, Theorem~\ref{thm:upperBoundAlgo} tells us that detection of $\Gamma$ is possible if one of the following conditions holds.
\begin{enumerate}
    \item The count test is successful if $k^{5/4}\approx |e(\Gamma)|\geq n^{1+o(1)}$, \emph{which also serves as a sufficient condition for detecting $\Gamma_1$ if planted alone}.
    \item The degree test is successful if $ k^{3/2}\approx d_{\max}(\Gamma)^2\geq C\cdot n\log n= n^{1+o(1)}$, \emph{which also serves as a sufficient condition for detecting $\Gamma_2$ if planted alone}.
\end{enumerate} 
The example above raises the following natural question: \emph{is it true that whenever $\Gamma$ can be decomposed into, say, $\Gamma_1\cup \Gamma_2$, the conditions under which detection of $\Gamma$ is possible are sufficient for the detection of at least one of its components $\Gamma_1$ or $\Gamma_2$? Or, equivalently, is it true that the conditions under which both $\Gamma_1$ and $\Gamma_2$ are undetectable are sufficient to rule out the possibility of detection for $\Gamma$?}

We prove that the answer to the above question is positive, thereby bridging the gap in Theorem~\ref{th:lowerVCD}. Loosely speaking, we show that any graph $\Gamma$ can be decomposed into, or reduced to, a finite union of edge-disjoint subgraphs, each of which is ``locally" balanced. Consequently, each subgraph can be controlled individually using (a close variant of) Theorem~\ref{th:lowerVCD}. This reduction relies on the following two meta-arguments:
\begin{itemize}
    \item[($\s{Arg}_1$)] Let $\Gamma=\bigcup_{\ell=1}^{M}\Gamma_\ell$ be a decomposition of $\Gamma$ into a sequence of edge-disjoint subgraphs $\{\Gamma_\ell\}_{\ell=1}^{M}$, where $M$ is independent of $n$. If $\Gamma_\ell$ satisfies \eqref{eq:LowerMainCond1} with $\lambda_M^2\triangleq (1+\lambda^2)^{M^2}-1$, \emph{for all} $\ell\in[M]$, then detection of $\Gamma$ is impossible, i.e., $\E_{\calH_0}[\s{L}(\s{G})^2]=O(1)$.
    %We show that for any decomposition of $\Gamma$ to finitely many disjoint subgraphs $\Gamma_1,\cdots,\Gamma_M$ (where $M$ is independent with $n$) such that for all $1\leq i\leq M$, $\Gamma_i$ satisfies (at least) one of the conditions \eqref{eq:LowerMainCond1}-\eqref{eq:LowerMainCond3}   of Theorem~\ref{th:lowerVCD} with $\lambda_M^2\triangleq (1+\lambda^2)^{M^2}-1$, then $\E_{\calH_0}[\s{L}(\s{G})^2]=O(1)$.
    \item[($\s{Arg}_2$)] For any $\varepsilon>0$, there exists a decomposition $\Gamma=\bigcup_{\ell=1}^{M}\Gamma_\ell$, with $M=\ceil{1/\varepsilon}$, such that,
    \begin{align}
        \vc(\Gamma_\ell)\cdot d_{\max}(\Gamma_\ell)\leq |e(\Gamma)|\cdot n^{\varepsilon},\label{eq:decompostioinProperty}
    \end{align}
    for all $\ell\in[M]$.
\end{itemize}
Before delving into each argument, let us first demonstrate that, when combined, they yield a lower bound that matches the upper bounds in Theorem~\ref{thm:upperBoundAlgo}. Specifically, using $(\s{Arg}_1)$, detection is impossible if,
\begin{align}
    \max_{\ell\in[M]}\frac{ \p{1+\chi^2(p||q)}^{\mu(\Gamma_\ell) }\cdot\max\p{\vc(\Gamma_\ell)d_{\max}(\Gamma_\ell),d^2_{\max}(\Gamma_\ell)}}{n-|v(\Gamma_\ell)|}\leq C.\label{eqn:CondOverAll}
\end{align}
Conjugated with $(\s{Arg}_2)$, and considering the facts that $|v(\Gamma_\ell)|\leq|v(\Gamma)|$, $d_{\max}(\Gamma_\ell)\leq d_{\max}(\Gamma)$, and $\mu(\Gamma_\ell)\leq\mu(\Gamma)$, for all $\ell\in[M]$, the condition in \eqref{eqn:CondOverAll} is certainly implied by,
\begin{align}
    \frac{ \p{1+\chi^2(p||q)}^{\mu(\Gamma) }\cdot\max\p{|e(\Gamma)|\cdot n^{\varepsilon},d^2_{\max}(\Gamma)}}{n-|v(\Gamma)|}\leq C,\label{eqn:CondOverAll2}
\end{align}
which, by the arbitrariness of $\varepsilon$, matches our upper bounds asymptotically. Further details can be found in Subsection~\ref{subsec:denseRegime}).  
%To see that, consider for example the regime where $\chi^2(p||q)=\lambda^2=\Theta(1)$, where it turns out that only the condition of \eqref{eq:LowerMainCond1} from Theorem~\ref{th:lowerVCD} is required. Since in the super-logarithmic case the lower bound depend on $\mu(\Gamma)$ solely (as explained in Section~\ref{sec:mainresults}), we assume that $\mu(\Gamma)=o\p{\log|v(\Gamma)|}$. In that case, using $(\text{Arg}_2)$, for a fixed  $\varepsilon>0$, we can find a corresponding disjoint decomposition of $\Gamma$ to $\Gamma_1,\dots,\Gamma_M$ such that \eqref{eq:decompostioinProperty} holds for all $i$. Note that if
%\begin{align}
%    \max\p{|e(\Gamma)|,d_{\max}(\Gamma)^2}\leq  n^{1-2\varepsilon},\label{eq:conditionfoDense}
%\end{align}
%then we clearly have 
%\begin{align}
%    \frac{(1+\lambda^2)^{\mu(\Gamma_i)}d_{\max}(\Gamma_i)}{n-|v(\Gamma_i)|}\max\p{\vc(\Gamma_i),d_{\max}(\Gamma_i)}\leq \frac{(1+\lambda^2)^{\mu(\Gamma)}}{n-|v(\Gamma)|}\max\p{n^{\varepsilon}|e(\Gamma)|,d_{\max}(\Gamma)^2}=o(1),
%\end{align}
%for all $i$. In particular, the conditions of $(\text{Arg}_1)$ are satisfied, implying that detection is impossible. The impossibility condition of \eqref{eq:conditionfoDense}, which is true for any $\varepsilon>0$, is complimentary to the possibility conditions from Theorem~\ref{thm:upperBoundAlgo}. 

Let us briefly walk through the ideas behind $(\s{Arg}_1)$ and $(\s{Arg}_2)$, starting with the former. For any given decomposition $\Gamma=\bigcup_{\ell=1}^M\Gamma_\ell$, applying H\"{o}lder's inequality on \eqref{eq:SecondMomentExpressionFolk}, we get, 
\begin{align}
    \E_{\calH_0}\pp{\s{L}(\s{G})^2}   &=\E_{\Gamma\sim\s{Unif}(\calS_{\Gamma})}\pp{\prod_{i,j=1}^M(1+\lambda^2)^{ |e(\Gamma_i\cap \Gamma'_j)|}}\\
      &\leq\prod_{i,j=1}^M\E_{\Gamma_i\sim\s{Unif}(\calS_{\Gamma_i})}\pp{(1+\lambda^2_M)^{ |e(\Gamma_i\cap \Gamma'_j)|}}^{\frac{1}{M^2}},\label{eq:GammaIeq2Outline}
    \end{align}
    where $\Gamma'_j$ is a fixed copy of $\Gamma_j$ in $\calK_n$, for all $j\in[M]$, induced by the fixed copy $\Gamma'$ of $\Gamma$. Equipped with \eqref{eq:GammaIeq2Outline}, our goal is to analyze the moment-generating function on the r.h.s. of, \eqref{eq:GammaIeq2Outline}; it can be shown that,
    \begin{align}
        \E_{\Gamma_i}\pp{(1+\lambda^2_M)^{ |e(\Gamma_i\cap \Gamma'_j)|}}=\sum_{\s{H}\subseteq \Gamma_j'}\lambda^{2|\s{H}|}\cdot \P_{\Gamma_i} \pp{\s{H} \subseteq \Gamma_i},\label{eq:RunTheSummation}
    \end{align}
    which results in formula similar to the one in \eqref{eq:SecondMomentExpressioniNT}. Consequently, by applying the same ideas as in the proof of Theorem~\ref{th:lowerVCD}, it can be shown that \eqref{eq:RunTheSummation} is bounded if, roughly speaking,
    \begin{align}
        \frac{\p{1+\chi^2(p||q)}^{\mu(\Gamma_i)\wedge\mu(\Gamma_j)}\cdot \max\p{\sqrt{ \tau(\Gamma_i) \tau(\Gamma_j) d_{\max}(\Gamma_i) d_{\max}(\Gamma_j)},d_{\max}(\Gamma_i) d_{\max}(\Gamma_j)}}{n-|v(\Gamma_i)|\wedge|v(\Gamma_j)|}\leq C.\label{eq:CrossingTerms}
    \end{align}   
    Accordingly, for \eqref{eq:GammaIeq2Outline} to be bounded, it is necessary for \eqref{eq:CrossingTerms} to hold for all $i,j\in[M]$. Nonetheless, it is not difficult to verify that this condition is satisfied even if \eqref{eq:CrossingTerms} holds only for $j=i$, and all $i\in[M]$. In other words, each $\Gamma_\ell$ satisfies \eqref{eq:LowerMainCond1} for all $\ell\in[M]$, as argued in $(\s{Arg}_1)$. A detailed analysis of this step is provided in Subsection~\ref{sec:ReudctionToBalanced}. Finally, it remains to establish the existence of a decomposition that satisfies \eqref{eq:decompostioinProperty}, as outlined in $(\s{Arg}_2)$.
    
    \paragraph{4. Balanced decomposition.} We now construct the sequence $\{\Gamma_\ell\}_{\ell=1}^M$ recursively. Let $v_1,\dots,v_{|v(\Gamma)|}$ represent the vertices of $\Gamma$, ordered in descending order according to their degrees in $\Gamma$, and fix $M\in\mathbb{N}$. For each $i\in[M]$, define $\ell_i\in[|v(\Gamma)|]$ as the maximal index such that $\deg(v_{\ell_i})\geq [d_{\max}(\Gamma)]^{\frac{M-i}{M}}$, and set $\ell_0\triangleq0$. Our construction proceeds as follows: We first define $\Gamma_1$ as the subgraph of $\Gamma$ obtained by taking all edges that intersect with any vertex in $\{v_1,\dots,v_{\ell_1}\}$, including all of their neighboring vertices. Next, given $\{\Gamma_1,\dots,\Gamma_{i-1}\}$, we define $\Gamma_{i}$ as the subgraph of $\Gamma$ that includes all edges intersecting with any vertex in the set $\{v_{\ell_{i-1}+1},\dots,v_{\ell_{i}}\}$, excluding any edges already included in the previous subgraphs $\{\Gamma_1\dots,\Gamma_{i-1}\}$. By construction, it is evident that for all $i\in[M]$,
    \begin{align}
         d_{\max}(\Gamma_i)\leq [d_{\max}(\Gamma)]^{\frac{M-i+1}{M}},
    \end{align}
    and,
    \begin{align}
        [\ell_{i}-(\ell_{i-1}+1)]\cdot d_{\max}^{\frac{M-i}{M}}(\Gamma)&\leq \sum_{\ell=\ell_{i-1}+1}^{\ell_{i}} \deg(v_\ell)\leq 2|e(\Gamma)|.
    \end{align}
    Now, by construction $\{v_{\ell_{i-1}+1},\dots,v_{\ell_{i}}\}$ is a vertex cover of $\Gamma_i$, and therefore,
    \begin{align}
        \tau(\Gamma_i)\cdot d_{\max}(\Gamma_{i})&\leq [\ell_{i}-(\ell_{i-1}+1)]\cdot d_{\max}^{\frac{M-i+1}{M}}(\Gamma)\leq 2|e(\Gamma)|\cdot d_{\max}^{\frac{1}{M}}(\Gamma) \leq |e(\Gamma)|\cdot n^{\varepsilon},\label{eq:GoodDecomposition}
    \end{align}
    as argued in $(\s{Arg}_2)$. Complete formal derivations appear in Subsection~\ref{subsec:coverBalancedDecomposition}.

\subsection{Preliminaries: Analysis via polynomial decomposition}\label{sec:Polynomial}
Our goal is to establish a lower bound on the optimal risk from below, thereby ruling out the possibility of successful detection. We begin by recalling the likelihood ratio for our problem, 
\begin{equation}
    \s{L}(\s{G})\triangleq\frac{\mathrm{d}\P_{\calH_1}}{\mathrm{d}\P_{\calH_0}}(\s{G}),\label{eqn:LIKlihood}
\end{equation} 
which is the Radon-Nikodym derivative of $\P_{\calH_1}$ w.r.t. the measure $\P_{\calH_0}$. It is well known (see, e.g., \cite[Theorem 2.2]{tsybakov2004introduction}) that the optimal test $\phi^{\star}$ that minimizes the risk $\s{R}_n$ is the likelihood ratio test defined as,
\begin{align}
\phi^{\star}\left(\s{G}\right) \triangleq \begin{cases}
1,\ &\text{if }\mathsf{L}\left(\s{G}\right) \geq 1\\
0,\ &\text{otherwise},   
\end{cases}
\end{align}
and the associated optimal risk is $ \mathsf{R}^{\star} \triangleq \mathsf{R}\left(\phi^{\star}\left(\s{G}\right)\right)=1-d_{\s{TV}}(\P_{\calH_0},\P_{\calH_1})$. Recalling that $\chi^2(\P_{\calH_1},\P_{\calH_0})=\E_{\calH_0}[\s{L}(\s{G})^2]-1$, it can be shown that (see, e.g., \cite[Sec. 2]{tsybakov2004introduction} and \cite[Prop. 3]{sason2014bounds}),
\begin{align}
    \chi^2(\P_{\calH_1},\P_{\calH_0})\geq \max\p{\frac{1}{2\p{1-d_{\s{TV}}(\P_{\calH_0},\P_{\calH_1})}}-1,\p{2d_{\s{TV}}(\P_{\calH_0},\P_{\calH_1})}^2},
\end{align}
and thus,
\begin{align}
    \mathsf{R}^{\star}&=1-d_{\s{TV}}(\P_{\calH_0},\P_{\calH_1})\geq \max\p{1-\frac{1}{2}\sqrt{\chi^2(\P_{\calH_1},\P_{\calH_0})},\frac{1}{2(1+\chi^2(\P_{\calH_1},\P_{\calH_0}))}}.\label{eqn:lowerBoundSecond}
\end{align}
In particular, we see that $\s{R}^\star$ is bounded away from zero, namely, strong detection is impossible, if $\E[\s{L}(\s{G})^2]$ is bounded. Similarly, $\s{R}^\star$ converge to unity, i.e., weak detection is impossible if $\E_{\calH_0}[\s{L}(\s{G})^2]=1+o(1)$. Accordingly, to rule out the possibility of detection (either strong or weak) it suffices to upper bound the second moment of the likelihood function. Throughout the rest of this section, we will bound $\E_{\calH_0}[\s{L}(\s{G})^2]$, and derive conditions under which $\E_{\calH_0}[\s{L}(\s{G})^2]$ remains bounded or converges to unity.

To that end, we take a Fourier-analytic approach and decompose $\s{L}(\s{G})$ w.r.t. an orthonormal polynomial basis. We denote by $L^2(\calH_0)$ the Hilbert space of random variables over the probability space on which $\P_{\calH_0}$ is defined, with a finite second moment, and equipped with the inner product,
\begin{align}
    \innerP{\varphi(\s{G}),\psi(\s{G})}_{\calH_0}\triangleq \E_{\calH_0}\pp{\varphi(\s{G})\cdot \psi(\s{G})}.
\end{align}
For any non-empty subset of edges in the complete graph $H\subseteq \binom{[n]}{2}$, define the Fourier character $\chi_H$ as,
\begin{align}
     \chi_{H}(\s{G}) \triangleq \prod_{\{i,j\}\in H}\frac{\s{G}_{ij}-q}{\sqrt{q(1-q)}},\label{eqn:Fouriercharacter} 
\end{align}
and $\chi_{\emptyset}(\s{G})\equiv1$, for each $\s{G}\in\{0,1\}^{\binom{n}{2}}$. Note that any subset of edges $H\subseteq \binom{[n]}{2}$ induces a subgraph $\s{H}\subseteq \calK_n$ containing no isolated vertices. In fact, there is a one-to-one correspondence between subgraphs without isolated vertices and subsets of edges. We therefore identify each character $\chi_H$ with a subgraph $\s{H}$ and denote it by $\chi_\s{H}$. Observe that $\chi_\s{H}$ is polynomial in the entries of $\s{G}$, with degree $|e(\s{H})|$, which, with slight abuse of notation, we also denote by $|\s{H}|$. It is well known (and easy to verify) that the set $\ppp{\chi_{\s{H}}}_{\s{H}\subseteq \binom{[n]}{2}}$ forms an orthonormal basis for $L^2(\calH_0)$. Hence, by Parseval's identity we have,
     \begin{align}
         \E_{\calH_0}\pp{\s{L}(\s{G})^2}=\norm{\s{L}(\s{G})}_{\calH_0}^2=\sum_{\s{H}\subseteq \binom{[n]}{2}}\abs{\innerP{\chi_{\s{H}}(\s{G}),\s{L}(\s{G})}}^2.\label{eq:ParsevalInital}
     \end{align}
     Building on \eqref{eq:ParsevalInital}, let us give two equivalent characterizations for $\E_{\calH_0}\pp{\s{L}(\s{G})^2}$.
     \begin{prop}\label{prop:LikelihoodMomentExperession}
         Consider the setting in Section~\ref{sec:problem}, and denote $\lambda^2\triangleq\chi^2(p||q)$. Then,
         \begin{align}
             \E_{\calH_0}\pp{\s{L}(\s{G})^2}&=\sum_{\s{H}'\subseteq \Gamma'}\lambda^{2|\s{H}'|}\cdot\P_{\Gamma}\pp{\s{H}'
             \subseteq \Gamma}\label{eq:SecondMomentExpressionProb1}\\
             &=\sum_{\s{H}'\subseteq \Gamma'}\lambda^{2|\s{H}'|}\cdot\P_{\s{H}}\pp{\s{H}\subseteq \Gamma'}\label{eq:SecondMomentExpressionProb2}\\
             &=\E_{\Gamma}\pp{(1+\lambda^2)^{|e(\Gamma\cap \Gamma')|}},\label{eq:SecondMomentExpression}
         \end{align}
         where $\Gamma'$ is a fixed arbitrary copy of $\Gamma$ in $\calK_n$, the probability in \eqref{eq:SecondMomentExpressionProb2} is taken w.r.t. a random copy $\s{H}$ of $\s{H}'$, the probability and expectation in \eqref{eq:SecondMomentExpressionProb1} and \eqref{eq:SecondMomentExpression} are taken w.r.t. a random copy $\Gamma\sim \s{Unif}(\calS_\Gamma)$, the summation is over subgraphs containing no isolated vertices, and $|e(\Gamma\cap\Gamma')|$ is the number edges in the intersection of $\Gamma$ and $\Gamma'$.
     \end{prop}
    In order to prove Proposition~\ref{prop:LikelihoodMomentExperession}, we will need two technical results. We start with the following definition.
    \begin{definition}\label{def:CountsSubg} Let $\Gamma'$ be a subgraph of $\calK_n$, and $\s{H}'\subseteq \Gamma'$ be a subgraph of $\Gamma'$. Define $\calN(\s{H}',\Gamma')$ as the number of copies of $\s{H}'$ in $\Gamma'$, and $\calM(\s{H}',\Gamma')$ as the number of copies of $\Gamma'$ in $\calK_n$ which contain $\s{H}'$.
    \end{definition}

    \begin{lemma}\label{lem:equivRandomCopyOrSubgraphcopy} Let $\Gamma'$ be a subgraph of $\calK_n$, and $\s{H}'\subseteq \Gamma'$  be a subgraph. Let $\s{H}\sim\s{Unif}(\calS_{\s{H}'})$ be a uniform random copy of $\s{H}'$ in $\calK_n$, and 
    let $\Gamma$ be a uniform random copy of $\Gamma'$. Then,
    \begin{align}
        \P_{\s{H}}\pp{\s{H}\subseteq \Gamma'}=\frac{\calN(\s{H}',\Gamma')}{\abs{\calS_{\s{H}}}}=\frac{\calM(\s{H}',\Gamma')}{|\calS_\Gamma|}=\P_{\Gamma}\pp{\s{H}'\subseteq \Gamma}. \label{eq:calMcalN relation}
    \end{align}
    \end{lemma}
    \begin{proof}
        The first and third equalities follow directly from the definition of the distribution of a random copy. It remains to show that 
        \begin{align}
            \calN(\s{H'},\Gamma')\cdot |\calS_\Gamma|=\calM(\s{H'},\Gamma')\cdot |\calS_{\s{H}}|, \label{eq:counMcalN}
        \end{align}
        which follows from the observation that the terms on both sides of \eqref{eq:counMcalN} represent the cardinality of the set,
        \begin{align}
            \ppp{\left.(\s{H}'',\Gamma'') \right| \s{H}''\subseteq \Gamma''\subseteq \calK_n,  \s{H}''\cong\s{H}', \Gamma''\cong \Gamma' }.
        \end{align}
    \end{proof}
    Thus, Lemma~\ref{lem:equivRandomCopyOrSubgraphcopy} establishes the second equality in \eqref{eq:SecondMomentExpressionProb2}. Next, we prove the third equality in \eqref{eq:SecondMomentExpression} through the following lemma. In fact, we derive a slightly more general identity, which will prove useful in establishing ($\s{Arg}_1$), described in Subsection~\ref{subsec:outline}.
    
    \begin{lemma}\label{lem:intersectionMomentGenGeneral}  Let $\Gamma_1$ and $\Gamma_2$ be two arbitrary (perhaps different) subgraphs of $\calK_n$. Then,
\begin{align}
    \E_{\Gamma_1\indep \Gamma_2}\pp{(1+\lambda^2)^{|e(\Gamma_1\cap\Gamma_2)|}}&=\sum_{\s{H}\subseteq \binom{[n]}{2}}\lambda^{2|\s{H}|}\cdot \P_{\Gamma_1}\pp{\s{H} \subseteq \Gamma_{1}}\cdot \P_{\Gamma_2} \pp{\s{H} \subseteq \Gamma_2}\label{eq:inetersectionEquivalent}\\
    &=\sum_{\s{H}\subseteq \Gamma_1'}\lambda^{2|\s{H}|}\cdot \P_{\Gamma_2} \pp{\s{H} \subseteq \Gamma_2},\label{eq:inetersectionEquivalent2}
\end{align}
where $\Gamma_1'$ is a fixed arbitrary copy of $\Gamma_1$ in $\calK_n$, the summation is over graphs containing no isolated vertices, and the probabilities and the expectation are taken w.r.t. two independent random copies $\Gamma_i\sim \s{Unif}(\calS_{\Gamma_i})$, for $i=1,2$.
\end{lemma}

\begin{proof}[Proof of Lemma~\ref{lem:intersectionMomentGenGeneral}]
Since we deal with graphs containing no isolated vertices, we slightly abuse notation by treating a graph $\s{H}$ as a subset of edges and denote $|\s{H}|\triangleq|e(\s{H})|$. Now, let us analyze the r.h.s. of \eqref{eq:inetersectionEquivalent}, beginning with the following chain of equalities,
\begin{align}
    \sum_{\s{H}\subseteq \binom{[n]}{2}}\lambda^{2|\s{H}|}\cdot \P_{\Gamma_1}\pp{\s{H} \subseteq \Gamma_1}\cdot \P_{\Gamma_2} \pp{\s{H} \subseteq \Gamma_2}&=\sum_{\s{H}\subseteq \binom{[n]}{2}}\lambda^{2|\s{H}|}\cdot \P_{\Gamma_1\indep \Gamma_2}\pp{\s{H} \subseteq \Gamma_1, \s{H} \subseteq \Gamma_2}\\
    &=\sum_{\s{H}\subseteq \binom{[n]}{2}}\lambda^{2|\s{H}|}\cdot \P_{\Gamma_1\indep \Gamma_2}\pp{\s{H} \subseteq e\p{\Gamma_1\cap  \Gamma_2}}\\
    &=\sum_{\s{H}\subseteq \binom{[n]}{2}}\lambda^{2|\s{H}|}\cdot \sum_{\substack{\s{H}'\subseteq \binom{[n]}{2}\\ \s{H}\subseteq \s{H}'}}\P_{\Gamma_1\indep \Gamma_2}\pp{\s{H}'=e\p{\Gamma_1\cap  \Gamma_2}}\\
    &=\sum_{\s{H}'\subseteq \binom{[n]}{2}} \P_{\Gamma_1\indep \Gamma_2}\pp{\s{H}'=e\p{\Gamma_1\cap  \Gamma_2}} \sum_{\substack{\s{H}\subseteq \binom{[n]}{2}\\ \s{H}\subseteq \s{H}'}}\lambda^{2|\s{H}|}\\
    &=\sum_{\s{H}'\subseteq \binom{[n]}{2}} \P_{\Gamma_1\indep \Gamma_2}\pp{\s{H}'=e\p{\Gamma_1\cap  \Gamma_2}} \sum_{i=0}^{\abs{\s{H}'}}\binom{|\s{H}'|}{i}\lambda^{2i}\\
    &=\sum_{\s{H}'\subseteq \binom{[n]}{2}} \P_{\Gamma_1\indep \Gamma_2}\pp{\s{H}'=e\p{\Gamma_1\cap  \Gamma_2}}\cdot\p{1+\lambda^{2}}^{\abs{\s{H}'}}\\
    &=  \E_{\Gamma_1\indep \Gamma_2}\pp{(1+\lambda^2)^{|e(\Gamma_1\cap\Gamma_2)|}},
\end{align}
which proves the first equality in \eqref{eq:inetersectionEquivalent}. For the second equality in \eqref{eq:inetersectionEquivalent2}, recall that by Lemma~\ref{lem:equivRandomCopyOrSubgraphcopy}, for any subgraph $\s{H}$ of a fixed copy $\Gamma_1'$ of $\Gamma_1$, we have that,
\begin{align}
    \label{eq:probcalc1}
    \pr_{\Gamma_1}\pp{\s{H}\subseteq \Gamma_1'}=\frac{\calM(\s{H},\Gamma_1')}{|\calS_{\Gamma_1'}|}=\frac{\calN(\s{H},\Gamma_1')}{|\calS_{\s{H}}|}.
\end{align}
By symmetry, the above expression is invariant under isomorphisms. Namely, for any $\s{H'}\subseteq \calK_n$,  if $\s{H}'$ is isomorphic to some subgraph $\s{H}\subseteq\Gamma_1'$, then 
\begin{align}
    \pr_{\Gamma_1}\pp{\s{H}\subseteq \Gamma_1'}=\pr_{\Gamma_1}\pp{\s{H'}\subseteq \Gamma_1'}.
\end{align}
If $\s{H}'$ is not isomorphic to any subgraph of $\Gamma_1$, then the probability above is clearly zero. We define an equivalence relation on subgraphs of $\Gamma_1'$, where two subgraphs are considered equivalent if and only if they are isomorphic. Let $[\s{H}]$ denote the equivalence class of a subgraph $\s{H}\subseteq \Gamma_1$, and let $\calP$ be the set of all equivalence classes. Then, using \eqref{eq:probcalc1} and the fact that $|[\s{H}]|=\calN(\s{H},\Gamma)$, we may write,
\begin{align}
     \sum_{\s{H}\subseteq \binom{[n]}{2}}\lambda^{2|\s{H}|}\cdot \P_{\Gamma_1}\pp{\s{H} \subseteq \Gamma_1}\cdot \P_{\Gamma_2} \pp{\s{H} \subseteq \Gamma_2}&=\sum_{[\s{H}]\in \calP}\sum_{\substack{\s{H}'\subseteq \binom{[n]}{2}\\ \s{H}'\cong \s{H} }}\lambda^{2|\s{H}'|}\cdot \P_{\Gamma_1}\pp{\s{H}' \subseteq \Gamma_1}\cdot \P_{\Gamma_2} \pp{\s{H}' \subseteq \Gamma_2}\\
     &=\sum_{[\s{H}]\in \calP}\sum_{\substack{\s{H}'\subseteq \binom{[n]}{2}\\ \s{H}'\cong \s{H} }}\lambda^{2|\s{H}|}\cdot \P_{\Gamma_1}\pp{\s{H} \subseteq \Gamma_1}\cdot \P_{\Gamma_2} \pp{\s{H} \subseteq \Gamma_2}\\
     &=\sum_{[\s{H}]\in \calP} |\calS_{\s{H}}|\lambda^{2|\s{H}|}\cdot \P_{\Gamma_1}\pp{\s{H} \subseteq \Gamma_1}\cdot \P_{\Gamma_2} \pp{\s{H} \subseteq \Gamma_2}\\
     &=\sum_{[\s{H}]\in \calP} \lambda^{2|\s{H}|}\cdot \calN(\s{H},\Gamma_1)\cdot \P_{\Gamma_2} \pp{\s{H} \subseteq \Gamma_2}\\
     &=\sum_{[\s{H}]\in \calP} \sum_{\s{H}'\in [\s{H}]}\lambda^{2|\s{H}|}\cdot \P_{\Gamma_2} \pp{\s{H} \subseteq \Gamma_2}\\
     &=\sum_{\s{H}\subseteq \Gamma_1'} \lambda^{2|\s{H}|}\cdot \P_{\Gamma_2} \pp{\s{H} \subseteq \Gamma_2}.
\end{align}
which completes the proof.
\end{proof}
We are now in a position to complete the proof of Proposition~\ref{prop:LikelihoodMomentExperession}.
\begin{proof}[Proof of Proposition~\ref{prop:LikelihoodMomentExperession}]
We recall \eqref{eq:ParsevalInital} and proceed by computing the Fourier coefficients. From the definition of $\s{L}(\s{G})$ in \eqref{eqn:LIKlihood}, it follows that for any character $\chi_{\s{H}}$, we have,
\begin{align}
    \innerP{\chi_{\s{H}}, \s{L}(\s{G})}_{\calH_0}=\E_{\calH_0}\pp{\chi_{\s{H}}(\s{G}) \cdot \s{L}(\s{G})}=\E_{\calH_1}\pp{\chi_{\s{H}}(\s{G}) }.
\end{align}
Let us compute $\bE_{\calH_1}[\chi_{\s{H}}(\s{G})]$, for each $\s{H}$. We have,
\begin{align}
    \bE_{\calH_1}\pp{\chi_{\s{H}}(\s{G})} &= \bE_{\Gamma}\pp{\bE_{\calH_1\vert\Gamma}\pp{\chi_{\s{H}}(\s{G})}}\\
    & = \bE_{\Gamma}\pp{\bE_{\calH_1\vert\Gamma}\pp{\prod_{\{i,j\}\in\s{H}}\frac{\s{G}_{ij}-q}{\sqrt{q(1-q)}}}}\\
    & = \bE_{\Gamma}\pp{\prod_{\{i,j\}\in\s{H}}\bE_{\calH_1\vert\Gamma}\pp{\frac{\s{G}_{ij}-q}{\sqrt{q(1-q)}}}},\label{eqn:ffDe}
\end{align}
where we have used the fact that conditioned on $\Gamma$, the edges of $\s{G}$ are independent. Depending on the edge $\{i,j\}$ and its relation to $\s{H}$ and $\s{G}$, there are two possible cases.
\begin{itemize}
\item If $\{i,j\}\in\s{H}$ is such that  $\{i,j\}\in e(\Gamma)$, then $\s{G}_{ij}\sim\s{Bern}(p)$, and accordingly,
\begin{align}
\bE\pp{\left.\frac{\s{G}_{ij}-q}{\sqrt{q(1-q)}}\right| \{i,j\}\in e(\Gamma)} = p\sqrt{\frac{1-q}{q}}-(1-p)\sqrt{\frac{q}{1-q}} = \frac{p-q}{\sqrt{q(1-q)}}\triangleq\lambda,
\end{align}
and we note that $\lambda=\sqrt{\chi^2(p||q)}$. 
\item If $\{i,j\}\in\s{H}$ is such that  $\{i,j\}\not\in e(\Gamma)$, then $\s{G}_{ij}\sim\s{Bern}(q)$, and accordingly,
\begin{align}
\bE\pp{\left.\frac{\s{G}_{ij}-q}{\sqrt{q(1-q)}}\right| \{i,j\}\not\in e(\Gamma)} = 0.
\end{align}
\end{itemize}
Combining the above results, we have,
\begin{align}
    \bE_{\calH_1\vert\Gamma}\pp{\frac{\s{G}_{ij}-q}{\sqrt{q(1-q)}}} = \lambda\cdot\Ind\ppp{\{i,j\}\in\s{H}\cap e(\Gamma)},
\end{align}
which in turn implies that,
\begin{align}\label{eq:coeffCalc}
    \prod_{\{i,j\}\in\s{H}}\bE_{\calH_1\vert\Gamma}\pp{\frac{\s{G}_{ij}-q}{\sqrt{q(1-q)}}} = \lambda^{|\s{H}|}\cdot\Ind_{\ppp{\s{H}\subseteq \Gamma}}.
\end{align}
Substituting in \eqref{eqn:ffDe} we get, 
\begin{align}
    \bE_{\calH_1}\pp{\chi_{\s{H}}(\s{G})} = \bE_\Gamma\pp{\lambda^{|\s{H}|}\cdot\Ind_{\ppp{\s{H}\subseteq e(\Gamma)}}}=\lambda^{|\s{H}|}\cdot\pr_{\Gamma}\pp{\s{H} \subseteq \Gamma}.\label{eq:coeffCalc2}
\end{align}
Therefore, inserting \eqref{eq:coeffCalc2} in \eqref{eq:ParsevalInital}, we get,
    \begin{align}
         \E_{\calH_0}\pp{\s{L}(\s{G})^2}&=\sum_{\s{H}\subseteq \binom{[n]}{2}}\bE_{\calH_1}\pp{\chi_{\s{H}}(\s{G})}^2\\
         &=\sum_{\s{H}\subseteq \binom{[n]}{2}}\lambda^{2|\s{H}|}\cdot\pr_{\Gamma}\pp{\s{H} \subseteq \Gamma}^2.
     \end{align}
     This proves the first equality in \eqref{eq:SecondMomentExpression}. The third equality in \eqref{eq:SecondMomentExpression} follows by setting $\Gamma_1=\Gamma_2=\Gamma$ in Lemma~\ref{lem:intersectionMomentGenGeneral}, and using the fact that, by symmetry, we can always consider $\Gamma_2$ as a fixed copy instead of a uniform random copy, while $\Gamma_1$ remains uniform and random.
\end{proof}

\begin{comment}
    THIS IS A BACKUP IS SOMETHING GOES WRONG

\end{comment}

\subsection{The dense regime}\label{subsec:denseRegime}

Following the outline in Subsection~\ref{subsec:outline}, in this subsection, we consider the dense regime, where $\chi^2(p||q)=\Theta(1)$ (in particular, $\delta<p-q<1-\delta$, for some $\delta>0$), and provide a complete formal proof for the statistical lower bounds in Theorem~\ref{th:lowerboundDense}. We will generalize this proof strategy in Subsection~\ref{subsec:otherRegimes} to other regimes.
\begin{theorem}\label{th:lowerboundDense}
    Assume that $\lambda^2=\chi^2(p||q)=\Theta(1)$, and that $\Gamma=(\Gamma_n)_n$ is any sequence of subgraphs such that $|v(\Gamma)|\leq (1-\delta)n$, for some fixed $\delta>0$.  \begin{enumerate}
        \item  For $\mu(\Gamma)\geq \alpha\log|v(\Gamma)|$, weak detection is impossible if
    \begin{align}
        \mu(\Gamma) \leq \frac{(1-\varepsilon)\alpha}{2+\alpha\log(1+\lambda^2)}\log n,\label{eqn:MainResDenseLower1}
    \end{align}
    for some $\varepsilon>0$, and strong detection is possible if 
    \begin{align}
        \mu(\Gamma)\geq (1+\varepsilon)\frac{\log n}{d_{\s{KL}}(p||q)},\label{eqn:MainResDenseUpper1}
    \end{align}
    for some $\varepsilon>0$.
    \item For $\mu(\Gamma)=o\p{\log|v(\Gamma)|}$, strong detection is impossible if  
    \begin{align}
        \max\p{|e(\Gamma)|,d_{\max}^2(\Gamma)}\leq n^{1-\varepsilon}, \label{eqn:MainResDenseLower2}
    \end{align}
    for some $\varepsilon>0$, and strong detection is possible if 
    \begin{align}
        \max\p{|e(\Gamma)|,d_{\max}^2(\Gamma)}\geq n^{1+f(n)}, \label{eqn:MainResDenseUpper2}
    \end{align}
    where $f(n)$ is an $o(1)$ function. 
    \end{enumerate}
\end{theorem}
The above result asserts that if $\Gamma$ has super-logarithmic density, i.e., $\mu(\Gamma)=\Omega(\log|v(\Gamma)|)$, then the statistical barrier is determined by $\mu(\Gamma)$ only. On the other hand, if $\Gamma$ has sub-logarithmic density, the statistical barrier is determined by $|e(\Gamma)|$ and $d_{\max}(\Gamma)$. As mentioned in Subsection~\ref{subsec:outline}, our proof consists of three main steps, as detailed below. Before we continue, we introduce the following two definitions, which play a significant role in our analysis.
\begin{definition}[Vertex cover]\label{def:coverNumber}
Let $\s{G}=(\s{V},\s{E})$ be an undirected graph. A set $\s{U}\subseteq \s{V}$ is called a vertex cover of $\s{G}$ if any edge in $\s{E}$ has a vertex in $\s{U}$. We define the vertex cover number $\vc(\s{G})$ of $\s{G}$ as the minimal cardinality vertex cover set in $\s{G}$. 
\end{definition}
\begin{definition}\label{def:coverdegreebalanced}
    A sequence of graphs $\Gamma=(\Gamma_n)_n$ is called \textit{vertex cover-degree balanced}, or simply $\vcd$-balanced, if 
    \begin{align}
        \lim_{n\to\infty} \frac{\log\p{\vc(\Gamma)\cdot d_{\max}(\Gamma)}}{\log|e(\Gamma)|}=1.
    \end{align}
\end{definition}
%The proof of our most general form of our lower bounds may appear to be technically complicated due to the case-by-case analysis required for capturing the delicate behaviors of the model in the diverse regimes of the parameters. Consequently, in order to make our ideas more transparent and  easy to capture, we first develop our proof strategy for the (perhaps most fundamental) dense regime, where it is assumed that $p,q=\Theta(1) $ (and in particular $\chi^2(p||q)=\Theta(1)$).     

\subsubsection{Combinatorial bound through vertex covers} \label{sec:VertexCovers}

In this subsection, we prove the impossibility result in Theorem~\ref{th:lowerVCD}. To that end, as mentioned in the previous subsection, it suffices to upper bound $\E_{\calH_0}[\s{L}(\s{G})^2]$. 
%As a first step towards the proof of our lower bound, in this section we will analyze the  likelihood's second moment and bound expression described in Proposition~\ref{prop:LikelihoodMomentExperession} w.r.t. the graph parameters. Our analysis results in an upper bound on $\E_{\calH_0}[\s{L}(\s{G})^2]$, which provides the sufficient condition \eqref{eq:LowerMainCond1} for the impossibility of (weak and strong) detection stated in Theorem~\ref{th:lowerVCD}, which is all we will require for the dense regime.
Throughout this subsection, for the benefit of readability, we let $k\triangleq|v(\Gamma)|$, $d\triangleq d_{\max}(\Gamma)$, $\vc\triangleq\vc(\Gamma)$, and $\mu\triangleq\mu(\Gamma)$. In light of Proposition~\ref{prop:LikelihoodMomentExperession}, our goal is to bound $\P_{\Gamma}[\s{H}'\subseteq \Gamma]$ (or, equivalently, $\P_{\s{H}}[\s{H}\subseteq \Gamma']$). The following result shows that, from the perspective of dependence on graph-theoretic properties of $\Gamma$, this probability can be upper bounded in terms of $(d_{\max}(\Gamma),\vc(\Gamma))$ alone.
\begin{lemma}\label{lem:probRandomSubgraph1}
    Let $\Gamma'$ be a fixed copy of $\Gamma$, $\s{H}'\subseteq\Gamma'$ be a subgraph containing no isolated vertices, with $\ell$ vertices and $m$ connected components, and $\s{H}$ a random copy of $\s{H}'$ in $\calK_n$. Then, 
    \begin{align}
        \P_{\Gamma}\pp{\s{H}'\subseteq\Gamma}=\P_{\s{H}}\pp{\s{H}\subseteq\Gamma'}\leq \frac{(2\vc)^m d^{\ell-m}}{(n-k)^\ell}\triangleq\vartheta(m,\ell).\label{eqn:UpperBoundProbBasic}
    \end{align}
\end{lemma}
For the proof of Lemma~\ref{lem:probRandomSubgraph1} we will need the following important observation.
\begin{obs}\label{obs:RandomCopy1}
    Let $\s{H}$ be a graph with $\ell\leq n$ vertices enumerated as $v_1,\dots,v_\ell$ and let $(X_1,\dots,X_\ell)$ be distributed as 
    \[X_1\sim \s{Unif}\p{[n]}, \quad \text{and}\quad X_i|X_1,\dots X_{i-1} \sim  \s{Unif}\p{[n]\setminus \set{X_1,\dots, X_{i-1}}},\]
    for all $1<i<\ell$. Then, the random graph $\s{H_U}=\p{v({\s{H_U}}),E({\s{H_U}})}$ with
    \[V_{\s{U}}=\set{X_1,\dots X_\ell}, \quad \text{and} \quad E_{\s{U}}=\set{\left.\p{X_i,X_j}\right| (v_i,v_j)\in \s{H}},\]
    is a uniform random copy of $\s{H}$ in $\calK_n$.
\end{obs}

\begin{proof}[Proof of Lemma~\ref{lem:probRandomSubgraph1}]
     Let $S\subseteq v(\Gamma)$ be a minimal vertex cover of  $\Gamma$, with $|S|=\vc$. Denote by $\s{C}_1,\dots,
     \s{C}_m$ the connected components of $\s{H}$. Since $\s{H}$ does not contain any isolated vertices, there exists an enumeration $v_1,\dots,v_\ell$ of the vertices of $\s{H}$, such that for any $1\leq i\leq m $ the pair $(v_{2i-1},v_{2i})$ is an edge in $\s{C_i}$, and furthermore, for all $2m+1\leq i\leq \ell$, the vertex $v_i$ is connected to some $v_j$, for $j<i$. Now, let $\s{H_U}$ be a random copy of $\s{H}$ in $\calK_n$, obtained by randomly picking the vertices $X_1,\dots X_\ell$ as described in Observation~\ref{obs:RandomCopy1}.
     For all $1\leq i \leq \ell$, let $\s{H}_{\s{U}}^{(i)}$ denote the subgraph induced by $X_1,\dots,X_i$, namely, 
     \begin{align}
         \s{H}_{\s{U}}^{(i)}&\triangleq\p{V_i, (V_i \times V_i)\cap E(\s{H_U}) },\\
         V_i&\triangleq\set{X_1,\dots,X_i}.
     \end{align}
     We next prove that 
     \begin{align}
     \P\pp{\s{H}_{\s{U}}^{(2m)}\subseteq \Gamma'}\leq \p{\frac{2 \vc d}{(n-k)^2}}^m.\label{eq:probConnectedComponents1}
     \end{align}
     Indeed, note that by Observation~\ref{obs:RandomCopy1}, and by the construction of the ordering $(v_i)_i$, the distribution of $X_1\dots,X_{2m}$ is uniform over all $2m$-tuples with no repetitions over $[n]$. Furthermore, $\s{H}_{\s{U}}^{(2m)}\subseteq \Gamma'$ only if $\ppp{X_{2i-1},X_{2i}}\in e(\Gamma')$, for all $1\leq i\leq m$. Recall that by the definition of a vertex cover, it must be that either $X_{2i-1}$ or $X_{2i}$ lies in $S$, and the other vertex must be in a neighborhood of other vertices in $\Gamma'$, which contains at most $d$ vertices. Thus, there are at most $2\vc d$ possible pairs $(X_{2i-1},X_{2i})$, and accordingly, 
     \begin{align}
         \P\pp{\s{H}_{\s{U}}^{(2m)}\subseteq \Gamma'}&=\P\pp{\bigcap_{i=1}^m \ppp{(X_{2i-1},X_{2i})\in e(\Gamma')}}\\
         &\leq \frac{(2\vc  d)^m}{\prod_{i=0}^{2m-1}(n-i)}\\
         &\leq \p{\frac{2\vc  d}{(n-k)^2}}^m.
     \end{align}
     Next, observe that for all $2m+1\leq i \leq \ell$,
     \begin{align}
         \P\pp{\s{H}_{\s{U}}^{(i)}\subseteq \Gamma'}=\P\pp{\left.\s{H}_{\s{U}}^{(i)}\subseteq \Gamma'\right|\s{H}_{\s{U}}^{(i-1)}\subseteq \Gamma'}\cdot \P\pp{\s{H}_{\s{U}}^{(i-1)}\subseteq \Gamma'}.\label{eq:recursiveProbCalc1}
     \end{align}
     Now, since for all $i\geq 2m+1$, the vertex $v_i$ must have a neighbor $v_j$ for some $j<i$, then given $\s{H}_{\s{U}}^{(i-1)}\subseteq \Gamma'$ it must be that vertex $X_i$ is one of the (at most $d$) neighbors of $X_j$ in $\Gamma'$. This in turn implies that, 
     \begin{align}
         \P\pp{\left.\s{H}_{\s{U}}^{(i)}\subseteq \Gamma' \right| \s{H}_{\s{U}}^{(i-1)}\subseteq \Gamma'}\leq \frac{d}{n-(i-1)}\leq \frac{d}{n-k}. \label{eq:probDegreeInGamma1}
     \end{align}
     Applying \eqref{eq:recursiveProbCalc1} and \eqref{eq:probDegreeInGamma1} recursively, and combining with \eqref{eq:probConnectedComponents1}, we finally obtain that,
     \begin{align}
         \P_{\s{H}}\pp{\s{H}\subseteq\Gamma'}&=\P\pp{\s{H}_{\s{U}}^{(\ell)}\subseteq\Gamma'}\\
         &\leq  \p{\frac{d}{n-k}}^{\ell-2m}\cdot \P\pp{\s{H}_{\s{U}}^{(2m)}\subseteq\Gamma'}\\
         &\leq \p{\frac{d}{n-k}}^{\ell-2m}\cdot \p{\frac{2\vc  d}{(n-k)^2}}^m \\
         &=\frac{(2 \vc )^m d^{\ell-m}}{(n-k)^\ell}.
     \end{align}
 \end{proof}
In light of Lemma~\ref{lem:probRandomSubgraph1} and \eqref{eq:SecondMomentExpressionProb1}, we see that the summands in \eqref{eq:SecondMomentExpressionProb1} depend on $\s{H}'$ only through the values of $(|v(\s{H}')|,|\s{H}'|,m(\s{H}'))$. Accordingly, for a fixed $2\leq \ell\leq k $, $1 \leq m \leq \floor{\frac{\ell}{2}}$ and $\ell -m \leq j\leq\floor{\mu \cdot \ell} $, let $\calS_{m,\ell,j}$ be the set of all subgraphs $\s{H}'\subseteq \Gamma $ with $m$ connected components, $\ell$ vertices, $j$ edges, and no isolated vertices. Then,
\begin{align}
    \E_{\calH_0}\pp{\s{L}(\s{G})^2}\leq\sum_{(m,\ell,j)}\lambda^{2j}\vartheta(m,\ell)|\calS_{m,\ell,j}|,\label{eqn:InterBound}
\end{align}
where the summation over $(m,\ell,j)$ is in the range described above. Thus, our next task is to upper bound $|\calS_{m,\ell,j}|$. 
One concept that has an important role in the following derivations is the notion of integer partitions.  
 \begin{definition}[Integer partitions]
 An integer partition $p$ of an integer $\ell$  to exactly $m$ elements is a superset of positive integers of size $m$, such that the sum of its elements is $\ell$. The set of integer partitions of $\ell$ to $m$ parts is denoted by $\s{Par}(\ell,m)$. 
 \end{definition}
We can also interpret an integer partition $p\in \s{Par}(\ell,m) $ as a function $p:[\ell] \to [\ell]$, where $p(i)$ denotes the number of parts of size $i$, or equivalently, the number of appearances of $i$ in $p$. Finally, from the definition of an integer partition, we have the following two identities,
 \begin{align}
     \sum_{i=1}^\ell p(i)=m \quad \s{and} \quad  \sum_{i=1}^\ell i\cdot p(i)=\ell.\label{eq:partitiondef1} 
 \end{align}
We have the following result.
\begin{lemma}\label{lem:subgraphsBoundedDegreeCount1}
     Fix $2\leq \ell\leq k $, $1 \leq m \leq \floor{\frac{\ell}{2}}$, $\ell-m \leq j\leq\floor{\mu  \ell }$, and define $\tilde{d}=\max(2,d-1)$. Then, 
    \begin{align}
         |\calS_{m,\ell,j} |&\leq  \abs{\s{Par}(\ell,m)} \vc^m (e\tilde{d})^{\ell-m}  \binom{\min\p{\floor{\mu  \ell},\binom{\ell}{2}}}{j}.\label{eq:subgraphcount1}
     \end{align}     
 \end{lemma}
To prove the above bound, we need the following definition and result.
\begin{definition}\label{def:connectedSet1}
     Let $G=(V,E)$ be an undirected graph and  $U\subseteq V$ be a set of vertices. The induced subgraph $G_U$ of $G$ obtained from $U$ is defined as $G_U=(U,U\times U \cap E)$. A set of vertices $U\subseteq V$ is called connected if the induced subgraph $G_U$ is connected. 
 \end{definition}
 \begin{lemma}\cite[pp. (129)-(130)]{bollobas2006art}\label{lem:Bollobas1} A graph $G=(V,E)$ of maximum degree $d\geq 3$ has at most $(e(d - 1))^{\ell-1}$ connected induced subgraphs (equivalently, connected sets) with $\ell $ vertices, one of which is a given vertex.
 \end{lemma}
We are now in a position to prove Lemma~\ref{lem:subgraphsBoundedDegreeCount1}.
\begin{proof}[Proof of Lemma~\ref{lem:subgraphsBoundedDegreeCount1}] 
    We begin our enumeration by first running over all possible sizes (i.e., number of vertices) of the connected components $\s{C}_1,\dots,\s{C}_m$. We observe that any choice of the sizes of these components, or, equivalently, the number of components of a given size, correspond to a unique integer partition of the number $\ell$ to exactly $m$ elements. Accordingly, fixing a partition $p$, let us denote by $S(p,j)$ the set of $\Gamma$-subgraphs with exactly $p(i)$ connected components containing exactly $i$ vertices, for all $i\in[\ell]$. Then, we can enumerate $\calS_{m,\ell,j}$ as,
    \begin{align}
        |\calS_{m,\ell,j} |= \sum_{p\in\s{Par}(\ell,m)}|S(p,j)|.\label{eqn:enumPar}
    \end{align}
    Next, we bound $|S(p,j)|$.
    We start by bounding the number of ways to choose the vertices of subgraphs $\s{H}$ with sizes that correspond to the partition $p$. For each such subgraph $\s{H}$, we have $p(i)$ connected components of size $i$, which are supported on $p(i)$ disjoint connected sets of vertices (see, Definition~\ref{def:connectedSet1}). Recall that $\tilde{d}=\max(d-1,2)$. Then, observe that there are at most
    \begin{align}
        \vc\cdot (e \tilde{d})^{i-1}\label{eq:quantity1}
    \end{align} 
    connected sets of size $i$ in $\Gamma$. To see that, let $S$ be a minimal vertex cover set in $\Gamma$, and assume first that $d\geq 3$. Now, any edge must share a vertex in $S$, and therefore, any connected set must contain a vertex in $S$. Since we have $\vc$ possibilities for choosing a representative vertex $v_s\in S$, by Lemma~\ref{lem:Bollobas1}, there are at most $(e\tilde{d})^{i-1}$ connected sets containing this vertex.
    If $d<3$, note that by adding edges to $\Gamma$ the number of connected sets can only increase. Thus, we may repeat the same argument above and use Lemma~\ref{lem:Bollobas1} with $\tilde{d}=\max(2,d-1)$ (indeed, we can add at most two edges to $\Gamma$, and apply the bound). This proves \eqref{eq:quantity1}. Multiplying \eqref{eq:quantity1} over all $m$ connected components, we get that there are at most
    \begin{align}
        \prod_{i=1}^\ell \p{\vc\cdot (e  \tilde{d})^{i-1}}^{p(i)}=\vc^m\cdot (e  \tilde{d})^{\ell-m} \label{eq:overcountvertices1}
    \end{align}
    ways to choose the $\ell$ vertices on which the graph $\s{H}$ is supported; the equality in \eqref{eq:overcountvertices1} follows from \eqref{eq:partitiondef1}. We remark that by multiplying \eqref{eq:overcountvertices1} over $i$, we surely over count, as we also consider connected sets which may overlap. 
    
     With the vertices of the graph fixed, let us now bound the number of connected subgraphs supported on these vertices, with a total number of $j$ edges. Let $U\in v(\Gamma)$ be a fixed subset of vertices on which our connected components are supported. Because $G_{U}\subseteq \Gamma$, we have,
    \begin{align}
        |e(G_U)|\leq \mu(\Gamma)   \cdot |v(G_U)|=\mu  \ell.
    \end{align}
    Furthermore, it is clear that, 
    \begin{align}
        |e(G_U)|\leq \binom{|U|}{2}=\binom{\ell}{2},
    \end{align}
    and when combined with the above,  the number of subgraphs of $G_U$ can be upper bounded by, 
    \begin{align}
        \binom{|e(G_U)|}{j}\leq \binom{\min\p{\floor{ \mu \ell},\binom{\ell}{2}}}{j}.\label{eqn:boundSubG_U}
    \end{align}
    Finally, combining \eqref{eqn:enumPar}, \eqref{eq:overcountvertices1}, and \eqref{eqn:boundSubG_U}, we get,
    \begin{align}
        |\calS_{m,\ell,j} |&=\sum_{p\in\s{Par}(\ell,m)}\abs{S(p,j)}\\
        &\leq \sum_{p\in\s{Par}(\ell,m)}\vc^m\cdot (e\tilde{d})^{\ell-m}\cdot \binom{\min\p{\floor{\mu  \ell},\binom{\ell}{2}}}{j}\\
        &= \abs{\s{Par}(\ell,m)}\vc^m\cdot (e\tilde{d})^{\ell-m}\cdot \binom{\min\p{\floor{\mu \ell},\binom{\ell}{2}}}{j},
    \end{align}
    which concludes the proof. 
    \end{proof}
Finally, we prove Theorem~\ref{th:lowerVCD} as follows.
 \begin{proof}[Proof of Theorem~\ref{th:lowerVCD}] Applying Lemmas~\ref{lem:probRandomSubgraph1} and \ref{lem:subgraphsBoundedDegreeCount1} on \eqref{eqn:InterBound}, we get,
 \begin{align}
      \E_{\calH_0}[\s{L(G)}^2]&\leq1+\sum_{\ell=2}^k \sum_{m=1}^{\floor{\ell/2}}\sum_{j=\ell-m}^{\floor{\mu \ell}} |\calS_{m,\ell,j} |\cdot \lambda^{2j} \cdot \frac{(2\vc)^m d^{\ell-m}}{(n-k)^\ell}\label{eq:GenralLB-GeneralMuBegin}\\
      &\leq 1+\sum_{\ell=2}^k \sum_{m=1}^{\floor{\ell/2}} \frac{(2\vc)^m d^{\ell-m}}{(n-k)^\ell} \sum_{j=\ell-m}^{\floor{\mu \ell }}  \lambda^{2j} \cdot \abs{\s{Par}(\ell,m)}\vc^m\cdot (e\tilde{d})^{\ell-m}\binom{\floor{\mu \ell }}{j}\\
      &\leq 1+\sum_{\ell=2}^k \sum_{m=1}^{\floor{\ell/2}} \abs{\s{Par}(\ell,m)}\frac{2^m \vc^{2m} (ed^2)^{\ell-m}}{(n-k)^\ell} \sum_{j=\ell-m}^{\floor{\mu \ell }}  \lambda^{2j} \cdot \binom{\floor{\mu \ell }}{j}\\
      &\leq  1+\sum_{\ell=2}^k \sum_{m=1}^{\floor{\ell/2}} \abs{\s{Par}(\ell,m)}\frac{2^m \vc^{2m} (ed^2)^{\ell-m}}{(n-k)^\ell}(1+\lambda^2)^{\floor{\mu \ell }}\\
      &\overset{(a)}{\leq} 1+\sum_{\ell=2}^k \sum_{m=1}^{\floor{\ell/2}} e^{c \sqrt{\ell}}\cdot \frac{2^m \vc^{2m} (ed^2)^{\ell-m}}{(n-k)^\ell}(1+\lambda^2)^{\mu \ell }\\
      &= 1+\sum_{\ell=2}^k  e^{c \sqrt{\ell}}\p{\frac{(1+\lambda^2)^{\mu}\cdot ed^2}{n-k}}^\ell\sum_{m=1}^{\floor{\ell/2}} \p{ \frac{2\vc^{2} }{ed^2}}^m\label{eq:innerSum1}\\
       &\overset{(b)}{\leq } 1+C_\varepsilon \cdot \sum_{\ell=2}^k  \p{\frac{(1+\varepsilon)(1+\lambda^2)^{\mu}\cdot ed^2}{n-k}}^\ell\sum_{m=1}^{\floor{\ell/2}} \p{ \frac{2\vc^{2} }{ed^2}}^m,\label{eq:generallowerSplit1}\\
\end{align}
where $(a)$ follows from the Hardy-Ramanujan formula (see, e.g., \cite[Chapter 1.3]{flajolet2009analytic}), \begin{equation}\label{eq:Hardy-Ramanujan}
    \abs{\s{Par}(\ell,m)}\leq e^{c\sqrt{\ell}},
\end{equation} 
for some constant $c>0$, and $(b)$ follows since $e^{c\sqrt{\ell}}$ is sub-exponential. We now separate our analysis into two complementary cases. In the first, we assume that,
\begin{align}
             \frac{2\vc^2}{ed^2}\geq 1+\alpha,\label{eq:CondDenseInProof}
         \end{align}
for some fixed positive $\alpha$, independent of $n$. In this case, the second sum on the r.h.s. of \eqref{eq:generallowerSplit1} reduces to,
        \begin{align}
            \sum_{m=1}^{\floor{\ell/2}} \p{ \frac{2\vc^{2} }{ed^2}}^m\leq C\cdot \p{\frac{2\vc^{2} }{ed^2}}^{\frac{\ell}{2}},\label{eq:FracGadokMe1}
        \end{align}
for some constant $C$. We therefore obtain,
        \begin{align}
             \E_{\calH_0}[\s{L(G)}^2]
      &\leq 1+C'_\varepsilon \cdot \sum_{\ell=2}^k  \p{\frac{(1+\varepsilon)(1+\lambda^2)^{\mu} \vc  \sqrt{2ed^2}}{n-k}}^\ell,\label{eq:GenralLB-GeneralMuEnd}
        \end{align}
which is bounded (and accordingly strong detection is impossible) provided that,
     \begin{align}
         \frac{(1+\varepsilon)(1+\lambda^2)^{\mu}\cdot \vc d \cdot \sqrt{2e}}{n-k}< 1-\delta.  \label{eq:innderCondDom1}
     \end{align}
for some fixed $\delta>0$ and $\varepsilon>0$, and converge to unity (and accordingly weak detection is impossible) if $1-\delta$ is replaced with an $o(1)$ function. Next, we move to the other case, where, 
         \begin{align}
             \frac{2\vc^2}{ed^2}\leq 1+o(1).
         \end{align}
     Here, for any $\varepsilon>0$, we have that for sufficiently large $n$, the second sum on the r.h.s. of  \eqref{eq:generallowerSplit1} satisfies,
        \begin{align}
            \sum_{m=1}^{\floor{\ell/2}} \p{ \frac{2\vc^{2} }{ed^2}}^m\leq C\cdot (1+\varepsilon)^\ell,
        \end{align}
and thus,
\begin{align}
         \E_{\calH_0}[\s{L(G)}^2]&\leq 1+C''_\varepsilon  \sum_{\ell=2}^k  \p{\frac{(1+\varepsilon)^2\cdot (1+\lambda^2)^{\mu}\cdot ed^2}{n-k}}^\ell.\label{eq:casesAnoying1}
     \end{align}
    The above is bounded provided that,
     \begin{align}
         \frac{(1+\varepsilon)^2(1+\lambda^2)^{\mu}\cdot ed^2}{n-k}<1-\delta, \label{eq:innderCondDom2}
     \end{align}
     for some fixed $\delta>0$ and $\varepsilon>0$, and converge to unity whenever the $1-\delta$ is replaced with any $o(1)$ function.     Finally, we observe that the condition in \eqref{eq:CondDenseInProof} holds exactly when the condition in \eqref{eq:innderCondDom1} dominates the one in \eqref{eq:innderCondDom2} (perhaps, up to a multiplicative constant factor). Thus, we can find a sufficiently small constant $C>0$, such that $\E_{\calH_0}[\s{L(G)}^2]$ remains bounded, as long as,
    \begin{align}
        \frac{(1+\lambda^2)^{\mu}}{n-k}\cdot \max(\vc d, d^2)<C,
    \end{align}
    which concludes the proof.
\end{proof}

\subsubsection{Reduction to the vertex cover-degree balanced case}\label{sec:ReudctionToBalanced}

There is a gap between the lower bound in Theorem~\ref{th:lowerVCD} and the upper bounds in Theorem~\ref{thm:upperBoundAlgo}, which closes if $\Gamma$ is $\vcd$-balanced. As discussed in Subsection~\ref{subsec:outline}, we address this gap and prove that Theorem~\ref{thm:upperBoundAlgo} is tight through two main arguments: 1) We show that if $\Gamma$ can be decomposed into a set of $\vcd$-balanced graphs, then detection is impossible as long as each subgraph in the decomposition is undetectable, and 2) we construct such a decomposition. In this subsection, we focus on the first argument. Consider the subsequent definition.
\begin{definition}
    Let $G=(V,E)$ be simple graph containing no isolated vertices. We say that $G$ is decomposed into a set of $M$ edge-disjoint graphs $\{G_\ell=(V_\ell,E_\ell)\}_{\ell=1}^M$, and denote $G=\bigcup_{i=1}^MG_i$, if $E$ is the disjoint union of the sets $E_1,\dots,E_M$, and for all $\ell$ we have $V_\ell\subseteq V$. 
\end{definition}
The following is the main result of this subsection. 
\begin{prop}\label{prop:decompisitionThD}
    Let $\Gamma=(\Gamma_n)_n$ be a sequence of graphs, and assume that each can be decomposed into a set of edge-disjoint graphs, i.e., $\Gamma=\bigcup_{\ell=1}^M \Gamma_\ell$. Then, $\E_{\calH_0}[\s{L}(\s{G})^2]=O(1)$ if 
    \begin{align}
        \max_{\ell=1,\dots, M}\p{\frac{(1+\lambda_M^2)^{\mu(\Gamma_\ell)}}{n-|v(\Gamma_\ell)|}\cdot\max\p{\vc(\Gamma_\ell)\cdot d_{\max}(\Gamma_\ell),d_{\max}(\Gamma_\ell)^2}}\leq C,\label{eq:condintersectionD}
    \end{align}
    where $\lambda_M^2\triangleq(1+\lambda^2)^{M^2}-1$, and $C>0$ is a universal constant.
\end{prop}
To prove Proposition~\ref{prop:decompisitionThD} we need the following generalization of Theorem~\ref{th:lowerVCD}, which provides the conditions under which the moment generating function of the intersection of \emph{two arbitrary} subgraphs is bounded.
\begin{lemma}\label{lem:interectionD}
      Let $\Gamma_1=(\Gamma_{1,n})_n$ and $\Gamma_2=(\Gamma_{2,n})_n$ be two sequences of graphs. If,
     \begin{align}
         \frac{(1+\lambda^2)^{\min(\mu(\Gamma_1),\mu(\Gamma_2))}}{n-\min\p{|v(\Gamma_1)|,|v(\Gamma_2)|}}\max\p{\sqrt{\vc(\Gamma_1)\vc(\Gamma_2)d_{\max}(\Gamma_1)d_{\max}(\Gamma_2)},d_{\max}(\Gamma_1)d_{\max}(\Gamma_2)}\leq C,\label{eq:condIntersectionD1}
     \end{align}
     for some $C>0$, then,
    \begin{align}
        \E_{\Gamma_2}\pp{(1+\lambda^2)^{|e(\Gamma_1'\cap \Gamma_2)|}}=O(1),\label{eq:condIntersectionD2}
    \end{align}
     where the expectation is taken w.r.t. $\Gamma_2\sim\s{Unif}(\calS_{\Gamma_2})$ and $\Gamma_1'$ is fixed copy of $\Gamma_1$ in $\calK_n$. 
\end{lemma}

\begin{proof}[Proof of Lemma~\ref{lem:interectionD}]
    Throughout this proof, we denote $\vc_i=\vc(\Gamma_i)$, $d_i=d_{\max}(\Gamma_i)$, $\mu_i=\mu(\Gamma_i)$, and $k_i=|v(\Gamma_i)|$, for $i=1,2$. The proof follows by generalizing the arguments in the proof of Theorem~\ref{th:lowerVCD}. Specifically, recall that by Lemma~\ref{lem:intersectionMomentGenGeneral} we have (by symmetry we can take $\Gamma_1$ to be a fixed copy $\Gamma_1'$),
   \begin{align}
    \E_{\Gamma_1}\pp{(1+\lambda^2)^{|e(\Gamma_1'\cap\Gamma_2)|}}
    &=\sum_{\s{H}\subseteq \Gamma_1'}\lambda^{2|\s{H}|}\cdot \P_{\Gamma_2} \pp{\s{H} \subseteq \Gamma_2}.\label{eqn:momenGenG1G2}
\end{align}
As in the proof of Theorem~\ref{th:lowerVCD}, we use Lemma~\ref{lem:probRandomSubgraph1} to upper bound the probability $\P_{\Gamma_2}[\s{H}\subseteq \Gamma_2]$, which in turn implies that the summand in \eqref{eqn:momenGenG1G2} depends on $\s{H}$ through $(|v(\s{H})|,|e(\s{H})|,m(\s{H}))$ only. Note that any subgraph $\s{H}$ of $\Gamma_1'$ for which the probability $\P_{\Gamma_2}[\s{H}\subseteq \Gamma_2]$ is non-zero, must be isomorphic to a subgraph of $\Gamma_2$. Thus, the number of vertices and edges in $\s{H}$ cannot exceed $\min(k_1,k_2)=\bar{k}$ and $|v(\s{H})|\cdot \min(\mu_1,\mu_2)\triangleq |v(\s{H})|\cdot \bar{\mu}$,  respectively. Denote by $\calS_{m,\ell,j}(\Gamma_1)$ the number of subgraphs of $\Gamma_1$ with exactly $m$ connected components, $\ell$ vertices, and $j$ edges. Using Lemma~\ref{lem:probRandomSubgraph1} and \ref{lem:subgraphsBoundedDegreeCount1}, we get,
\begin{align}
      \E_{\Gamma_2}&\pp{(1+\lambda^2)^{|e(\Gamma_1'\cap\Gamma_2)|}}\\
      &=\sum_{\s{H}\subseteq \Gamma_1'}\lambda^{2|\s{H}|}\cdot \P_{\Gamma_2}[\s{H}\subseteq \Gamma_2]\\
      & \leq1+\sum_{\ell=2}^{\bar{k}} \sum_{m=1}^{\floor{\ell/2}}\sum_{j=\ell-m}^{\floor{\ell \bar{\mu}}} |\calS_{m,\ell,j}(\Gamma_1)|\cdot \lambda^{2j} \cdot \frac{(2\vc_2)^m d_2^{\ell-m}}{(n-{\bar{k}})^\ell}\\
      &\leq 1+\sum_{\ell=2}^{\bar{k}} \sum_{m=1}^{\floor{\ell/2}} \frac{(2\vc_2)^m d_2^{\ell-m}}{(n-{\bar{k}})^\ell} \sum_{j=\ell-m}^{\floor{\ell \bar{\mu} }}  \lambda^{2j} \cdot \abs{\s{Par}(\ell,m)}\vc_1^m\cdot (ed_1)^{\ell-m}\binom{\floor{\ell \bar{\mu}}}{j}\\
      &\leq 1+C_\varepsilon \cdot \sum_{\ell=2}^{\bar{k}}  \p{\frac{(1+\varepsilon)(1+\lambda^2)^{\bar{\mu}}\cdot e  d_1d_2}{n-{\bar{k}}}}^\ell\sum_{m=1}^{\floor{\ell/2}} \p{ \frac{2 \vc_1\vc_2 }{e d_1 d_2}}^m.\label{eq:intersectionSums1}
\end{align}
On one hand, if, 
\begin{align}
   \frac{2\vc_1\vc_2}{ed_1d_2}\geq 1+\alpha,\label{eq:CondDenseInProofInt}
\end{align} 
for some $\alpha>0$, then the inner sum on the r.h.s. of \eqref{eq:intersectionSums1} is dominated by $\p{\frac{2\vc_1\vc_2}{ed_{1}d_2}}^{\ell/2}$, and accordingly,
\begin{align}
    \E_{\Gamma_2}\pp{(1+\lambda^2)^{|e(\Gamma_1'\cap \Gamma_2)|}}
    &\leq 1+C'_\varepsilon \cdot \sum_{\ell=2}^{\bar{k}}  \p{\frac{(1+\varepsilon)(1+\lambda^2)^{\bar{\mu}}\cdot \sqrt{2e \cdot  \vc_1\vc_2 \cdot d_1 d_2}}{n-{\bar{k}}}}^\ell ,
\end{align}
which is bounded provided that,
\begin{align}
    \frac{(1+\lambda^2)^{\bar{\mu}}\cdot \sqrt{ \vc_1\vc_2 \cdot d_1 d_2}}{n-{\bar{k}}}\leq C,\label{eq:innderCondDom1Int}
\end{align}
for a sufficiently small universal constant $C>0$. On the other hand, if,
\begin{align}
    \frac{2 \vc_1\vc_2}{e d_1 d_2 }\leq 1+o(1),
\end{align}
then the inner sum on the r.h.s. of \eqref{eq:intersectionSums1} is dominated by $(1+\varepsilon)^{\ell}$, and accordingly,
\begin{align}
    \bE_{\Gamma_2}\pp{(1+\lambda^2)^{|e(\Gamma_1'\cap \Gamma_2)|}}& \leq 1+C''_\varepsilon \cdot \sum_{\ell=2}^{\bar{k}}  \p{\frac{(1+\varepsilon)^2(1+\lambda^2)^{\bar{\mu}}\cdot e d_1 d_2}{n-{\bar{k}}}}^\ell ,
\end{align}
which is bounded provided that,
\begin{align}
    \frac{(1+\lambda^2)^{\bar{\mu}}\cdot d_1 d_2}{n-{\bar{k}}}\leq C,\label{eq:innderCondDom2Int}
\end{align}
for a sufficiently small constant  universal $C$. Finally, we note that \eqref{eq:CondDenseInProofInt} holds exactly when \eqref{eq:innderCondDom1Int} dominates \eqref{eq:innderCondDom2Int} (perhaps, up to a multiplicative constant factor). Thus, the condition in \eqref{eq:condIntersectionD1} is sufficient for \eqref{eq:condIntersectionD2} to hold.
\end{proof}

We can now prove Proposition~\ref{prop:decompisitionThD}.
\begin{proof}[Proof of Proposition~\ref{prop:decompisitionThD}] 
For any $1\leq\ell\leq M$, we denote $k_\ell\triangleq|v(\Gamma_\ell)|$, $d_{\ell}\triangleq d_{\max}(\Gamma_\ell)$, $\mu_\ell\triangleq\mu(\Gamma_\ell)$, and $\vc_\ell\triangleq\vc(\Gamma_\ell)$. Below, we first assume that $M$ is fixed, independent with $n$, and treat the case where $M = M_n$. Using Proposition~\ref{prop:LikelihoodMomentExperession} and H\"{o}lder's inequality, we have,
\begin{align}
    \E_{\calH_0}\pp{\s{L}(\s{G})^2}&=\E_{\Gamma}\pp{(1+\lambda^2)^{|e(\Gamma\cap \Gamma')|}}\\
    &=\E_{\Gamma}\pp{(1+\lambda^2)^{\bigcup_{i,j=1}^M |e(\Gamma_i\cap \Gamma'_j)|}}\\
    &=\E_{\Gamma}\pp{\prod_{i,j=1}^M(1+\lambda^2)^{ |e(\Gamma_i\cap \Gamma'_j)|}}\\
    &\leq  \prod_{i,j=1}^M\E_{\Gamma}^{\frac{1}{M^2}}\pp{(1+\lambda^2)^{ M^2\cdot |e(\Gamma_i\cap \Gamma'_j)|}}\\&
      =\prod_{i,j=1}^M\E_{\Gamma_i}^{\frac{1}{M^2}}\pp{(1+\lambda^2)^{ M^2\cdot |e(\Gamma_i\cap \Gamma'_j)|}} \label{eq:GammaIeq}\\
      &=\prod_{i,j=1}^M\E_{\Gamma_i}^{\frac{1}{M^2}}\pp{(1+\lambda^2_M)^{ |e(\Gamma_i\cap \Gamma'_j)|}},\label{eq:GammaIeq2}
    \end{align}
    where for all $i,j\in[M]$, $\Gamma'_j$ is a fixed copy of $\Gamma_j$ in $\calK_n$ (induced by the fixed copy $\Gamma'$), and the expectation in \eqref{eq:GammaIeq} and \eqref{eq:GammaIeq2} is taken w.r.t. a uniformly distributed random copy of $\Gamma_i$ in $\calK_n$. Let $C>0$ be the smallest constant for which Theorem~\ref{th:lowerVCD} and Lemma~\ref{lem:interectionD} hold simultaneously. For $i=j$, by Theorem~\ref{th:lowerVCD}, we know that,
    \begin{align}
        \E_{\calH_0}\pp{(1+\lambda_M^2)^{|e(\Gamma_i\cap\Gamma_i)|}}=O(1),
    \end{align}
    provided that,
    \begin{align}
        \frac{(1+\lambda_M^2)^{\mu_i}\max(\vc_i d_i,d_i^2)}{n-k_i}\leq C.
    \end{align}
    This, however, is satisfied when \eqref{eq:condintersectionD} holds because, 
    \begin{align}
        \frac{(1+\lambda_M^2)^{\mu_i}\max(\vc_i d_i,d_i^2)}{n-k_i}\leq \max_{1\leq j\leq M}{\frac{(1+\lambda_M^2)^{\mu_j}\max(\vc_j d_j,d_j^2)}{n-k_j}}\leq C.
    \end{align}
    Next, for $i\neq j$, we note that,
    \begin{align}
        \frac{(1+\lambda_M^2)^{\min(\mu_i,\mu_j)}\max(\sqrt{\vc_i \vc_j d_i d_j},d_i d_j)}{n-\min(k_i,k_j)}&\leq  \frac{(1+\lambda_M^2)^{\min(\mu_i,\mu_j)}\max(\sqrt{\vc_i \vc_j d_i d_j},d_i d_j)}{n-\min(k_i,k_j)}\\
        &\leq  \frac{(1+\lambda_M^2)^{\min(\mu_i,\mu_j)}\max(\vc_i d_i,\vc_j d_j, d_i^2, d_j^2)}{n-\min(k_i,k_j)}\\
        &\leq \max_{\ell=i,j} \frac{(1+\lambda_M^2)^{\mu_\ell}\max(\vc_\ell d_\ell ,d_\ell^2)}{n-k_\ell}\\
        &\leq \max_{1\leq \ell\leq M} \frac{(1+\lambda_M^2)^{\mu_\ell}\max(\vc_\ell d_\ell,d_\ell^2)}{n-k_\ell}\leq C,
    \end{align}
    where the last equality is due to \eqref{eq:condintersectionD}, and therefore, by Lemma~\ref{lem:interectionD},
    \begin{align}
        \E_{\calH_0}\pp{(1+\lambda_M^2)^{|e(\Gamma_i\cap\Gamma_j)|}}=O(1).
    \end{align}
    Thus, we proved that under \eqref{eq:condintersectionD} each term in the product on the r.h.s. of \eqref{eq:GammaIeq2} is bounded, and since $M$ is finite, the entire product of \eqref{eq:GammaIeq2} is bounded as well. Finally, to deal with the case where $M=M_n$ depends on $n$, we note that the proofs of Theorem~\ref{th:lowerVCD} and Lemma~\ref{lem:interectionD} provide uniform bounds on the second moment. To be precise, it is easily seen from the proofs that there is a sufficiently small constant  $C>0$ such that if the conditions in \eqref{eq:condIntersectionD1} and \eqref{eq:LowerMainCond1} are satisfied, then,
    \begin{align}
        \E_{\calH_0}\pp{(1+\lambda_M^2)^{|e(\Gamma_i\cap\Gamma_j')|}}\leq 2,
    \end{align}
    and accordingly, the geometric average in \eqref{eq:GammaIeq2} is bounded by $2$ as well. This concludes the proof.
 
\end{proof}

\subsubsection{Degree-vertex cover balanced decomposition}\label{subsec:coverBalancedDecomposition} 

In this subsection, we construct a locally $\vcd$-balanced decomposition of $\Gamma$, and thus establishing our second argument. We have the following result.
\begin{prop}\label{prop:GoodDecompositionExists}
    Any non-empty graph $\Gamma$ can be decomposed into a set of edge-disjoint subgraphs $\{\Gamma_\ell\}_{\ell=1}^M$, such that,
    \begin{align}
        \max_{1\leq \ell \leq M} \p{\vc(\Gamma_\ell)\cdot d_{\max}(\Gamma_\ell)} \leq 2 |e(\Gamma)|\cdot d_{\max}(\Gamma)^{\frac{1}{M}},\label{eq:decompositionPropoety}
    \end{align}
    where the decomposition may includes some empty graphs.
\end{prop}

\begin{proof}[Proof of Proposition~\ref{prop:GoodDecompositionExists}]
Let $k\triangleq|v(\Gamma)|$, and fix $M\in\mathbb{N}$. Consider the following recursive decomposition of $\Gamma$ to at most $M$ edge-disjoint subgraphs $\{\Gamma_\ell\}_{\ell=1}^M$. Let $v_1,\dots,v_{k}$ be an enumeration of $v(\Gamma)$ ordered in a descending ordering according to the graph degrees in $\Gamma$.  Let $\ell_i\in[|v(\Gamma)|]$ denote the maximal index for which $\deg(v_{\ell_i})\geq [d_{\max}(\Gamma)]^{\frac{M-i}{M}}$, that is, 
\begin{align}
    \ell_i\triangleq\sup\ppp{\ell\in[k]:\deg_{\Gamma}(v_\ell)\geq d_{\max}(\Gamma)^{\frac{M-i}{M}}}.
\end{align}
and set $\ell_0\triangleq 0$. We now define $\Gamma_1,\dots,\Gamma_M$ recursively. For $\ell=1$, let $E_1$ be the set of edges with (at least) one of their endpoints is in the set $S_1\triangleq\{v_1,\ldots,v_{\ell_1}\}$ and $V_1$ be the set of all vertices in $v(\Gamma)$ which participate in an edge of $E_1$. We define $\Gamma_1$ as $\Gamma_1=(V_1,E_1)$. Given $\Gamma_1,\dots, \Gamma_{i-1}$, we define $\Gamma_i$. We let $E_i$ be the set of all edges in $e(\Gamma)\setminus \cup_{j=1}^{i-1}E_j$ with (at least) one of their endpoints is in the set $S_i\triangleq\{v_{\ell_{i-1}+1},\ldots,v_{\ell_i}\}$, and $V_i$ be the set of all vertices in $v(\Gamma)$ which are one of the endpoints of any edge in $E_i$. 

Now, we note that for all $i\in[M]$, by the definition of $\Gamma_i$, the set $S_i$ is a vertex cover of $\Gamma_i$ (perhaps, not a minimal one). Furthermore, we observe that,
\begin{align}
      d_{\max}(\Gamma_i) \leq  d_{\max}(\Gamma)^{\frac{M-i+1}{M}}.\label{eq:degreeDocomp}
\end{align}
Indeed, let $v\in V_i$ be a vertex. If $v\in S_i$, then by the definition of $S_i$,
\begin{align}
    \deg_{\Gamma_i}(v)\leq  \deg_{\Gamma}(v) \leq  d_{\max}(\Gamma)^{\frac{M-i+1}{M}}.
\end{align}
In the other case, if $v\in V_i\setminus S_i$, then assume to the contrary that,
\begin{align}
    \deg_{\Gamma_i}(v)> d_{\max}(\Gamma)^{\frac{M-i+1}{M}}.
\end{align}
By the definition of $S_1,\dots,S_i$, we have that $v\in S_j$ for some $j<i$. In particular, for any edge $e$ which includes $v$ we have $e\in E_j$, and therefore,
\begin{align}
    e\notin\p{ e(\Gamma)\setminus \bigcup_{\ell=1}^{i-1}E_\ell}  \supseteq E_i.
\end{align}
In particular $v\notin V_i$, which is a contradiction. Next, since the smallest degree (in $\Gamma$) of any vertex in $S_i$ is at least $d_{\max}(\Gamma)^{\frac{M-i}{M}}$, we have,
\begin{align}
    |S_i|\cdot  d_{\max}(\Gamma)^{\frac{M-i}{M}}\leq \sum_{\ell=\ell_{i-1}+1}^{\ell_i} \deg_{\Gamma}(v_\ell)\leq \sum_{\ell=1}^{k} \deg_{\Gamma}(v_\ell) \leq 2 |e(\Gamma)|,
\end{align}
and thus,
\begin{align}
    \vc(\Gamma_i)&\leq |S_i|\leq 2|e(\Gamma)|\cdot d_{\max}^{-\frac{M-i}{M}}(\Gamma).
\end{align}
Combining the above we get, 
\begin{align}
  \vc(\Gamma_i)\cdot  d_{\max}(\Gamma_i)\leq 2|e(\Gamma)|\cdot d_{\max}(\Gamma) ^{\frac{M-i+1}{M}-\frac{M-i}{M}}= 2|e(\Gamma)|\cdot d_{\max}(\Gamma)^{\frac{1}{M}},
\end{align}
for all $i\in[M]$. It should be emphasized that it could be the case that some of the subgraphs in the decomposition above might be empty; indeed it could be that $\ell_i=\ell_{i+1}$, for some $i\in[M]$, and then $S_i$ is empty. Nonetheless, the above decomposition is certainly not empty because $\Gamma_1\neq\emptyset$. This concludes the proof.
\end{proof}

\subsubsection{Combining everything together}
In this subsection, we combine all the results established in the previous subsections and prove Theorem~\ref{th:lowerboundDense}.
\begin{proof}[Proof of Theorem~\ref{th:lowerboundDense}]
    Let $C>0$ be the universal constant from Proposition~\ref{prop:decompisitionThD}, and let $\Gamma=(\Gamma_n)_n$ be a sequence of graphs. We separate our proof into two cases. We first consider graphs $\Gamma$ with super-logarithmic density, and then those with sub-logarithmic density.

    \paragraph{Super-logarithmic density.}
    Assume that $\mu(\Gamma)\geq \alpha\log|v(\Gamma)|$ for a fixed $\alpha>0$. Using Theorem~\ref{th:lowerVCD}, we know that strong detection is impossible if,
    \begin{align}
        \mu(\Gamma)\log(1+\lambda^2)+\log\p{\max\p{\vc(\Gamma)d_{\max}(\Gamma),d_{\max}(\Gamma)^2}}\leq  \log C + \log n+\log\delta,\label{eq:explainDenseSuer}     \end{align}
    In the super-logarithmic density regime we have,
    \begin{align}
     \log\p{\max\p{\vc(\Gamma),d_{\max}(\Gamma)}}\leq \log|v(\Gamma)|\leq \frac{\mu(\Gamma)}{\alpha},
    \end{align}
   and therefore,
    \begin{align}
        \mu(\Gamma)\log(1+\lambda^2)+\log\p{\max\p{\vc(\Gamma)d_{\max}(\Gamma),d_{\max}(\Gamma)^2}}\leq \frac{\alpha\log(1+\lambda^2)+2}{\alpha}\cdot \mu(\Gamma).
    \end{align}
    Thus, \eqref{eq:explainDenseSuer} holds if,
    \begin{align}
        \mu(\Gamma)\leq \frac{(1-\varepsilon)\alpha}{2+\alpha\log(1+\lambda^2)}\log n,
    \end{align}
    for some $\varepsilon>0$, which proves \eqref{eqn:MainResDenseLower1}. The corresponding upper bound in \eqref{eqn:MainResDenseUpper1} follows from Theorem~\ref{thm:upperBoundAlgo} immediately.
    
    \paragraph{Sub-logarithmic density.} Assume now that $\Gamma$ has a sub-logarithmic density, i.e., $\mu(\Gamma)=o(\log|v(\Gamma)|)$. Fix $\varepsilon>0$. From Proposition~\ref{prop:GoodDecompositionExists} we know that there exists a decomposition $\Gamma=\bigcup_{\ell=1}^M\Gamma_\ell$, for which \eqref{eq:decompositionPropoety} holds, and we take,
    \begin{align}
        M=\ceil{\frac{2}{\varepsilon}}.
    \end{align} By Proposition~\ref{prop:decompisitionThD} strong detection is impossible if,
    \begin{align}
         \max_{\ell=1,\dots, M}\p{\frac{(1+\lambda_M^2)^{\mu(\Gamma_\ell)}}{n-|v(\Gamma_\ell)|}\cdot\max\p{\vc(\Gamma_\ell)\cdot d_{\max}(\Gamma_\ell),d_{\max}(\Gamma_\ell)^2}}\leq C,\label{eq:CondFinalD}
    \end{align}
     where 
     \begin{align}
         \lambda_M^2&=(1+\lambda^2)^{M^2}-1=(1+\lambda^2)^{4\varepsilon^{-2}}-1.\label{eq:lambdaMsize}
     \end{align}
     Now, by \eqref{eq:decompositionPropoety} and the definition of $\mu(\Gamma)$, we have,
     \begin{align}
          \max_{\ell=1,\dots, M}&\p{\frac{(1+\lambda_M^2)^{\mu(\Gamma_\ell)}}{n-|v(\Gamma_\ell)|}\cdot\max\p{\vc(\Gamma_\ell)\cdot d_{\max}(\Gamma_\ell),d_{\max}(\Gamma_\ell)^2}}\\
          &\qquad\leq   \frac{(1+\lambda_M^2)^{\mu(\Gamma)}}{n-|v(\Gamma)|}\cdot\max\p{|e(\Gamma)|\cdot d_{\max}(\Gamma)^{\frac{1}{M_n}},d_{\max}(\Gamma)^2}\\
          &\qquad\leq  \frac{(1+\lambda_M^2)^{\mu(\Gamma)}}{\delta n}\cdot\max\p{|e(\Gamma)|\cdot d_{\max}(\Gamma)^{\frac{1}{M_n}},d_{\max}(\Gamma)^2}\\
            &\qquad\leq \frac{(1+\lambda^2)^{4\frac{\mu(\Gamma)}{\varepsilon^2}}}{\delta n^{1-\frac{1}{M_n}}}\cdot\max\p{|e(\Gamma)|,d_{\max}(\Gamma)^2}\\
            &\qquad= \frac{|v(\Gamma)|^{\frac{4\log(1+\lambda^2)\mu(\Gamma)}{\varepsilon^2\log|v(\Gamma)|}}}{\delta n^{1-\frac{1}{M_n}}}\cdot\max\p{|e(\Gamma)|,d_{\max}(\Gamma)^2}\\
            &\qquad\leq \frac{n^{\frac{4\log(1+\lambda^2)\mu(\Gamma)}{\varepsilon^2\log|v(\Gamma)|}}}{\delta n^{1-\frac{1}{M_n}}}\cdot\max\p{|e(\Gamma)|,d_{\max}(\Gamma)^2},
            \label{eq:DenseNiceEq}
     \end{align}
     where the first equality follows from the fact that $\mu(\Gamma_\ell)\leq\mu(\Gamma)$ and $|v(\Gamma_\ell)|\leq |v(\Gamma)|$, for all $\ell\in[M]$, the second equality is because $|v(\Gamma)|\leq (1-\delta)n$, for some fixed $\delta>0$, the third equality follows from $d_{\max}(\Gamma)\leq |v(\Gamma)|$, and the last equality is because $|v(\Gamma)|\leq n$. Since $\mu(\Gamma)=o(\log|v(\Gamma)|)$, we have that for a sufficiently large $n$,
     \begin{align}
         \frac{4\log(1+\lambda^2)\mu(\Gamma)}{\varepsilon^2\log|v(\Gamma)|}\leq \frac{\varepsilon}{2},
     \end{align}
     and therefore, \eqref{eq:DenseNiceEq} is upper bounded by, 
     \begin{align}
         \frac{\max\p{|e(\Gamma)|,d_{\max}(\Gamma)^2}}{ n^{1-\varepsilon}}.
     \end{align}
     Thus, if 
     \begin{align}
         \max\p{|e(\Gamma)|,d_{\max}(\Gamma)^2}\leq  n^{1-\varepsilon},
     \end{align}
     then \eqref{eq:CondFinalD} is satisfied for a sufficiently large $n$ (regardless of the constant $C$), which proves \eqref{eqn:MainResDenseLower2}. Finally, the corresponding upper bound in \eqref{eqn:MainResDenseUpper2} follows from Theorem~\ref{thm:upperBoundAlgo}, for, e.g., $f(n)=\log_n(17\log n)=o(1)$ in \eqref{eqn:MainResDenseUpper2}.
    
\end{proof}
%\Dor{Should we discuss other decompositions? perhaps in light of Example~\ref{ex:StartUnbalanced}}

%\subsubsection{Collecting everything together}
%\Dor{Not sure what to write here, it also relates to the question of how we want to present our results. I'll skip that for now.}

\subsection{Other regimes}\label{subsec:otherRegimes}
In this subsection, we will present and prove the most refined version of our lower bounds, and subsequently analyze those bounds
in the two distinct regimes that garnered the most attention in the literature:
\begin{enumerate}
    \item \emph{Sparse regime}, where $\chi^2(p||q)=\Theta(n^{-\alpha})$, for $0<\alpha\leq 2$. 
    \item \emph{Critical regime}, where $p=1-o(1)$ and $q = \Theta(n^{-\alpha})$, for $0\leq \alpha\leq 2$ (in particular, $\chi^2(p||q)=\Theta(n^{\alpha})$). 
\end{enumerate}
\subsubsection{General bounds}
Building on the ideas and concepts introduced in Section~\ref{subsec:denseRegime}, in this subsection we prove the following result, which can be viewed as an extension of Theorem~\ref{th:lowerVCD}.
\begin{theorem}\label{th:lowerVCFull}
  There exists a constant $C>0$ such that for any sequence of graphs $\Gamma=(\Gamma_n)_n$, we have that $\E_{\calH_0}[\s{L}(\s{G})^2]=O(1)$, if at least one of the following hold.
             \begin{enumerate}
                 \item If, \begin{align}
                      \frac{ \p{1+\lambda^2}^{\mu(\Gamma) } \cdot d_{\max}(\Gamma)}{n-|v(\Gamma)|}\cdot \max\p{\vc(\Gamma),d_{\max}(\Gamma)}\leq C.\label{eq:LowerMainCond1Other}
                 \end{align}
                 \item If $\mu(\Gamma)\geq 1$ and, 
                 \begin{align}
                     \frac{ \p{1+\lambda^2}^{\mu(\Gamma)-1 }\cdot \lambda^2 \cdot d_{\max}(\Gamma)}{n-|v(\Gamma)|}\cdot \max\p{\vc(\Gamma)\cdot \sqrt{\frac{1+\lambda^2}{\lambda^2}},d_{\max}(\Gamma)\cdot \mu(\Gamma)}\leq C.\label{eq:LowerMainCond2}
                 \end{align}

                 \item If $\mu(\Gamma)<1$ and, 
                 \begin{align}
                      \frac{ \lambda\cdot d_{\max}(\Gamma)}{n-|v(\Gamma)|}\cdot \max\p{\vc(\Gamma),d_{\max}(\Gamma)\cdot \lambda}\leq C.\label{eq:LowerMainCond3}
                 \end{align}
                 \item If $\mu(\Gamma)<1$ and,
                 \begin{align}
                     \frac{\vc(\Gamma)^2}{\lambda^2 \cdot d_{\max}(\Gamma)^2}<1-\delta, \quad \text{and }\quad \frac{e\cdot \lambda^2 \cdot d_{\max}^2}{n-|v(\Gamma)|}\geq 1+\delta,\label{eq:LowerMainCond4}
                 \end{align}
                 for some fixed $\delta>0$ and,
                 \begin{align}
                     \frac{\vc(\Gamma)^2\cdot \p{e  \lambda^2  d_{\max}(\Gamma)^2}^{|v(\Gamma)|-1}}{(n-|v(\Gamma)|)^{|v(\Gamma)|}}\leq C.\label{eq:LowerMainCond5}
                 \end{align}
                 \item If $\limsup \mu(\Gamma)=1-\delta<1$, for some $0<\delta<1$, and,
                \begin{align}
                     \frac{|v(\Gamma)|^2}{\lambda^2}\cdot \max\p{1,\frac{\delta^2 (n-|v(\Gamma)|)}{|v(\Gamma)|^2}, \frac{1}{\delta 
                     }\p{\frac{ e \lambda^{2}}{\delta\p{n-|v(\Gamma)|}}}^{\frac{1}{\delta}}}\leq C.\label{eq:LowerMainCond6}
                 \end{align}
             \end{enumerate}
         \end{theorem}
         The proof of the above theorem follows the same approach to the proof of Theorem~\ref{th:lowerVCD}. The key difference is that we use an extended version of Lemma~\ref{lem:subgraphsBoundedDegreeCount1}, and bound $|\calS_{m,i,j}|$ in multiple different ways (depending on the value of $\mu(\Gamma)$), which surprisingly, are all essential. For simplicity of presentation, we define the following functions,
\begin{align}
   \kappa(m,\ell,\vc,d)&\triangleq \abs{\s{Par}(\ell,m)}  \vc^m  (e  \tilde{d})^{\ell-m},\\
   \theta(m,\ell,j,\mu)&\triangleq e^m(2\mu)^{\ell-2m}  \binom{\floor{\mu   \ell}-\ell+m}{j-\ell+m},
\end{align}
for $2\leq \ell\leq k $, $1 \leq m \leq \floor{\frac{\ell}{2}}$ and $\ell -m \leq j\leq\floor{\mu   \ell}$, and we let $\tilde{d}=\max(2,d-1)$. 
 \begin{lemma}\label{lem:subgraphsBoundedDegreeCount2}
     Fix $2\leq \ell\leq k $ and $1 \leq m \leq \floor{\frac{\ell}{2}}$.
     \begin{itemize}
     \item If $\mu\geq 1$, then for $\ell-m \leq j\leq\floor{\mu \cdot \ell }$,  
     \begin{align}
         |\calS_{m,\ell,j}| &\leq  \kappa(m,\ell,\vc,d)\cdot \min\p{\binom{\floor{\mu  \ell}}{j},\theta(m,\ell,j,\mu)},\label{eq:subgraphcount1Geq}
     \end{align}
      and for $\ell-m \leq j\leq \binom{\ell}{2}$,
      \begin{align}
         |\calS_{m,\ell,j}| &\leq  \kappa(m,\ell,\vc,d)\cdot \min\p{\binom{\binom{\ell}{2}}{j},\theta(m,\ell,j,\mu)}.\label{eq:subgraphcount1Geq.1}
     \end{align}
    \item If $\mu<1$, then,
          \begin{align}
         |\calS_{m,\ell,j}| &\leq \begin{cases}
             \kappa(m,\ell,\vc,d) & j=\ell-m\\
             0 & j\neq \ell-m.
         \end{cases}\label{eq:subgraphcount2}
     \end{align}
     \end{itemize}
     
 \end{lemma}
 We will also need a slightly modified version of Lemma~\ref{lem:subgraphsBoundedDegreeCount1}, specifically for the case where $\limsup \mu(\Gamma)<1$.
  \begin{lemma}\label{lem:subgraphsBoundedDegreeCountL}
     Assume that $\Gamma$ has exactly $L(\Gamma)=L$ connected components, each of size at most $t$, and let us assume that that $\mu<1$. 
     For $1\leq i \leq L$,  $
     2i\leq \ell\leq \floor{i t} $, $i \leq m \leq \floor{\frac{\ell}{2}}$ and $ j=\ell-m$, let $\calS_{i,m,\ell, j}$ be the number of subgraphs of $\Gamma$ with exactly $m$ connected components, $\ell$ vertices, $j$ edges, which intersects exactly $i$ out of the $L$ connected components of $\Gamma$.  Then,
     \begin{align} 
        |\calS_{i,m,\ell,j}| &\leq\abs{  \s{Par}(\ell,m)} \binom{L}{i} t^i \binom{i(t-1)}{m-i} m!  (e\tilde{d})^{\ell-m}.\label{eq:subgraphcountL}
     \end{align}
 \end{lemma}
In order to prove the above results we will need some combinatorial results concerning spanning trees.
\begin{definition}\label{def:spaningTree}
     Let $G=(V,E)$ be a connected undirected graph. A subset of edges $T\subseteq E$ is called a spanning tree of $G$ if any vertex in $T$ is connected to an edge in $T$, and $T=|V|-1$.
 \end{definition}
 The following is a corollary of \cite[Theorem 1]{grone1988bound}.
  \begin{lemma}\label{lem:AlonSpanning} Let $G=(V,E)$ be a connected graph. Then, 
    \begin{align}
        \cT(G)\triangleq\abs{\set{T\subseteq E:T \text{ is a spanning tree of }G }}\leq e(2\mu(G))^{|V|-2}.
    \end{align}
 \end{lemma}
 \begin{proof}
     It was shown in \cite[Theorem 1]{grone1988bound} that,
     \begin{align}
         \cT(G)\leq \p{\frac{|V|}{|V|-1}}^{|V|-1}\cdot \frac{\prod_{v\in V}\deg(v)}{\sum_{v\in V}\deg(v)},
     \end{align}
     where $\deg(v)$ is the degree of $v$ in $G$. Since $\sum_{v}\deg(v)=2|E|$, by the AM-GM inequality we have,
     \begin{align}
         \cT(G)&\leq \p{1+\frac{1}{|V|-1}}^{|V|-1}\frac{\prod_{v\in V}\deg(v)}{2|E|}\\
         &\leq e \frac{\p{\frac{1}{|V|}\sum_{v\in V}\deg(v)}^{|V|}}{2|E|}\\
         & =e \frac{\p{\frac{2|E|}{|V|}}^{|V|}}{2|E|}\\
         &=e \p{\frac{2|E|}{|V|}}^{|V|-1}\cdot \frac{1}{|V|}\\
         &\overset{(a)}{\leq}   e  \p{\frac{2|E|}{|V|}}^{|V|-1}\cdot \frac{1}{2(|E|/|V|)}\\
         &\leq  e   (2\mu(G))^{|V|-2},
     \end{align}
     where $(a)$ is because $|E|\leq \binom{|V|}{2}$.
 \end{proof}
We are now in a position to prove Lemmas~\ref{lem:subgraphsBoundedDegreeCount2} and \ref{lem:subgraphsBoundedDegreeCountL}.

\begin{proof}[Proof of Lemma~\ref{lem:subgraphsBoundedDegreeCount2}] 
    Recall that $\calS_{m,\ell,j} $ is the set of all subgraphs $\s{H}\subseteq \Gamma $ containing no isolated vertices, and $\ell$ vertices, $j$ edges, and $m$ connected components. Our analysis starts with the identity in \eqref{eqn:enumPar}, namely,
    \begin{align}
        |\calS_{m,\ell,j}|=\sum_{p\in\s{Par}(\ell,m)}|S(p,j)|,
    \end{align}
    where $S(p,j)$ is the set of all subgraphs in $\calS_{m,\ell,j}$, whose connected components' sizes correspond to a given partition $p$. As in Lemma~\ref{lem:subgraphsBoundedDegreeCount1}, we bound $|S(p,j)|$ by first bounding the number of ways to choose the set $U$ of $\ell$ vertices that supports the subgraph by $\vc^m(e\tilde{d})^{\ell-m}$. Next, we bound the number of ways to choose the $j$ edges of $\s{H}$. This is where our analysis diverges from Lemma~\ref{lem:subgraphsBoundedDegreeCount1}, in which we ignored the connectivity constraints when choosing edges. Here, we account for the connectivity constraint, and analyze the cases where $\mu\geq 1$ and $\mu<1$ separately.
    \begin{itemize}[leftmargin=*]
        \item \underline{\textbf{Proof of \eqref{eq:subgraphcount1Geq}--\eqref{eq:subgraphcount1Geq.1}.}} Let us now bound the number of subgraphs of $G_U$. Recall that at this point, we have fixed the vertex set $U$, which contains exactly $p(i)$ connected subsets of size $i$. Our goal is to bound the number of ways to choose $j$ edges on these graphs while respecting the connectivity constraints. Since every connected component must contain a spanning tree, we can bound the number of ways to choose these edges by multiplying the number of spanning trees of the connected components by the number of ways to choose the remaining edges. Using Lemma~\ref{lem:AlonSpanning}, the number of spanning trees $\cT(G_U)$ of an induced subgraph $G_U$, supported on a set $U$ with $i$ vertices, is bounded as,
   \begin{align}
       \cT(G_U)\leq e\p{2\mu(G_U)}^{i-2}\leq e\p{2\mu}^{i-2},
   \end{align} 
   where the second inequality follows from the definition of $\mu(\Gamma)$. Multiplying over all connected components we have at most,
    \begin{align}
        \prod_{i=1}^\ell \p{e\p{2\mu}^{i-2}}^{p(i)}=e^{m}(2\mu)^{\ell-2m}
    \end{align}
    ways to choose $m$ spanning trees, one for each component. Since each spanning tree of a component contains $|v(\s{C}_i)|-1$ edges exactly, there are at most,
    \begin{align}
        |e(G_U)|-\sum_{i=1}^m (|v(\s{C}_i)|-1)= |e(G_U)|-\ell+m\leq \min\p{\floor{ \mu \ell},\binom{\ell}{2}} -\ell+m
    \end{align}
    edges in the induced subgraph, from which we are to choose the remaining $j-(\ell-m)$ edges. Thus, the number of subgraphs for a fixed set of vertices $U$ can be upper bounded by,
    \begin{align}
         e^m(2\mu)^{\ell-2m}  \binom{\min\p{\floor{ \mu \ell},\binom{\ell}{2}}-\ell+m}{j-\ell+m}\triangleq f(\ell,m,j).
    \end{align}
    Summing over all possible partitions we obtain,
    \begin{align}
        \calS_{m,\ell,j}&=\sum_{p\in\s{Par}(\ell,m)}\abs{S(p,j)}\\
        &\leq \sum_{p\in\s{Par}(\ell,m)}\vc^m\cdot (e\tilde{d})^{\ell-m}\cdot f(\ell,m,j).\\
        &\leq \abs{\s{Par}(\ell,m)}\vc^m\cdot (e\tilde{d})^{\ell-m}\cdot f(\ell,m,j),
    \end{align}
    which completes the proof.
    \item \underline{\textbf{Proof of \eqref{eq:subgraphcount2}.}} When $\mu<1$ the connected components must be all trees. Indeed, since $\mu<1$, for any connected component ${\s{C}}_i$ we have, 
    \begin{align}
        \abs{e(\s{C}_i)}\leq \mu(\Gamma) \cdot \abs{v(\s{C}_i)}< 1\cdot \abs{v(\s{C}_i)}.
    \end{align}
    Because $\abs{v(\s{C}_i)}$ and $\abs{e(\s{C}_i)}$ are integers, we have $\abs{e(\s{C}_i)}\leq \abs{v(\s{C}_i)}- 1$. In the other direction, since $\s{C}_i$ is connected we also have $\abs{e(\s{C}_i)}\geq \abs{v(\s{C}_i)}-1$, and thus equality holds. Hence, in this case, the number of edges of a subgraph $\s{H}$ with exactly $m$ connected components is $\ell-m$. The proof then follows by repeating the exact same steps as in the case where $\mu\geq 1$, with the only observation that now a choice of the $m$ connected components determine a unique subgraph in $\calS_{\ell,m,j}$, for any $j$ such that $\calS_{\ell,m,j}\neq\emptyset$.
    \end{itemize}
       
    \end{proof}

\begin{proof}[Proof of Lemma~\ref{lem:subgraphsBoundedDegreeCountL}]
The proof follows the same approach as the case where $\mu<1$ in Lemma~\ref{lem:subgraphsBoundedDegreeCount1}, but now we refine the enumeration of the connected components. Specifically, we proceed as follows: we consider partitions in $\s{Par}(\ell,m)$, which dictate how the $\ell$ vertices are distributed among the $m$ different connected components. Once a partition $p$ is fixed, we proceed to choose $m$ connected components with the corresponding sizes in the following manner: First, we select $m$ fixed vertices, which we call \emph{anchors}, from exactly $i$ of the $L$ connected components of $\Gamma$, one anchor per each connected component. Then, we select the remaining vertices for each of these $m$ connected components, each containing exactly one of the previously chosen anchors. We observe that the number of ways to choose $m$ anchors from exactly $i$ connected components is at most,
 \begin{align}
     \binom{L}{i} t^i \binom{i(t-1)}{m-i} m!. \label{eq:HishuvBenZona}
 \end{align}
 Indeed, the factor $\binom{L}{i}$ counts the number of ways to choose which $i$ connected components of $\Gamma$ intersects the subgraph $\s{H}$. Once these components are fixed, there are at most $t^i$ ways to choose one anchor from each of these $i$ components. After that, we have at most $\binom{i(t-1)}{m-i}$ ways to choose the remaining $m-i$ anchors. Finally, we multiply by $m!$ to account for the number of ways to assign a size to the component containing each anchor (note that the sizes of the $m$ components are specified by the partition). The rest of the proof follows exactly as in the proof of \eqref{eq:subgraphcount2} in Lemma~\ref{lem:subgraphsBoundedDegreeCount1}, where the number of ways to choose a connected set of size $j$ that contains a fixed anchor is bounded using Lemma~\ref{lem:Bollobas1} by $(e\cdot \tilde{d})^{j-1}$. Combining this with \eqref{eq:HishuvBenZona} we get,
 \begin{align}
     \calS_{i,m,\ell,\ell-m}&\leq \sum_{p\in \s{Par}(\ell,m)} \binom{L}{i} t^i \binom{i(t-1)}{m-i} m!  \prod_{j=1}^{\ell}(e \tilde{d})^{p(j)\cdot(j-1)}\\
     &= \abs{\s{Par}(\ell,m)} \binom{L}{i} t^i \binom{i(t-1)}{m-i} m!  (e\tilde{d})^{\ell-m},
 \end{align}
 as claimed.     
 \end{proof}
We are now ready to prove Theorem~\ref{th:lowerVCFull}. 
\begin{proof}[Proof of Theorem~\ref{th:lowerVCFull}] 
Note that \eqref{eq:LowerMainCond1Other} is in fact the same condition in Theorem~\ref{th:lowerVCD}, which we already proved in Section~\ref{subsec:denseRegime}. It is left to prove \eqref{eq:LowerMainCond2}--\eqref{eq:LowerMainCond6}. To that end, we follow the steps of the proof of Theorem~\ref{th:lowerVCD}, but now apply the refined bounds on $|\calS_{m,\ell,j}|$ in Lemmas~\ref{lem:subgraphsBoundedDegreeCount2} and \ref{lem:subgraphsBoundedDegreeCountL}. 
\begin{itemize}[leftmargin=*]
    \item \underline{\textbf{Proof of \eqref{eq:LowerMainCond2}.}} From \eqref{eq:subgraphcount1Geq} we know that,
    \begin{align}
       |\calS_{m,\ell,j}|\leq  \kappa(m,\ell,\vc,d)\cdot \theta(m,\ell,j,\mu).\label{eqn:215bound}
    \end{align}
    Then, applying Lemmas~\ref{lem:probRandomSubgraph1} and \eqref{eqn:215bound} on \eqref{eqn:InterBound}, we get,
    \begin{align}
        &\E_{\calH_0}[\s{L(G)}^2]\\
      &\leq 1+\sum_{\ell=2}^k \sum_{m=1}^{\floor{\ell/2}}\sum_{j=\ell-m}^{\floor{\mu \ell}} |\calS_{m,\ell,j}| \lambda^{2j}  \frac{(2\vc)^m d^{\ell-m}}{(n-k)^\ell}\\
      &\leq 1+\sum_{\ell=2}^k \sum_{m=1}^{\floor{\ell/2}} \frac{(2\vc)^m d^{\ell-m}}{(n-k)^\ell} \sum_{j=\ell-m}^{\floor{\mu \ell }}  \lambda^{2j} \abs{\s{Par}(\ell,m)}\vc^m (ed)^{\ell-m}e^m(2\mu)^{\ell-2m}  \binom{\floor{\mu \ell}-\ell+m}{j-\ell+m}\\
      &\leq 1+C_{\varepsilon}\sum_{\ell=2}^k \p{\frac{2e(1+\varepsilon) \lambda^2 d^2 \mu}{n-k} }^\ell \sum_{m=1}^{\floor{\ell/2}} \p{\frac{\vc^2}{2\lambda^2 d^2 \mu^2}}^m\sum_{j=\ell-m}^{\floor{\mu \ell }}  \lambda^{2(j-\ell+m)} \binom{\floor{\mu \ell}-\ell+m}{j-\ell+m}\\
      &=1+C_{\varepsilon}\sum_{\ell=2}^k \p{\frac{2e(1+\varepsilon)\lambda^2   d^2 \mu}{n-k} }^\ell \sum_{m=1}^{\floor{\ell/2}} \p{\frac{\vc^2}{2 \lambda^2 d^2 \mu^2}}^m\sum_{j'=0}^{\floor{\mu \ell }-\ell+m}  \lambda^{2j'}  \binom{\floor{\mu \ell}-\ell+m}{j'}\\
      &\leq 1+C_{\varepsilon}\sum_{\ell=2}^k \p{\frac{2e(1+\varepsilon)\lambda^2  d^2 \mu}{n-k} }^\ell \sum_{m=1}^{\floor{\ell/2}} \p{\frac{\vc^2}{2 \lambda^2 d^2  \mu^2}}^m \cdot (1+\lambda^2)^{ \mu \ell-\ell+m}\\
      &\leq 1+C_{\varepsilon}\sum_{\ell=2}^k \p{\frac{2e(1+\varepsilon)(1+\lambda^2)^{\mu-1}\cdot\lambda^2  d^2 \mu}{n-k} }^\ell \cdot \sum_{m=1}^{\floor{\ell/2}} \p{\frac{(1+\lambda^2) \vc^2}{2\lambda^2 d^2  \mu^2}}^m.\label{eq:GenBoundSparse}
      \end{align}
      We now separate our analysis into two complementary cases. In the first, we assume that,
      \begin{align}
             \frac{ (1+\lambda^2) \vc^2}{ 2\lambda^2  d^2 \mu^2}\geq 1+\alpha, \label{eq:repatativeDOm2}
         \end{align}
         for some fixed positive $\alpha$. In this case, the inner sum on the r.h.s. of \eqref{eq:GenBoundSparse} is dominated by $\p{\frac{(1+\lambda^2)\vc^2}{2\lambda^2 d^2 \mu^2}}^{\ell/2}$, and therefore,
         \begin{align}
             \E_{\calH_0}\pp{\s{L}(\s{G})^2}\leq 1+C'_\varepsilon\cdot \sum_{\ell=2}^k \p{\frac{e (1+\varepsilon)  (1+\lambda^2)^{\mu-\frac{1}{2}}\lambda \vc  d}{\sqrt{2}(n-k)} }^\ell,
         \end{align}
         which is bounded (and accordingly strong detection is impossible) provided that,
         \begin{align}
         \frac{(1+\lambda^2)^{\mu-\frac{1}{2}}\lambda \vc d }{n-k}<C,\label{eq:rePETATIVEthresh1}
         \end{align}
         for some universal $C>0$. Next, we move to the other case, where,
         \begin{align}
                 \frac{ (1+\lambda^2) \vc^2}{ 2\lambda^2  d^2 \mu^2}\leq 1+o(1),
         \end{align}
        for some $o(1)$ function. Here, for a sufficiently large $n$, the inner sum on the r.h.s. of \eqref{eq:GenBoundSparse} is dominated by $(1+\varepsilon)^\ell$, and therefore,
        \begin{align}
            \E_{\calH_0}[\s{L(G)}^2]&\leq 1+C'_\varepsilon\cdot \sum_{\ell=2}^k \p{\frac{2e (1+\varepsilon)^2(1+\lambda^2)^{\mu-1}  \lambda^2   d^2  \mu }{n-k} }^\ell.\label{eq:AnoyningCases2}
        \end{align}
        The above is bounded provided that,
        \begin{align}
            \frac{(1+\lambda^2)^{\mu-1} \lambda^2   d^2  \mu }{n-k} <C, \label{eq:rePETATIVEthresh2}
        \end{align}
        for some universal constant $C>0$. Since \eqref{eq:repatativeDOm2} holds exactly when \eqref{eq:rePETATIVEthresh1} dominates \eqref{eq:rePETATIVEthresh2} (perhaps up to a multiplicative constant factor), we have that \eqref{eq:LowerMainCond2} is sufficient for $\E_{\calH_0}[\s{L}(\s{G})^2]$ to be bounded.
        \item \underline{\textbf{Proof of \eqref{eq:LowerMainCond3}-\eqref{eq:LowerMainCond5}.}} We are now in the regime where $\mu<1$. We repeat the analysis as in the pervious case, but bound $|\calS_{m,\ell,j}|$ using \eqref{eq:subgraphcount2}. Specifically, in this case, we sum over the vertices $\ell$ and connected components $m$ of $\s{H}$, while the number of edges $j$ is fixed and given by $j=\ell-m$, as explained in the proof of \eqref{eq:subgraphcount2}. Specifically, we have,
          \begin{align}
               \E_{\calH_0}[\s{L(G)}^2]
      &=1+\sum_{\ell=2}^k \sum_{m=1}^{\floor{\ell/2}} |\calS_{m,\ell,j}|\cdot \lambda^{2(\ell-m)} \cdot \frac{(2\vc)^m d^{\ell-m}}{(n-k)^\ell}\\
      &\leq 1+\sum_{\ell=2}^k \sum_{m=1}^{\floor{\ell/2}} \frac{(2\vc)^m d^{\ell-m}}{(n-k)^\ell} \lambda^{2(\ell-m)} \cdot \abs{\s{Par}(\ell,m)}\vc^m (ed)^{\ell-m}\\
      &\leq 1+C_\varepsilon \cdot \sum_{\ell=2}^k \p{\frac{e(1+\varepsilon)\lambda^2  d^2}{n-k}}^\ell \sum_{m=1}^{\floor{\ell/2}} \p{\frac{2\vc^{2} }{e \lambda^2  d^2} }^m  .\label{eq:GenBoundMuLeq1}
          \end{align}
    Again, we separate our analysis into two complementary cases. In the first, we assume that,
      \begin{align}
             \frac{ 2 \vc^2}{ e\lambda^2  d^2 }\geq 1+\alpha, \label{eq:repatativeDOm3}
         \end{align}
         for some fixed positive $\alpha$. In this case, the inner sum on the r.h.s. of \eqref{eq:GenBoundSparse} is dominated by $\p{\frac{ 2 \vc^2}{ e\lambda^2  d^2 }}^{\ell/2}$, and therefore, we have,
        \begin{align}
          \E_{\calH_0}\pp{\s{L}(\s{G})^2}\leq  1+C'_\varepsilon \cdot \sum_{\ell=2}^k \p{\frac{\sqrt{ 2e }(1+\varepsilon)\lambda \vc  d}{n-k}}^\ell. 
        \end{align}
        The above is bounded whenever 
        \begin{align}
            \frac{\lambda\vc  d}{n-k}<C,\label{eq:AnoyningCases31}
        \end{align} 
        for some constant $C>0$. Next, we move to the other case, where,
         \begin{align}
             \frac{ 2\vc^2}{ e\lambda^2d^2}\leq 1+o(1),
         \end{align}
        for some $o(1)$ function. Here, for a sufficiently large $n$, the inner sum on the r.h.s. of \eqref{eq:GenBoundSparse} is dominated by $\p{1+\varepsilon}^\ell$, and thus,
        \begin{align}
            \E_{\calH_0}[\s{L(G)}^2]&\leq 1+C''_\varepsilon\cdot  \sum_{\ell=2}^k \p{\frac{e(1+\varepsilon)^2\lambda^2 d^2}{n-k}}^\ell<C.  \label{eq:AnoyningCases3}
        \end{align}
        The above is bounded provided that,
        \begin{align}
           \frac{\lambda^2 d^2 }{n-k}<C,
        \end{align}
        for some constant $C$. Since \eqref{eq:repatativeDOm3} holds exactly when \eqref{eq:AnoyningCases31} dominates \eqref{eq:AnoyningCases3} (perhaps, up to a multiplicative constant factor), we have that \eqref{eq:LowerMainCond3} is sufficient for $\E_{\calH_0}[\s{L}(\s{G})^2]$ to be bounded.    
          
          In order to prove \eqref{eq:LowerMainCond4}, consider the scenario where \begin{align}
               \frac{ \vc^2}{ e\lambda^2d^2}\leq  1-\alpha,
               \quad \text{and} \quad \frac{e\lambda^2 d^2}{n-k}>1+\alpha,
          \end{align}
          for some fixed constant $\alpha$. In this case, the inner sum on the r.h.s. of \eqref{eq:GenBoundMuLeq1} is dominated by $\frac{\vc^2}{e\lambda^2 d^2}$, while the outer sum is dominated by $\p{\frac{e(1+\varepsilon)\lambda^2d^2}{n-k}}^k$. We therefore have,
          \begin{align}
            \E_{\calH_0}[\s{L(G)}^2]&\leq 1+C''_\varepsilon\cdot \p{\frac{\vc^2}{e\lambda^2 d^2 }} \p{\frac{e(1+\varepsilon)\lambda^2\cdot d^2}{n-k}}^k,
        \end{align}
        which is bounded if, 
        \begin{align}
            \frac{\vc^2\cdot (e\lambda^2 d^2)^{k-1}}{(n-k)^k}=O(1).
        \end{align}
        This proves that \eqref{eq:LowerMainCond4} and \eqref{eq:LowerMainCond5} are sufficient for the impossibility of strong detection. 
        
        \item\underline{\textbf{Proof of \eqref{eq:LowerMainCond6}.}} Consider the case where,
          \begin{equation}
              \limsup_{n\to\infty} \mu=1-\delta, \label{eq:limsupMu}
          \end{equation} 
          for $\delta>0$. Here, as explained before, $\Gamma$ must be a forest. We begin with the observation that in this case each tree in $\Gamma$ contains at most $\frac{1}{\delta}$ vertices. Indeed, if $\s{T}$ is a tree in $\Gamma$ with $\ell$ vertices, we have 
          \begin{align}
              1-\frac{1}{\ell}=\frac{\ell-1}{\ell}=\mu(\s{T})\leq \mu \leq 1-\delta.
          \end{align}
          Let us denote $t=\frac{1}{\delta}$, which as stated above, is a constant that bounds the sizes of the connected components  in $\Gamma$, and in particular, the maximum degree of $\Gamma$. Recall that $L=L(\Gamma)$ denotes the number of connected components of $\Gamma$. We note that,
          \begin{align}
             \frac{k}{t} \leq L \leq \frac{k}{2}.
          \end{align}
          We also observe that any vertex cover set must contain at least one vertex from each connected component of $\Gamma$, and therefore, we have $L\leq \vc$. On the other hand, from the definition of $\vc$ we have $ \vc\leq k$. Combining together, we get,
          \begin{align}
             \frac{k}{t}\leq L \leq \vc \leq k\leq t\cdot L, \label{eq:L-S.relations}
          \end{align}
          where the first inequity holds for a sufficiently large $n$, as $t$ is bounded with $n$.
          
          We next evaluate \eqref{eq:SecondMomentExpressionProb2} by running over subgraphs with fixed number of vertices and connected components, but also take into count the number $i(\s{H})$ of connected components of $\Gamma$ which intersects with $\s{H}$. Formally, let $\calS_{i,m,\ell,j}$ be the number of subgraphs $\s{H}\subseteq \Gamma$ with $m$ connected components, $\ell$ vertices, and $j$ edges, which satisfy $i(\s{H})=i$. As in the previous case, if $j\neq \ell-m$, we have $S_{i,m,\ell,j}=0$. We note that with $i$ fixed, the number of vertices $\ell$ must be between $2i$ and $it$ (which bounds the maximal number of vertices in the $i$ connected components of $\Gamma)$. With $i$ and $\ell$ fixed, the number of connected components $m$ must be between $i$ and $\floor{\frac{\ell}{2}}$. Combining the above observations with \eqref{eq:SecondMomentExpressionProb2}, Lemmas~\ref{lem:probRandomSubgraph1} and \ref{lem:subgraphsBoundedDegreeCountL} we obtain,
          \begin{align}
              &\E_{\calH_0}[\s{L(G)}^2]\\
      &\leq1+\sum_{i=1}^{L}\sum_{\ell=2i}^{\floor{it}} \sum_{m=i}^{\floor{\ell/2}} |\calS_{i,m,\ell,\ell-m}| \lambda^{2(\ell-m)} \cdot \frac{(2\vc)^m d^{\ell-m}}{(n-k)^\ell}\\
      &\leq 1+\sum_{i=1}^{L}\sum_{\ell=2i}^{\floor{it}} \sum_{m=i}^{\floor{\ell/2}} |\calS_{i,m,\ell,\ell-m}| \lambda^{2(\ell-m)} \cdot \frac{(2t  L)^m t^{\ell-m}}{(n-k)^\ell}\\
      &\leq 1+\sum_{i=1}^{L}\sum_{\ell=2i}^{\floor{it}} \sum_{m=i}^{\floor{\ell/2}} \abs{  \s{Par}(\ell,m)} \binom{L}{i} t^i \binom{i(t-1)}{m-i} m!  (ed)^{\ell-m} \lambda^{2(\ell-m)} \cdot \frac{(2t  L)^m t^{\ell-m}}{(n-k)^\ell}\\
       &\leq 1+\sum_{i=1}^{L}\sum_{\ell=2i}^{\floor{it}} \sum_{m=i}^{\floor{\ell/2}} \abs{  \s{Par}(\ell,m)} \binom{L}{i} t^i \binom{i(t-1)}{m-i} m!  (et)^{\ell-m} \lambda^{2(\ell-m)} \cdot \frac{(2t  L)^m t^{\ell-m}}{(n-k)^\ell}\\
      &= 1+C_{\varepsilon} \sum_{i=1}^{L}t^i\sum_{\ell=2i}^{\floor{it}} \p{\frac{e(1+\varepsilon)\lambda^2 t^2}{n-k}}^\ell  \sum_{m=i}^{\floor{\ell/2}} \frac{L!}{(L-i)!}\cdot \frac{(i(t-1))!}{\p{i(t-1)-(m-i)}!} \binom{m}{i} \p{\frac{2L}{et  \lambda^2}}^m\\
      &\overset{(a)}{\leq }1+C_{\varepsilon} \sum_{i=1}^{L}t^i\sum_{\ell=2i}^{\floor{it}} \p{\frac{e(1+\varepsilon)\lambda^2 t^2}{n-k}}^\ell  \sum_{m=i}^{\floor{\ell/2}} \frac{L!}{(L-i)!}\cdot \frac{(L(t-1))!}{\p{L(t-1)-(m-i)}!}  \p{\frac{e  m}{i}}^i  \p{\frac{2L}{et  \lambda^2}}^m\\
      &\overset{(b)}{\leq }1+C_{\varepsilon} \sum_{i=1}^{L}t^i\sum_{\ell=2i}^{\floor{it}} \p{\frac{e(1+\varepsilon)\lambda^2 t^2}{n-k}}^\ell  \sum_{m=i}^{\floor{\ell/2}} L^i  (L(t-1))^{m-i}  \p{\frac{et}{2}}^i  \p{\frac{2L}{et  \lambda^2}}^m\\
       &\overset{(b)}{\leq }1+C_{\varepsilon} \sum_{i=1}^{L}t^i\sum_{\ell=2i}^{\floor{it}} \p{\frac{e(1+\varepsilon)\lambda^2 t^2}{n-k}}^\ell  \sum_{m=i}^{\floor{\ell/2}} L^i  (Lt)^{m-i}  \p{\frac{et}{2}}^i  \p{\frac{2L}{et  \lambda^2}}^m\\
      &\leq 1+C_{\varepsilon} \sum_{i=1}^{L}\p{\frac{ et}{2}}^i\sum_{\ell=2i}^{\floor{it}} \p{\frac{e(1+\varepsilon)\lambda^2 t^2}{n-k}}^\ell  \sum_{m=i}^{\floor{\ell/2}} \p{\frac{2L^2}{e  \lambda^2}}^m\\
      &\leq 1+C_{\varepsilon} \sum_{i=1}^{k}\p{\frac{ et}{2}}^i\sum_{\ell=2i}^{\floor{it}} \p{\frac{e(1+\varepsilon)\lambda^2 t^2}{n-k}}^\ell  \sum_{m=i}^{\floor{\ell/2}} \p{\frac{2k^2}{e  \lambda^2}}^m,
          \end{align}
          where (a) follows from the fact that $\binom{n}{m}\leq\p{\frac{en}{m}}^m$, for any $m\leq n$, and (b) is because $m\leq \frac{1}{2}\ell\leq \frac{1}{2}\cdot t\cdot i$. We separate our analysis into two complementary cases. In the first, we assume that,
          \begin{align}
              \frac{2k^2}{e\lambda^2}\leq 1-\alpha,
          \end{align}
          and 
          \begin{align}
              \frac{e(1+\varepsilon)\lambda^2 t^2}{n-k}\geq 1+\alpha,
          \end{align}
          for a fixed $\alpha>0$. Then, we bound the above sum as,
          \begin{align}
              \E_{\calH_0}[\s{L(G)}^2]&\leq 1+C'_\varepsilon \cdot \sum_{i=1}^{k}\p{\frac{ k^2 t}{ \lambda^2}}^i\sum_{\ell=2i}^{\floor{it}} \p{\frac{e(1+\varepsilon)\lambda^2 t^2}{n-k}}^\ell.  \\
              &\leq 1+C''_\varepsilon \cdot \sum_{i=1}^{k}\p{\frac{ k^2 t  \p{e(1+\varepsilon)  t  \lambda^2}^{t}}{\lambda^2  (n-k)^t}}^i,
          \end{align}
          which is bounded provided that, 
          \begin{align}
              \frac{k^2 t^{t+1}\lambda^{2(t-1)}}{(n-k)^t}\leq C,
          \end{align}
          which concludes the proof of \eqref{eq:LowerMainCond6}. We remark that we do not need to analyze the regimes where 
          \begin{align}
              \frac{2k^2}{e\lambda^2}>1+\alpha,
          \end{align}
          or,
          \begin{align}
              \frac{e(1+\varepsilon)\lambda^2 t^2}{n-k}\leq  1-\alpha,
          \end{align}
          because the resulting conditions overlap with \eqref{eq:LowerMainCond3}--\eqref{eq:LowerMainCond5}. 
\end{itemize}

 \end{proof}

\subsubsection{Reduction to the vertex cover-degree balanced case}

In this subsection, we extend the decomposition argument in Proposition~\ref{prop:decompisitionThD} in the dense regime, to the more general scenario we consider here, using the refined combinatorial analysis developed in the proof of Theorem~\ref{th:lowerVCFull}. 
\begin{prop}\label{prop:decompisitionThFull}
Let $\Gamma=(\Gamma_n)_n$ be a sequence of graphs, and assume that each can be decomposed into a set of edge-disjoint graphs, i.e., $\Gamma=\bigcup_{\ell=1}^M \Gamma_\ell$. Then, $\E_{\calH_0}[\s{L}(\s{G})^2]=O(1)$ if     \begin{align}
        \max_{\ell=1,\dots, M}\p{\frac{(1+\lambda_M^2)^{\tilde{\mu}(\Gamma_\ell)-1}\lambda_M^2}{n-|v(\Gamma_\ell)|}\cdot\max\p{\vc(\Gamma_\ell) d_{\max}(\Gamma_\ell)\sqrt{\frac{(1+\lambda_M^2)}{\lambda_M^2}},d_{\max}(\Gamma_\ell)^2\mu(\Gamma_\ell)}}\leq C,\label{eq:condintersectionFull}
    \end{align}
    for some universal $C>0$, where $\lambda_M^2\triangleq(1+\lambda^2)^{M^2}-1$ and $\tilde{\mu}(\Gamma_\ell)=\max(\mu(\Gamma_\ell),1 )$.
\end{prop}

\begin{proof}
    The proof follows the exact same steps as in the proof of Proposition~\ref{prop:decompisitionThD}. Specifically, applying H\"{o}lder inequality on \eqref{eq:SecondMomentExpression} we get,
    \begin{align}
        \E_{\Gamma}\pp{(1+\lambda^2)^{|e(\Gamma\cap\Gamma')|}}\leq \prod_{i,j=1}^M \E_{\Gamma_i}\pp{(1+\lambda^2_M)^{|e(\Gamma_i\cap\Gamma_j')|}}^{\frac{1}{M^2}}.\label{eq:prodMfull}
    \end{align}
    Accordingly, in order to find the conditions under which \eqref{eq:prodMfull} is bounded it remains to generalize the argument of Lemma~\ref{lem:interectionD}, for our case. Let us follow the notation used in the proof of Lemma~\ref{lem:interectionD}, where $\vc_i=\vc(\Gamma_i)$, $d_i=d_{\max}(\Gamma_i)$, $\mu_i=\mu(\Gamma_i)$, $k_i=|v(\Gamma_i)|$, $\bar{k}=\min(k_1,k_2)$, and $\bar{\mu}=\min(\mu_1,\mu_2)$. Below we prove that $\E_{\Gamma_2}\pp{(1+\lambda^2)^{|e(\Gamma_1'\cap\Gamma_2)|}}$ is bounded if the following hold:
    \begin{enumerate}
        \item If $\bar{\mu}<1$, and 
        \begin{align}
            \frac{\lambda }{n-\bar{k}}\cdot \max\p{\sqrt{\vc_1 d_1 \vc_2 d_2}, d_1 d_2 \lambda }\leq C,\label{eq:lemCondInter1}
        \end{align}
        for some universal $C>0$.
        \item If $\bar{\mu}\geq1$, and 
        \begin{align}
            \frac{(1+\lambda^2)^{\bar{\mu}-1}\lambda^2}{n-\bar{k}}\cdot \max\p{\sqrt{\frac{\vc_1 d_1\vc_2 d_2  (1+\lambda^2)}{\lambda^2}},d_1 d_2 \bar{\mu}}\leq C,\label{eq:lemCondInter2}
        \end{align}
        for some universal $C>0$.
    \end{enumerate}
    We start with the case where $\bar{\mu}>1$. By symmetry, the distributions of $|e(\Gamma_1'\cap \Gamma_2)|$ and $|e(\Gamma_1\cap \Gamma_2')|$ are exactly the same, where $\Gamma_i$ is a uniform random copy and $\Gamma_j'$ is a fixed copy in $\calK_n$. Thus, we can assume, without loss of generality, that $\bar{\mu}=\mu_1$. As in the proof of Lemma~\ref{lem:interectionD}, we apply Lemma~\ref{lem:probRandomSubgraph1} and the bound in \eqref{eq:subgraphcount1Geq} on \eqref{eq:inetersectionEquivalent2}, and get,
    \begin{align}
        \E_{\Gamma_2}&\pp{(1+\lambda^2)^{|e(\Gamma_1'\cap\Gamma_2)|}}\\
      &=\sum_{\s{H}\subseteq \Gamma_1'}\lambda^{2|\s{H}|}  \P_{\Gamma_2}[\s{H}\subseteq \Gamma_2]\\
      & \leq1+\sum_{\ell=2}^{\bar{k}} \sum_{m=1}^{\floor{\ell/2}}\sum_{j=\ell-m}^{\floor{\ell \bar{\mu}}} |\calS_{m,\ell,j}(\Gamma_1)|  \lambda^{2j}   \frac{(2\vc_2)^m d_2^{\ell-m}}{(n-{\bar{k}})^\ell}\\
      &\leq 1+\sum_{\ell=2}^{\bar{k}} \sum_{m=1}^{\floor{\ell/2}} \frac{(2\vc_2)^m d_2^{\ell-m}}{(n-{\bar{k}})^\ell} \sum_{j=\ell-m}^{\floor{\ell \bar{\mu} }}  \lambda^{2j}   \abs{\s{Par}(\ell,m)}\vc_1^m  (ed_1)^{\ell-m}  e^m(2\mu_1)^{\ell-2m}\binom{\floor{\ell \bar{\mu}}-\ell+m}{j-\ell+m}\\
      &\leq 1+C_\varepsilon\sum_{\ell=2}^{\bar{k}} \p{\frac{ 2e(1+\varepsilon)d_1d_2\lambda^2 \bar{\mu}}{n-{\bar{k}}}}^\ell \sum_{m=1}^{\floor{\ell/2}} \p{\frac{\vc_1\vc_2}{2d_1d_2\lambda^2 \bar{\mu}^2}}^m\sum_{j=0}^{\floor{\ell \bar{\mu} }-\ell+m}  \lambda^{2j'} \binom{\floor{\ell \bar{\mu}}-\ell+m}{j'}\\
      &= 1+C_\varepsilon\sum_{\ell=2}^{\bar{k}} \p{\frac{ 2e(1+\varepsilon)(1+\lambda^2)^{\bar{\mu}-1}d_1d_2\lambda^2 \bar{\mu}}{n-{\bar{k}}}}^\ell \sum_{m=1}^{\floor{\ell/2}} \p{\frac{\vc_1\vc_2(1+\lambda^2)}{2d_1d_1\lambda^2  \bar{\mu}^2}}^m.\label{eqn:GenBounG1G2}
    \end{align}
    We separate our analysis into two complementary cases. In the first, we assume that, 
    \begin{align}
        \frac{\vc_1\vc_2(1+\lambda^2)}{2d_1d_1\lambda^2  \bar{\mu}^2}> 1+\alpha, \label{eq:condIntersectBoring1}
    \end{align}
    for some fixed positive $\alpha$. Then, we bound the above sum as,
    \begin{align}
        \E_{\Gamma_2}&\pp{(1+\lambda^2)^{|e(\Gamma_1'\cap\Gamma_2)|}}\leq 1+C_\varepsilon'\sum_{\ell=2}^{\bar{k}} \p{\frac{ e(1+\varepsilon)(1+\lambda^2)^{\bar{\mu}-\frac{1}{2}}\lambda\sqrt{2d_1d_2 \vc_1\vc_2} }{n-{\bar{k}}}}^\ell,
    \end{align}
    which is bounded provided that, 
    \begin{align}
        \frac{ (1+\lambda^2)^{\bar{\mu}-\frac{1}{2}}\lambda\sqrt{d_1d_2 \vc_1\vc_2} }{n-{\bar{k}}}<C,\label{eq:condIntersectBoringCond1}
    \end{align}
    for some universal $C>0$. Next, we move to the other case, where,
    \begin{align}
        \frac{\vc_1\vc_2(1+\lambda^2)}{2d_1d_1\lambda^2  \bar{\mu}^2}\leq1+o(1), \label{eq:condIntersectBoring2}
    \end{align}
    in which case the inner sum on the r.h.s. of \eqref{eqn:GenBounG1G2} is dominated by $(1+\varepsilon)^\ell$, and therefore,
    \begin{align}
       \E_{\Gamma_2}\pp{(1+\lambda^2)^{|e(\Gamma_1'\cap\Gamma_2)|}}\leq 1+C_\varepsilon''\sum_{\ell=2}^{\bar{k}}\p{\frac{ 2e(1+\varepsilon)^2(1+\lambda^2)^{\bar{\mu}-1}d_1d_2\lambda^2 \bar{\mu}}{n-{\bar{k}}}}^\ell,
    \end{align}
    which is bounded provided that,
     \begin{align}
        \frac{ (1+\lambda^2)^{\bar{\mu}-1}\lambda^2 d_1d_2\bar{\mu} }{n-{\bar{k}}}<C,\label{eq:condIntersectBoringCond2}
    \end{align}
    for some universal $C>0$. Since \eqref{eq:condIntersectBoring1} holds exactly when \eqref{eq:condIntersectBoringCond1} dominates \eqref{eq:condIntersectBoringCond2}, we get that $\E_{\Gamma_2}\pp{(1+\lambda^2)^{|e(\Gamma_1'\cap\Gamma_2)|}}=O(1)$ if \eqref{eq:lemCondInter2} holds.

Next, we proceed to the case where $\mu_1=\bar{\mu}<1$. This case follows similarly to the previous one, with the only difference being the use of the bound in \eqref{eq:subgraphcount2}. Specifically, we obtain,
\begin{align}
    \E_{\Gamma_2}\pp{(1+\lambda^2)^{|e(\Gamma_1'\cap\Gamma_2)|}}&
      \leq1+\sum_{\ell=2}^{\bar{k}} \sum_{m=1}^{\floor{\ell/2}}|\calS_{m,\ell,\ell-m}(\Gamma_1)|  \lambda^{2(\ell-m)}   \frac{(2\vc_2)^m d_2^{\ell-m}}{(n-{\bar{k}})^\ell}\\
      &\leq1+\sum_{\ell=2}^{\bar{k}} \sum_{m=1}^{\floor{\ell/2}}|\s{Par}(\ell,m)|\vc_1^m (ed_1)^{\ell-m}
      \lambda^{2(\ell-m)}   \frac{(2\vc_2)^m d_2^{\ell-m}}{(n-{\bar{k}})^\ell}\\
      &=1+C_\varepsilon\sum_{\ell=2}^{\bar{k}} \p{\frac{e(1+\varepsilon) d_1d_2 \lambda^2}{n-\bar{k}}}^\ell\sum_{m=1}^{\floor{\ell/2}}\p{\frac{2\vc_1\vc_2}{ed_1d_2\lambda^2}}^m. \label{eq:SumSAgainDay}
\end{align}
Similarly to the previous case, if the inner sum on the r.h.s. of \eqref{eq:SumSAgainDay} is dominated by the last term then \eqref{eq:SumSAgainDay} is bounded if,
\begin{align}
\frac{\lambda \sqrt{d_1d_2\vc_1\vc_2}}{n-k}<C,    \label{eq:condIntersectBoringCond4}
\end{align}
for some universal $C>0$. Otherwise, if the common ratio of the inner sum is bounded by $1+o(1)$, then \eqref{eq:SumSAgainDay} is bounded if,
\begin{align}
    \frac{\lambda^2d_1d_2}{n-k}<C,   \label{eq:condIntersectBoringCond5}
\end{align}
for some universal $C>0$. Since the inner sum is dominated by the last term exactly when the condition in \eqref{eq:condIntersectBoringCond4} dominates \eqref{eq:condIntersectBoringCond5}, we get that $\E_{\Gamma_2}\pp{(1+\lambda^2)^{|e(\Gamma_1'\cap\Gamma_2)|}}=O(1)$ if \eqref{eq:lemCondInter1} holds. Finally, note that for any $i,j\in[M]$, the terms on the l.h.s. of \eqref{eq:lemCondInter1} and \eqref{eq:lemCondInter2} are bounded by the term on the l.h.s. of \eqref{eq:condintersectionFull}. This implies that \eqref{eq:condintersectionFull}, with an appropriate choice of $C>0$, is a sufficient condition for 
\begin{align}
    \E_{\Gamma_i}\pp{(1+\lambda^2_M)^{|e(\Gamma_i\cap\Gamma_j')|}}^{\frac{1}{M^2}}\leq 2,
\end{align}
and therefore the entire product in \eqref{eq:prodMfull} is bounded as well, which concludes the proof.
\end{proof}

\subsubsection{The sparse regime}\label{subsec:SparseRegimeProofs}

In this subsection, we specialize the general bounds from Theorem~\ref{th:lowerVCFull} to the sparse regime, where $\chi^2(p||q)=(n^{-\alpha})$, for $0<\alpha\leq 2$. This occurs when $|p-q|,p,q=\Theta(n^{-\alpha})$. In this regime, we focus on the polynomial (or more precisely, signomial, as we allow for non-integer exponents) growth/decay of each of the parameters, rather than the sub-polynomial factors, although our bounds apply also to those. Accordingly, we assume that $k\triangleq|v(\Gamma_n)|=\Theta(n^{\beta})$, for $0<\beta\leq 1$, and investigate the statistical limits of our detection problem in terms of the exponential relations between the parameters $\mu(\Gamma),|e(\Gamma)|, |v(\Gamma)|$ and $d_{\max}(\Gamma)$. Consider the following definition.
\begin{definition} A sequence of graphs  $\Gamma=(\Gamma_n)_n$ is called an $(\epsilon, \delta,\smu)$-polynomial family if,
\begin{align}
    \lim_{n\to \infty }\frac{\log|e(\Gamma)|}{\log|v(\Gamma)|}=\epsilon, \quad   \lim_{n\to \infty }\frac{\log d_{\max}(\Gamma)}{\log|v(\Gamma)|}=\delta, \quad   \lim_{n\to \infty }\frac{\log\mu(\Gamma)}{\log|v(\Gamma)|}=\smu.
\end{align}
\end{definition}
The following is our main result for the sparse regime.
\begin{theorem}\label{th:polynomial}
    Let $\Gamma=(\Gamma_n)_n$ be an $(\epsilon, \delta, \smu)$-polynomial family of graphs with  $k=|v(\Gamma)|=\Theta(n^{\beta})$, $0<\beta< 1$, and let $\chi^2(p||q)=\Theta(n^{-\alpha})$, for $0<\alpha< 2$. Then, weak detection is impossible if,
        \begin{align}
            \beta< \min\p{\frac{\alpha}{\smu}, \frac{1+\alpha}{2\delta+\smu},\frac{2+\alpha}{2\epsilon}},\label{eq:sparceLowerBoundCond}
        \end{align}
         while strong detection is possible if, 
        \begin{align}
            \beta> \min\p{\frac{\alpha}{\smu}, \frac{1+\alpha}{2\delta},\frac{2+\alpha}{2\epsilon}},\label{eq:sparceUpperBoundCond}
        \end{align}
        where $\frac{\alpha}{a}\triangleq \infty$ if $a=0$.
\end{theorem}
\begin{proof}
    The barrier in \eqref{eq:sparceUpperBoundCond} is obtained by projecting Theorem~\ref{thm:upperBoundAlgo} onto the $(\epsilon, \delta, \smu)$-polynomial family. 
    Specifically, it is immediately seen that the scan test is successful if $\beta>\frac{\alpha}{\smu}$. The count test is successful if  \begin{align}
        \beta>\max\p{\frac{2+\alpha}{2\epsilon}, \frac{\alpha}{\epsilon}}.
    \end{align} Note that in region $\alpha \in [0,2]$ we always have  $\frac{2+\alpha}{2\epsilon}>\frac{\alpha}{\epsilon}$ meaning that the maximum is dominated by $\frac{2+\alpha}{2\epsilon}$. Finally, the condition in \eqref{eq:MaxDegGeneral} implies that the maximum degree test succeeds if 
     \begin{align}
        \beta>\max\p{\frac{1+\alpha}{2\delta}, \frac{\alpha}{\delta}}.\label{eqn:maxDegsparse}
    \end{align}  
    Since $0<\beta\leq1$ we claim that $\frac{1+\alpha}{2\delta}$ dominates the maximum in \eqref{eqn:maxDegsparse} in the feasible regime. Indeed, observe that $\frac{\alpha}{\delta}$ dominates the maximum exactly when $\alpha>1$. Because $0<\delta\leq 1$, we have $\frac{\alpha}{\delta}>1$, and therefore, $\beta$ cannot exceed $\frac{\alpha}{\delta}$.
    %Nevertheless, the sufficient condition of \eqref{eq:MaxDegGeneral}, is not a necessary condition in the sparse regime. Indeed, as seen in the proof of Theorem~\ref{thm:upperBoundAlgo}, the count test succeeds if 
    %\begin{align}
    %    \exp{\left(\log{n}-
    %\frac{d^2_{\max}\left(\Gamma\right) \cdot {\left(p-q\right)}^2/8}{\left(n-1\right)q(1-q)+d_{\max}\left(\Gamma\right)\cdot{\left(p-q\right)}}\right)}\xrightarrow[n\to\infty ]{} -\infty.
    %\end{align}
    %Since in the sparse regime it is assumed that  $(p-q)=\Theta(q)$, and since $d_{\max}(\Gamma)\leq n-1$, the term $(n-1)q(1-q)$ dominates the denominator in the above expression.and the entire fraction becomes a $\Theta\p{\frac{d_{\max}^2(\Gamma)\chi^2(p||q)}{n}}$ function. Thus,  the condition 
    %\begin{align}
    %    \frac{d_{\max}^2(\Gamma)\chi^2(p||q)}{n\log(n)}\xrightarrow[n\to\infty ]{} \infty
    %\end{align}
    %is sufficient to guarantee that the maximum degree is successful. 
    Therefore, the maximum degree test achieves strong detection provided that $\beta>\frac{1+\alpha}{2\delta}$, which concludes the proof of \eqref{eq:sparceUpperBoundCond}.
    
    Let us now prove \eqref{eq:sparceLowerBoundCond}. For simplicity, we start by assuming that $\Gamma$ is $\vcd$-balanced, and then generalize. For $\vcd$-balanced graphs,
    \begin{align}
    \vc(\Gamma)\cdot d_{\max}(\Gamma)=|e(\Gamma)|^{1+o(1)}=k^{\epsilon+o(1)}.    
    \end{align} 
    Then, Theorem~\ref{th:lowerVCFull} implies that weak detection is impossible if \eqref{eq:LowerMainCond2} holds, which translates to,
    \begin{align}
        \frac{(1+n^{-\alpha+o(1)})^{k^{\smu+o(1)}}n^{-\alpha+o(1)}}{n^{1+o(1)}}\max\p{\frac{k^{\epsilon+o(1)}}{n^{-\frac{\alpha}{2}+o(1)}},k^{2\delta+\smu+o(1)}}\leq C.
    \end{align}
    Since $k=\Theta(n^{\beta})$, the above is equivalent to,
    \begin{align}
        \frac{(1+n^{-\alpha+o(1)})^{n^{\smu\beta}}n^{-\alpha}}{n^{1+o(1)}}\max\p{\frac{n^{\epsilon\beta}}{n^{-\frac{\alpha}{2}}},n^{(2\delta+\smu)\beta}}\leq C.\label{eq:RawCondSparse}
    \end{align}
    A simple calculation reveals that the l.h.s. of \eqref{eq:RawCondSparse} decays to zero if, 
    \begin{align}
        \beta< \min\p{\frac{\alpha}{\smu}, \frac{1+\alpha}{2\delta+\smu},  \frac{2+\alpha}{2\epsilon}},
    \end{align}
    as stated in \eqref{eq:sparceLowerBoundCond}. In order to reduce the general case to the $\vcd$-balanced scenario, we repeat the ideas in the proof of Theorem~\ref{th:lowerboundDense}, where we have used the $\vcd$-balanced decomposition argument presented in Subsection~\ref{subsec:coverBalancedDecomposition}. Assume that \eqref{eq:sparceLowerBoundCond} holds, and let us denote,
    \begin{align}
        2\rho \triangleq \beta- \min\p{\frac{\alpha}{\smu}, \frac{1+\alpha}{2\delta+\smu},  \frac{2+\alpha}{2\epsilon}} >0.\label{eq:AssumptionSparse}
    \end{align}
    Using Proposition~\ref{prop:GoodDecompositionExists}, consider the decomposition of $\Gamma$ into $\Gamma_1,\dots,\Gamma_M$, where $M=\ceil{\rho^{-1}}$, such that for all $1\leq \ell\leq M$, we have 
    \begin{align}
        \vc(\Gamma_\ell)d_{\max}(\Gamma_\ell)&\leq |e(\Gamma)|d_{\max}(\Gamma)^{\frac{1}{M}}\leq n^{\beta (\epsilon+\rho)+o(1)}.\label{eq:decomposintiosparse}
    \end{align}
    Denote $k_\ell\triangleq|v(\Gamma_\ell)|$, $d_{\ell}\triangleq d_{\max}(\Gamma_\ell)$, $\mu_\ell\triangleq\mu(\Gamma_\ell)$, and $\vc_\ell\triangleq\vc(\Gamma_\ell)$, for all $1\leq \ell\leq M$. Furthermore, recall that $d\triangleq d_{\max}(\Gamma)$, $\mu\triangleq\mu(\Gamma)$, and $\vc\triangleq\vc(\Gamma)$. By 
   \eqref{eq:decomposintiosparse} observe that for all $\ell\in[M]$,  
    \begin{align}
        &\frac{(1+\lambda_M^2)^{\max(\mu_\ell,1)-1}\lambda_M^2}{n-k_\ell}\max \p{\vc_\ell d_\ell\sqrt{\frac{1+\lambda_M^2}{\lambda_M^2}},d_\ell^2\mu_\ell}\\
         &\qquad \leq \frac{(1+\lambda_M^2)^{\max(\mu,1)-1}\lambda_M^2}{n-k}\max \p{\vc_\ell d_\ell\sqrt{\frac{1+\lambda_M^2}{\lambda_M^2}},d^2\mu}\\
         &\qquad \leq \frac{(1+\lambda_M^2)^{n^{\beta\smu+o(1)}}\lambda_M^2}{n^{1+o(1)}}\max \p{n^{\epsilon\beta+\rho}\sqrt{\frac{1+\lambda_M^2}{\lambda_M^2}},n^{\beta(2\delta +\smu)}}.\label{eq:SparceDecompAlmost}
    \end{align}
    Furthermore, note that for a finite $M$ independent of $n$,
    \begin{align}
        \lambda_M^2=(1+\lambda^2)^{M^2}-1=M^2\lambda^2 +O(\lambda^4)=n^{-\alpha+o(1)}.
    \end{align}
    Thus, using \eqref{eq:SparceDecompAlmost} we get,
    \begin{align}
        &\frac{(1+\lambda_M^2)^{\max(\mu_\ell,1)-1}\lambda_M^2}{n-k_\ell}\max \p{\vc_\ell d_\ell\sqrt{\frac{1+\lambda_M^2}{\lambda_M^2}},d_\ell^2\mu_\ell}\\
         &\qquad \leq \frac{(1+n^{-\alpha+o(1)})^{n^{\beta\smu+o(1)}}n^{-\alpha}}{n^{1+o(1)}}\max \p{\frac{n^{\epsilon\beta+\rho}}{n^{-\frac{\alpha}{2}}},n^{\beta(2\delta +\smu)}}.\label{eq:finalSparseSTAT}
    \end{align}
    Under \eqref{eq:AssumptionSparse}, the expression in \eqref{eq:finalSparseSTAT} is upper bounded by $n^{-\rho+o(1)}=o(1)$. In particular, the condition in Proposition~\ref{prop:decompisitionThFull} is satisfied, which imply that weak detection is impossible.
\end{proof}

As can seen from Theorem~\ref{th:polynomial}, our bounds are not tight in general. In fact, for any set of parameters $(\smu,\delta,\epsilon)$ for which the set,
\begin{align}
    A\triangleq\ppp{0<\alpha<2 ~\Big\vert~ \frac{1+\alpha}{2\delta+\smu}>\min\p{\frac{\alpha}{\smu}, \frac{2+\alpha}{2\epsilon}} },\label{eq:EmptylikeSet}
\end{align}
is non-empty, the lower and upper bounds in Theorem~\ref{th:polynomial} do not match. Nonetheless, in the following, we show that for several special non-trivial families of graphs our bounds are, in fact, tight.

\paragraph{A condition for tightness.} A simple algebraic analysis shows that in some non-empty region of the possible values of $(\smu,\delta,\epsilon)$ in the parameter space, the set $A$ can be empty (and then our bounds are tight). This is captured in the following result.
\begin{lemma}[Optimality of the count and scan tests]
    Let $\Gamma=(\Gamma_n)_n$ be an $(\epsilon, \delta, \smu)$-polynomial family of graphs with $ \delta,\smu>0$, $k=|v(\Gamma)|=\Theta(k^{\beta})$, for $0<\beta< 1$, and let $\chi^2(p||q)=\Theta(n^{-\alpha})$, for $0<\alpha< 2$. If
    \begin{align}
    \epsilon> 2\delta+\frac{\smu}{2},\label{eq:whenBoundIsGood}
\end{align} 
         then weak detection is impossible if, 
        \begin{align}
            \beta< \min\p{\frac{\alpha}{\smu},\frac{2+\alpha}{2\epsilon}},
        \end{align}
        while strong detection is possible if,
        \begin{align}
            \beta> \min\p{\frac{\alpha}{\smu},\frac{2+\alpha}{2\epsilon}}.
        \end{align}
\end{lemma}

\begin{proof} For the set $A$ in \eqref{eq:EmptylikeSet} is empty, we investigate when $\beta\leq \frac{2+\alpha}{2\epsilon}$ is dominated by $\beta\leq \frac{1+\alpha}{2\delta+\smu}$. In the parameter space,these two lines cross at,
\begin{align}
    \alpha = \alpha_{0}\triangleq\frac{2(\epsilon-2\delta-\smu)}{2\delta+\smu-2\epsilon}.
\end{align}
Accordingly, the inequality $\frac{2+\alpha}{2\epsilon}> \frac{1+\alpha}{2\delta+\smu}$ is possible only in one of the following two cases:
\begin{itemize}
    \item If, \begin{equation}
\epsilon<\frac{2\delta+\smu}{2} \quad \text{and} \quad   \alpha>\alpha_{0}.\label{eq:StupiedCalc1}
    \end{equation} In this case, for $A$ to be empty, the intersection point must satisfy $\alpha_0>2$, which holds if and only if, \begin{align}
        \epsilon>\frac{2(2\delta+\smu)}{3}.\label{eq:StupiedCalc2}
    \end{align}
    The intersection of the conditions on the l.h.s. of \eqref{eq:StupiedCalc1} and \eqref{eq:StupiedCalc2} is empty. 
    \item  If, \begin{equation}
\epsilon>\frac{2\delta+\smu}{2}  \quad \text{and} \quad   \alpha<\alpha_{0}.\label{eq:StupiedCalc3}
    \end{equation}  
    In this case, for $A$ to be empty, the interval $[0,\alpha_0)$ must be contained in the region where the line $\beta=\frac{\alpha}{\smu}$ is below $\beta=\frac{1+\alpha}{2\delta+\smu}$. A straightforward calculation shows that these lines intersect at $\alpha_1\triangleq\frac{\smu}{2\delta}$, and accordingly, the line $\beta=\frac{\alpha}{\smu}$ is below the line $\beta=\frac{1+\alpha}{2\delta+\smu}$, if $\alpha_0<\alpha_1$. Thus, the set $A$ is empty if $\alpha_1>\alpha_0$ which boils down to,
    \begin{align}
       \frac{2(\epsilon-2\delta-\smu)}{2\delta+\smu-2\epsilon}<\frac{\smu}{\delta},\label{eq:StupiedCalc4}
    \end{align}
    and is equivalent to,
    \begin{align}
        \epsilon > \frac{6\smu\delta +\smu^2+8\delta^2}{4\delta+2\smu}= \frac{(2\delta+\smu)^2+4\delta^2+2\smu\delta}{2(2\delta+\smu)} =\frac{2\delta+\smu}{2}+\frac{4\delta^2+2\smu\delta}{2(2\delta+\smu)} =2\delta+\frac{\smu}{2}.
    \end{align}
    We conclude by noting that the condition $\epsilon>2\delta+\frac{\smu}{2}$ dominates the condition $\epsilon>\frac{2\delta+\smu}{2}$.
\end{itemize}
    
\end{proof}

The condition in \eqref{eq:whenBoundIsGood} is satisfied for several non-trivial families of graphs. As an example consider the family of graphs with moderate degrees, defined as follows: Let $(\epsilon,\delta,\smu)$ be a polynomial sequence of graphs, where $\delta\leq \frac{2}{5}$ (namely, $d_{\max}(\Gamma)\ll k^{2/5}$). In this case, the set $A$ is empty and our lower and upper bounds align. Indeed, note that by the definition of maximal density, we have $\delta>\smu$, and since we consider only graphs containing no isolated vertices we also have that,
\begin{align}
    \epsilon&\geq 1 \geq \frac{5\delta}{2}\geq 2\delta+\frac{\smu}{2}.
\end{align}
Thus, the condition in \eqref{eq:whenBoundIsGood} is satisfied.

\paragraph{Sub-polynomial density.} We consider now a family of graphs, for which the condition in Lemma~\eqref{eq:whenBoundIsGood} is not (necessarily) satisfied, but still Theorem~\ref{th:polynomial} gives tight results. Consider the following definition.
\begin{definition}
    A sequence of graphs $\Gamma=(\Gamma_n)_n$ has sub-polynomial density if, 
    \begin{align}
        \limsup_{n\to\infty} \frac{\log(\mu(\Gamma))}{\log|v(\Gamma)|}=0.
    \end{align}
    Otherwise, $\Gamma$ has super-polynomial density. 
\end{definition}
For graphs $\Gamma$ with sub-polynomial density it must be that $\smu=0$, and then we see that the lower and upper bounds in Theorem~\ref{th:polynomial} coincide. Specifically, weak detection is impossible if,
 \begin{align}
     \beta<\min\p{\frac{2+\alpha}{2\epsilon},\frac{1+\alpha}{2\delta}},
 \end{align}
 while strong detection is possible if, \begin{align}
      \beta>\min\p{\frac{2+\alpha}{2\epsilon},\frac{1+\alpha}{2\delta}}.
 \end{align} 
Note that here the upper bounds are achieved by the count and maximum degree tests, which run in polynomial-time. As so, there is no computational-statistical gap for detecting graphs with sub-polynomial density. In Section~\ref{sec:complowerbound}, we will prove that for graphs with super-polynomial density, a computational-statistical gap exist always.

\paragraph{Super-dense graphs.} Consider the following family of graphs.
\begin{definition}
    A sequence of graphs $\Gamma=(\Gamma)_n$ is super-dense if,
\begin{align}
    \lim_{n\to\infty}=\frac{\log(\mu(\Gamma))}{\log|v(\Gamma)|}=1. \label{eq:superDenseDef}
\end{align}
\end{definition}
Examples for super-dense graphs are: cliques, and balanced bipartite graphs (i.e., bipartite graphs with the same order of left and right vertices). We have the following result. 
\begin{lemma}\label{obs:SuperDense}
    Any family of super-dense graphs is a $\vcd$-balanced $(2,1,1)$-polynomial family.
\end{lemma}
\begin{proof}[Proof of Lemma~\ref{obs:SuperDense}]
    We begin by observing that from the definition of $\mu(\Gamma)$, we have $\mu(\Gamma)\leq d_{\max}(\Gamma)$. Thus, \eqref{eq:superDenseDef} implies that,
    \begin{align}
        \lim_{n\to\infty}\frac{\log d_{\max}(\Gamma)}{\log|v(\Gamma)|}\geq\lim_{n\to\infty}\frac{\log \mu(\Gamma)}{\log|v(\Gamma)|}=1.\label{eq:superDenseParameters}
    \end{align}
    In addition, note that if $\mu(\Gamma)=k^{1+o(1)}$ then $|e(\Gamma)|=k^{2+o(1)}$. To see that, let $G'=(V',E')$ be a subgraph of $\Gamma$ with maximal density. We have,
        \begin{align}
            k^{1+o(1)}=\mu(\Gamma)=\frac{|E'|}{|V'|}\leq \frac{\binom{|V'|}{2}}{|V'|}\leq \frac{|V'|}{2},
        \end{align}
        and in particular,
        \begin{align}
            \frac{k^2}{2}\geq |e(\Gamma)|\geq |E'|=\mu(\Gamma)\cdot |V'|\geq 2\cdot \mu(\Gamma)^2=k^{2+o(1)}.
        \end{align}
        Combined with \eqref{eq:superDenseParameters}, we get that $\Gamma$ is a $(2,1,1)$-polynomial family. Finally, we note that for any graph $\Gamma$, we have $\mu(\Gamma)\leq \vc(\Gamma)$. To see that, let $S\subseteq V$ be a minimal vertex cover of $\Gamma$, and let $\Gamma'=(V',E') $ be a subgraph such that $|E'|/|V'|=\mu(\Gamma)$. Since any edge in $\Gamma$ must contain a vertex in $S$, the number of edges in $\Gamma$, which intersect $V'$, is at most $|V'|\cdot |S|$. In particular, we have $|E'|\leq |V'|\cdot |S|$, and therefore,
    \begin{align}
        k^{1+o(1)}=\mu(\Gamma)=\frac{|E'|}{|V'|}\leq \frac{|S|\cdot |V'|}{|V'|}=|S|=\vc(\Gamma)\leq k.
    \end{align}
    This proves that $\Gamma$ is $\vcd$-balanced.
\end{proof}

 Consider the case where $\Gamma_n$ is a clique with $n^{\beta}$ vertices. The bound in Theorem~\ref{th:polynomial} imply that detection is impossible if,
    \begin{align}
        \beta < \min\p{\alpha, \frac{\alpha+1}{3} }.
    \end{align}
    This bound is loose. Indeed, as was shown in \cite{hajek2015computational}, detection is impossible if,
    \begin{align}
        \beta <\min\p{\alpha, \frac{\alpha}{4}+\frac{1}{2}},
    \end{align}
    while possible if $\beta$ is exceeds the above threshold. This sub-optimality is somewhat expected in the light of Lemma~\ref{obs:SuperDense}. Indeed, the bounding technique used in Theorem~\ref{th:lowerVCFull} hings on the evaluation of the second moment of the likelihood, by running over subgraphs with a fixed number of connected components, and bounding the number of such subgraphs in terms of $\mu(\Gamma)$ and $\vc(\Gamma)$. When $\vc(\Gamma)$ and $\mu(\Gamma)$ are large w.r.t. $|v(\Gamma)|$, this bound becomes loose. Furthermore, our bound in Lemma~\ref{lem:probRandomSubgraph1}, i.e.,
\begin{align}
    \P_{\s{H}}\pp{\s{H}\subseteq \Gamma'} \leq \frac{(2\vc(\Gamma))^m\cdot d_{\max}(\Gamma)^{|v(\s{H})|-m} }{(n-|v(\Gamma)|^{|v(\s{H})|})},
\end{align}
becomes exponentially equivalent to the trivial bound,
\begin{align}
    \P_{\Gamma}\pp{\s{H}'\subseteq \Gamma} \leq \P\pp{v(\s{H}')\subseteq v(\Gamma)} \leq \p{\frac{|v(\s{H}')|}{n-|v(\Gamma)|}}^{|v(\s{H}')|},\label{eq:trivialProbBound}
\end{align}
indicating that bounding the likelihood's second moment by summing subgraphs according to their connected components does not exploit any special properties of the planted structure, and therefore, results in a sub-optimal lower bound. 

As a remedy, we now show that by slightly modifying our techniques\textemdash{}essentially by ignoring the number of connected components\textemdash{}we obtain tight results. 
\begin{prop}\label{prop:denseGraphSparseRegimeStat}
    Let $\Gamma=(\Gamma_n)$ be a sequence of super-dense graphs, with $|v(\Gamma)|=\Theta\p{n^{\beta}}$, and let $\chi^2(p||q)=\Theta\p{n^{-\alpha}}$, with $0<\alpha\leq 2$ and $0\leq \beta\leq 1$. Then, weak detection is impossible if, 
         \begin{align}
        \beta <\min\p{\alpha, \frac{\alpha}{4}+\frac{1}{2}},\label{eq:condSuperDenneSparseReg}
    \end{align}
    while strong detection is possible if, 
    \begin{align}
        \beta >\min\p{\alpha, \frac{\alpha}{4}+\frac{1}{2}}.\label{eq:condSuperDenneSparseReg2}
    \end{align}
\end{prop}

\begin{proof}
        The barrier in \eqref{eq:condSuperDenneSparseReg2} follows immediately from Theorem~\ref{th:polynomial} and Lemma~\ref{obs:SuperDense}. Indeed, since $\Gamma$ is a $(2,1,1)$-polynomial family, a simple calculation shows that, 
        \begin{align}
            \min\p{\frac{\alpha}{1},\frac{1+\alpha}{2\cdot 1}, \frac{2+\alpha}{2\cdot 2}}=\min\p{\alpha,\frac{\alpha}{4}+\frac{1}{2}}.
        \end{align}
        We now prove that detection is impossible when \eqref{eq:condSuperDenneSparseReg} holds. Assume that $\beta<\alpha$. We repeat the main arguments used in the proof of Theorem~\ref{th:lowerVCFull}, where we sum over subgraphs with a fixed number of edges $j$ and vertices $\ell$; we do not constrain/fix the number of connected components. Note that for a fixed $\ell$, the number of subgraphs of $\Gamma$ with $\ell$ vertices and $\ell/2 \leq j\leq \binom{\ell}{2}$ edges is at most,  \begin{align}
            \binom{k}{\ell}\cdot \binom{\binom{\ell}{2}}{j}.
        \end{align}
        Thus, by Lemma~\ref{lem:probRandomSubgraph1} and \eqref{eq:trivialProbBound} we have,
        \begin{align}
            \E\pp{\s{L}(\s{G})}^2&=\sum_{\s{H}'\subseteq \Gamma'} \P_{\Gamma}\pp{\s{H}'\subseteq \Gamma}  \lambda^{2|\s{H}|}\\
            &\overset{(a)}{\leq } 1+\sum_{\ell=2}^k\sum_{j=\ceil{\frac{\ell}{2}}}^{\binom{\ell}{2}}\binom{k}{\ell}  \binom{\binom{\ell}{2}}{j}\lambda^{2j}  \p{\frac{k}{n-k}}^{\ell}\\
            &\overset{(b)}{\leq} 1+\sum_{\ell=2}^k\p{\frac{e   k^2}{(n-k)  \ell}}^{\ell}  \sum_{j=\ceil{\frac{\ell}{2}}}^{\binom{\ell}{2}}\p{  \frac{e\binom{\ell}{2}}{j} \lambda^{2} }^{j}\\
            &\leq 1+\sum_{\ell=2}^k\p{\frac{e   k^2}{(n-k)  \ell}}^{\ell}  \sum_{j=\ceil{\frac{\ell}{2}}}^{\binom{\ell}{2}}\p{e  (\ell-1) \lambda^{2} }^{j}\\
            &\overset{(c)}{\leq} 1+C  \sum_{\ell=2}^k\p{\frac{e^2   k^2  \sqrt{\lambda^2}}{(n-k)}\cdot \frac{\sqrt{\ell-1}}{\ell}}^{\ell} \label{eq:superdenseFinal}
        \end{align}
        where $(a)$ is because we sum over subgraphs containing no isolated vertices which implies that $j\geq \ceil{\ell/2}$, $(b)$ follows from the fact that $\binom{n}{m}\leq\p{\frac{en}{m}}^m$, for any $m\leq n$, and in $(c)$ we used the assumption that $\beta<\alpha$, and so $(\ell-1)\lambda^2\leq k\lambda^2=o(1)$. The sum in \eqref{eq:superdenseFinal} vanishes if $k^2=o\p{n/\lambda}$, and in the $(\alpha,\beta)$ parameter space this is equivalent to $\beta<\frac{1}{2}+\frac{\alpha}{4}$, which concludes the proof. 
    \end{proof}

     \begin{remark}
         The results of Proposition~\ref{prop:denseGraphSparseRegimeStat} are somewhat expected in the light of \cite{rotenberg2024planted}, where the detection problem of a planted bipartite graph was studied. Specifically, it is shown in \cite{rotenberg2024planted} that in the balanced case (where both sides of the bipartite graph have approximately the same number of vertices), the computational and statistical thresholds for the possibility and impossibility of detection are the same as those in Proposition~\ref{prop:denseGraphSparseRegimeStat}. By Szemer\'{e}di's regularity lemma, any super-dense sequence of graphs can be approximated by multi-partite graphs, where the edges between different parts are regular, and the parts are of the same size. In particular, there always exists a decomposition of $\Gamma$ into finitely many (possibly overleaping) approximately regular bipartite graphs $\{\Gamma^{(i)}\}_i$, such that at least one of the parts contains $\Theta(|e(\Gamma)|)$ edges. Using a decomposition-type argument (as in Section~\ref{sec:ReudctionToBalanced}), the risk of detecting $\Gamma$ can be bounded by the second moment $\E_{\calH_0}[\s{L}(\Gamma^{(i)})^2]$ of the largest component $\Gamma^{(i)}$, which, according to \cite{rotenberg2024planted}, is bounded exactly when the condition in Proposition~\ref{prop:denseGraphSparseRegimeStat} is satisfied. Finally, using the same argument as above, the computational thresholds in Proposition~\ref{prop:denseGraphSparseRegimeStat} can be explained as well. 
     \end{remark}

%%%%%%%%%%%%%%%%%%%%%%%%%%%%%%%%%%%%%%%%%%%%%%%%%%%%%%%%%%%%%%%%%%%%%%%%%%%%%%%%%%%%%%%%%%%%%%%%%%%%%%%%%%%%%%%%%%%%%%%%%%%%%%%%%%%%%%%%%%%%%%%%%%%%%%%%%%%%%
\subsubsection{The critical regime}\label{sec:CriticalRegime}

In this subsection, we study the regime where $p=1-o(1)$ and $q=\Theta(n^{-\alpha})$, for $0<\alpha\leq 2$, and thus, $\chi^2(p||q)=\Theta(n^{\alpha})$. We refer to this regime as the \textit{critical regime}; as will be seen later on, in this regime phase transition phenomena occur. 

We remark that while our $\vcd$-balanced decomposition argument, presented in Section~\ref{sec:ReudctionToBalanced}, is generally applicable, it is somewhat ineffective in the critical regime. Indeed, when $\Gamma$ is decomposed into $M$ parts, the conditions of Proposition~\ref{prop:decompisitionThD} are formulated using $\lambda_M^2=(1+\lambda^2)^{M^2}-1$. When $\lambda^2\approx n^{\alpha}$, we have $\lambda_M^2\approx n^{M^2\alpha}$. As $M^2$ is expected to be large for a ``good $\vcd$-balanced approximation", the condition of Proposition~\ref{prop:decompisitionThD} becomes meaningless (since $M^2\alpha$ will exceed $1$ when $\alpha$ is fixed). Despite the general applicability of our vertex cover-based lower bounds, for simplicity of representation, we will focus our investigation on $\vcd$-balanced graphs. It is worth noting that in the bounded degree case\textemdash{}of particular interest due to the phase transition phenomena it exhibits\textemdash{}the $\vcd$-balanceness constraint is automatically satisfied. Finally, as will be evident in the sequel, it is quite convenient to separate the investigation into three distinct cases: $\mu(\Gamma)\geq 1$, $\mu(\Gamma)<1-\delta$, and $\mu(\Gamma)=1-o(1)$. 
\begin{itemize}[leftmargin=*]
    \item \underline{\textbf{The case where $\mu(\Gamma)=1-o(1)$}}: We begin with the case where $\mu(\Gamma)$ converge to unity from below, which is perhaps the most intriguing case, due to the emergence of a phase transition that occurs when $\alpha=1$. We have the following result.
\begin{theorem}\label{thm:phaseTansitionsCrit}
Let $\Gamma=(\Gamma_n)_n$ be a $\vcd$-balanced sequence of  graphs such that $1-o(1)\leq \mu(\Gamma_n)<1$ for some $o(1)$ function, and assume that $p=1-o(1)$ and $q=\Theta(n^{-\alpha})$. Then, there exists a function $f(n)=o(1)$ such that the following hold.
\begin{enumerate}[leftmargin=*]
    \item Fix $0<\alpha<1$. Then, weak detection is impossible if,
    \begin{align}
        |e(\Gamma)|\leq n^{1-\frac{\alpha}{2}-\varepsilon} \quad \text{and} \quad d_{\max}(\Gamma)\leq n^{\frac{1-\alpha}{2}-\varepsilon},\label{eq:CriticalCond1}
    \end{align}
    for some $\varepsilon>0$, while strong detection is possible if 
    \begin{align}
        |e(\Gamma)|\geq n^{1-\frac{\alpha}{2}+f(n)} \quad \text{or} \quad d_{\max}(\Gamma)\geq n^{\frac{1-\alpha}{2}+f(n)}.\label{eq:CriticalCond1.1}
    \end{align}
    \item Assume $q=\frac{\sigma}{n}$ (in particular, $\alpha=1$), for some $\sigma>0$.
    \begin{enumerate}
        \item If $d_{\max}(\Gamma)=\Omega\p{|v(\Gamma)|^\beta}$, for $0<\beta\leq 1$, then weak detection is impossible if, 
        \begin{align}
            |e(\Gamma)|\leq n^{\frac{1}{2}-\varepsilon} \quad \text{and} \quad  d_{\max}^{\frac{1}{\beta}}(\Gamma)\cdot \log(d_{\max}(\Gamma))\leq \frac{1-\varepsilon}{2} \log n, \label{eq:CriticalCond2}
        \end{align}
        while strong detection is possible if, 
        \begin{align}
             |e(\Gamma)|\leq n^{\frac{1}{2}+\varepsilon} \quad \text{or} \quad  d_{\max}(\Gamma)\geq (16+\varepsilon)\log n,
        \end{align}
        for some  fixed $\varepsilon>0$.
        \item If $d_{\max}(\Gamma)=O(1)$, then:
        \begin{enumerate}
            \item If $\sigma > 2e{d^2}$,
            then weak detection is impossible if, 
        \begin{align}
          |e(\Gamma)|\leq n^{\frac{1}{2}-\varepsilon},
        \end{align}
        while strong detection is possible if 
        \begin{align}
           |e(\Gamma)|\geq n^{\frac{1}{2}+\varepsilon},
        \end{align}
        for some $\varepsilon>0$.
            \item If $ \sigma <1$ and $\mu(\Gamma)\geq 1-|v(\Gamma)|^{-\beta}$, for $0<\beta\leq 1$,
            then weak detection is impossible if, 
        \begin{align}
         |v(\Gamma)|\leq \frac{\log\p{\frac{e d^2}{\sigma}}}{1+\varepsilon}\cdot \log n,\label{eq:CriticalCond4.1}
        \end{align}
        while strong detection is possible if, 
        \begin{align}
            |v(\Gamma)|\geq (1+\varepsilon)\cdot\p{\log n}^{\frac{1}{\beta}},\label{eq:CriticalCond4.2}
        \end{align}
        for some $\varepsilon>0$.
        \end{enumerate}
    \end{enumerate}
\end{enumerate}

\end{theorem}

 \begin{proof}[Proof of Theorem~\ref{thm:phaseTansitionsCrit}]
    The proof essentially follows by analyzing the conditions of  Theorems~\ref{thm:upperBoundAlgo} and \ref{th:lowerVCFull}. Let  $\lambda^2\triangleq\chi^2(p||q)=\Theta(n^{\alpha})$,  $k\triangleq|v(\Gamma)|$, $\vc\triangleq\vc(\Gamma)$, $d\triangleq d_{\max}(\Gamma)$, and $\calE\triangleq|e(\Gamma)|$. Let us start with the case where $\alpha<1$. 
    \begin{enumerate}[leftmargin=*]
    \item Assume $0<\alpha<1$. Note that the $\vcd$-balances assumption implies that,
    \begin{align}
        \vc\cdot d=\calE^{1+o(1)}\ll \calE^{1+\frac{\varepsilon}{2}},\label{eq:balancedCriticial}
    \end{align}
    for every $\varepsilon>0$, and a sufficiently large $n$. Using the fact that $k\leq \calE\ll n$, condition \eqref{eq:LowerMainCond3} in Theorem~\ref{th:lowerVCFull} implies that weak detection is impossible if, 
    \begin{align}
        \max\p{\frac{\vc d\lambda}{n},\frac{d^2\lambda^2}{n^2}}=o(1).\label{eq:StopEq}
    \end{align}
    By \eqref{eq:balancedCriticial}, since $\lambda^2=\Theta(n^{\alpha})$, the condition in \eqref{eq:CriticalCond1} guarantees,
    \begin{align}
       \max\p{\frac{\vc d\lambda}{n},\frac{d^2\lambda^2}{n^2}}\leq  \max\p{\frac{\calE^{1+\frac{\varepsilon}{2}}\lambda}{n},\frac{d^2\lambda^2}{n^2}} =O\p{n^{-\varepsilon/4}},
    \end{align}
    and in particular \eqref{eq:StopEq} holds. The upper bound in \eqref{eq:CriticalCond1.1} follows immediately from Theorem~\ref{thm:upperBoundAlgo}, by taking,
    \begin{align}
        f(n)=\log_n\p{17\log n}=o(1).
    \end{align}
    \item Assume $q=\frac{\sigma}{n}$. Let us begin with the scenario where $\Gamma$ has polynomial order maximum degree, i.e., \begin{align}
        d= \Omega(k^{\beta} ),\label{eq:AssumptLog}
    \end{align} for some $\beta>0$. Assuming that $\calE\leq n^{\frac{1}{2}-\varepsilon}$, because, 
    \begin{align}
        \vc=\vc(\Gamma)\leq |v(\Gamma)|=|e(\Gamma)|=\calE \quad \text{and} \quad \lambda^2=(1+o(1))\cdot \frac{n}{\sigma},\label{eq:balancedIsGood}
    \end{align}
    we get,
    \begin{align}
        \frac{\vc^2}{\lambda^2 d^2 }\leq \frac{\calE^2}{\lambda^2}\leq n^{-\varepsilon+o(1)}=o(1). 
    \end{align}
    Furthermore, the assumption in \eqref{eq:AssumptLog} implies that, 
    \begin{align}
        \frac{e\lambda^2d^2}{n-k}\geq C  k^{\beta}=\omega(1),
    \end{align}
    where the last equality is because it is assumed that $k\to\infty$. The above two inequalities show that the condition in \eqref{eq:LowerMainCond4} holds. Thus, by \eqref{eq:LowerMainCond5} in Theorem~\ref{th:lowerVCFull}, weak detection is impossible if, 
    \begin{align}
        \frac{\vc^2}{n-k}\cdot\p{ \frac{e\cdot \lambda^2 d^2}{n-k}}^{k-1}=o(1).\label{eq:ineqCritLog}
    \end{align}
    We note that since $k\leq \calE$, we have $k\leq n^{\frac{1}{2}-\varepsilon}$ and thus, 
    \begin{align}
       \p{ \frac{\lambda^2}{n-k}}^{k-1}&\leq \p{\frac{n}{\sigma(n-k)}}^k\\
        &=\frac{1}{\sigma^k}\p{1+\frac{k}{n-k}}^k\\& \leq \frac{1}{\sigma^k}\cdot\exp\overset{o(1)}{\overbrace{\p{\frac{k^2}{n-k}}}}=\frac{1+o(1)}{\sigma^k}.
    \end{align}
    Since $\vc\leq k$, the above shows that \eqref{eq:ineqCritLog} holds if, 
    \begin{align}
        \frac{k^2}{n-k}\cdot \p{\frac{(1+\varepsilon)ed^2}{\sigma}}^k=o(1),
    \end{align}
    which is equivalent to,
    \begin{align}
        2\log k-\log(n-k) + k(1+\log(1+\varepsilon)-\log\sigma +2\log d )\to -\infty,
    \end{align}
    and clearly holds if, 
    \begin{align}
        2k\log d\leq 2d^{\frac{1}{\beta
        }}\log d \leq (1-\varepsilon) \log n,
    \end{align}
    for some fixed $\varepsilon>0$. For the possibility proof, note that by Theorem~\ref{thm:upperBoundAlgo}, as in the previous case, the count test succeeds if, 
    \begin{align}
        \calE\geq n^{1+\frac{\alpha}{2}+\varepsilon}.
    \end{align}
    Furthermore, by Theorem~\ref{thm:upperBoundAlgo} the maximum degree test achieves strong detection if,
    \begin{align}
        \min\p{\frac{d^2(1+o(1))}{\sigma \log n},\frac{d(1-n^{-\alpha})}{\log n}}=\min\p{\frac{d^2\lambda^2}{n\log n},\frac{d(p-q)}{\log n}}>16,
    \end{align}
    which concludes the proof.
    \item Assume $q=\frac{\sigma}{n}$, and let us now handle the case of bounded degrees, and begin with the case where,
    \begin{align}
        \sigma>(1+\varepsilon)2ed^2.
    \end{align}
    A close inspection of the proof of \eqref{eq:LowerMainCond3} in Theorem~\ref{th:lowerVCFull} reveals that the constant $C$ in \eqref{eq:LowerMainCond3} can be upper bounded by $2e$. Thus, assuming further that $\calE\leq n^{\frac{1}{2}-\varepsilon}$, by using \eqref{eq:balancedIsGood} we obtain, 
    \begin{align}
        \frac{d\lambda}{n-k}\max\p{\vc,d\lambda }\leq (1+o(1))\max\p{\frac{\calE^{1+\varepsilon}}{(\sigma n)^{\frac{1}{2}}},\frac{d^2}{\sigma}}\leq \frac{2e}{1+\varepsilon}.
    \end{align}
    Thus, \eqref{eq:LowerMainCond3} holds and detection is impossible. The possibility thresholds follow from Theorem~\ref{thm:upperBoundAlgo}, as in the previous case.
    
    Finally, we handle the case where $\sigma<1$. Here, assuming that $k\ll \sqrt{n}$, we have  
    \begin{align}
         \frac{\vc^2}{\lambda^2 d^2 }\leq \frac{\sigma k^2}{n}=o(1),
    \end{align}
    and 
    \begin{align}
         \frac{e\lambda^2d^2}{n-k}\geq C  k^{\beta}\geq (1+o(1))\frac{ed^2}{\sigma}>1.
    \end{align}
    In particular, the condition in \eqref{eq:LowerMainCond4} is satisfied. Similarly to the scenario of unbounded degrees, the condition in \eqref{eq:LowerMainCond5} is satisfied if,
    \begin{align} 2\log k-\log(n-k) + k(1+\log(1+\varepsilon)-\log\sigma +2\log d )\to -\infty, \end{align}
    which holds if, 
    \begin{align}
        k\leq \frac{\log\p{\frac{e d^2}{\sigma}}}{1+\varepsilon}\cdot \log n,
    \end{align}
    for some $\varepsilon>0$. On the  other hand, by Theorem~\ref{thm:upperBoundAlgo}, the scan test succeeds if,
    \begin{align}
        \liminf_{n\to\infty} \mu\cdot \frac{d_{\s{KL}}(p||q)}{\log n}>1.\label{eq:condKL}
    \end{align}
    Note that under the assumption that $p=1-o(1)$ and $q=\frac{\sigma}{n}$, we get, 
    \begin{align}
        d_{\s{KL}}(p||q)=(1+o(1))\log\p{\frac{n}{\sigma}}=(1+o(1)) \p{\log n-\log\sigma}.
    \end{align}
    Since $\mu\geq  1-k^{-\beta}$, for some constant $\alpha$, we have,
    \begin{align}
         \mu\cdot \frac{d_{\s{KL}}(p||q)}{\log n}&\geq \p{ 1- k^{-\beta}}\p{1+\frac{\log(\sigma^{-1})}{\log n}}.
    \end{align}
    In particular, \eqref{eq:condKL} holds if, 
    \begin{align}
        k>(1+\varepsilon)\cdot\p{\log n}^{\frac{1}{\beta}}, 
    \end{align}
    for some fixed $\varepsilon>0$.
    \end{enumerate}
\end{proof}

Before we proceed to the other scenarios, let us compare Theorem~\ref{thm:phaseTansitionsCrit} with the findings of \cite{massoulie19a} on the case where $\Gamma$ is a regular tree and $q=\sigma/n$. In the situation where the maximum degree is polynomial w.r.t. the number of vertices, our bounds suggest that the statistical limits are determined by both the number of edges (as compared to $\sqrt{n}$), and the maximum degree (as compared to poly-logarithmic function of $n$). This includes the example of a planted star, studied in \cite{massoulie19a}, which, in our terminology, corresponds to the case where $d_{\max}(\Gamma)=|v(\Gamma)|^\beta$, with $\beta=1$. In this case our lower bound matches the one in \cite{massoulie19a} (up to a factor of $2$), and the upper bound is the same up to a factor of $\log\log d_{\max}$. We comment that the upper bound in \cite{massoulie19a} uses the same test (maximum degree) that we employ, but analyzed specifically for the critical regime, leading to a slightly sharper result. 

In the bounded degree case, our results generalize those of \cite{massoulie19a} concerning $D$-regular trees. Note that bounded degree graphs are $\vcd$-balanced, and the assumption of Theorem~\ref{thm:phaseTansitionsCrit} is automatically satisfied. Our results then reveal that the phase transition observed in \cite{massoulie19a} is, in fact, a general phenomena. Indeed, Theorem~\ref{thm:phaseTansitionsCrit} suggests that for any sequence of bounded degree graphs, there are thresholds $\underline{\sigma}(d_{\max})\leq \overline{\sigma}(d_{\max})$, such that for $\sigma<\overline{\sigma}(d_{\max})$, the barrier for detection is determined by comparing $|v(\Gamma)|$ to a poly-logarithmic function of $n$, while if $\sigma>\overline{\sigma}(d_{\max})$, the barrier for detection is governed by the count test, namely, $|v(\Gamma)|=\omega(\sqrt{n})$. 

We conclude this discussion by pointing out that the proof techniques used for proving the lower bounds in \cite{massoulie19a}, in the case where $\alpha=1$, do not generalize well to other values of $\alpha$. For example, consider the case where $\Gamma$ is a path with $k$ edges. Then, when $0\leq \alpha<1$, the coupling approach used in \cite{massoulie19a} shows that detection is impossible for $k\ll \sqrt{n}$, which is loose as compared to Theorem~\ref{thm:phaseTansitionsCrit}, which states that detection is impossible as long as $k\ll n^{1-\frac{\alpha}{2}}$.

%%%%%%%%%%%%%%%%%%%%%%%%%%%%%%%%%%%%%%%%%%%%%%%%%%%%%%%%%%%%%%%%%%%%%%%%%%%%%

\item \underline{\textbf{The case where $\mu(\Gamma)<1-\delta$:}} Let $\Gamma$ be a sequence of graphs such that,
\begin{align}
    \lim_{n\to \infty} \mu(\Gamma_n)\triangleq \mu =1-\delta,\label{eq:condOneDelta}
\end{align}
for some $\delta>0$. Below, we show that if $\Gamma$ is $\vcd$-balanced and satisfies \eqref{eq:condOneDelta}, then our lower and upper bounds coincide. 
\begin{theorem}\label{thm:CritOneDelta}
    Let $\Gamma=(\Gamma_n)_n$ be a sequence of graphs satisfying \eqref{eq:condOneDelta}. Then, for all $0\leq \alpha < \frac{1}{\mu}$, weak detection is impossible if,
    \begin{align}
        |v(\Gamma)|= o\p{n^{g_{\mu}(\alpha)}},
    \end{align}
    while strong detection is possible if,
    \begin{align}
        |v(\Gamma)|=\omega\p{n^{1-\frac{\alpha}{2}}},
    \end{align}
    where 
    \begin{align}
        g_{\mu}(\alpha)\triangleq\begin{cases}
            1-\frac{\alpha}{2} & 0\leq \alpha \leq 1\\
            \frac{1-\mu \alpha }{2(1-\mu)} & 1\leq \alpha < \frac{1}{\mu}.
        \end{cases}
    \end{align}

\end{theorem}
\begin{proof}[Proof of Theorem~\ref{thm:CritOneDelta}]
    When $0<\alpha\leq 1$, the analysis of the lower bound is exactly the same as in the proof of the first case in Theorem~\ref{thm:phaseTansitionsCrit}, where we rely on the fact that if $\mu<1-\delta$, then $\Gamma$ is a sequence of bounded degree graphs (see, the proof of Lemma~\ref{lem:subgraphsBoundedDegreeCountL}, for a detailed explanation). Hence, $\Gamma$ is clearly $\vcd$-balanced and,
    \begin{align}
        |v(\Gamma)| \leq |e(\Gamma)|\leq v(\Gamma)\cdot d_{\max}(\Gamma). \label{eq:boundedDegG}
    \end{align}
    We therefore focus on the case where $1<\alpha\leq \mu^{-1}$. Denote $k\triangleq|v(\Gamma)|$, and note that if $k\ll n^{g_\mu(\alpha)}$, then $k\ll \sqrt{n}$. Thus, 
    \begin{align}
        \frac{k^2}{\lambda^2}\max\p{1,\frac{n-k}{k^2}}\leq C\cdot \frac{k^2}{n^{\alpha}}\max\p{1,\frac{n-k}{k^2}}=o(1).
    \end{align}
    Accordingly, due to \eqref{eq:LowerMainCond6} in Theorem~\ref{th:lowerVCFull}, weak detection is impossible if, 
    \begin{align}
        \frac{k^2}{\delta \lambda^2}\p{\frac{e\lambda^2}{\delta(n-k)}}^{\frac{1}{\delta}}=o(1).
    \end{align}
    Since $\delta=\frac{1}{1-\mu}$, the above holds provided that, 
    \begin{align}
        k=o\p{n^{\frac{1-\mu\alpha}{2(1-\mu)}}}. 
    \end{align}
    Finally, the upper bound follows from the performance of the count test in Theorem~\ref{thm:upperBoundAlgo}, combined with \eqref{eq:boundedDegG}, which asserts that $k=\Theta(|e(\Gamma)|)$.
\end{proof}

We note that there is a gap between our lower and upper bounds in the regime where $1<\alpha<\frac{1}{\mu}$; we suspect that the upper bound is loose, and the lower bound, is tight. Indeed, we note that for $\alpha>\frac{1}{\mu}$, detection is trivial, even if $\Gamma$ is finite. To see that, note that by Theorem~\ref{thm:upperBoundAlgo}, the scan test achieves detection if, 
\begin{align}
    \mu>(1+\varepsilon)\frac{\log n}{d_{\s{KL}}(p||q)}>\frac{\log n}{\alpha\log n+C}=\frac{1}{\alpha}(1+o(1)).\label{eq:trivialMU}
\end{align}
We also observe that $g_\mu(\alpha)$ in Proposition~\ref{thm:CritOneDelta} approaches zero when $\alpha\to \frac{1}{\mu
}$, and the lower bound becomes trivial exactly when detection becomes trivial. %Nevertheless, it is not impossible that the upper count test is possible, and the threshold for detection exhibits a sharp phase transition at the point $\alpha=\frac{1}{\mu}$.  

\item \underline{\textbf{The case where $\mu(\Gamma)\geq 1$:}} Let $\mu$ be defined as the limit value of $\mu(\Gamma)$, as in \eqref{eq:condOneDelta}. Note that by the analysis of \eqref{eq:trivialMU}, if $\alpha>\frac{1}{\mu}$, then detection becomes trivial because the scan test is successful even if $\Gamma$ is finite. We therefore assume that $0<\alpha<\frac{1}{\mu}$. 
\begin{theorem}\label{thm:lastTheorem}
            Let $\Gamma=(\Gamma_n)_n$ be a sequence of $\vcd$-balanced graphs such that $\mu=\lim \mu(\Gamma_n)<\infty$, and let $0<\alpha<\frac{1}{\mu}$. Then weak detection is impossible if 
        \begin{align}
            \max\p{|e(\Gamma)|,d_{\max}^2(\Gamma)}\leq n^{1-\alpha \mu-\varepsilon},
        \end{align}
        for some $\varepsilon>0$, while strong detection is possible if 
        \begin{align}
             |e(\Gamma)|=n^{1+\frac{\alpha}{2}+f(n)} \quad \text{or}\quad d_{\max}^2\geq n^{\frac{1+\alpha}{2}+f(n)},
        \end{align}
        where $f(n)$ is and $o(1)$ function.
\end{theorem}

\begin{proof}[Proof of Theorem~\ref{thm:lastTheorem}]
    The proof of the upper bound follows immediately from Theorem~\ref{thm:upperBoundAlgo}. The lower bounds are a direct consequence of \eqref{eq:LowerMainCond1} in Theorem~\ref{th:lowerVCFull}, combined with the $\vcd$-balanceness assumption.
\end{proof}
\end{itemize}

%%%%%%%%%%%%%%%%%%%%%%%%%%%%%%%%%%%%%%%%%%%%%%%%%%%%%%%%%%%%%%%%%%%%%%%%%%%%%%%%%%%%%%%%%%%%%%%%%%%%%%%%%%%%%%%%%%%%%%%%%%%%%%%%%%%%%%%%%%%%%%%%%%%%%%%%%%%%%%%%%%%%%%%%%%%%%%%%%%%%%%%%%%%%%%%%%%
\section{Computational Lower Bounds}\label{sec:complowerbound}

This section is devoted to establishing and proving general computational lower bounds for the detection problem introduced in Section~\ref{sec:problem}. We start with a brief background on the low-degree polynomial (LDP) method; the interested reader is referred to \cite{Dmitriy19}, for a detailed exposition.

\subsection{Background: LDP method}\label{subsec:LDPBack}

The LDP framework hings on the hypothesis that all polynomial-time algorithms for solving detection problems are captured/represented by low-degree polynomials. To date, there is increasing and compelling evidence supporting this conjecture. The concepts described below were developed through a fundamental sequence of works in the sum-of-squares optimization literature \cite{barak2016nearly,Hopkins18,hopkins2017bayesian,hopkins2017power}. 

We begin by outlining the fundamentals of of the LDP framework, adhering to the notations and definitions established in \cite{Hopkins18,Dmitriy19}. Recall that any distribution $\pr_{\calH_0}$
defines an inner product of measurable functions $f,g:\Omega_n\to\mathbb{R}$ given by $\left\langle f,g \right\rangle_{\calH_0} = \bE_{\calH_0}[f(\s{G})g(\s{G})]$, along with a norm defined as $\norm{f}_{\calH_0} = \left\langle f,f \right\rangle_{\calH_0}^{1/2}$. Also, recall that the space $L^2(\calH_0)$ represents the Hilbert space of function $f$ with  $\norm{f}_{\calH_0}<\infty$, equipped with the above inner product. The core idea of the LDP method is to identify the ``low-degree" polynomial that distinguishes $\pr_{\calH_0}$ from $\pr_{\calH_1}$ best in the $L^2$ sense. To define this mathematically, let $V_{n, \leq\s{D}}\subset L^2(\calH_0)$ denote the subspace of polynomials of degree at most $\s{D}\in\mathbb{N}$. Then, the \emph{$ \s{D}$-low-degree likelihood ratio} $\s{L}_{n,\leq  \s{D}}$ is defined as the projection of the likelihood $\s{L}_{n}$ onto $V_{n, \leq\s{D}}$, w.r.t. $\left\langle \cdot,\cdot \right\rangle_{\calH_0}$. Now, recall that the likelihood ratio is the optimal test to distinguish $\pr_{\calH_0}$ from $\pr_{\calH_1}$, in the $L^2$ sense. As it turns out, the $\s{D}$-low-degree likelihood ratio shares a similar property \cite{hopkins2017bayesian,hopkins2017power,Dmitriy19}.
%Let us describe the basics of the LDP framework. We follow the notations and definition in \cite{Hopkins18,Dmitriy19}. Recall that any distribution $\pr_{\calH_0}$ induces an inner product of measurable functions $f,g:\Omega_n\to\mathbb{R}$ given by $\left\langle f,g \right\rangle_{\calH_0} = \bE_{\calH_0}[f(\s{G})g(\s{G})]$, and norm $\norm{f}_{\calH_0} = \left\langle f,f \right\rangle_{\calH_0}^{1/2}$. Also, recall that $L^2(\calH_0)$ be the Hilbert space of functions $f$ with $\norm{f}_{\calH_0}<\infty$, and endowed with the inner product above. The idea behind the LDP framework is to find the ``low-degree" polynomial that best distinguishes $\pr_{\calH_0}$ from $\pr_{\calH_1}$ in the $L^2$ sense. Let $V_{n, \leq\s{D}}\subset L^2(\calH_0)$ denote the linear subspace of polynomials of degree at most $ d\in\mathbb{N}$. Then, the \emph{$ d$-low-degree likelihood ratio} $\s{L}_{n,\leq  d}$ is the projection of a function $\s{L}_{n}$ to $\s{L}_{n, \leq\s{D}}$, where the projection is w.r.t. the inner product $\left\langle \cdot,\cdot \right\rangle_{\calH_0}$. As is well known, the likelihood ratio is the optimal test to distinguish $\pr_{\calH_0}$ from $\pr_{\calH_1}$, in the $L^2$ sense. Accordingly, it can be shown that over $\s{L}_{n, \leq\s{D}}$, the function $\s{L}_{n, \leq\s{D}}$ exhibits the same property \cite{hopkins2017bayesian,hopkins2017power,Dmitriy19}. 
\begin{lemma}[Optimally of $\s{L}_{n, \leq\s{D}}$ {\cite{hopkins2017bayesian,hopkins2017power,Dmitriy19}}]\label{lem:Dmitriy}
Consider the following optimization problem:
\begin{equation}
\begin{aligned}
\mathrm{max}
\;\bE_{\calH_1}f(\s{G})
\quad\mathrm{s.t.}
\quad\bE_{\calH_0}f^2(\s{G}) = 1,\; f\in V_{n, \leq\s{D}},
\end{aligned}\label{eqn:optimizationProblem}
\end{equation}
Then, the unique solution $f^\star$ for \eqref{eqn:optimizationProblem} is the $\s{D}$-low degree likelihood ratio $f^\star = \s{L}_{n, \leq\s{D}}/\norm{\s{L}_{n, \leq\s{D}}}_{\calH_0}$, and the value of the optimization problem is $\norm{\s{L}_{n, \leq\s{D}}}_{\calH_0}$. 
\end{lemma}
As we have already seen in Section~\ref{sec:Polynomial}, a key characteristic of the likelihood ratio is that when $\norm{\s{L}_n}_{\calH_0}=O(1)$, then $\pr_{\calH_0}$ and $\pr_{\calH_1}$ are statistically indistinguishable (or, strong detection is impossible). The conjecture below extends this principle to the computational realm. In a nutshell, it suggests that polynomials of degree $\approx\log n$ can effectively represent/cover all polynomial-time algorithms. The statement below is inspired by  \cite{Hopkins18,hopkins2017bayesian,hopkins2017power}, and \cite[Conj. 2.2.4]{Hopkins18}. Here, we present the informal version which appears in \cite[Conj. 1.16]{Dmitriy19}, while more formal statements can be found in, e.g., \cite[Conj. 2.2.4]{Hopkins18} and \cite[Sec. 4]{Dmitriy19}.
\begin{conjecture}[Low-degree conj., informal]\label{conj:1}
Given a sequence of probability measures $\pr_{\calH_0}$ and $\pr_{\calH_1}$, if there exists $\epsilon>0$ and $\s{D} =  \s{D}(n)\geq (\log n)^{1+\epsilon}$, such that $\norm{\s{L}_{n, \leq\s{D}}}_{\calH_0}$ remains bounded as $n\to\infty$, then there is no polynomial-time algorithm that distinguishes $\pr_{\calH_0}$ and $\pr_{\calH_1}$.
\end{conjecture}

Let us derive a general expression for $\norm{\s{L}_{n, \leq\s{D}}}_{\calH_0}$ in our setting. To that end, we expand the likelihood ratio $\s{L}_n$ in a basis of orthogonal polynomials w.r.t. $\pr_{\calH_0}$. Specifically, suppose that $f_0,f_1,\ldots,f_B:\Omega^n\to\mathbb{R}$ is an orthonormal basis for the coordinate degree-$\s{D}$ functions, and that $f_0$ is the unit constant function. Then, we have,
\begin{align}
\norm{\s{L}_{n, \leq\s{D}}}_{\calH_0}^2 &= \sum_{1\leq i\leq B}\left\langle f_i,\s{L}_{n, \leq\s{D}} \right\rangle_{\calH_0}^2\\
& = \sum_{1\leq i\leq B}\pp{\bE_{\calH_0}\pp{\s{L}_n(\s{G})f_i(\s{G})}}^2\\
& = \sum_{1\leq i\leq B}\pp{\bE_{\calH_1}f_i(\s{G})}^2,\label{eqn:orthogonalDecomposition}
\end{align}
where we have used the fact that $(\s{L}_n-\s{L}_{n, \leq\s{D}})$ is orthogonal to $\{f_i\}_{i=0}^B$. Therefore, it remains to compute $\bE_{\calH_1}f_i(\s{G})$. As we have already seen, in our setting, the corresponding basis is the Fourier character $\{\chi_{\s{H}}\}_{\s{H}\subseteq\binom{[n]}{2}}$ defined in \eqref{eqn:Fouriercharacter}. As established in \cite{odonnell_2014}, the set $\{\chi_{\s{H}}\}_{\s{H}\subseteq\binom{[n]}{2},|e(\s{H})|\leq \s{D}}$ forms an orthonormal basis for the degree-$\s{D}$ functions w.r.t. $\pr_{\calH_0}$. Thus, using \eqref{eq:coeffCalc2}, we obtain that,
\begin{align}
\norm{\s{L}_{n, \leq\s{D}}}_{\calH_0}^2 &= \sum_{\substack{\s{H}'\subseteq \binom{[n]}{2}\\ |e(\s{H}')|\leq \s{D}}} \E_{\calH_0}\pp{\s{L}\chi_{\s{H}'}}^2\\
&=\sum_{\substack{\s{H}'\subseteq \binom{[n]}{2}\\ |e(\s{H}')|\leq \s{D}}} \E_{\calH_1}\pp{\chi_{\s{H}'}}^2\\
        &=\sum_{\substack{\s{H}'\subseteq \binom{[n]}{2}\\ |e(\s{H}')|\leq \s{D}}} \lambda^{2|e(\s{H}')|}\P_\Gamma\pp{\s{H}'\subseteq \Gamma}^2.\label{eq:LDPstart}
\end{align}
Furthermore, following the exact same steps in the proof of Lemma~\ref{lem:intersectionMomentGenGeneral}, an alternative equivalent expression is,
    \begin{align}
         \norm{\s{L}_{n, \leq\s{D}}}_{\calH_0}^2 &=\sum_{\substack{\s{H}'\subseteq \binom{e(\Gamma)}{2}\\ |e(\s{H}')|\leq \s{D}}} \lambda^{2\abs{e(\s{H}')}} \P_{\s{H}}[\s{H}\subseteq \Gamma],\label{eq:LDPstart2}
    \end{align}
where $\s{H}$ is an independent copy of $\s{H}'$ in $\calK_n$. We remark that \eqref{eq:LDPstart2} also follows as a corollary of Lemma~\ref{lem:intersectionMomentGenGeneralComp}.

\subsection{LDP analysis for planted subgraphs}\label{subsec:DenseLDP}

Our investigation begins with the following fundamental property: all planted graphs $\Gamma$ with super-logarithmic density exhibit a (conjecturally) non-empty hard region, where detection is statistically possible, yet no polynomial-time testing algorithm exists. In the second part of this section, we establish general computational lower bounds.

\subsubsection{Warm up: super-logarithmic density and hardness}\label{subsec:supperLDP}

Before delving into the LDP analysis of our detection problem, we begin with a simple yet intriguing observation. Recall that a sequence $(\Gamma_n)_n$ has super-logarithmic density if 
\begin{align}
    \mu(\Gamma_n)=\Omega\p{\log |v(\Gamma_n)|}. 
\end{align}
As mentioned in Section~\ref{sec:mainresults}, super-logarithmic density is a necessary condition/property for the existence of a computational-statistical gap. Conversely, for planted graphs with sub-logarithmic density, no such gap exists. In this subsection we prove that super-logarithmic density is also a sufficient condition, namely, all planted graphs with this property exhibit a computational-statistical gap. To establish this, we start with the following result.
\begin{theorem}\label{obs:lowdegreeClique}
    Let $\s{D}=(\s{D}_n)_n$ be an increasing sequence such that $\omega(1)  \leq \s{D}\leq o(\sqrt{n})$, and assume that $\lambda^2=\chi^2(p||q)= o(\sqrt{\s{D}})$. Then, for every sequence of graphs $\Gamma=(\Gamma_n)_n$, such that,
    \begin{align}
        |v(\Gamma)|= O\p{\p{2\lambda^2 \s{D}}^{-\frac{\sqrt{\s{D}}}{2\sqrt{2}}}\sqrt{n}},\label{eqn:CompCliqueBased}
    \end{align}
     we have $\norm{\s{L}_{n, \leq\s{D}}}_{\calH_0} = O(1)$.
\end{theorem}
\begin{proof}[Proof of Theorem~\ref{obs:lowdegreeClique}]
Recall \eqref{eq:LDPstart}, and consider the following chain of inequalities,
    \begin{align}
        \norm{\s{L}_{n, \leq\s{D}}}_{\calH_0}^2&\leq\sum_{\substack{\s{H}\subseteq \binom{[n]}{2}\\ |\s{H}|\leq \s{D}}} \lambda^{2|\s{H}|}\P\pp{v(\s{H})\subseteq v(\Gamma)}^2\\
        &\overset{(a)}{=}\sum_{\substack{\s{H}\subseteq \binom{[n]}{2}\\ |\s{H}|\leq \s{D}}} \lambda^{2|\s{H}|}\p{\frac{\binom{n-|v(\s{H})|}{|v(\Gamma)|-|v(\s{H})|}}{\binom{n}{|v(\Gamma)|}}}^2\\
        &=\sum_{\substack{\s{H}\subseteq \binom{[n]}{2}\\ |\s{H}|\leq \s{D}}} \lambda^{2|\s{H}|}\p{\frac{\p{n-|v(\s{H})|}!}{n!}\cdot \frac{|v(\Gamma)|!}{|v(\s{H})|!}}^2\\
        &\leq\sum_{\substack{\s{H}\subseteq \binom{[n]}{2}\\ |\s{H}|\leq \s{D}}} \lambda^{2|\s{H}|}\p{\frac{| v(\Gamma)|}{n-|v(\s{H})|}}^{2|v(\s{H})|}.
    \end{align}
    where $(a)$   follows since the vertices of a random copy of $\Gamma$ are uniform over $\binom{[n]}{|v(\Gamma)|}$. Next, we upper bound the above sum by considering all possible graphs with $m$ vertices and $\ell \leq \s{D}$ edges. Specifically, let $\calS_{\ell,m,n}$ denote the set of all subgraphs $\s{H}\subseteq\binom{n}{\ell}$ with $|v(\s{H})|=m$ vertices and $\abs{e(\s{H})}=\ell$ edges. The size of $\calS_{\ell,m,n}$ is clearly upper bounded by,
    \begin{align}
        \abs{\calS_{\ell,m,n}}\leq \binom{n}{m}\binom{\binom{m}{2}}{\ell}\leq n^m\cdot \p{\frac{m^2}{2}}^\ell.
    \end{align}
    We also note that it must be the case that $\ell\geq \frac{m}{2}$ and $m\leq 2\s{D}$, and thus, 
\begin{align}
        \norm{\s{L}_{n, \leq\s{D}}}_{\calH_0}^2
        &\leq\sum_{\substack{\s{H}\subseteq \binom{[n]}{2}\\ |\s{H}|\leq \s{D}}} \lambda^{2|\s{H}|}\p{\frac{| v(\Gamma)|}{n-|v(\s{H})|}}^{2|v(\s{H})|}\\
        &=1+\sum_{\substack{\s{H}\subseteq \binom{[n]}{2}\\ 1\leq |\s{H}|\leq \s{D}}} \lambda^{2|\s{H}|}\p{\frac{| v(\Gamma)|}{n-|v(\s{H})|}}^{2|v(\s{H})|}\\
        &\overset{(a)}{=}1+\sum_{m=2}^{2\s{D}}\sum_{\ell=m/2}^{\min\p{\binom{m}{2},\s{D}}} \lambda^{2\ell}\p{\frac{| v(\Gamma)|}{n-m}}^{2m} \abs{\calS_{\ell,m,n}}\\
        &=1+\sum_{m=2}^{2\s{D}}\sum_{\ell=m/2}^{\min\p{\binom{m}{2},\s{D}}} \lambda^{2\ell}\p{\frac{| v(\Gamma)|}{n-m}}^{2m} n^m \p{\frac{m^2}{2}}^\ell\\
        &=1+\sum_{m=2}^{2\s{D}} \p{\frac{| v(\Gamma)|^2n}{(n-m)^2}}^{m}
 \sum_{\ell=m/2}^{\min\p{\binom{m}{2},\s{D}}} \p{\frac{m^2\lambda^{2} }{2}}^\ell\\
        &\leq 1+\sum_{2\leq m \leq \sqrt{2\s{D}} } \p{\frac{| v(\Gamma)|^2n}{(n-\s{D})^2}}^{m}
 \sum_{\ell=m/2}^{\binom{m}{2}} \p{\frac{m^2\lambda^{2} }{2}}^\ell\\
 &\qquad + \sum_{\sqrt{2\s{D}}\leq m \leq 2\s{D} } \p{\frac{| v(\Gamma)|^2n}{(n-\s{D})^2}}^{m}
 \sum_{\ell=m/2}^{\s{D}} \p{\frac{m^2\lambda^{2} }{2}}^\ell\\
  &\overset{(b)}{\leq } 1+(1+o(1))\sum_{2\leq m \leq \sqrt{2\s{D}} } \p{\frac{| v(\Gamma)|^2}{n}}^{m}
 \sum_{\ell=m/2}^{\binom{m}{2}} \p{\frac{2\s{D}\lambda^{2} }{2}}^\ell\\
 &\qquad + (1+o(1))\sum_{\sqrt{2\s{D}}\leq m \leq 2\s{D} } \p{\frac{| v(\Gamma)|^2}{n}}^{m}
 \sum_{\ell=m/2}^{\s{D}} \p{\frac{m^2\lambda^{2} }{2}}^\ell\\
  &\overset{(c)}{\leq } 1+(1+o(1))\sum_{2\leq m \leq \sqrt{2\s{D}} } \p{\frac{| v(\Gamma)|^2}{n}}^{m}
 \p{ \s{D}\lambda^{2} }^{\binom{m}{2}}\\
 &\qquad + (1+o(1))\sum_{\sqrt{2\s{D}}\leq m \leq 2\s{D} } \p{\frac{| v(\Gamma)|^2}{n}}^{m}
  \p{\frac{m^2\lambda^{2} }{2}}^{\s{D}},\label{eq:TwoBadSums}
    \end{align}
    where (a) follows since we only sum over graphs containing no isolated vertices, (b) is because $\s{D}=o(\sqrt{n})$, and (c) is because we assume that $\lambda^2=\chi^2(p||q)= o(\sqrt{\s{D}})$, and thus for $m\geq \sqrt{2\s{D}}$, we have that $m^2\lambda^2=\omega(1)$, which implies that,
    \begin{align}
        \sum_{\ell=m/2}^{m/2+C}\p{\frac{m \lambda^2}{2}}^\ell= (1+o(1))\p{\frac{m \lambda^2}{2}}^C.
    \end{align}   
    Let us now handle the second term on the r.h.s. of \eqref{eq:TwoBadSums}. We have,
    \begin{align}
        \sum_{2\leq m \leq \sqrt{2\s{D}} } \p{\frac{| v(\Gamma)|^2}{n}}^{m}\p{\s{D}\lambda^{2} }^{\binom{m}{2}}&\leq \sum_{2\leq m \leq \sqrt{2\s{D}} } \p{\frac{| v(\Gamma)|^2}{n}}^{m}\p{\s{D}\lambda^{2} }^{\frac{m^2}{2}}\\
        & \leq \sum_{2\leq m \leq \sqrt{2\s{D}} } \p{\frac{| v(\Gamma)|^2}{n}\s{D}^{\frac{\sqrt{\s{D}}}{2}}\lambda^{2}}^{m}.
    \end{align}
    Clearly, the above expression is bounded provided that, 
    \begin{align}
        \frac{\lambda^2\abs{v(\Gamma)}^2\s{D}^{\frac{\sqrt{\s{D}}}{2}}}{n}<1 \iff \abs{v(\Gamma)}<\frac{\sqrt{n}}{\lambda \s{D}^\frac{\sqrt{\s{D}}}{4}}.
    \end{align}
    As for the last term on the r.h.s. of \eqref{eq:TwoBadSums}, we have,
    \begin{align}
        \sum_{\sqrt{2\s{D}}\leq m \leq 2\s{D} } \p{\frac{| v(\Gamma)|^2}{n}}^{m}
  \p{\frac{m^2\lambda^{2} }{2}}^{\s{D}}&\leq \p{2\s{D}\lambda^{2}}^{\s{D}}\sum_{\sqrt{2\s{D}}\leq m \leq 2\s{D} } \p{\frac{| v(\Gamma)|^2}{n}}^{m}\\
  &\leq (1+o(1))\p{2\s{D}\lambda^{2}}^{\s{D}}\p{\frac{| v(\Gamma)|^2}{n}}^{\sqrt{2\s{D}}},
    \end{align}
    where the last inequality is because we are in the regime where $|v(\Gamma)|\ll \sqrt{n}$. Combining the above, we see that \eqref{eq:TwoBadSums} is bounded if and only if \eqref{eqn:CompCliqueBased} holds, which completes the proof.
    
\end{proof}

We are now ready to formally state the observation mentioned at the beginning of this subsection: conditioned on the LDP conjecture, super-logarithmic density is not only a necessary condition but also a sufficient one for the existence of a computational-statistical gap.
\begin{theorem}
         Assume that $\chi^2(p||q)=O(1)$, $\mu(\Gamma)\gg \log|v(\Gamma)|$, and that Conjecture~\ref{conj:1} holds. Then, there is an  $\Omega(1)$ function and an $\omega(1)$ function $f_{\Gamma}$, which depends on $\Gamma$, such that whenever, 
    \begin{equation}
        \Omega(1)\cdot \log n \leq \mu(\Gamma) \leq f_\Gamma (n)\cdot \log n,
    \end{equation}
    while detection is statistically possible, a polynomial-time algorithm that achieves strong detection does not exist.
\end{theorem}

\begin{proof}
    Using Theorem~\ref{obs:lowdegreeClique} and Conjecture~\ref{conj:1}, we deduce that there is no polynomial-time algorithm the distinguishes between $\calH_0$ and $\calH_1$ if for some arbitrarily small $\varepsilon$ we have,
    \begin{align}
        \log|v(\Gamma)|&\leq \log\p{ \frac{\sqrt{n}}{(2\lambda^2 (\log n)^{1+\varepsilon})^\frac{\sqrt{(\log n)^{1+\varepsilon}}}{2\sqrt{2}}}}\\
        &=\frac{1}{2}\log n-\frac{\sqrt{(\log n)^{1+\varepsilon}}}{2\sqrt{2}}\cdot \log\p{2\lambda^2 (\log n)^{1+\varepsilon}}\\
        &=\frac{\log n}{2}\cdot (1+o(1)),
    \end{align}
    and if stated in terms of $\mu(\Gamma)$,
        \begin{align}
        \mu(\Gamma)\leq (1+o(1)) \frac{\log n}{2}\cdot \overset{f_\Gamma(n)=\omega(1)}{\overbrace{\p{\frac{\mu(\Gamma)}{\log|v(\Gamma)|}}}}=\omega(1)\cdot \log n.
    \end{align}
    On the other hand, by Theorem~\ref{thm:upperBoundAlgo}, detection is possible (via the scan test) if $\mu(\Gamma)=\Omega(\log n)$. 
    \end{proof}

\subsubsection{General LDP analysis}\label{subsec:LDPboundDense}

In this subsection, we establish tighter computational lower bounds than those derived in the previous subsection. To that end, we adopt the same approach and reasoning as in the proof of our statistical lower bounds in Section~\ref{sec:statlowerbound}, focusing in particular on the representation in \eqref{eq:LDPstart2}. We have the following result.
\begin{theorem}\label{th:lowDegreeConbinatorial}
Let $\s{D}=(\s{D}_n)_n$ be an increasing sequence such that $\omega(1) \leq \s{D}\leq o(n)$ and let $\Gamma=(\Gamma_n)_n$ be a sequence.
\begin{enumerate}
    \item Assume that $\lambda^2=\chi^2(p||q)= \Theta(1)$. If,
    \begin{align}
        \frac{ \p{\max(\chi^2(p||q),1)\cdot 2\s{D}^2}^{\sqrt{\s{D}/2}}\cdot\max\p{\vc(\Gamma)d_{\max}(\Gamma),d^2_{\max}(\Gamma)}}{n-|v(\Gamma)|}\leq C, \label{eq:CompGeneralDense}
            \end{align}
    then $\norm{\s{L}_{n, \leq\s{D}}}_{\calH_0} = O(1)$.
    \item Assume that $\chi^2(p||q)=o(D^{-2})$. If,
    \begin{align}
        \frac{\sqrt{\chi^2(p||q)}\cdot\s{D}\cdot\max\p{\vc(\Gamma)d_{\max}(\Gamma),d^2_{\max}(\Gamma)\sqrt{\chi^2(p||q)}\cdot\s{D}}}{n-|v(\Gamma)|}\leq C, \label{eq:CompGeneralSparse}
            \end{align}
    then $\norm{\s{L}_{n, \leq\s{D}}}_{\calH_0} = O(1)$.
\end{enumerate}
\end{theorem}
\begin{proof}[Proof of Theorem~\ref{th:lowDegreeConbinatorial}]
    We start by proving \eqref{eq:CompGeneralDense}. To that end, we recall \eqref{eq:LDPstart2}. Our approach is reminiscent to the proof of Theorem~\ref{th:lowerVCD}, but now, the maximal number of vertices is $2\s{D}$,  and we run only on graphs with at most  $ \s{D}$ edges, the maximal number of edges $j$ in a subgraph with $\ell$ vertices is dominated by $\binom{\ell}{2}\wedge \s{D}$. Hence, using the same arguments as in \eqref{eq:GenralLB-GeneralMuBegin}--\eqref{eq:GenralLB-GeneralMuEnd} and Lemma~\ref{lem:subgraphsBoundedDegreeCount1}, we obtain,
    \begin{align}
         \norm{\s{L}_{n,\leq \s{D}}}_{\calH_0}^2&=\sum_{\substack{\s{H}\subseteq \binom{e(\Gamma)}{2}\\ |\s{H}|\leq \s{D}}} \lambda^{2\abs{\s{H}}} \P_{\s{H}}[\s{H}\subseteq \Gamma]\\
         &\leq 1+\sum_{\ell=2}^{2\s{D}} \sum_{m=1}^{\floor{\ell/2}}\sum_{j=\ell-m}^{\min\p{\binom{\ell}{2},\s{D}}} |\calS_{m,\ell,j}|  \lambda^{2j}  \frac{(2\vc)^m d^{\ell-m}}{(n-k)^\ell}\\
         &\leq 1+\sum_{\ell=2}^{2\s{D}} \sum_{m=1}^{\floor{\ell/2}}\sum_{j=\ell-m}^{\min\p{\binom{\ell}{2},\s{D}}} |\calS_{m,\ell,j}|  \lambda^{2j}   \frac{(2\vc)^m d^{\ell-m}}{(n-k)^\ell}\\
         &\leq 1+\sum_{\ell=2}^{2\s{D}} \sum_{m=1}^{\floor{\ell/2}}\sum_{j=\ell-m}^{\min\p{\binom{\ell}{2},\s{D}}} \abs{\s{Par}(\ell,m)}  \vc^m  (e  d)^{\ell-m}  \binom{\binom{\ell}{2}}{j}  \lambda^{2j}   \frac{(2\vc)^m d^{\ell-m}}{(n-k)^\ell}\\
         &\leq 1+C_\varepsilon  
          \sum_{\ell=2}^{2\s{D}} \p{\frac{(1+\varepsilon)  e  d^2}{n-k}}^\ell\sum_{m=1}^{\floor{\ell/2}}\p{\frac{2\vc^2}{e  d^2}}^m\sum_{j=\ell-m}^{\min\p{\binom{\ell}{2},\s{D}}}\binom{\binom{\ell}{2}}{j}  \lambda^{2j} \\
          &\leq 1+C_\varepsilon  
          \sum_{\ell=2}^{2\s{D}} \p{\frac{(1+\varepsilon)  e  d^2}{n-k}}^\ell\sum_{m=1}^{\floor{\ell/2}}\p{\frac{2\vc^2}{e  d^2}}^m\sum_{j=\ell-m}^{\min\p{\binom{\ell}{2},\s{D}}}\binom{\ell}{2}^{j}  \lambda^{2j}\label{eqn:LDPForSparse} \\
          &\leq
          1+C_\varepsilon  
          \sum_{\ell=2}^{2\s{D}} \p{\frac{(1+\varepsilon)  e  d^2}{n-k}}^\ell\sum_{m=1}^{\floor{\ell/2}}\p{\frac{2\vc^2}{e  d^2}}^m  \p{\frac{\ell^2  \lambda^2}{2}}^{\min\p{\binom{\ell}{2},\s{D}}}\label{eq:LDPmediate}\\
          &\leq
          1+C_\varepsilon  
          \sum_{\ell=2}^{2\s{D}} \p{\frac{(1+\varepsilon)  e  d^2}{n-k}}^\ell  \p{\frac{\ell^2  \lambda^2}{2}}^{\min\p{\frac{\ell^2}{2},\s{D}}}\sum_{m=1}^{\floor{\ell/2}}\p{\frac{2\vc^2}{e  d^2}}^m,\label{eq:LDPinital}
    \end{align}
    where in the fourth inequality we have used the Hardy-Ramanujan formula, and the sixth inequality is because we assume without loss of generality that $\lambda^2>1$; otherwise, we may replace $\lambda^2$ by $1$ and proceed with the analysis. We now separate our analysis into two complementary cases, starting with the case where,
    \begin{align}
        \frac{2\vc^2}{e  d^2}\geq 1+\alpha,\label{eq:blablac}
        \end{align}
        for some fixed $\alpha>0$, independent of $n$. In this case, we have,
        \begin{align}
             \norm{\s{L}_{n,\leq \s{D}}}_{\calH_0}^2&\leq 1+C_\varepsilon'  
          \sum_{\ell=2}^{2\s{D}} \p{\frac{(1+\varepsilon)  e  d^2}{n-k}}^\ell  \p{\frac{\ell^2  \lambda^2}{2}}^{\min\p{\frac{\ell^2}{2},\s{D}}}\p{\frac{2\vc^2}{e  d^2}}^{\frac{\ell}{2}}\\
          &=1+C_\varepsilon'  \Bigg[
          \sum_{\ell=2}^{\floor{\sqrt{2\s{D}}}} \p{\frac{(1+\varepsilon)\sqrt{2e} \vc d }{n-k}}^\ell  \p{\frac{\ell^2  \lambda^2}{2}}^{\frac{\ell^2}{2}}\\
          \nonumber&\hspace{3.5cm} + \sum_{\ell=\floor{\sqrt{2\s{D}}}+1}^{2\s{D}} \p{\frac{(1+\varepsilon)\sqrt{2e} \vc d }{n-k}}^\ell  \p{\frac{\ell^2  \lambda^2}{2}}^{\s{D}}\Bigg]\\
          &=1+C_\varepsilon'  \Bigg[
          \sum_{\ell=2}^{\floor{\sqrt{2\s{D}}}} \p{\frac{(1+\varepsilon)\sqrt{2e} \vc d }{n-k}}^\ell  \p{\s{D}  \lambda^2}^{\frac{\sqrt{2\s{D}}  \ell}{2}}\\
          \nonumber&\hspace{3.5cm} + \p{\frac{(2\s{D})^2  \lambda^2}{2}}^{\s{D}}  \sum_{\ell=\floor{\sqrt{2\s{D}}}+1}^{2\s{D}} \p{\frac{(1+\varepsilon)\sqrt{2e} \vc d }{n-k}}^\ell\Bigg]\\
          &=1+C_\varepsilon'  \Bigg[
          \sum_{\ell=2}^{\floor{\sqrt{2\s{D}}}} \p{\frac{(1+\varepsilon)\sqrt{2e} \vc d   (\s{D}  \lambda^2)^{\sqrt{\s{D}/2}}}{n-k}}^\ell\\
          &\hspace{3.5cm} + \p{\frac{(2\s{D})^2  \lambda^2}{2}}^{\s{D}}  \sum_{\ell=\floor{\sqrt{2\s{D}}}+1}^{2\s{D}} \p{\frac{(1+\varepsilon)\sqrt{2e} \vc d }{n-k}}^\ell\Bigg].\label{eq:lowdegreeUglySum}
        \end{align}
        Now, the first sum on the l.h.s. of \eqref{eq:lowdegreeUglySum} converges when 
        \begin{align}
            \frac{\vc  d  (\s{D}\lambda^2)^{\sqrt{\s{D}/2}}}{n-k}<C,\label{eq:LDPcond1}
        \end{align}
        for a sufficiently small constant $C>0$. A necessary condition for the second term on the r.h.s. of \eqref{eq:lowdegreeUglySum} to converge is that,
        \begin{align}
            \frac{(1+\varepsilon)  \sqrt{2e}\vc  d }{n-k}< 1-\alpha,
        \end{align}
        for some fixed $\alpha>0$, in which case, we have
        \begin{align}
        \sum_{\ell=\floor{\sqrt{2\s{D}}}+1}^{2\s{D}} \p{\frac{(1+\varepsilon)  2\vc  d}{n-k}}^\ell
             \leq C_\alpha 
             \p{\frac{(1+\varepsilon)\sqrt{2e} \vc d }{n-k}}^{\sqrt{2\s{D}}}.
        \end{align}
        Hence, the second term on the r.h.s. of \eqref{eq:lowdegreeUglySum} is bounded if,
        \begin{align}
            \frac{\vc  d   \p{2\s{D}^2   \lambda^2}^{\sqrt{\s{D}/2}} }{n-k}<C,\label{eq:LDPcond2}
        \end{align}
         for a sufficiently small constant $C>0$. In the limit $\s{D}\to\infty$, the condition in \eqref{eq:LDPcond1} is satisfied if \eqref{eq:LDPcond2} holds, which implies that $\norm{\s{L}_{n,\leq \s{D}}}_{\calH_0}$ is bounded provided that \eqref{eq:LDPcond2} holds. Next, we move to the complementary case where, 
         \begin{align}
        \frac{2\vc^2}{e  d^2}\leq 1+o(1),
         \end{align}
        for some $o(1)$ function. Here, for a sufficiently large $n$, the inner sum on the r.h.s. of \eqref{eq:lowdegreeUglySum} is dominated by $\p{1+\varepsilon}^\ell$, and thus,
        \begin{align}
             \norm{\s{L}_{n,\leq \s{D}}}_{\calH_0}^2&\leq 1+C_\varepsilon' 
          \sum_{\ell=2}^{2\s{D}} \p{\frac{(1+\varepsilon)^2  e  d^2}{n-k}}^\ell  \p{\frac{\ell^2  \lambda^2}{2}}^{\min\p{\frac{\ell^2}{2},\s{D}}}.
          \end{align}
          Accordingly, if we repeat the exact same steps as in the first case above, we finally obtain that,
          \begin{align}
             \norm{\s{L}_{n,\leq \s{D}}}_{\calH_0}^2&\leq1+C_\varepsilon'  \Bigg[
          \sum_{\ell=2}^{\floor{\sqrt{2\s{D}}}} \p{\frac{(1+\varepsilon)^2 e  d^2  (\s{D}  \lambda^2)^{\sqrt{\s{D}/2}}}{n-k}}^\ell\label{eq:lowdegreeUglySum2}\\
          \nonumber&\hspace{3.5cm} + \p{\frac{(2\s{D})^2  \lambda^2}{2}}^{\s{D}}  \sum_{\ell=\floor{\sqrt{2\s{D}}}+1}^{2\s{D}} \p{\frac{(1+\varepsilon)^2  e  d^2}{n-k}}^\ell\Bigg],
          \end{align}
          which converges if,
    \begin{align}
            \frac{d^2  (2\s{D}^2\lambda^2)^{\sqrt{\s{D}/2}}}{n-k}<C,\label{eq:LDPcond3}
        \end{align}
         for a sufficiently small $C>0$.  We conclude by noting that \eqref{eq:blablac} holds exactly when the condition \eqref{eq:LDPcond2} dominates the condition \eqref{eq:LDPcond3} (up to a fixed multiplicative constant, and therefore \eqref{eq:CompGeneralDense} sufficient condition.

         Finally, we prove \eqref{eq:CompGeneralSparse}. This is quite similar to the above analysis, and our derivations proceed from \eqref{eqn:LDPForSparse} as follows,
         \begin{align}
             \norm{\s{L}_{n,\leq \s{D}}}_{\calH_0}^2&\leq 1+C_\varepsilon  
          \sum_{\ell=2}^{2\s{D}} \p{\frac{(1+\varepsilon)  e  d^2}{n-k}}^\ell\sum_{m=1}^{\floor{\ell/2}}\p{\frac{2\vc^2}{e  d^2}}^m\sum_{j=\ell-m}^{\min\p{\binom{\ell}{2},\s{D}}}\binom{\ell}{2}^{j}\lambda^{2j}\\
          &\leq 1+C_\varepsilon  
          \sum_{\ell=2}^{2\s{D}} \p{\frac{(1+\varepsilon)  e  d^2}{n-k}}^\ell\sum_{m=1}^{\floor{\ell/2}}\p{\frac{2\vc^2}{e  d^2}}^m\sum_{j=\ell-m}^{\min\p{\binom{\ell}{2},\s{D}}}\p{2\s{D}^2\lambda^2}^j\\
          &\leq 1+C_\varepsilon  
          \sum_{\ell=2}^{2\s{D}} \p{\frac{(1+\varepsilon)  e  d^2}{n-k}}^\ell\sum_{m=1}^{\floor{\ell/2}}\p{\frac{2\vc^2}{e  d^2}}^m\frac{(2\s{D}^2\lambda^2)^{\ell-m}}{1-2\s{D}^2\lambda^2}\\
          &\leq 1+C'_\varepsilon  
          \sum_{\ell=2}^{2\s{D}} \p{\frac{2(1+\varepsilon)  e  d^2\s{D}^2\lambda^2}{n-k}}^\ell\sum_{m=1}^{\floor{\ell/2}}\p{\frac{\vc^2}{e  d^2\s{D}^2\lambda^2}}^m,\label{eqn:LDPLastBoundSparse}
         \end{align}
         where third inequality follows from the assumption that $\lambda^2 = o(\s{D}^{-2})$. Accordingly, in a similar fashion to the analysis of the previous case, we see that if,
        \begin{align}
            \frac{\vc^2}{e  d^2\s{D}^2\lambda^2}\geq 1+\alpha,
        \end{align}
        for some positive $\alpha>0$, then \eqref{eqn:LDPLastBoundSparse} is bounded provided that,
        \begin{align}
            \frac{\vc d \s{D} \lambda}{n-k}<C,
        \end{align}
        for a sufficiently small constant $C$. Otherwise, if,
        \begin{align}
            \frac{\vc^2}{e  d^2\s{D}^2\lambda^2}\leq 
            1+o(1),
         \end{align}
        for some $o(1)$ function, then \eqref{eqn:LDPLastBoundSparse} is bounded provided that,
        \begin{align}
            \frac{d^2 \s{D}^2 \lambda^2}{n-k}<C.
        \end{align}
        Combining the above conditions, we obtain \eqref{eq:CompGeneralSparse}, which concludes the proof.
\end{proof}

\subsection{Reduction to the vertex cover-degree balanced case}\label{subsec:ReducLDP}

There is a gap between the computational lower bound in Theorem~\ref{th:lowDegreeConbinatorial} and the performance of the efficient algorithms (i.e., the count and maximum degree tests) in Theorem~\ref{thm:upperBoundAlgo}, which disappear if $\Gamma$ is $\vcd$-balanced. We remedy this and prove that Theorem~\ref{thm:upperBoundAlgo} is tight using a similar strategy as in Subsection~\ref{sec:ReudctionToBalanced}. We start with the following variant of Proposition~\ref{prop:decompisitionThD}.
    \begin{prop}\label{prop:decompisitionThDCcomp}
    Let $\Gamma=(\Gamma_n)_n$ be a sequence of graphs, and assume that each can be decomposed into a set of edge-disjoint graphs, i.e., $\Gamma=\bigcup_{\ell=1}^M \Gamma_\ell$. 
    \begin{enumerate}
        \item If $\chi^2(p||q)=\Theta(1)$, and there exists a constant $C$ such that,
                \begin{align}
        \max_{\ell=1,\dots, M}\frac{ \p{\bar\lambda_M^2\cdot M^4\s{D}^2}^{\sqrt{M^2\s{D}/2}}\cdot\max\p{\vc(\Gamma_\ell)d_{\max}(\Gamma_\ell),d^2_{\max}(\Gamma_\ell)}}{n-|v(\Gamma_\ell)|}\leq C,\label{eq:condintersectionDComp}
    \end{align}
    where $\bar\lambda_M^2\triangleq \max(1,2eM^2\lambda^2,\lambda_M^2,C_\lambda(e\lambda^2)^{M^2})$, $\lambda_M^2=(1-\lambda^2)^{M^2}-1$, and $C_{\lambda}>0$ is a universal constant, then $\norm{\s{L}_{n,\leq \s{D}}}_{\calH_0}^2=O(1)$.
        \item If $\chi^2(p||q)=o(D^{-2})$, and there exists a constant $C$ such that,
    \begin{align}
        \max_{\ell=1,\dots, M}\frac{\bar\lambda_M\s{D}\cdot\max\p{\vc(\Gamma_\ell)d_{\max}(\Gamma_\ell),d^2_{\max}(\Gamma_\ell)\bar\lambda_M\s{D}}}{n-|v(\Gamma_\ell)|}\leq C, \label{eq:condintersectionDComp2}
            \end{align}
            where
    \begin{align}
        \bar\lambda_M^2\triangleq \max\p{\lambda_M^2,eM^4\s{D}\lambda^2,(eM^2\s{D}\lambda^2)^{M^2},\s{D}M^2\p{e M^2\s{D}^2\lambda^2}^{M^2}\p{M^4\s{D}}^{M^4}},
    \end{align}
    then $\norm{\s{L}_{n, \leq\s{D}}}_{\calH_0} = O(1)$.
\end{enumerate}
\end{prop}
To prove Proposition~\ref{prop:decompisitionThDCcomp} we need the following generalization of Theorem~\ref{th:lowDegreeConbinatorial}, which provides the conditions under which the truncated moment generating function of the intersection of \emph{two arbitrary} subgraphs is bounded. %For simplicity of notation, throughout this section, we use $\Gamma_1\cap \Gamma_2$ to denote the set intersecting edges between $\Gamma_1$ and $\Gamma_2$, two subgraphs of $\calK_n$ (which may be random copies or fixed copies, depending on the context). 

\begin{lemma}\label{lem:interectionDComp}
      Let $\Gamma_1=(\Gamma_{1,n})_n$ and $\Gamma_2=(\Gamma_{2,n})_n$ be two sequences of graphs. Then, there exists a universal constant $C$ such that:
      \begin{enumerate}
          \item If $\chi^2(p||q)=\Theta(1)$,  and
          \begin{align}
         \frac{ \p{\max(\lambda^2,1)\cdot 2\s{D}^2}^{\sqrt{\s{D}/2}}\max\p{\sqrt{\vc(\Gamma_1)d_{\max}(\Gamma_1)\vc(\Gamma_2)d_{\max}(\Gamma_2)},d_{\max}(\Gamma_1)d_{\max}(\Gamma_2)}}{n-\min\p{|v(\Gamma_1)|,|v(\Gamma_2)|}}\leq C,\label{eq:condIntersectionD1Comp}
     \end{align}
      then,
    \begin{align}
        \bE_{ \Gamma_2}\pp{\sum_{m=0}^{\min(\s{D},|\Gamma_1'\cap\Gamma_2|)}\binom{|\Gamma_1'\cap\Gamma_2|}{m}\lambda^{2m}}=O(1),\label{eq:condIntersectionD2Comp}
    \end{align}
     where the expectation is taken w.r.t. $\Gamma_2\sim\s{Unif}(\calS_{\Gamma_2})$ and $\Gamma_1'$ is fixed copy of $\Gamma_1$ in $\calK_n$. 
     \item If $\chi^2(p||q)=o(D^{-2})$, and
     \begin{align}
        &\frac{\sqrt{\chi^2(p||q)}\cdot\s{D}\cdot\max\p{\sqrt{\vc(\Gamma_1)d_{\max}(\Gamma_1)\vc(\Gamma_2)d_{\max}(\Gamma_2)},d_{\max}(\Gamma_1)d_{\max}(\Gamma_2)\sqrt{\chi^2(p||q)}\cdot\s{D}}}{n-\min\p{|v(\Gamma_1)|,|v(\Gamma_2)|}}\nonumber\\
        &\hspace{4cm}\leq C, \label{eq:condIntersectionD1Comp2}
            \end{align}
            then,
    \begin{align}
        \bE_{ \Gamma_2}\pp{\sum_{m=0}^{\min(\s{D},|\Gamma_1'\cap\Gamma_2|)}\binom{|\Gamma_1'\cap\Gamma_2|}{m}\lambda^{2m}}=O(1).\label{eq:condIntersectionD2Comp2}
    \end{align}
      \end{enumerate}
      
\end{lemma}

Before we prove Lemma~\ref{lem:interectionDComp}, we will require the following simple generalization of Lemma~\ref{lem:intersectionMomentGenGeneral}.
\begin{lemma}\label{lem:intersectionMomentGenGeneralComp}  Let $\Gamma_1$ and $\Gamma_2$ be two arbitrary (perhaps different) subgraphs of $\calK_n$ and $\s{D}$ be an integer. Then,
\begin{align}
    \E_{\Gamma_1\indep \Gamma_2}\pp{\sum_{m=0}^{\min\p{\s{D},|e(\Gamma_1\cap \Gamma_2)|}}\binom{|e(\Gamma_1\cap\Gamma_2)|}{m}\lambda^{2m}}&=\sum_{\substack{\s{H}\subseteq \binom{[n]}{2}\\ |\s{H}|\leq \s{D}}}\lambda^{2|\s{H}|}\cdot \P_{\Gamma_1}\pp{\s{H} \subseteq \Gamma_{1}}\cdot \P_{\Gamma_2} \pp{\s{H} \subseteq \Gamma_2}\label{eq:inetersectionEquivalentComp}\\
    &=\sum_{\substack{\s{H}\subseteq \Gamma_1' \\ |\s{H}|\leq \s{D}}}\lambda^{2|\s{H}|}\cdot \P_{\Gamma_2} \pp{\s{H} \subseteq \Gamma_2},\label{eq:inetersectionEquivalent2Comp}
\end{align}
where $\Gamma_1'$ is a fixed arbitrary copy of $\Gamma_1$ in $\calK_n$, the summation is over graphs containing no isolated vertices, and the probabilities and the expectation are taken w.r.t. two independent random copies $\Gamma_i\sim \s{Unif}(\calS_{\Gamma_i})$, for $i=1,2$. %In particular, when $\s{D}\to \infty$, where the binomial sum of the l.h.s. of \eqref{eq:inetersectionEquivalent} becomes $(1+\lambda^2)^{|e(\Gamma_1\cap \Gamma_2)|}$, we get the result of Lemma~\ref{lem:intersectionMomentGenGeneral}.
\end{lemma}
Note that in the limit $\s{D}\to \infty$, the binomial sum on the l.h.s. of \eqref{eq:inetersectionEquivalent} converges to $(1+\lambda^2)^{|e(\Gamma_1\cap \Gamma_2)|}$, and we obtain Lemma~\ref{lem:intersectionMomentGenGeneral}.

\begin{proof}[Proof of Lemma~\ref{lem:intersectionMomentGenGeneralComp}]
We follow the exact same steps as in the proof of Lemma~\ref{lem:intersectionMomentGenGeneral}, and start with the analysis of the term on the r.h.s. of \eqref{eq:inetersectionEquivalentComp}. We have,
\begin{align}
    \sum_{\substack{\s{H}\subseteq \binom{[n]}{2}\\ |\s{H}|\leq \s{D}}}\lambda^{2|\s{H}|}\cdot \P_{\Gamma_1}\pp{\s{H} \subseteq \Gamma_1}\cdot \P_{\Gamma_2} \pp{\s{H} \subseteq \Gamma_2}
    &=\sum_{\substack{\s{H}\subseteq \binom{[n]}{2}\\ |\s{H}|\leq \s{D}}}\lambda^{2|\s{H}|}\cdot \P_{\Gamma_1\indep \Gamma_2}\pp{\s{H} \subseteq e\p{\Gamma_1\cap  \Gamma_2}}\\
    &=\sum_{\substack{\s{H}\subseteq \binom{[n]}{2}\\ |\s{H}|\leq \s{D}}}\lambda^{2|\s{H}|}\cdot \sum_{\substack{\s{H}'\subseteq \binom{[n]}{2}\\ \s{H}\subseteq \s{H}'}}\P_{\Gamma_1\indep \Gamma_2}\pp{\s{H}'=e\p{\Gamma_1\cap  \Gamma_2}}\\
    &=\sum_{\s{H}'\subseteq \binom{[n]}{2}} \P_{\Gamma_1\indep \Gamma_2}\pp{\s{H}'=e\p{\Gamma_1\cap  \Gamma_2}} \sum_{\substack{\s{H}\subseteq \binom{[n]}{2}\\ \s{H}\subseteq \s{H}'\\ |\s{H}|\leq \s{D}}}\lambda^{2|\s{H}|}\\
    &=\sum_{\s{H}'\subseteq \binom{[n]}{2}} \P_{\Gamma_1\indep \Gamma_2}\pp{\s{H}'=e\p{\Gamma_1\cap  \Gamma_2}} \sum_{i=0}^{\min\p{\abs{\s{H}'},\s{D}}}\binom{|\s{H}'|}{i}\lambda^{2i}\\
    &=\E_{\Gamma_1\indep \Gamma_2}\pp{\sum_{m=0}^{\min\p{\s{D},|e(\Gamma_1\cap \Gamma_2)|}}\binom{|e(\Gamma_1\cap\Gamma_2)|}{m}\lambda^{2m}},
\end{align}
which proves the first equality in \eqref{eq:inetersectionEquivalentComp}. For the second equality in \eqref{eq:inetersectionEquivalent2Comp}, recall that as explained in the proof of Lemma~\ref{lem:intersectionMomentGenGeneral}, for any graph $\s{H'}\subseteq \calK_n$, if $\s{H}'$ is isomorphic to some subgraph $\s{H}\subseteq\Gamma_1'$, then $\pr_{\Gamma_1}\pp{\s{H}\subseteq \Gamma_1'}=\pr_{\Gamma_1}\pp{\s{H'}\subseteq \Gamma_1'}\neq 0$. If $\s{H}'$ is not isomorphic to any subgraph of $\Gamma_1$, then the probability above is clearly zero. We define an equivalence relation on subgraphs of $\Gamma_1'$, where two subgraphs are considered equivalent if and only if they are isomorphic. Let $[\s{H}]$ denote the equivalence class of a subgraph $\s{H}\subseteq \Gamma_1$, and let $\calP$ be the set of all equivalence classes corresponding to graphs with at most $\s{D}$ edges.
Then,
\begin{align}
     \sum_{\substack{\s{H}\subseteq \binom{[n]}{2}\\ |\s{H}|\leq \s{D}}}\lambda^{2|\s{H}|}\cdot \P_{\Gamma_1}\pp{\s{H} \subseteq \Gamma_1}\cdot \P_{\Gamma_2} \pp{\s{H} \subseteq \Gamma_2}&=\sum_{[\s{H}]\in \calP_{\s{D}}}\sum_{\substack{\s{H}'\subseteq \binom{[n]}{2}\\ \s{H}'\cong \s{H} }}\lambda^{2|\s{H}'|}\cdot \P_{\Gamma_1}\pp{\s{H}' \subseteq \Gamma_1}\cdot \P_{\Gamma_2} \pp{\s{H}' \subseteq \Gamma_2}\\
     &=\sum_{[\s{H}]\in \calP_{\s{D}}}\sum_{\substack{\s{H}'\subseteq \binom{[n]}{2}\\ \s{H}'\cong \s{H} }}\lambda^{2|\s{H}|}\cdot \P_{\Gamma_1}\pp{\s{H} \subseteq \Gamma_1}\cdot \P_{\Gamma_2} \pp{\s{H} \subseteq \Gamma_2}\\
     &=\sum_{[\s{H}]\in \calP_{\s{D}}} |\calS_{\s{H}}|\lambda^{2|\s{H}|}\cdot \P_{\Gamma_1}\pp{\s{H} \subseteq \Gamma_1}\cdot \P_{\Gamma_2} \pp{\s{H} \subseteq \Gamma_2}\\
     &=\sum_{[\s{H}]\in \calP_{\s{D}}} \lambda^{2|\s{H}|}\cdot \calN(\s{H},\Gamma_1)\cdot \P_{\Gamma_2} \pp{\s{H} \subseteq \Gamma_2}\\
     &=\sum_{[\s{H}]\in \calP_{\s{D}}} \sum_{\s{H}'\in [\s{H}]}\lambda^{2|\s{H}|}\cdot \P_{\Gamma_2} \pp{\s{H} \subseteq \Gamma_2}\\
     &=\sum_{\substack{\s{H}\subseteq \Gamma_1'\\ |\s{H}|\leq \s{D}}} \lambda^{2|\s{H}|}\cdot \P_{\Gamma_2} \pp{\s{H} \subseteq \Gamma_2}.
\end{align}
which completes the proof.
\end{proof}

\begin{proof}[Proof of Lemma~\ref{lem:interectionDComp}]
    The proof follows the exact same steps as in the proof of Lemma~\ref{lem:interectionD} with the modifications as in the proof of Theorem~\ref{th:lowDegreeConbinatorial}. Let us follow the notation used in the proof of Lemma~\ref{lem:interectionD}, where $\vc_i=\vc(\Gamma_i)$, $d_i=d_{\max}(\Gamma_i)$, $\mu_i=\mu(\Gamma_i)$, $k_i=|v(\Gamma_i)|$, $\bar{k}=\min(k_1,k_2)$, and $\bar{\mu}=\min(\mu_1,\mu_2)$. Let $\calS(m,\ell,j)(\Gamma_1)$ be the set of subgraphs of $\Gamma_1$ with $m$ connected components, $\ell$ vertices, and $j$ edges. We repeat the steps as in the proof of Lemma~\ref{lem:interectionD}, but now, the maximal number of vertices is $2\s{D}$, and we consider only graphs with at most $\s{D}$ edges. Furthermore, note that the maximal number of edges $j$ in a subgraph with $\ell$ vertices is upper bounded by $\binom{\ell}{2}\wedge \s{D}$.  
    We begin with the case where $\lambda^2=\Theta(1)$.  Using Lemmas~\ref{lem:probRandomSubgraph1}, \ref{lem:subgraphsBoundedDegreeCount1}, and \ref{lem:intersectionMomentGenGeneralComp}, and following similar derivations as in the proof of Lemma~\ref{lem:interectionD}, we get, 
    \begin{align}
          \E_{\Gamma_1\indep \Gamma_2}&\pp{\sum_{m=0}^{\min\p{\s{D},|e(\Gamma_1\cap \Gamma_2)|}}\binom{|e(\Gamma_1\cap\Gamma_2)|}{m}\lambda^{2m}}\\
          &=\sum_{\substack{\s{H}\subseteq \Gamma_1' \\ |\s{H}|\leq \s{D}}}\lambda^{2|\s{H}|}\cdot \P_{\Gamma_2} \pp{\s{H} \subseteq \Gamma_2}\\
         &\leq 1+\sum_{\ell=2}^{2\s{D}} \sum_{m=1}^{\floor{\ell/2}}\sum_{j=\ell-m}^{\min\p{\binom{\ell}{2},\s{D}}} |\calS_{m,\ell,j}(\Gamma_1)|  \lambda^{2j}  \frac{(2\vc_2)^m d_2^{\ell-m}}{(n-\bar{k})^\ell}\\
         &\leq 1+\sum_{\ell=2}^{2\s{D}} \sum_{m=1}^{\floor{\ell/2}}\sum_{j=\ell-m}^{\min\p{\binom{\ell}{2},\s{D}}} \abs{\s{Par}(\ell,m)}  \vc_1^m  (e  d_1)^{\ell-m}  \binom{\binom{\ell}{2}}{j}  \lambda^{2j}   \frac{(2\vc_2)^m d_2^{\ell-m}}{(n-\bar{k})^\ell}\\
         &\leq 1+C_\varepsilon  
          \sum_{\ell=2}^{2\s{D}} \p{\frac{(1+\varepsilon)  e  d_1 d_2}{n-\bar{k}}}^\ell\sum_{m=1}^{\floor{\ell/2}}\p{\frac{2\vc_1 \vc_2}{e  d_1 d_2}}^m\sum_{j=\ell-m}^{\min\p{\binom{\ell}{2},\s{D}}}\binom{\binom{\ell}{2}}{j}  \lambda^{2j} \\
          &\leq 1+C_\varepsilon  
          \sum_{\ell=2}^{2\s{D}} \p{\frac{(1+\varepsilon)  e  d_1 d_2}{n-\bar{k}}}^\ell\sum_{m=1}^{\floor{\ell/2}}\p{\frac{2\vc_1 \vc_2}{e  d_1 d_2}}^m\sum_{j=\ell-m}^{\min\p{\binom{\ell}{2},\s{D}}}\binom{\ell}{2}^{j}  \lambda^{2j}\label{eqn:LDPForSparseInt} \\
          &{\leq}
          1+C_\varepsilon  
          \sum_{\ell=2}^{2\s{D}} \p{\frac{(1+\varepsilon)  e  d_1 d_2}{n-\bar{k}}}^\ell\sum_{m=1}^{\floor{\ell/2}}\p{\frac{2\vc_1 \vc_2}{e  d_1 d_2}}^m  \p{\frac{\ell^2  \lambda^2}{2}}^{\min\p{\binom{\ell}{2},\s{D}}}\label{eq:LDPmediateInt}\\
          &\leq
          1+C_\varepsilon  
          \sum_{\ell=2}^{2\s{D}} \p{\frac{(1+\varepsilon)  e  d_1 d_2}{n-\bar{k}}}^\ell  \p{\frac{\ell^2  \lambda^2}{2}}^{\min\p{\frac{\ell^2}{2},\s{D}}}\sum_{m=1}^{\floor{\ell/2}}\p{\frac{2\vc_1 \vc_2}{e  d_1 d_2}}^m,\label{eq:LDPinitalIntersec}
    \end{align}
    where the last inequality is because we assume, without loss of generality, that $\lambda^2>1$; otherwise, we can replace $\lambda^2$ by $1$ and proceed with the same analysis. We separate the following analysis into to complementary cases. In the first, we assume that,
    \begin{align}
        \frac{2\vc_1 \vc_2}{e  d_1 d_2}\geq 1+\alpha\label{eq:dominationCompInt}
        \end{align}
        for some fixed $\alpha>0$, independent of $n$. In this case, we have,
        \begin{align}
             \norm{\s{L}_{n,\leq \s{D}}}_{\calH_0}^2&\leq 1+C_\varepsilon'  
          \sum_{\ell=2}^{2\s{D}} \p{\frac{(1+\varepsilon)  e  d_1 d_2}{n-\bar{k}}}^\ell  \p{\frac{\ell^2  \lambda^2}{2}}^{\min\p{\frac{\ell^2}{2},\s{D}}}\p{\frac{2\vc_1 \vc_2}{e  d_1 d_2}}^{\frac{\ell}{2}}\\
          &=1+C_\varepsilon'  \Bigg[
          \sum_{\ell=2}^{\floor{\sqrt{2\s{D}}}} \p{\frac{(1+\varepsilon)  \sqrt{2e  \vc_1 \vc_2 d_1 d_2}}{n-\bar{k}}}^\ell  \p{\frac{\ell^2  \lambda^2}{2}}^{\frac{\ell^2}{2}}\\
          \nonumber&\hspace{3.5cm} + \sum_{\ell=\floor{\sqrt{2\s{D}}}+1}^{2\s{D}} \p{\frac{(1+\varepsilon)  \sqrt{2e  \vc_1 \vc_2 d_1 d_2}}{n-\bar{k}}}^\ell  \p{\frac{\ell^2  \lambda^2}{2}}^{\s{D}}\Bigg]\\
          &=1+C_\varepsilon'  \Bigg[
          \sum_{\ell=2}^{\floor{\sqrt{2\s{D}}}} \p{\frac{(1+\varepsilon)  \sqrt{2e  \vc_1 \vc_2 d_1 d_2}  (\s{D}  \lambda^2)^{\sqrt{\s{D}/2}}}{n-\bar{k}}}^\ell\\
          &\hspace{3.5cm} + \p{\frac{(2\s{D})^2  \lambda^2}{2}}^{\s{D}}  \sum_{\ell=\floor{\sqrt{2\s{D}}}+1}^{2\s{D}} \p{\frac{(1+\varepsilon)  \sqrt{2e  \vc_1 \vc_2 d_1 d_2}}{n-\bar{k}}}^\ell\Bigg].\label{eq:lowdegreeUglySumInt}
        \end{align}
        The first sum on the l.h.s. of \eqref{eq:lowdegreeUglySumInt} converges when, 
        \begin{align}
            \frac{\sqrt{\vc_1 \vc_2 d_1 d_2} (\s{D}\lambda^2)^{\sqrt{\s{D}/2}}}{n-\bar{k}}<C,\label{eq:LDPcond1Int}
        \end{align}
        for a sufficiently small constant $C>0$. A necessary condition for the second term on the r.h.s. of \eqref{eq:lowdegreeUglySumInt} to converge is that,
        \begin{align}
            \frac{(1+\varepsilon)  \sqrt{2e\vc_1 \vc_2 d_1 d_2}}{n-\bar{k}}< 1-\alpha
        \end{align}
         for some fixed $\alpha>0$, in which case, we have
        \begin{align}
        \sum_{\ell=\floor{\sqrt{2\s{D}}}+1}^{2\s{D}} \p{\frac{(1+\varepsilon) \sqrt{2e\vc_1 \vc_2 d_1 d_1}}{n-\bar{k}}}^\ell
             \leq C_\alpha 
             \p{\frac{(1+\varepsilon)  \sqrt{2e\vc_1 \vc_2 d_1 d_1}}{n-\bar{k}}}^{\sqrt{2\s{D}}}.
        \end{align}
        Hence, the second term on the r.h.s. of \eqref{eq:lowdegreeUglySumInt} is bounded if,
        \begin{align}
            \frac{\sqrt{\vc_1 \vc_2  d_1 d_2}   (2\s{D}^2   \lambda^2)^{\sqrt{\s{D}/2}} }{n-\bar{k}}<C,\label{eq:LDPcond2Int}
        \end{align}
         for a sufficiently small constant $C>0$. In the limit $\s{D}\to\infty$, the condition in \eqref{eq:LDPcond1Int} is satisfied if \eqref{eq:LDPcond2Int} holds, which implies that $\norm{\s{L}_{n,\leq \s{D}}}_{\calH_0}$ is bounded provided that \eqref{eq:LDPcond1} holds. Next, we move to the complementary case where, 
         \begin{align}
        \frac{2\vc_1 \vc_2}{e  d_1 d_2}\leq 1+o(1),
         \end{align}
        for some $o(1)$ function. Here, for a sufficiently large $n$, the inner sum on the r.h.s. of \eqref{eq:lowdegreeUglySumInt} is dominated by $\p{1+\varepsilon}^\ell$, and thus,
        \begin{align}
             \norm{\s{L}_{n,\leq \s{D}}}_{\calH_0}^2&\leq 1+C_\varepsilon' 
          \sum_{\ell=2}^{2\s{D}} \p{\frac{(1+\varepsilon)^2  e  d_1 d_2}{n-k}}^\ell  \p{\frac{\ell^2  \lambda^2}{2}}^{\min\p{\frac{\ell^2}{2},\s{D}}}.
          \end{align}
          Accordingly, if we repeat the exact same steps as in the first case above, we finally obtain that,
          \begin{align}
             \norm{\s{L}_{n,\leq \s{D}}}_{\calH_0}^2&\leq1+C_\varepsilon'  \Bigg[
          \sum_{\ell=2}^{\floor{\sqrt{2\s{D}}}} \p{\frac{(1+\varepsilon)^2 e  d_1 d_2  (\s{D}  \lambda^2)^{\sqrt{\s{D}/2}}}{n-k}}^\ell\label{eq:lowdegreeUglySum2Int}\\
          \nonumber&\hspace{3.5cm} + \p{\frac{(2\s{D})^2  \lambda^2}{2}}^{\s{D}}  \sum_{\ell=\floor{\sqrt{2\s{D}}}+1}^{2\s{D}} \p{\frac{(1+\varepsilon)^2  e  d_1 d_2}{n-k}}^\ell\Bigg],
          \end{align}
          which converges if,
    \begin{align}
            \frac{d_1 d_2  (2\s{D}^2\lambda^2)^{\sqrt{\s{D}/2}}}{n-k}<C,\label{eq:LDPcond3Int}
        \end{align}
         for a sufficiently small $C>0$. Combining the above conditions, we obtain \eqref{eq:condIntersectionD1Comp}, which concludes the proof for the case where $\lambda^2=\Theta(1)$.

         Next, we handle the case where $\lambda^2=o(D^{-2})$ in a similar fashion. We proceed our analysis from \eqref{eqn:LDPForSparseInt}, and get,
         \begin{align}
              \E_{\Gamma_1\indep \Gamma_2}&\pp{\sum_{m=0}^{\min\p{\s{D},|e(\Gamma_1\cap \Gamma_2)|}}\binom{|e(\Gamma_1\cap\Gamma_2)|}{m}\lambda^{2m}}\\
              &\leq 1+C_\varepsilon  
          \sum_{\ell=2}^{2\s{D}} \p{\frac{(1+\varepsilon)  e  d_1 d_2}{n-\bar{k}}}^\ell\sum_{m=1}^{\floor{\ell/2}}\p{\frac{2\vc_1 \vc_2}{e  d_1 d_2}}^m\sum_{j=\ell-m}^{\min\p{\binom{\ell}{2},\s{D}}}\binom{\ell}{2}^{j}  \lambda^{2j}\\
          & \leq 1+C_\varepsilon  
          \sum_{\ell=2}^{2\s{D}} \p{\frac{(1+\varepsilon)  e  d_1 d_2}{n-\bar{k}}}^\ell\sum_{m=1}^{\floor{\ell/2}}\p{\frac{2\vc_1 \vc_2}{e  d_1 d_2}}^m\sum_{j=\ell-m}^{\min\p{\binom{\ell}{2},\s{D}}}\p{ 2\s{D}^2 \lambda^{2}}^j\\
          &\overset{(a)}{\leq} 1+C_\varepsilon'  
          \sum_{\ell=2}^{2\s{D}} \p{\frac{(1+\varepsilon)  e  d_1 d_2}{n-\bar{k}}}^\ell\sum_{m=1}^{\floor{\ell/2}}\p{\frac{2\vc_1 \vc_2}{e  d_1 d_2}}^m\p{ 2\s{D}^2 \lambda^{2}}^{\ell-m}\\
          &=1+C_\varepsilon'  
          \sum_{\ell=2}^{2\s{D}} \p{\frac{(1+\varepsilon) 2 e  d_1 d_2 \lambda^2 \s{D}^2}{n-\bar{k}}}^\ell\sum_{m=1}^{\floor{\ell/2}}\p{\frac{\vc_1 \vc_2}{e  d_1 d_2 \lambda^2 \s{D}^2}}^m,\label{eq:innerouterIntersect}
         \end{align}
         where $(a)$ follows from the assumption that $\lambda^2=o(\s{D}^{-2})$. Now, consider the case where \begin{align}
               \frac{\vc_1 \vc_2}{e  d_1 d_2 \lambda^2 \s{D}^2}>1+\alpha,\label{eq:RHSalpha}
          \end{align}
          for some fixed constant $\alpha$. In this case, the inner sum in the r.h.s. of \eqref{eq:innerouterIntersect} is dominated by the last term, and we have,
          \begin{align}
               \E_{\Gamma_1\indep \Gamma_2}\pp{\sum_{m=0}^{\min\p{\s{D},|e(\Gamma_1\cap \Gamma_2)|}}\binom{|e(\Gamma_1\cap\Gamma_2)|}{m}\lambda^{2m}}\leq 1+C_\varepsilon''  
          \sum_{\ell=2}^{2\s{D}} \p{\frac{2(1+\varepsilon) \lambda \s{D} \sqrt{ e  d_1 d_2 \vc_1 \vc_2} }{n-\bar{k}}}^\ell,
          \end{align}
          which is bounded if 
          \begin{align}
              \frac{\lambda \s{D} \sqrt{  d_1 d_2 \vc_1 \vc_2} }{n-\bar{k}}<C,\label{eq:condInters1}
          \end{align}
          for some universal constant $C$. Similarly if the l.h.s. of \eqref{eq:RHSalpha} is upper bounded by $1+o(1)$, then the inner sum in the r.h.s. of is dominated by $(1+\varepsilon)^\ell$, and, accordingly, \eqref{eq:innerouterIntersect} is bounded if, 
          \begin{align}
             \frac{ d_1 d_2 \lambda^2 \s{D}^2}{n-\bar{k}}\leq C. \label{eq:condInters2}
          \end{align}
          Combining the conditions in \eqref{eq:condInters1} and \eqref{eq:condInters2}, we obtain \eqref{eq:condIntersectionD2Comp2}.
\end{proof} 

Before proving Proposition~\ref{prop:decompisitionThDCcomp}, we note that while the argument of Lemma~\ref{lem:interectionD} extends relatively easily to the low-degree scenario, generalizing the H\"{o}lder's inequality argument used in Proposition~\ref{prop:decompisitionThD} is significantly more technically involved. Specifically, for the statistical lower bounds, the second moment $\E_{\calH_0}[\s{L}(\s{G})^2]$ is given by $\E_{\Gamma}[\exp(c\cdot |e(\Gamma\cap \Gamma'|))]$. Substituting the  decomposition of $\Gamma$ into $\Gamma_1,\dots,\Gamma_M$, and applying H\"{o}lder's inequality allows us to analyze each crossing term $\E[\exp(c_M\cdot |e(\Gamma_i\cap \Gamma_j'|))]$, separately. In the low-degree case, on the other hand, this is not the case anymore. Indeed, as seen from Lemma~\ref{lem:intersectionMomentGenGeneralComp}, the low-degree likelihood's second moment is given by,
\begin{align}
    \norm{\s{L}_{n,\leq \s{D}}}_{\calH_0}^2&=\bE_{\Gamma\indep\Gamma'}\pp{\sum_{m=0}^{\min(\s{D},|e(\Gamma\cap \Gamma')|)}\binom{|e(\Gamma\cap \Gamma')|}{m}\lambda^{2m}}\\
    &=\bE_{\Gamma\indep\Gamma'}\pp{\sum_{m=0}^{\min(\s{D},\sum_{i,j}|e(\Gamma_i\cap \Gamma_j')|)}\binom{\sum_{i,j}|e(\Gamma_i\cap \Gamma_j')|)}{m}\lambda^{2m}}.
\end{align}
At first glance, it is rather unclear whether the above expression can be bounded using a product of expressions of the form
\begin{align}
    \bE_{\Gamma_i\indep\Gamma_j'}\pp{\sum_{m=0}^{\min(\s{D},|e(\Gamma_i\cap \Gamma_j')|)}\binom{|e(\Gamma_i\cap \Gamma_j')|)}{m}\bar{\lambda}_M^{2m}},
\end{align}
for some ``well-behaved" $\bar{\lambda}_M=f(\lambda^2,M,\s{D})$, as in the statistical lower bound case. In Proposition~\ref{prop:decompisitionThDCcomp}, we provide a positive answer to this question, which forms the main argument in the proof.

\begin{proof}[Proof of Proposition~\ref{prop:decompisitionThDCcomp}]
    First, we note that by Lemma~\ref{lem:intersectionMomentGenGeneralComp}, we have 
    \begin{align}
        \norm{\s{L}_{n,\leq \s{D}}}_{\calH_0}^2&=\sum_{\substack{\s{H}\subseteq \binom{[n]}{2}\\ |\s{H}|\leq \s{D}}} \lambda^{2\abs{\s{H}}}\P_{\Gamma}\pp{\s{H}\subseteq \Gamma}^2\\
        & =\bE_{\Gamma\indep\Gamma'}\pp{\sum_{m=0}^{\min(\s{D},|e(\Gamma\cap \Gamma')|)}\binom{|e(\Gamma\cap \Gamma')|}{m}\lambda^{2m}},\label{eqn:compStartToUpperDecom}
    \end{align}
    \begin{comment}
        Lets save a copy of this in case we decide to remove the preivious lemma 

         Consider the following chain of inequalities,
    \begin{align}
        \norm{\s{L}_{n,\leq \s{D}}}_{\calH_0}^2&=\sum_{\substack{\s{H}\subseteq \binom{[n]}{2}\\ |\s{H}|\leq \s{D}}} \lambda^{2\abs{\s{H}}}\P_{\Gamma}\pp{\s{H}\subseteq \Gamma}^2\\
        &=\sum_{\substack{\s{H}\subseteq \binom{[n]}{2}\\ |\s{H}|\leq \s{D}}} \lambda^{2|\s{H}|}\P_{\Gamma}\pp{\s{H}\subseteq \Gamma}\cdot\P_{\Gamma'}\pp{\s{H}\subseteq \Gamma'}\\
        &=\sum_{\substack{\s{H}\subseteq \binom{[n]}{2}\\ |\s{H}|\leq \s{D}}} \lambda^{2|\s{H}|}\P_{\Gamma\indep\Gamma'}\pp{\s{H}\subseteq \Gamma\cap \Gamma'}
        \\&=\sum_{\substack{\s{H}\subseteq \binom{[n]}{2}\\ |\s{H}|\leq \s{D}}} \lambda^{2|\s{H}|}\sum_{\substack{\s{H}'\in \binom{[n]}{2}\\\s{H}\subseteq\s{H}'}}\P_{\Gamma\indep\Gamma'}\pp{\s{H}'= \Gamma\cap \Gamma'}\\
        &=\sum_{\substack{\s{H}'\subseteq \binom{[n]}{2} }} \P\pp{\s{H}'= \Gamma\cap \Gamma'}\sum_{\substack{\s{H}\subseteq \s{H}'\\|\s{H}|\leq \s{D}}}\lambda^{2|\s{H}|}\\
        &=\sum_{\substack{\s{H}'\subseteq \binom{[n]}{2} }} \P_{\Gamma\indep\Gamma'}\pp{\s{H}'= \Gamma\cap \Gamma'}\sum_{m=0}^{\min(\s{D},|e(\s{H}')|)}\binom{|e(\s{H}')|}{m}\lambda^{2m}\\
        & = \bE_{\Gamma\indep\Gamma'}\pp{\sum_{m=0}^{\min(\s{D},|e(\Gamma\cap \Gamma')|)}\binom{|e(\Gamma\cap \Gamma')|}{m}\lambda^{2m}},\label{eqn:compStartToUpperDecom}
    \end{align}
    \end{comment}
    where $\Gamma'$ is an independent random copy of $\Gamma$, both drawn uniformly at random over $\calS_\Gamma$. While all the expectations throughout the following analysis are with respect to the distribution of $\Gamma$ and $\Gamma'$, we drop the $\Gamma\indep \Gamma'$ notation from the expectations for the sake of readability. %Below, we assume for simplicity that $\lambda^2\geq1$; otherwise, if $\lambda^2<1$, then just replace $\lambda^2$ with unity. 
    Consider an edge-disjoint decomposition $\Gamma = \bigcup_{i=1}^M\Gamma_i$. We next show that,
    \begin{align}
        \bE\pp{\sum_{m=0}^{\min(\s{D},|e(\Gamma\cap \Gamma')|)}\binom{|e(\Gamma\cap \Gamma')|}{m}\lambda^{2m}}\leq \bE\pp{\prod_{i,j\in[M]}\sum_{m_{i,j}=0}^{\min(\s{D},|\Gamma_i\cap\Gamma_j'|)}\binom{|\Gamma_i\cap\Gamma_j'|}{m}\lambda^{2m}}.\label{eqn:DecompoStep1}
    \end{align}
     To that end, let us rewrite the term on the r.h.s. of \eqref{eqn:DecompoStep1} as follows. We let 
     \begin{align}
         \calJ_{\s{D}}\triangleq\ppp{(m_{i,j})_{i,j=1}^{M}\in [D]^{M^2}:0\leq m_{i,j}\leq \min(\s{D},|\Gamma_i\cap\Gamma_j'|),\forall1\leq i,j\leq M}
     \end{align}
     Below, for brevity, we will use the following shorthand notation for multiple sums,
    \begin{align}
        \sum_{m_{1,1}=0}^{\min(\s{D},|\Gamma_1\cap\Gamma_1'|)}\sum_{m_{1,2}=0}^{\min(\s{D},|\Gamma_1\cap\Gamma_2'|)}\ldots\sum_{m_{M,M}=0}^{\min(\s{D},|\Gamma_M\cap\Gamma_M'|)} \equiv\sum_{(m_{i,j})_{i,j}\in\calJ_{\s{D}}}.
    \end{align}
    Then, clearly,
    \begin{align}
        \prod_{i,j\in[M]}\sum_{m=0}^{\min(\s{D},|\Gamma_i\cap\Gamma_j'|)}\binom{|\Gamma_i\cap\Gamma_j'|}{m}\lambda^{2m} = \sum_{(m_{i,j})_{i,j}\in\calJ_{\s{D}}}\prod_{i,j\in[M]}\binom{|\Gamma_i\cap\Gamma_j'|}{m_{i,j}}\lambda^{2m_{i,j}}.\label{eqn:Step1Exp0}
    \end{align}
    Now, we note that,
    \begin{align}
        &\sum_{(m_{i,j})_{i,j}\in\calJ_{\s{D}}}\prod_{i,j\in[M]}\binom{|\Gamma_i\cap\Gamma_j'|}{m_{i,j}}\lambda^{2m_{i,j}}\nonumber\\
        &\hspace{1cm}= \sum_{m=0}^{\sum_{i,j\in[M]}\min(\s{D},|\Gamma_i\cap\Gamma_j'|)}\lambda^{2m}\sum_{\substack{(m_{i,j})_{i,j}\in\calJ_{\s{D}}\\\sum_{i,j}m_{i,j}=m}}\prod_{i,j\in[M]}\binom{|\Gamma_i\cap\Gamma_j'|}{m_{i,j}}.\label{eqn:Step1Exp}
    \end{align}
    We next show that \eqref{eqn:Step1Exp} is lower bounded by the sum on the l.h.s. of \eqref{eqn:DecompoStep1}. First, we notice that,
    \begin{align}
        \sum_{i,j\in[M]}\min(\s{D},|\Gamma_i\cap\Gamma_j'|)\geq \min\p{\s{D},\sum_{i,j\in[M]}|\Gamma_i\cap\Gamma_j'|} = \min(\s{D},|\Gamma\cap\Gamma'|),
    \end{align}
    and thus,
    \begin{align}
        \sum_{(m_{i,j})_{i,j}\in\calJ_{\s{D}}}\prod_{i,j\in[M]}\binom{|\Gamma_i\cap\Gamma_j'|}{m_{i,j}}\lambda^{2m_{i,j}}\geq \sum_{m=0}^{\min(\s{D},|\Gamma\cap\Gamma'|)}\lambda^{2m}\sum_{\substack{(m_{i,j})_{i,j}\in\calJ_{\s{D}}\\\sum_{i,j}m_{i,j}=m}}\prod_{i,j\in[M]}\binom{|\Gamma_i\cap\Gamma_j'|}{m_{i,j}}.\label{eqn:Step1Exp2}
    \end{align}
    We recall Vandermonde's identity,
    \begin{align}
        \binom{|\Gamma\cap\Gamma'|}{m} = \sum_{\substack{(m_{i,j})_{i,j}\in\calJ_{\infty}\\\sum_{i,j}m_{i,j}=m}}\prod_{i,j\in[M]}\binom{|\Gamma_i\cap\Gamma_j'|}{m_{i,j}},\label{eqn:Step1Exp3}
    \end{align}
    where 
    \begin{align}
        \calJ_\infty\triangleq \ppp{(m_{i,j})_{i,j=1}^{M}:0\leq m_{i,j}\leq |\Gamma_i\cap\Gamma_j'|,\forall1\leq i,j\leq M}.
    \end{align}
    Then, the sum on the l.h.s. of \eqref{eqn:DecompoStep1} can be rewritten as,
    \begin{align}
        \sum_{m=0}^{\min(\s{D},|e(\Gamma\cap \Gamma')|)}\binom{|e(\Gamma\cap \Gamma')|}{m}\lambda^{2m}=\sum_{m=0}^{\min(\s{D},|e(\Gamma\cap \Gamma')|)}\lambda^{2m}\sum_{\substack{(m_{i,j})_{i,j}\in\calJ_{\infty}\\\sum_{i,j}m_{i,j}=m}}\prod_{i,j\in[M]}\binom{|\Gamma_i\cap\Gamma_j'|}{m_{i,j}}.\label{eqn:Rep2}
    \end{align}
    Thus, in light of \eqref{eqn:Step1Exp2} and \eqref{eqn:Rep2}, in order to prove \eqref{eqn:DecompoStep1}, it suffices to show that,
    \begin{align}
        \sum_{\substack{(m_{i,j})_{i,j}\in\calJ_{\s{D}}\\\sum_{i,j}m_{i,j}=m}}\prod_{i,j\in[M]}\binom{|\Gamma_i\cap\Gamma_j'|}{m_{i,j}}\geq \sum_{\substack{(m_{i,j})_{i,j}\in\calJ_{\infty}\\\sum_{i,j}m_{i,j}=m}}\prod_{i,j\in[M]}\binom{|\Gamma_i\cap\Gamma_j'|}{m_{i,j}}.\label{eqn:combind}
    \end{align}
    \sloppy
    Indeed, the key here is the range of $m$, namely, $0\leq m\leq \min(\s{D},|\Gamma\cap\Gamma'|) = \min(\s{D},\sum_{i,j\in[M]}|\Gamma_i\cap\Gamma_j'|)$. Let $\calQ$ denote the pairs $(i,j)$ for which  $\min(\s{D},|\Gamma_{i}\cap\Gamma_{j}'|)=|\Gamma_{i}\cap\Gamma_{j}'|$, and then $\calQ^c = [M]^2\setminus\calQ$. Note that pairs $(i,j)\in\calQ$, the summations that corresponds to $(i,j)$ have the sum range in both sides of \eqref{eqn:combind}, and as so, we may focus on $\calQ^c$. In this case, $\min(\s{D},|\Gamma_{i}\cap\Gamma_{j}'|)=\s{D}$, and thus, the range of these pairs $(i,j)$ in the summation on the l.h.s. of \eqref{eqn:combind} is $0\leq m_{i,j}\leq \s{D}$, which contained in the range of those on the r.h.s. of \eqref{eqn:combind}, which is $0\leq m_{i,j}\leq |\Gamma_{i}\cap\Gamma_{j}'|$. While at first sight, this might seem problematic, since for both terms in \eqref{eqn:combind} we have 
    \begin{align}
        0\leq \sum_{i,j}m_{i,j}=m\leq \min(\s{D},|\Gamma\cap\Gamma'|) = \min\bigg(\s{D},\sum_{i,j\in[M]}|\Gamma_i\cap\Gamma_j'| \bigg)=\s{D},
    \end{align}it must be that $0\leq m_{i,j}\leq \s{D}$ for the term on the r.h.s. of \eqref{eqn:combind} as well. Thus, \eqref{eqn:combind} holds, and in fact, with equality. Combining \eqref{eqn:Step1Exp0}, \eqref{eqn:Step1Exp}, \eqref{eqn:Step1Exp2}, \eqref{eqn:Step1Exp3}, and \eqref{eqn:combind}, proves \eqref{eqn:DecompoStep1}. 

    Next, applying H\"{o}lder's inequality on  \eqref{eqn:DecompoStep1}, we get that,
    \begin{align}
        \bE\pp{\prod_{i,j\in[M]}\sum_{m=0}^{\min(\s{D},|\Gamma_i\cap\Gamma_j'|)}\binom{|\Gamma_i\cap\Gamma_j'|}{m}\lambda^{2m}}\leq \prod_{i,j\in[M]}\p{\bE\pp{\sum_{m=0}^{\min(\s{D},|\Gamma_i\cap\Gamma_j'|)}\binom{|\Gamma_i\cap\Gamma_j'|}{m}\lambda^{2m}}^{M^2}}^{\frac{1}{M^2}}.\label{eqn:afterHoldComp}
    \end{align}
    We now analyze the inner expectation on the r.h.s. of  \eqref{eqn:afterHoldComp}, for any given pair $(i,j)$. Let $\varphi\in[0,1]$ be defined as $\varphi\triangleq\frac{\lambda^2}{1+\lambda^2}$, and as so $\lambda^2 = \frac{\varphi}{1-\varphi}$. Then,
    \begin{align}
        \sum_{m=0}^{\min(\s{D},|\Gamma_i\cap\Gamma_j'|)}\binom{|\Gamma_i\cap\Gamma_j'|}{m}\lambda^{2m} &= (1-\varphi)^{-|\Gamma_i\cap\Gamma_j'|}\cdot\sum_{m=0}^{\min(\s{D},|\Gamma_i\cap\Gamma_j'|)}\binom{|\Gamma_i\cap\Gamma_j'|}{m}\varphi^{m}(1-\varphi)^{|\Gamma_i\cap\Gamma_j'|-m}\\
        & = (1+\lambda^2)^{|\Gamma_i\cap\Gamma_j'|}\cdot\pr\pp{\s{B}\leq \min(\s{D},|\Gamma_i\cap\Gamma_j'|)},
    \end{align}
    where $\s{B}\triangleq\s{Binomial}(|\Gamma_i\cap\Gamma_j'|,\varphi)$. Therefore,
    \begin{align}
        \pp{\sum_{m=0}^{\min(\s{D},|\Gamma_i\cap\Gamma_j'|)}\binom{|\Gamma_i\cap\Gamma_j'|}{m}\lambda^{2m}}^{M^2} = (1+\lambda^2)^{M^2|\Gamma_i\cap\Gamma_j'|}\cdot\p{\pr\pp{\s{B}\leq \min(\s{D},|\Gamma_i\cap\Gamma_j'|)}}^{M^2}.
    \end{align}
    Now, let $\ppp{\s{B}_{\ell}}_{\ell=1}^{M^2}$ be $M^2$ independent random copies of $\s{B}$. Then, we have,
    \begin{align}
        \p{\pr\pp{\s{B}\leq \min(\s{D},|\Gamma_i\cap\Gamma_j'|)}}^{M^2} &= \pr\pp{\max_{1\leq \ell\leq M^2}\s{B}_\ell\leq \min(\s{D},|\Gamma_i\cap\Gamma_j'|)}\\
        &\leq \pr\pp{\frac{1}{M^2}\sum_{\ell=1}^{M^2}\s{B}_\ell\leq \min(\s{D},|\Gamma_i\cap\Gamma_j'|)}\\
        &= \pr\pp{\bar{\s{B}}\leq\min(M^2\s{D},M^2|\Gamma_i\cap\Gamma_j'|)},
    \end{align}
    where $\bar{\s{B}}\triangleq\s{Binomial}\p{M^2|\Gamma_i\cap\Gamma_j'|,\varphi}$. Thus,
    \begin{align}
        &\pp{\sum_{m=0}^{\min(\s{D},|\Gamma_i\cap\Gamma_j'|)}\binom{|\Gamma_i\cap\Gamma_j'|}{m}\lambda^{2m}}^{M^2} \leq (1+\lambda^2)^{M^2|\Gamma_i\cap\Gamma_j'|}\cdot\pr\pp{\bar{\s{B}}\leq\min(M^2\s{D},M^2|\Gamma_i\cap\Gamma_j'|)}\\
        &=(1+\lambda^2)^{M^2|\Gamma_i\cap\Gamma_j'|}\sum_{m=0}^{\min(M^2\s{D},M^2|\Gamma_i\cap\Gamma_j'|)}\binom{M^2|\Gamma_i\cap\Gamma_j'|}{m}\varphi^{m}(1-\varphi)^{M^2|\Gamma_i\cap\Gamma_j'|-m}\\
        &=\sum_{m=0}^{\min(M^2\s{D},M^2|\Gamma_i\cap\Gamma_j'|)}\binom{M^2|\Gamma_i\cap\Gamma_j'|}{m}\lambda^{2m},\label{eqn:lastStepCompDec}
    \end{align}
    where in the last equality we have used the fact that $(1+\lambda^2)^{M^2|\Gamma_i\cap\Gamma_j'|}(1-\varphi)^{M^2|\Gamma_i\cap\Gamma_j'|}=1$. Therefore, combining \eqref{eqn:compStartToUpperDecom}, \eqref{eqn:DecompoStep1}, \eqref{eqn:afterHoldComp}, and \eqref{eqn:lastStepCompDec}, we get that,
    \begin{align}
        \norm{\s{L}_{n,\leq \s{D}}}_{\calH_0}^2\leq \prod_{i,j\in[M]}\p{\bE\pp{\sum_{m=0}^{M^2\min(\s{D},|\Gamma_i\cap\Gamma_j'|)}\binom{M^2|\Gamma_i\cap\Gamma_j'|}{m}\lambda^{2m}}}^{\frac{1}{M^2}}.\label{eqn:almostTheEnd}
    \end{align}
    %\begin{rmk}
    %    Actually with a bit more work we can show that
     %   \begin{align}
      %  \norm{\s{L}_{n,\leq \s{D}}}_{\calH_0}^2\leq \prod_{i,j\in[M]}\p{\bE\pp{\sum_{m=0}^{M^2\min(\s{D},|\Gamma_i\cap\Gamma_j'|)}\binom{M^2|\Gamma_i\cap\Gamma_j'|}{m}\lambda^{2m}f(\s{D},|\Gamma_i\cap\Gamma_j|)}}^{\frac{1}{M^2}}.
    %\end{align}
    %where
    %\begin{align}
    %    f(\s{D},|\Gamma_i\cap\Gamma_j|) = \sum_{(k_1,\ldots,k_{M^2})\in\calK}\pr\pp{\bigcap_{i=1}^{M^2}H_i=k_i},
    %\end{align}
    %where is mutli-Hypergeometric distribution 
    %\end{rmk}
    \begin{rmk}
        As a sanity check, note that if we take $\s{D}\to\infty$, we end up with,
    \begin{align}
        \norm{\s{L}_{n,\leq \infty}}_{\calH_0}^2&\leq \prod_{i,j\in[M]}\p{\bE\pp{\sum_{m=0}^{M^2|\Gamma_i\cap\Gamma_j'|}\binom{M^2|\Gamma_i\cap\Gamma_j'|}{m}\lambda^{2m}}}^{\frac{1}{M^2}}\\
        &=\prod_{i,j\in[M]}\p{\bE\pp{(1+\lambda^2)^{M^2|\Gamma_i\cap\Gamma_j'|}}}^{\frac{1}{M^2}},
    \end{align}
    which is the expression we got in our derivation of the statistical lower bound.
    \end{rmk}
    Let us simplify \eqref{eqn:almostTheEnd}. Note that,
    \begin{align}
       \sum_{m=0}^{M^2\min(\s{D},|\Gamma_i\cap\Gamma_j'|)}\binom{M^2|\Gamma_i\cap\Gamma_j'|}{m}\lambda^{2m} \leq  \s{A}+\s{B},\label{eqn:DecompositionAandB}
    \end{align}
    where
    \begin{align}
        \s{A}\triangleq\Ind\ppp{|\Gamma_i\cap\Gamma_j'|\leq\s{D}}\sum_{m=0}^{M^2|\Gamma_i\cap\Gamma_j'|}\binom{M^2|\Gamma_i\cap\Gamma_j'|}{m}\lambda^{2m},\label{eqn:sAdef}
    \end{align}
    and
    \begin{align}
        \s{B}\triangleq\Ind\ppp{|\Gamma_i\cap\Gamma_j'|\geq\s{D}}\cdot\pp{1+\sum_{m=1}^{M^2\s{D}}\binom{M^2|\Gamma_i\cap\Gamma_j'|}{m}\lambda^{2m}}.\label{eqn:sBdef}
    \end{align}
    As for $\s{A}$, using the Binomial theorem, we readily see that,
    \begin{align}
        \sum_{m=0}^{M^2|\Gamma_i\cap \Gamma_j'|}\binom{M^2|\Gamma_i\cap\Gamma_j'|}{m}\lambda^{2m}&=(1+\lambda^2)^{M^2|\Gamma_i\cap\Gamma_j|}=(1+\lambda_M^2)^{|\Gamma_i\cap\Gamma_j|},
    \end{align}
    for $\lambda_M^2 = (1+\lambda^2)^{M^2}-1$, and therefore
    \begin{align}
        \s{A}=\sum_{m=0}^{|\Gamma_i\cap\Gamma_j'|}\binom{|\Gamma_i\cap\Gamma_j'|}{m}\lambda_{M}^{2m}\Ind\ppp{|\Gamma_i\cap\Gamma_j'|\leq\s{D}}.
    \end{align}
      Now, for $\s{B}$, we split our analysis into two parts, where $\s{D}\leq|\Gamma_i\cap\Gamma_j'|\leq M^2\s{D}$ and $|\Gamma_i\cap\Gamma_j'|\geq M^2\s{D}$, and start our analysis with the later. Our goal is to  find a parameter $\bar\lambda\in\mathbb{R}_+$ such that,
    \begin{align}
        \sum_{m=1}^{M^2\s{D}}\binom{M^2|\Gamma_i\cap\Gamma_j'|}{m}\lambda^{2m}\leq \sum_{m=1}^{M^2\s{D}}\binom{|\Gamma_i\cap\Gamma_j'|}{m}\bar\lambda^{2m}.
    \end{align}
    To that end, on one hand, we will use the following observation. Let $f(\bar{m})\triangleq\binom{\bar{n}}{\bar{m}}\alpha^{\bar{m}}$, for $1\leq \bar{m}\leq \bar{k}$ and $\alpha>0$. We would like to find the value of $\bar{m}$ which maximizes $f(\bar{m})$ in the interval $0\leq \bar{m}\leq \bar{k}$. We have,
    \begin{align}
        \frac{f(\bar{m}-1)}{f(\bar{m})} = \frac{1}{\alpha}\frac{\bar{m}}{\bar{n}-\bar{m}+1}.
    \end{align}
    Thus, we have $f(\bar{m}-1)\leq f(\bar{m})$ for,
    \begin{align}
        \bar{m}\leq \frac{\alpha(\bar{n}+1)}{1+\alpha}.
    \end{align}
    Therefore, the value $m^\star\triangleq\min\p{\ceil{\frac{\alpha(\bar{n}+1)}{1+\alpha}},\bar{k}}$ maximizes $f(\bar{m})$ over $0\leq \bar{m}\leq \bar{k}$. Using this fact, in our case we note that,
    \begin{align}
        \sum_{m=1}^{M^2\s{D}}\binom{M^2|\Gamma_i\cap\Gamma_j'|}{m}\lambda^{2m}\leq M^2\s{D}\binom{M^2|\Gamma_i\cap\Gamma_j'|}{m^\star}\lambda^{2m^\star},
    \end{align}
    where 
    \begin{align}
        m^\star = \min\p{\ceil{\frac{\lambda^2 (M^2|\Gamma_i\cap\Gamma_j'|+1)}{\lambda^2+1}},M^2\s{D}}.\label{eq:mStar}
    \end{align} 
    Let us now assume that $\lambda^2=\Theta(1)$, and address the case where $\lambda=o(\s{D}^{-2})$ later. When $\lambda=\Theta(1)$ we may assumed without loss of generality that $\lambda^2\geq1$ and thus $\lambda^2/(1+\lambda^2)\geq  1/2$. 
    In the regime where $|\Gamma_i\cap\Gamma_j'|\geq M^2\s{D}$, we have $m^\star=cM^2\s{D}$, for some $0<c\leq1$ (in fact, since $M\geq 2$,  under the assumption $|\Gamma_i\cap\Gamma_j'|\geq M^2\s{D}$ we have that $m^\star = M^2\s{D}$ and $c=1$).  Thus,
    \begin{align}
        \sum_{m=1}^{M^2\s{D}}\binom{M^2|\Gamma_i\cap\Gamma_j'|}{m}\lambda^{2m}&\leq M^2\s{D}\cdot\binom{M^2|\Gamma_i\cap\Gamma_j'|}{cM^2\s{D}}\lambda^{2cM^2\s{D}}\\
        &\leq M^2\s{D}\cdot\p{\frac{eM^2|\Gamma_i\cap\Gamma_j'|}{cM^2\s{D}}}^{cM^2\s{D}}\lambda^{2cM^2\s{D}}\label{eq:maximizingShit}\\
        & \leq M^2\s{D}\cdot\p{\frac{|\Gamma_i\cap\Gamma_j'|}{c\s{D}}}^{cM^2\s{D}}(e\lambda^2)^{cM^2\s{D}},\label{eqn:UpperBinomwith}
    \end{align}
    where we have used the fact that $\binom{n}{m}\leq\p{\frac{en}{m}}^m$, for $m\leq n$. On the other hand, for any $\bar{\lambda}$:
    \begin{align}
        \sum_{m=1}^{M^2\s{D}}\binom{|\Gamma_i\cap\Gamma_j'|}{m}\bar\lambda^{2m}&\geq \binom{|\Gamma_i\cap\Gamma_j'|}{cM^2\s{D}}\bar\lambda^{2cM^2\s{D}}\\
        &\geq\p{\frac{|\Gamma_i\cap\Gamma_j'|}{c\s{D}}}^{cM^2\s{D}}M^{-2cM^2\s{D}} \bar\lambda^{2cM^2\s{D}},\label{eqn:LowerBinomwith}
    \end{align}
    where we have used the fact that $\p{\frac{n}{m}}^m\leq\binom{n}{m}$, for $m\leq n$. Then, it can be checked that the term on the r.h.s. of \eqref{eqn:LowerBinomwith} is larger than the term on the r.h.s. of \eqref{eqn:UpperBinomwith} for any,
    \begin{align}
        \bar\lambda^2\geq \lambda^2_{M,2}\triangleq \pp{M^2\s{D}M^{2cM^2\s{D}}(e\lambda^2)^{cM^2\s{D}}}^{\frac{1}{cM^2D}} = eM^2[M^2D]^{\frac{1}{cM^2\s{D}}}\lambda^2,\label{eq:LambdaBar1}
    \end{align}
    We mention here that $\bar\lambda^2$ is bounded as $M$ is assumed fixed and $c$ is bounded away from $0$. To conclude, for $|\Gamma_i\cap\Gamma_j'|\geq M^2\s{D}$,
    \begin{align}
        \sum_{m=1}^{M^2\s{D}}\binom{M^2|\Gamma_i\cap\Gamma_j'|}{m}\lambda^{2m}\leq \sum_{m=1}^{M^2\s{D}}\binom{|\Gamma_i\cap\Gamma_j'|}{m}\lambda_{M,2}^{2m}.\label{eqn:LargeValInter}
    \end{align}
   For $\s{D}\leq|\Gamma_i\cap\Gamma_j'|\leq M^2\s{D}$, we first show that there exists $\bar\lambda\in\mathbb{R}_+$ such that,
    \begin{align}
        \sum_{m=1}^{M^2\s{D}}\binom{M^2|\Gamma_i\cap\Gamma_j'|}{m}\lambda^{2m}\leq \sum_{m=1}^{\s{D}}\binom{|\Gamma_i\cap\Gamma_j'|}{m}^{M^2}\bar\lambda^{2m}.\label{eqn:combComp1}
    \end{align}
    This follows exactly from the same arguments above. Indeed, as in the previous case, there is a $0<c\leq  1$ such that $m^\star=cM^2\s{D}$ such that the $m^{\star}$ term is the maximal among the summands of the l.h.s. of \eqref{eqn:combComp1}, and therefore multiplying it by $M^2\s{D}$ upper bounds the l.h.s. of \eqref{eqn:combComp1}. If $c\s{D}$ is an integer, similarly to the previous case we have
    \begin{align}
        \sum_{m=1}^{\s{D}}\binom{|\Gamma_i\cap\Gamma_j'|}{m}^{M^2}\bar\lambda^{2m}&\geq \binom{|\Gamma_i\cap\Gamma_j'|}{c\s{D}}^{M^2}\bar\lambda^{2c\s{D}}\\
        &\geq\p{\frac{|\Gamma_i\cap\Gamma_j'|}{c\s{D}}}^{cM^2\s{D}}\bar\lambda^{2c\s{D}},
    \end{align}
    Thus, \eqref{eqn:combComp1} holds for, 
    \begin{align}
        \bar\lambda^2\geq\lambda_{M,3}^{2}\triangleq(M^2\s{D})^{\frac{1}{c\s{D}}}(e\lambda^2)^{M^2}.\label{eq:lambdaBARkashe}
    \end{align}
     On the other hand, if $c\s{D}$ is not an integer, note that the function $f_{\hat{m}}(x)=\p{\frac{\hat{m}}{x}\lambda^{2}}^{M^2 x}$ has at most one extrema point in $(0,\hat{m}]$, at $\frac{\hat{m}}{e}$. Consider $\hat{m}=|\Gamma_i\cap \Gamma_j'|$. In the regime where $\lambda^2=\Theta(1)$, we may assume without loss of generality that $\bar{\lambda}^2$ is sufficiently large such that $\bar{\lambda}^2/(1+\bar{\lambda}^2)$ is arbitrarily close to one.  Thus from the definition of $m^\star$ and since we are in the regime where $|\Gamma_i\cap \Gamma_j'|\geq \s{D}$, we can assume that $c$ is bounded away from $e^{-1}$. In particular $[c\s{D}-1,c\s{D}+1]$ does not contain the maximum of $f_{\hat{m}}$. Hence, it must hold that,
     \begin{align}
       \p{\frac{|\Gamma_i\cap\Gamma_j'|}{c\s{D}}}^{cM^2\s{D}}\bar\lambda^{2c\s{D}}\leq \max\p{\p{\frac{|\Gamma_i\cap\Gamma_j'|}{\ceil{c\s{D}}}}^{\ceil{c\s{D}}M^2}\bar\lambda^{2\ceil{c\s{D}}}, \p{\frac{|\Gamma_i\cap\Gamma_j'|}{\floor{c\s{D}}}}^{\floor{c\s{D}}M^2}\bar\lambda^{2\floor{c\s{D}}} }.\label{eq:StillmAzhik}
     \end{align}
     Thus, there exists $0<\tilde{c}\leq 1$ such that $\tilde{c}\s{D}\in \{\floor{c\s{D}}, \ceil{c\s{D}} \}$ and, 
     \begin{align}
         \sum_{m=1}^{\s{D}}\binom{|\Gamma_i\cap\Gamma_j'|}{m}^{M^2}\bar\lambda^{2m}&\geq \binom{|\Gamma_i\cap\Gamma_j'|}{\tilde{c}\s{D}}^{M^2}\bar\lambda^{2\tilde{c}\s{D}}\\
        &\geq\p{\frac{|\Gamma_i\cap\Gamma_j'|}{\tilde{c}\s{D}}}^{\tilde{c}M^2\s{D}}\bar\lambda^{2\tilde{c}\s{D}}\\&\geq\p{\frac{|\Gamma_i\cap\Gamma_j'|}{c\s{D}}}^{cM^2\s{D}}\bar\lambda^{2c\s{D}}.
     \end{align}
     In particular, \eqref{eqn:combComp1} holds for the same $\overline{\lambda}^2$ from \eqref{eq:lambdaBARkashe}.
     
     Next, we show further that there exists $\tilde\lambda\in\mathbb{R}_+$ such that,
    \begin{align}
       \sum_{m=0}^{\s{D}}\binom{|\Gamma_i\cap\Gamma_j'|}{m}^{M^2}\lambda_{M,3}^{2m}\leq \sum_{m=0}^{\s{D}}\binom{|\Gamma_i\cap\Gamma_j'|}{m}\tilde\lambda^{2m},\label{eqn:combComp2}
    \end{align}
    Indeed, let $<0<c\leq 1$ be such that $m^{\star} =c\s{D}$ maximizes the summands of the l.h.s. of \eqref{eqn:combComp2}.  On one hand, we note that,
    \begin{align}
        \sum_{m=1}^{\s{D}}\binom{|\Gamma_i\cap\Gamma_j'|}{m}^{M^2}\lambda_{M,3}^{2m}&\leq \s{D}\p{\frac{|\Gamma_i\cap\Gamma_j'|}{c\s{D}}}^{M^2}\lambda_{M,3}^{2cD}\\
        &\leq \s{D}\p{\frac{M^2D}{c\s{D}}}^{M^2}\lambda_{M,3}^{2cD}\\
        & =  \s{D}\p{\frac{M^2}{c}}^{M^2}\lambda_{M,3}^{2cD},\label{eqn:UpperBinomwith2}
    \end{align}
    where we have used the fact that we are in the regime where $\s{D}\leq|\Gamma_i\cap\Gamma_j'|\leq M^2\s{D}$. On the other hand,
    \begin{align}
        \sum_{m=1}^{\s{D}}\binom{|\Gamma_i\cap\Gamma_j'|}{m}\tilde\lambda^{2m}\geq \p{\frac{|\Gamma_i\cap\Gamma_j'|}{c\s{D}}}^{c\s{D}}\tilde\lambda^{2c\s{D}}\geq \p{\frac{1}{c}}^{c\s{D}}\tilde\lambda^{2c\s{D}},\label{eqn:LowerBinomwith2}
    \end{align}
    where we have used the fact that we are in the regime where $\s{D}\leq|\Gamma_i\cap\Gamma_j'|$. Then, it can be checked that the term on the r.h.s. of \eqref{eqn:LowerBinomwith2} is larger than the term on the r.h.s. of \eqref{eqn:UpperBinomwith2} for any,
    \begin{align}
        \tilde\lambda^2\geq \lambda^2_{M,4}\triangleq \pp{c^{-c\s{D}}\s{D}\p{\frac{M^2}{c}}^{M^2}\lambda_{M,3}^{2cD}}^{\frac{1}{c\s{D}}} = \frac{\s{D}^{\frac{1}{c\s{D}}}\p{\frac{M^2}{c}}^{\frac{M^2}{c\s{D}}}}{c}\lambda_{M,3}^{2}=O(1).
    \end{align}
    Thus, combining \eqref{eqn:combComp1} and \eqref{eqn:combComp2}, we get that for $\s{D}\leq|\Gamma_i\cap\Gamma_j'|\leq M^2\s{D}$,
    \begin{align}
        \sum_{m=0}^{M^2\s{D}}\binom{M^2|\Gamma_i\cap\Gamma_j'|}{m}\lambda^{2m}\leq \sum_{m=0}^{\s{D}}\binom{|\Gamma_i\cap\Gamma_j'|}{m}\lambda_{M,4}^{2m}.\label{eqn:midValInter}
    \end{align}
    Recalling the definition of $\s{B}$ in \eqref{eqn:sBdef}, we get from \eqref{eqn:LargeValInter} and \eqref{eqn:midValInter},
    \begin{align}
        \s{B}&\triangleq\Ind\ppp{|\Gamma_i\cap\Gamma_j'|\geq\s{D}}\sum_{m=0}^{M^2\s{D}}\binom{M^2|\Gamma_i\cap\Gamma_j'|}{m}\lambda^{2m}\\
        &\leq  \Ind\ppp{\s{D}\leq|\Gamma_i\cap\Gamma_j'|\leq M^2\s{D}}\sum_{m=0}^{\s{D}}\binom{|\Gamma_i\cap\Gamma_j'|}{m}\lambda_{M,4}^{2m}\nonumber\\
        &\quad+\Ind\ppp{|\Gamma_i\cap\Gamma_j'|\geq M^2\s{D}}\sum_{m=0}^{M^2\s{D}}\binom{|\Gamma_i\cap\Gamma_j'|}{m}\lambda_{M,2}^{2m}.
    \end{align}
    Let $\bar\lambda_M^2\triangleq\max(\lambda_M^2,\lambda_{M,2}^2,\lambda_{M,4}^2)$. Then, from \eqref{eqn:DecompositionAandB}, we have,
    \begin{align}
       &\bE\pp{\sum_{m=0}^{M^2\min(\s{D},|\Gamma_i\cap\Gamma_j'|)}\binom{M^2|\Gamma_i\cap\Gamma_j'|}{m}\lambda^{2m}} = \bE\pp{\s{A}+\s{B}}\\
       &\hspace{2.5cm} \leq \bE\pp{\sum_{m=0}^{|\Gamma_i\cap\Gamma_j'|}\binom{|\Gamma_i\cap\Gamma_j'|}{m}\lambda_{M}^{2m}\Ind\ppp{|\Gamma_i\cap\Gamma_j'|\leq\s{D}}}\nonumber\\
       &\hspace{2.5cm}\quad+\bE\pp{\sum_{m=0}^{\s{D}}\binom{|\Gamma_i\cap\Gamma_j'|}{m}\lambda_{M,4}^{2m}\Ind\ppp{\s{D}\leq|\Gamma_i\cap\Gamma_j'|\leq M^2\s{D}}}\nonumber\\
       &\hspace{2.5cm}\quad+\bE\pp{\sum_{m=0}^{M^2\s{D}}\binom{|\Gamma_i\cap\Gamma_j'|}{m}\lambda_{M,2}^{2m}\Ind\ppp{|\Gamma_i\cap\Gamma_j'|\geq M^2\s{D}}}\\
       &\hspace{2.5cm}\leq \bE\pp{\sum_{m=0}^{|\Gamma_i\cap\Gamma_j'|}\binom{|\Gamma_i\cap\Gamma_j'|}{m}\bar\lambda_M^{2m}\Ind\ppp{|\Gamma_i\cap\Gamma_j'|\leq M^2\s{D}}}\nonumber\\
       &\hspace{2.5cm}\quad+\bE\pp{\sum_{m=0}^{|\Gamma_i\cap\Gamma_j'|}\binom{|\Gamma_i\cap\Gamma_j'|}{m}\bar\lambda_M^{2m}\Ind\ppp{\s{D}\leq|\Gamma_i\cap\Gamma_j'|\leq M^2\s{D}}}\nonumber\\
       &\hspace{2.5cm}\quad+\bE\pp{\sum_{m=0}^{M^2\s{D}}\binom{|\Gamma_i\cap\Gamma_j'|}{m}\bar\lambda_M^{2m}\Ind\ppp{|\Gamma_i\cap\Gamma_j'|\geq M^2\s{D}}}\\
       &\hspace{2.5cm}\leq 2\cdot\bE\pp{\sum_{m=0}^{\min(M^2\s{D},|\Gamma_i\cap\Gamma_j'|)}\binom{|\Gamma_i\cap\Gamma_j'|}{m}\bar\lambda_M^{2m}}.
    \end{align}
    Substituting the above in \eqref{eqn:almostTheEnd}, we finally obtain that,
    \begin{align}
        \norm{\s{L}_{n,\leq \s{D}}}_{\calH_0}^2\leq 2\prod_{i,j\in[M]}\p{\bE\pp{\sum_{m=0}^{\min(M^2\s{D},|\Gamma_i\cap\Gamma_j'|)}\binom{|\Gamma_i\cap\Gamma_j'|}{m}\bar\lambda_M^{2m}}}^{\frac{1}{M^2}}.\label{eqn:upperBoundL2Dnorm}
    \end{align}
    Note that $\bar\lambda_M^2$ can be simplified as follows. We have,
    \begin{align}
        \lambda_{M,2}^2=eM^2[M^2D]^{\frac{1}{cM^2\s{D}}}\lambda^2\leq 2eM^2\lambda^2,
    \end{align}
    where we have used the fact that $[M^2D]^{\frac{1}{cM^2\s{D}}}\leq 2$. Similarly,
    \begin{align}
        \lambda_{M,4}^2=\frac{\s{D}^{\frac{1}{c\s{D}}}\p{\frac{M^2}{c}}^{\frac{M^2}{c\s{D}}}}{c}(M^2\s{D})^{\frac{1}{c\s{D}}}(e\lambda^2)^{M^2}\leq C(e\lambda^2)^{M^2},
    \end{align}
    for some universal $C>0$ independent of $M$ and $\s{D}$. Thus,
    \begin{align}
        \bar\lambda_M^2\leq\max\ppp{1,2eM^2\lambda^2,\lambda_M^2,C(e\lambda^2)^{M^2}}.
    \end{align}

   Let us now consider the case where $\lambda^2=o(\s{D^{-2}})$. The analysis of this case remains almost the same, with some minor adjustments. First, we observe that $m^{\star}=cM^2\s{D}$, which maximizes the l.h.s. of \eqref{eq:maximizingShit}, can be as small as $1$. Therefore, the constant $c$ in \eqref{eqn:LowerBinomwith} can vary in the interval $[(M^2\s{D})^{-1},1]$. In particular, $\bar{\lambda}_{M,2}^2$ is bounded by, 
   \begin{align}
      \lambda^2_{M,2}\triangleq \pp{M^2\s{D}M^{2cM^2\s{D}}(e\lambda^2)^{cM^2\s{D}}}^{\frac{1}{cM^2D}} = eM^2[M^2D]^{\frac{1}{cM^2\s{D}}}\lambda^2\leq eM^4\s{D}\lambda^2.
   \end{align}
   Using similar arguments, our new bounds on $\lambda_{M,3}^2$ and $\lambda_{M,4}^2$ are bounded as,
   \begin{align}
       \lambda_{M,3}^2\leq \p{e M^2\s{D}\lambda^2}^{M^2} \quad\s{and}\quad \lambda_{M,4}^2\leq \s{D}M^2\p{e M^2\s{D}^2\lambda^2}^{M^2}\p{M^4\s{D}}^{M^4}.
   \end{align}
   Finally, we note that the inequality in \eqref{eq:StillmAzhik} still holds true. Indeed, under the assumption that $\lambda^2=o(D^{-2})$, we have that $m^{\star}$, which maximizes \eqref{eqn:combComp1}, must be $1$. Since it is assumed $\s{D}\leq |\Gamma_i\cap\Gamma_j| \leq M^2\s{D}$, we have, for sufficiently large $n$,
   \begin{align}
        m^\star = \min\p{\ceil{\frac{\lambda^2 (M^2|\Gamma_i\cap\Gamma_j'|+1)}{\lambda^2+1}},M^2\s{D}}=\ceil{\frac{\lambda^2 (M^2\cdot M^2\s{D}+1)}{\lambda^2+1}}=1.
   \end{align}
   In particular, $c$ in that case must be $(M^2\s{D})^{-1}$, which is also bounded away from $e^{-1}$, that value that maximizes $f_{\hat{m}}$. Therefore, $\eqref{eq:StillmAzhik}$ holds.

    We are now in a position to prove \eqref{eq:condintersectionDComp}. Let $C>0$ be the smallest constant for which Theorem~\ref{th:lowDegreeConbinatorial} and Lemma~\ref{lem:interectionDComp} hold simultaneously. Also, define $\nu(M,\s{D})\triangleq\p{\bar\lambda_M^2 M^4\s{D}^2}^{\sqrt{M^2\s{D}/2}}$. For $i=j$, using \eqref{eq:CompGeneralDense} in Theorem~\ref{th:lowDegreeConbinatorial}, we know that,
    \begin{align}
        \bE\pp{\sum_{m=0}^{\min(M^2\s{D},|\Gamma_i\cap\Gamma_i'|)}\binom{|\Gamma_i\cap\Gamma_i'|}{m}\bar\lambda_M^{2m}}=O(1),
    \end{align}
    provided that,
    \begin{align}
        \frac{\nu(M,\s{D})\cdot\max\p{\vc_id_i,s_i}}{n-k_i}\leq C.
    \end{align}
    This, however, is satisfied when \eqref{eq:condintersectionDComp} holds because, 
    \begin{align}
         \frac{\nu(M,\s{D})\cdot\max\p{\vc_id_i,d^2_i}}{n-k_i}\leq \max_{1\leq j\leq M} \frac{\nu(M,\s{D})\cdot\max\p{\vc_jd_j,d^2_j}}{n-k_j}\leq C.
    \end{align}
    Next, for $i\neq j$, we note that,
    \begin{align}
        \frac{ \nu(M,\s{D})\max\p{\sqrt{\vc_id_i\vc_jd_j},d_id_j}}{n-\min\p{k_i,k_j}}&\leq  \frac{\nu(M,\s{D})\max(\sqrt{\vc_i \vc_j d_i d_j},d_i d_j)}{n-\min(k_i,k_j)}\\
        &\leq  \frac{\nu(M,\s{D})\max(\vc_i d_i,\vc_j d_j, d_i^2, d_j^2)}{n-\min(k_i,k_j)}\\
        &\leq \max_{\ell=i,j} \frac{\nu(M,\s{D})\max(\vc_\ell d_\ell ,d_\ell^2)}{n-k_\ell}\\
        &\leq \max_{1\leq \ell\leq M} \frac{\nu(M,\s{D})\max(\vc_\ell d_\ell,d_\ell^2)}{n-k_\ell}\leq C,
    \end{align}
    where the last equality is due to \eqref{eq:condintersectionDComp}, and therefore, by Lemma~\ref{lem:interectionDComp},
    \begin{align}
        \bE\pp{\sum_{m=0}^{\min(M^2\s{D},|\Gamma_i\cap\Gamma_j|)}\binom{|\Gamma_i\cap\Gamma_j|}{m}\bar\lambda^{2m}}=O(1).
    \end{align}
    Thus, we proved that under \eqref{eq:condintersectionDComp} each term in the product on the r.h.s. of \eqref{eqn:upperBoundL2Dnorm} is bounded, and since $M$ is finite, the entire product of \eqref{eqn:upperBoundL2Dnorm} is bounded as well. The proof of \eqref{eq:condintersectionDComp2} follows from the exact same arguments as above by using \eqref{eq:CompGeneralSparse} in Theorem~\ref{th:lowDegreeConbinatorial} Lemma~\ref{lem:interectionDComp}, and \eqref{eq:condIntersectionD2Comp2}.

\end{proof}

\subsection{Implications}\label{subsec:OtherLDP} The computational lower bounds in Theorem~\ref{th:lowDegreeConbinatorial} provide general conditions under which $\norm{\s{L}_{n, \leq\s{D}}}_{\calH_0} = O(1)$. While these conditions are formulated in terms of the vertex cover number, our upper bounds (given by the performances of the count test and the maximum degree test) are formulated in terms of $|e(\Gamma)|$ and $d_{\max}(\Gamma)$ solely. In this subsection, we bridge this gap for the dense and sparse regimes. %We state and prove our main result, where we determine the exact thresholds for efficient detection in terms of the number edges and maximum degree. Similarly to the proof of our lower bounds, the proof relies on a decomposition argument, and a reduction argument (to the $\vcd|$-balanced case), building on the results of Section~\ref{subsec:coverBalancedDecomposition} and Section~\ref{subsec:ReducLDP}. 
\begin{theorem}\label{thm:LDPFinal}
Let $\s{D}=(\s{D}_n)_n$ be a sequence such that $\s{D} = O(\log n)$, and let $\Gamma=(\Gamma_n)_n$ be a sequence of graphs. 
\begin{enumerate}
    \item Assuming that $\lambda^2=\chi^2(p||q)= \Theta(1)$,  if
    \begin{align}
        \max\p{|e(\Gamma)|,d_{\max}^2(\Gamma)}\leq n^{1-\varepsilon}, \label{eqn:MainResDenseLower2Comp}
    \end{align}
    for some $\varepsilon>0$, we have $\norm{\s{L}_{n, \leq\s{D}}}_{\calH_0} = O(1)$.
    \item Assume that $\lambda^2=\Theta(n^{-\alpha})$, where $0<\alpha<2$, and that $\Gamma$ is an $(\epsilon,\delta,\smu)$-polynomial family with $|v(\Gamma)|=\Theta(n^\beta)$, for $0<\beta<1$. Then, $\norm{\s{L}_{n, \leq\s{D}}}_{\calH_0} = O(1)$ if
    \begin{align}
        \beta <\min\p{\frac{2+\alpha}{2\varepsilon},\frac{1+\alpha}{2\delta}}.\label{eqn:MainResSparseLower2Comp}
    \end{align}
\end{enumerate} 
\end{theorem}

\begin{proof}[Proof of Theorem~\ref{thm:LDPFinal}]
    From Proposition~\ref{prop:GoodDecompositionExists} we know that there exists a decomposition $\Gamma=\bigcup_{\ell=1}^M\Gamma_\ell$, for which \eqref{eq:decompositionPropoety} holds, and we take,
    \begin{align}
        M_n=\ceil{\frac{1}{\varepsilon}}.
    \end{align} By Proposition~\ref{prop:decompisitionThDCcomp} strong detection is impossible if,
    \begin{align}
         \max_{\ell=1,\dots, M}\p{\frac{\nu(M,\s{D})}{n-|v(\Gamma_\ell)|}\cdot\max\p{\vc(\Gamma_\ell)\cdot d_{\max}(\Gamma_\ell),d_{\max}(\Gamma_\ell)^2}}\leq C,\label{eq:CondFinalDComp}
    \end{align}
    where $\nu(M,\s{D})\triangleq\p{\max(\bar\lambda_M^2,1)\cdot 2 M^4\s{D}^2}^{\sqrt{M^2\s{D}/2}}$. Now, by \eqref{eq:decompositionPropoety}, we have,
     \begin{align}
          \max_{\ell=1,\dots, M}&\p{\frac{\nu(M,\s{D})}{n-|v(\Gamma_\ell)|}\cdot\max\p{\vc(\Gamma_\ell)\cdot d_{\max}(\Gamma_\ell),d_{\max}(\Gamma_\ell)^2}}\\
          &\qquad\leq   \frac{\nu(M,\s{D})}{n-|v(\Gamma)|}\cdot\max\p{|e(\Gamma)|\cdot d_{\max}(\Gamma)^{\frac{1}{M_n}},d_{\max}(\Gamma)^2}\\
          &\qquad\leq  \frac{\nu(M,\s{D})}{\delta n}\cdot\max\p{|e(\Gamma)|\cdot d_{\max}(\Gamma)^{\frac{1}{M_n}},d_{\max}(\Gamma)^2}\\
            &\qquad\leq \frac{\nu(M,\s{D})}{\delta n^{1-\frac{1}{M_n}}}\cdot\max\p{|e(\Gamma)|,d_{\max}(\Gamma)^2}
            \label{eq:DenseNiceEqComp}
     \end{align}
     where the first equality follows from the fact that $|v(\Gamma_\ell)|\leq |v(\Gamma)|$, for all $\ell\in[M]$, the second equality is because $|v(\Gamma)|\leq (1-\delta)n$, for some fixed $\delta>0$, the third equality follows from $d_{\max}(\Gamma)\leq |v(\Gamma)|$, and the last equality is because $|v(\Gamma)|\leq n$. Since for a fixed $M$ and $\s{D} = O(\log n)$, we have $\nu(M,\s{D}) = n^{o(1)}$, we finally get that if, 
     \begin{align}
         \max\p{|e(\Gamma)|,d_{\max}(\Gamma)^2}\leq  n^{1-\varepsilon},
     \end{align}
     then \eqref{eq:CondFinalDComp} is satisfied for a sufficiently large $n$ (regardless of the constant $C$), which proves \eqref{eqn:MainResDenseLower2Comp}.

     Next, we follow the lines of the proof of Theorem~\ref{th:polynomial} for the proving the condition \eqref{eqn:MainResSparseLower2Comp} in the sparse regime. Let 
      \begin{align}
        2\rho \triangleq \beta- \min\p{, \frac{1+\alpha}{2\delta},  \frac{2+\alpha}{2\epsilon}} >0,\label{eq:AssumptionSparseComp}
    \end{align}
    and let $\Gamma$ be decomposed as 
    $\Gamma=\Gamma_1\cup\cdots\cup\Gamma_M$, where $M=\ceil{\rho^{-1}}$, such that for all $1\leq \ell\leq M$, we have 
    \begin{align}
        \vc(\Gamma_\ell)d_{\max}(\Gamma_\ell)&\leq |e(\Gamma)|d_{\max}(\Gamma)^{\frac{1}{M}}\leq n^{\beta(\epsilon +\rho)+o(1)},\label{eq:decomposintiosparseComp}
    \end{align}
    where such a decomposition exists due to Proposition~\ref{prop:GoodDecompositionExists}. By Proposition~\ref{prop:decompisitionThDCcomp}, $\norm{\s{L}_{n,\leq \s{D}}}_{\calH_0}$ is bounded if 
    \begin{align}
         \max_{\ell=1,\dots, M}\frac{\bar\lambda_M\s{D}\cdot\max\p{\vc(\Gamma_\ell)d_{\max}(\Gamma_\ell),d^2_{\max}(\Gamma_\ell)\bar\lambda_M\s{D}}}{n-|v(\Gamma_\ell)|}\leq C, \label{eq:condinproof}
    \end{align}
    where \begin{align}
        \bar\lambda_M^2\triangleq \max\p{(1+\lambda^2)^{M^2}-1,eM^4\s{D}\lambda^2,(eM^2\s{D}\lambda^2)^{M^2},\s{D}M^2\p{e M^2\s{D}^2\lambda^2}^{M^2}\p{M^4\s{D}}^{M^4}}.\label{eq:HowtochooseRHO}
    \end{align} Note that under the assumption that $\lambda^2=\Theta(n^{-\alpha})$, $M$ is finite, and $\s{D}=O(\log(n))$, we also have that $\bar{\lambda}^2_M=n^{-\alpha+o(1)}$. Thus, from the definition of an $(\epsilon,\delta,\smu)$-polynomial sequence, and since $|v(\Gamma)|=\Theta(n^{\beta})$ and $\s{D}=O(\log(n))=n^{o(1)}$ for all $\ell$, 
    \begin{align}
        \frac{\bar\lambda_M\s{D}\cdot\max\p{\vc(\Gamma_\ell)d_{\max}(\Gamma_\ell),d^2_{\max}(\Gamma_\ell)\bar\lambda_M\s{D}}}{n-|v(\Gamma_\ell)|}&\leq \frac{n^{-\frac{\alpha}{2}+o(1)}\max\p{n^{\epsilon\beta+\rho},n^{2\delta\beta-\frac{\alpha}{2}+o(1)}}}{n^{1-o(1)}}\\
        &=n^{\max\p{\beta\epsilon +\rho -1-\frac{\alpha}{2}, 2\beta\delta -1 -\alpha}}+o(1).\label{eq:oneBeforeEND}
    \end{align}
    Note that by \eqref{eq:HowtochooseRHO} we have 
    \begin{align}
        \max\p{\beta\epsilon +\rho -1-\frac{\alpha}{2}, 2\beta\delta -1 -\alpha}<n^{-\rho},
    \end{align}
    and theretofore \eqref{eq:oneBeforeEND} decays to zero. In particular, the condition in \eqref{eq:condinproof} holds.
\end{proof}

We conclude this section by pointing out that the bounds in Theorem~\ref{thm:LDPFinal} are tight. Indeed, by Theorem~\ref{thm:upperBoundAlgo}, in the dense regime, using the count and maximum degree tests (which runs in polynomial time), strong detection is possible whenever,
\begin{align}
    \max\p{|e(\Gamma)|,d_{\max}(\Gamma)^2}\geq n^{1+o(1)}.
\end{align}
In the sparse regime, under the setting of Theorem~\ref{thm:LDPFinal}, detection using the count and the maximum degree tests is possible when,
\begin{align}
    \beta>\min\p{\frac{2+\alpha}{2\varepsilon},\frac{1+\alpha}{2\delta}}.
\end{align}
These complement our computational lower bounds. Furthermore, the statistical and computational bounds in Theorem~\ref{th:StatCompLimitSparse} imply that a statistical-computational gap exist if and only if $\Gamma$ has a positive polynomial density, i.e., $\smu>0$. Indeed, we observe that as long as $\alpha>0$, the line $\frac{\mu}{\alpha}$ is strictly below the lines $\frac{2+\alpha}{2\epsilon}$ and $\frac{1+\alpha}{2\delta}$ in the region \begin{align}
    \alpha\in \p{0, \min\p{\frac{2\smu}{2\epsilon-\smu},\frac{\smu}{2\delta-\mu}}}\triangleq I_{\epsilon,\delta,\smu}.
\end{align} In particular, for any $\alpha \in I_{\epsilon,\delta,\smu}$, and $\frac{\alpha}{\smu}<\beta<\min(\frac{1+\alpha}{2\delta},  \frac{2+\alpha}{2\epsilon})$, detection is possible (using the scan test), and by Theorem~\ref{thm:LDPFinal}, detection in polynomial time is impossible.

%\begin{cor}
%The computational lower bound in Theorem~\ref{thm:LDPFinal} is tight, and fully characterized the hard regime in the phase diagram. Indeed we note that the count test and maximum degree test (which are efficient), achieves detection provided that 
%\begin{align}
%    \beta>\min\p{ \frac{1+\alpha}{2\delta},  \frac{2+\alpha}{2\epsilon}}.
%\end{align}

%In addition, the statistical and computational bounds of Theorem~\ref{th:StatCompLimitSparse} proves that a statistical-computational gap is evident if and only if $\Gamma$ has positive polynomial density, i.e., $\smu>0$. To see that, we observe that as long as $\alpha>0$ the line $\frac{\mu}{\alpha}$ is strictly below the lines $\frac{2+\alpha}{2\epsilon}$ and $\frac{1+\alpha}{2\delta}$ in the region \begin{align}
%    \alpha\in \p{0, \min\p{\frac{2\smu}{2\epsilon-\smu},\frac{\smu}{2\delta-\mu}}}\triangleq I_{\epsilon,\delta,\smu}.
%\end{align} In particular, for any $\alpha \in I_{\epsilon,\delta,\smu}$, and $\frac{\alpha}{\smu}<\beta<\min(\frac{1+\alpha}{2\delta},  \frac{2+\alpha}{2\epsilon})$, detection is possible (using the scan test), and by Theorem~\ref{thm:LDPFinal}, detection in polynomial time is impossible.
    
%\end{cor}

%\section{Conclusion}\label{sec:conc}

\bibliographystyle{alpha}
\bibliography{bibfile}
\appendix

\section{Additional Proofs}\label{app:1}

\subsection{Essential number of planted edges}\label{app:simpleLowerBound}

Let us show that strong detection is information-theoretically impossible if $|p-q|\cdot|e(\Gamma)|<c<1$. We already saw that the optimal risk satisfies (see, \eqref{eqn:lowerBoundSecond}),
\begin{align}
    \mathsf{R}^{\star}&=1-d_{\s{TV}}(\P_{\calH_0},\P_{\calH_1}).
\end{align}
Recall the following two facts about the total-variation distance (see, e.g., \cite[Theorem 7.5]{Polyanskiy_Wu_2025} and \cite[Lemma 7]{ma2015computational}).
\begin{enumerate}
    \item \emph{Convexity}: $(P,Q)\mapsto d_{\s{TV}}(P,Q)$ is a jointly convex function.
    \item \emph{Tensorization}: Let $\{P_i\}_{i=1}^n$ and $\{Q_i\}_{i=1}^n$ be distributions on a measurable space. Then,
    \begin{align}
        d_{\s{TV}}\p{\prod_{i=1}^nP_i,\prod_{i=1}^nQ_i}\leq\sum_{i=1}^nd_{\s{TV}}(P_i,Q_i).
    \end{align}
\end{enumerate}
Accordingly, using the above two properties we readily obtain that,
\begin{align}
    d_{\s{TV}}(\P_{\calH_0},\P_{\calH_1}) &= d_{\s{TV}}(\P_{\calH_0},\bE_{\Gamma}\P_{\calH_1\vert\Gamma})\\
    &\leq \bE_{\Gamma}d_{\s{TV}}(\P_{\calH_0},\P_{\calH_1\vert\Gamma})\\
    & = \bE_{\Gamma}d_{\s{TV}}\p{\prod_{(i,j)\in\binom{[n]}{2}}\s{Bern}(q),\prod_{(i,j)\in\binom{[n]}{2}\setminus\Gamma}\s{Bern}(q)\times\prod_{(i,j)\in\Gamma}\s{Bern}(p)}\\
    &\leq |p-q|\cdot|e(\Gamma)|,
\end{align}
and thus, $\mathsf{R}^{\star}\geq1-|p-q|\cdot|e(\Gamma)|$, which concludes the proof.    

\subsection{Graphs with sub-logarithmic degrees: alternative probabilistic proof}\label{app:probablisticProofs}

Our main results were proved using a combination of combinatorial and probabilistic approaches. Many of the existing results in the literature, for specific structures, rely on probabilistic arguments only. In this appendix, we present a pure probabilistic proof, which give tight lower bounds for the family of sug-logarithmic degree graphs (i.e., $d_{\max}(\Gamma)=o(\log(n))$), which might be of independent interest. We prove the following result.
\begin{lemma}\label{lem:intersectionBoundedDegree}
    Let $\Gamma_1$ and $\Gamma_2$ be two graphs drawn uniformly and independently at random from $\calS_{\Gamma_1}$ and $\calS_{\Gamma_2}$, respectively. Assume further that $d_{\max}(\Gamma_1)=O(\log n)$ and $|e(\Gamma_2)|=O(n^{1-\varepsilon})$, for some $\varepsilon>0$. Then, for any constant $c>0$,
    \begin{align}
        \bE_{\Gamma_1\indep\Gamma_2}\pp{(1+\lambda^2)^{c\cdot |e(\Gamma_1\cap\Gamma_2)|}} = 1+o(1).
    \end{align}
\end{lemma}

\begin{proof}[Proof of Lemma~\ref{lem:intersectionBoundedDegree}]
    Throughout this proof, let $k_i$ denote the number of vertices in $\Gamma_i$ and $d_i$ denote the maximum degree of $\Gamma_i$, for $i=1,2$. Let us also denote the fixed deterministic versions of $\Gamma_1$ and $\Gamma_2$ by $\Gamma_1^{\s{F}}$  and $\Gamma_2^{\s{F}}$  respectively.
      By symmetry we may fix $\Gamma_1$ to be a deterministic copy of $\Gamma_1^{\s{F}}$ in $\calK_n$ with vertex set $\set{1,2,\ldots,k_1}$. Let $\Gamma_2$ be a  uniform random copy of $\Gamma_2^{\s{F}}$ generated by the random variables $X_1,\dots,X_{k_2}$ distributed on $[n]$ as described in Observation~\ref{obs:RandomCopy1}, with respect to an enumeration 
      \begin{align}
    v(\Gamma_2^{\s{F}})=\set{v_1,\dots,v_{k_2}}.
      \end{align} For any $i\in [k]$, let $\deg_{\Gamma_1\cap\Gamma_2}^{\s{p}}(i)$ be the past degree of $v_i$ (with respect to the ordering of $X_1,\dots,X_{k_2}$) in $\Gamma_1\cap \Gamma_2$. That is,
      \begin{equation}
      \deg_{\Gamma_1\cap\Gamma_2}^{\s{p}}(i)\triangleq \abs{\ppp{j\leq i-1:\set{v_i,v_j}\in e(\Gamma_2^{\s{F}} )\text{ and }\set{X_i,X_j}\in e(\Gamma_1)}},
      \end{equation}
      for $2\leq i\leq k_2$, and we define $\deg_{\Gamma_1\cap\Gamma_2}^{\s{p}}(1)\triangleq0$. Observe that,
      \begin{align}
          e(\Gamma_1\cap\Gamma_2)=\sum_{i=1}^k \deg_{\Gamma_1\cap \Gamma_2}^{\s{p}}(i).\label{eq:edgesPastDegrees}
      \end{align}
      Furthermore, for any $i\in[k_2]$ we have, 
      \begin{align}
          \deg_{\Gamma_1\cap\Gamma_2}^{\s{p}}(i)\leq \min\p{\deg_{\Gamma_2}^{\s{p}}(i),d_1} \cdot Z_i,
      \end{align}
      with probability one, where 
      \begin{align}
          Z_i&\triangleq\Ind_{\set{\deg_{\Gamma_1\cap\Gamma_2}^{\s{p}}(i)>0}},\\
          \deg_{\Gamma_2}^{\s{p}}(i)&\triangleq\abs{\ppp{j\leq i-1: \set{v_i,v_j} \in e(\Gamma_2^{\s{F}})}},%\cdot \Ind_{[0,\infty)}(i),
      \end{align}
    for $2\leq i\leq k_2$, and we define $\deg_{\Gamma_2}^{\s{p}}(1)\triangleq0$.
    Thus, by \eqref{eq:edgesPastDegrees}, we have 
    \begin{align}
        e(\Gamma_1\cap\Gamma_2)\leq \sum_{i=1}^k a_i\cdot Z_i,\label{eq:edgesPastDegreesSolo}
    \end{align}
    where $a_i\triangleq \min\p{\deg_{\Gamma_2}^{\s{p}}(i),d_1}$.
    Let $B_1,\dots,B_{k_2}$ be independent random variables, independent of $\Gamma_1$ and $\Gamma_2$, where  $B_i\sim \s{Ber}(p_i)$, with
    \begin{align}
        p_i\triangleq \min\p{1, \frac{\deg^{\s{p}}_{\Gamma_2^{\s{F}}}(i)\cdot d_1}{n-k}},
    \end{align}
    for $i\in[k_2]$. In what follows, we prove that 
    \begin{align}
        \sum_{i=1}^ka_i\cdot Z_i \preceq \sum_{i=1}^k a_i\cdot B_i.\label{eq:stochasitcDominanceDegrees}
    \end{align}
    We will prove the above relation inductively.  That is, we will prove that for all $1\leq \ell\leq k_2$, 
    \begin{equation}
       \overset{\text{denote by }S_\ell }{\overbrace{ \sum_{i=1}^\ell a_i\cdot Z_i}} \preceq \overset{\text{denote by }Y_\ell}{\overbrace{\sum_{i=1}^\ell a_i\cdot B_i}}.
    \end{equation}
    Indeed, by definition, we have
    $Z_1=B_1=0$ with probability one, and thus,
    \begin{equation}
        a_1\cdot Z_1 \preceq  a_1\cdot B_1.
    \end{equation}
    Assume by induction that the statement is true for $ 1\leq \ell< k_2$.  We recall the following result.
    \begin{lemma}\label{lem:S.D.Sums}\cite[Theorem 5.4]{mosler1991some} Let $P_{Z,X}$ and $P_{Y,X}$ be two joint distributions on $(Z,X)$ and $(Z,Y)$ respectively such that for all $x$, $P_{Y|X=x}\preceq P_{Z|X=x}$, then,
\[ Z+X\preceq Y+X.\]
\end{lemma}
Using Lemma~\ref{lem:S.D.Sums}, we show that it is sufficient to show that  any value of $S_\ell=s_\ell$, the conditional distribution of $a_{\ell+1}\cdot Z_{\ell+1}$ given the event $\set{S_\ell=s_\ell}$ is stochastically dominated by the conditional distribution of $a_{\ell+1}\cdot B_{\ell+1}$ given $\set{S_\ell=s_\ell}$. Having that proved, we obtain,
    \begin{align}
        S_{\ell+1}&=S_{\ell}+ a_{\ell+1}\cdot Z_{\ell+1}\label{eq:stochasticChain1}\\& \overset{(a)}{\preceq} S_{\ell}+ a_{\ell+1}\cdot B_{\ell+1}\\
        &\overset{(b)}{\preceq} Y_{\ell}+ a_{\ell+1}\cdot B_{\ell+1}=Y_{\ell+1},\label{eq:stochasticChain2}
    \end{align}
    where $(a)$ follows from Lemma~\ref{lem:S.D.Sums}, and $(b)$ follows  from Lemma~\ref{lem:S.D.Sums}, the induction hypothesis and from the independence of $Y_\ell$ and $S_\ell$ with $B_{\ell+1}$. Since $\preceq$ is a partial order over probability measures, the derivation of \eqref{eq:stochasticChain1}-\eqref{eq:stochasticChain2} implies that $S_{\ell+1}\preceq Y_{\ell+1}$ as desired.

    Since the partial order of $\preceq$ is invariant under multiplication with a non-negative constants, in order to complete the proof, it is sufficient to show that,
    \begin{align}
        P_{Z_{\ell+1}|S_\ell=s_\ell}\preceq P_{B_{\ell+1}|S_\ell=s_\ell}. 
    \end{align}
    Since $Z_{\ell+1}$ and $B_{\ell}$ are Bernoulli distributed random variables, by Lemma~\ref{lem:S.D.Sums}, it is sufficient to show that,
    \begin{align}
        \P[Z_{\ell+1}=1\vert S_\ell=s_\ell]\leq \P[B_{\ell+1}=1\vert S_\ell=s_\ell]=p_{\ell+1}.
    \end{align}
    Note that the event $\set{S_\ell=s_\ell}$ is in the $\sigma$-algebra generated by $X_1,\dots,X_\ell$. Let us denote $(X_1,\dots,X_\ell)$ as $\overline{X}_\ell$ and define $A_{\ell}\in[n]^\ell$ to be the set of all assignments $\overline{X}_\ell=\overline{x}_\ell$ such that the event $\set{S_\ell=s_\ell}$ is decomposed as,
    \begin{align}
    \set{S_\ell=s_\ell}=\bigcup_{\overline{x}_{\ell}\in A_{\ell}}\set{\overline{X}_\ell = \overline{x}_\ell}.
    \end{align}
    We now have,
    \begin{align}
        \P[Z_{\ell+1}=1~|~S_\ell=s_\ell]&
        =\frac{\P[Z_{\ell+1}=1 \cap S_\ell=s_\ell]}{\P[ S_\ell=s_\ell]}\\
        & =\sum_{\overline{x}_\ell\in A_\ell}\frac{\P[Z_{\ell+1}=1 ~|~ \overline{X}_\ell=\overline{x}_\ell]\cdot \P[\overline{X}_\ell=\overline{x}_\ell]}{\P[ S_\ell=s_\ell]}\\
        &\overset{(a)}{\leq }\sum_{\overline{x}_\ell\in A_\ell}\frac{\P[\overline{X}_\ell=\overline{x}_\ell]}{\P[ S_\ell=s_\ell]}\cdot \frac{1}{n-\ell}\cdot\sum_{\substack{j\leq\ell\\ \set{v_{\ell+1},v_j}\in e(\Gamma_2^{\s{F}})\\ x_j\in [k_1]}}\deg_{\Gamma_1}(x_j) \\
        &\leq \sum_{\overline{x}_\ell\in A_\ell}\frac{\P[\overline{X}_\ell=\overline{x}_\ell]}{\P[ S_\ell=s_\ell]}\cdot \frac{1}{n-k_{2}}\cdot  \sum_{\substack{j\leq\ell\\ \set{\ell+1,j}\in e(\Gamma_2^{\s{F}})}}d_1 \\
        &= \sum_{\overline{x}_\ell\in A_\ell}\frac{\P[\overline{X}_\ell=\overline{x}_\ell]}{\P[ S_\ell=s_\ell]}\cdot  \frac{d_1\cdot \deg^{\s{p}}_{\Gamma_2^{\s{F}}}(\ell+1)}{n-k_{2}}\\
        &=p_{\ell+1}
    \end{align}
    where $(a)$ follows since conditioned on $\set{\overline{X}_\ell=\overline{x}_\ell}$, $Z_{\ell+1}=1$ only if $X_{\ell+1}$ is in the $\Gamma_0$-neighborhood of $x_{j}$ for some $j\leq\ell$ such that $x_j\in [k_1]$. Namely, there are at most 
    \begin{align}
        \sum_{j\leq\ell,\; \set{\ell+1,j}\in e(\Gamma_2^{\s{F}}),\; x_j\in [k_1]}\deg_{\Gamma_1}(x_j)
    \end{align} locations at which $X_{\ell+1}$ can be, out of $n-\ell$ possibilities in total. This concludes the proof of the induction step. 

    Since $\lambda^2,c\geq 0$, the function $f(x)=(1+\lambda^2)^{c\cdot x}$ is monotonically increasing w.r.t. $x\geq0$. Thus, combining \eqref{eq:SecondMomentExpression}, \eqref{eq:edgesPastDegrees}, and \eqref{eq:stochasitcDominanceDegrees}, we get,
    \begin{align}
    \bE_{\Gamma_1\indep\Gamma_2}\pp{(1+\lambda^2)^{e(\Gamma\cap\Gamma')}}&=\bE_{\Gamma_2}\pp{(1+\lambda^2)^{c\cdot\sum_{i=1}^{k_2} \deg_{\Gamma_1\cap \Gamma_2}^{\s{p}}(i)}}\\
        &\leq \bE_{\Gamma}\pp{(1+\lambda^2)^{c\cdot\sum_{i=1}^{k_2}a_i\cdot Z_i}}\\
        &\leq \bE\pp{(1+\lambda^2)^{c\cdot\sum_{i=1}^{k_2} a_i\cdot B_i}}\\
        &=\prod_{i=1}^{k_2} \bE\pp{(1+\lambda^2)^{c\cdot a_i\cdot B_i}}\\
        &=\prod_{i=1}^{k_2} \p{1+p_i\cdot \p{(1+\lambda^2)^{c\cdot a_i}-1}}\\
        &\leq \prod_{i=1}^{k_2} \p{1+\frac{\deg^{\s{p}}_{\Gamma_2^{\s{F}}}(i)\cdot d_1}{n-k_{2}}\cdot \p{(1+\lambda^2)^{c\cdot d_1}}}\\
        &\leq\exp \p{\frac{ d_1}{n-k_{2}}\cdot \p{\sum_{i=1}^{k_2}\deg^{\s{p}}_{\Gamma_2^{\s{F}}}(i)}\cdot \p{(1+\lambda^2)^{c\cdot d_1}}}\\
        &=\exp \p{\frac{ d_1}{n-k_{2}}\cdot |e(\Gamma_2)|\cdot \p{(1+\lambda^2)^{c\cdot d_1}}}.\label{eq:probabilisticIntersectionFinal}
    \end{align}
    By the assumption that $d_1=o(\log n)$ and $\lambda^2=\Theta(1)$ we have that, 
    \begin{align}
        d_1 \cdot (1+\lambda^2)^{c\cdot d_1}=n^{o(1)},
    \end{align}
    and \eqref{eq:probabilisticIntersectionFinal} converge to unity if $|e(\Gamma_2)|=O(n^{1-\varepsilon})$ for some arbitrarily small $\varepsilon>0$.

\end{proof}

    Note that an impossibility result for low degree graphs follows immediately from Lemma~\ref{lem:intersectionBoundedDegree}.
\begin{cor}\label{cor:LogarthmicDegreeLower}
    Let $\Gamma$ be such that $d_{\max}(\Gamma)=o(\log n)$, and $\chi^2(p||q)=\Theta(1)$. Then, strong and weak detection is impossible if 
    \begin{align}
        |e(\Gamma)| =O\p{n^{1-\varepsilon}},\label{eq:logarithmicDegThrahshold}
    \end{align}
    for an arbitrary small $\varepsilon>0$.
\end{cor}

\section{Useful Lemmata}\label{app:lemmata}
In this section, we record a few well known auxiliary results.
\begin{lemma}[Chernoff's inequality]
    For any $n\in\mathbb{N}$ and $p,q\in(0,1)$ such that $p\geq q$,
    \begin{align}
        \pr\pp{\s{Binomial}(n,q)\geq pn}\leq\exp\pp{-n\cdot d_{\s{KL}}(p||q)}.
    \end{align}
\end{lemma}
\begin{lemma}[Multiplicative Chernoff's inequality]
Suppose $\{S_i\}_{i=1}^n$ are independent random variables taking values in $\{0,1\}$, where $n\in\mathbb{N}$. Let $S_n\triangleq\sum_{i=1}^nS_i$, and denote $\eta\triangleq\bE[X]$. Then, for any $0<\delta<1$,
    \begin{align}
        \pr\pp{S_n\leq(1-\delta)\eta}\leq\exp\pp{-\frac{\delta^2\eta}{2}}.
    \end{align}
\end{lemma}
\begin{lemma}[Bernstein's inequality]
    For any $n\in\mathbb{N}$, $p\in(0,1)$, and $t>0$, 
    \begin{align}
        \pr\pp{\s{Binomial}(n,p)\geq np+t}\leq\exp\pp{-\frac{t^2}{2[np(1-p)+t/3]}}.
    \end{align}
\end{lemma}

\section{Failure of Local Subgraph Counts Bounds}\label{app:JansonType}
%\Dor{I think that the best way to introduce this failure is with the most simple example of a path where the Janson based analysis gives $k\ll n^{4/5} $ in the dense regime. We already have the details written in section 6.5.2 in the Temp file.}
%\wasim{Agree}

One idea to handle the expression in \eqref{eq:SecondMomentExpressionProb1} is by using Janson's type bounds. Indeed, note that by using Lemma~\ref{lem:equivRandomCopyOrSubgraphcopy}, we have,
\begin{align}
    \E_{\calH_0}\pp{\s{L}(\s{G})^2}&=\sum_{\s{H}'\subseteq \Gamma'}\lambda^{2|\s{H}'|}\cdot\P_{\Gamma}\pp{\s{H}'
             \subseteq \Gamma}\\
             & = \sum_{\s{H}'\subseteq \Gamma'}\lambda^{2|\s{H}'|}\cdot\frac{\calN(\s{H}',\Gamma')}{\abs{\calS_{\s{H}}}}.\label{eqn:2momentJans}
\end{align}
Now, let $\calN(n,m,\s{H})$ be the largest number of copies of $\s{H}$ that can be packed in $n$ vertices and $m$ edges. Clearly $\calN(\s{H},\Gamma)$ can be bounded by $\calN\p{|v(\Gamma)|,|e(\Gamma)|,\s{H}}$, which received significant attention in the literature. Indeed, many bounds on $\calN(v(\Gamma),e(\Gamma),\s{H})$ exist in the literature, e.g.,  \cite{Alon1981OnTN,Janson,Friedgut1998OnTN,Jaikumar,Wojciech}, and here we will focus the following recent result from \cite{Wojciech}. Specifically, for every graph $\s{H}$ without isolated vertices, and for all $e(\s{H})\leq e(\Gamma)$ and $v(\s{H})\leq v(\Gamma)$, we have \cite[Theorem 5.4]{Wojciech},
\begin{align}
    \calN(v(\Gamma),e(\Gamma),\s{H})\leq [2e(\Gamma)]^{v(\s{H})-\alpha^\star(\s{H})}\cdot\min\ppp{2e(\Gamma),v(\Gamma)}^{2\alpha^\star(\s{H})-v(\s{H})},\label{eqn:GenBoundcalN1}
\end{align}
\begin{comment}
\cite[Theorem 1.3]{Janson}
\begin{align}
    \calN(v(\Gamma),e(\Gamma),\s{H})\leq [v(\s{H})]^{v(\s{H})}\cdot\begin{cases}
        [e(\Gamma)]^{\alpha^\star(\s{H})}\ &\s{if}\;e(\Gamma)\leq v(\Gamma) \\
        [e(\Gamma)]^{v(\s{H})-\alpha^\star(\s{H})}[v(\Gamma)]^{2\alpha^\star(\s{H})-v(\s{H})}\ &\s{if}\;v(\Gamma)\leq e(\Gamma)\leq \binom{v(\Gamma)}{2}\\
        [v(\Gamma)]^{v(\s{H})}\ &\s{if}\;e(\Gamma)\geq \binom{v(\Gamma)}{2},
    \end{cases}
\end{align}
\end{comment}
where for a given graph $\s{H}$, the fractional independence number of $\s{H}$, denoted by $\alpha^\star(\s{H})$ is the largest value of $\sum_v\alpha_v$ over all assignments of weights $\alpha_v\in[0,1]$ to the vertices of $\s{H}$ satisfying the condition $\alpha_v+\alpha_u\leq1$, for all edges $(u,v)$ of $\s{H}$. It was shown in \cite[Lemma A.1]{Janson} that if $e(\s{H})>0$ then,
\begin{align}
    1\leq\frac{v(\s{H})}{2}\leq \alpha^\star(\s{H})\leq v(\s{H})-\frac{e(\s{H})}{\Delta(\s{H})}\leq  v(\s{H})-1,
\end{align}
where $\Delta(\s{H})$ is the maximum degree in $\s{H}$. Therefore, combining \eqref{eqn:2momentJans} and \eqref{eqn:GenBoundcalN1} we obtain,
\begin{align}
    \norm{\s{L}(\s{G})}_{\calH_0}^2&\leq 1+\sum_{\emptyset\neq \s{H}\subseteq \binom{e(\Gamma)}{2}}\lambda^{2|\s{H}|}\frac{[2e(\Gamma)]^{v(\s{H})-\alpha^\star(\s{H})}\cdot\min\ppp{2e(\Gamma),v(\Gamma)}^{2\alpha^\star(\s{H})-v(\s{H})}}{|\calS_{\s{H}}|}\\
    &= 1+\sum_{\emptyset\neq \s{H}\subseteq \binom{e(\Gamma)}{2}}\lambda^{2|\s{H}|}\frac{[2e(\Gamma)]^{v(\s{H})-\alpha^\star(\s{H})}\cdot\min\ppp{2e(\Gamma),v(\Gamma)}^{2\alpha^\star(\s{H})-v(\s{H})}}{\binom{n}{|v(\s{H})|}\cdot \frac{|v(\s{H})|!}{|\s{Aut}(\s{H})|}}.\label{eqn:bound}
\end{align}
The goal is now to derive rather sufficient conditions under which the former is bounded. While at first sight, this might seem like a promising approach, it is quite loose. To illustrate this, let us consider the case of planted path in the dense regime; our bounds in this paper show that the detection problem in this case is impossible whenever $|v(\Gamma)| = o(n)$, which is rather an intuitive result. We next analyze the above local Janson-type approach. 

Let us begin by recalling from \cite{Janson}, that for a path of with $s$ edges (and $s+1$ vertices) the fractional covering number $\alpha^\star$ is $(s+1)/2$ if $s$ is odd, and $(s+2)/2$ if $s$ is even. We also observe that any subgraph $\s{H}$ with $\ell$ edges of a path with $k$ edges is a disjoint union of $P(H)=m$ paths of lengths $\ell_1,\dots,\ell_m$ such that $\sum_i \ell_i=\ell$. Since $|v(\s{H})|\leq k+1 $ and 
\[|v(\s{H})|=\sum_{i=1}^m \overset{\text{length of the $i$'th path}}{\overbrace{\ell_i+1}}=m+\ell,\]
we have
\[m \leq k-\ell+1.\]
On the other hand the number of paths $m$ will always be at most the number of edges, and therefore $m\leq \ell$.
Immediately from the definition of $\alpha^\star$,  we observe that for a disjoint union of graphs $\s{H}_1,\dots, \s{H}_{\ell}$, 
\[\alpha^\star\p{\bigcup_{i=1}^\ell \s{H}_i}=\sum_{i=1}^\ell \alpha^\star\p{\s{H}_i}. \]
Combining the above, any subgraph of $\Gamma$, a path of length $k$ edges, is a disjoint union of $m$ paths of lengths $\ell_1,\dots, \ell_m$, where $1\leq \ell \leq \min(k+1-\ell,\ell)$, and 
\[\frac{\ell}{2}=\sum_{i=1}^m \frac{\ell_i}{2}\leq \alpha^\star (\s{H})\leq \sum_{i=1}^m \frac{\ell_i+2}{2}=m+\frac{\ell}{2}.\]
We also recall the following result.
\begin{lemma} \label{lem:automorhisms}\cite{bams1183513745}
     Let $\s{G}$ be a graph composed as a union of disjoint graphs $\s{G}=\bigcup_{i=1}^\ell \bigcup_{j=1}^{h_i} \s{G}_{i,j}$, such that for any $i$, $\s{G}_{i,1}\dots \s{G}_{i,h_i}$ are isomorphic and for any $i'\neq i$ and each $j'$ the graphs $\s{G}_{i,j}$ and $\s{G}_{i',j'}$ are not isomorphic. Then:
     \begin{equation}
         \abs{\s{Aut}(\s{G})}=\prod_{i=1}^\ell\abs{\s{Aut}(\s{G}_{i,1})}^{h_i}\cdot h_i!
     \end{equation}
 \end{lemma}  
To wit, the size of the automorphism group of a disjoint union of $m$ paths of lengths $\ell_1,\dots,\ell_m$ with $\sum_i \ell_i=\ell$ is given by 
\[ |\s{Aut}(\s{H})|=2^m\prod_{i=1}^\ell p_{\s{H}}(i)!,\] 
where $p_{\s{H}}(i)$ is the number sub-paths with length $i$.

No we can further analyze \eqref{eqn:bound}:
\begin{align}
    \norm{\s{L(G)}}_{\calH_0}^2 &\leq 1+\sum_{\emptyset\neq \s{H}\subseteq \binom{e(\Gamma)}{2}}\lambda^{2|\s{H}|}\frac{(2|e(\Gamma)|)^{\alpha^\star(\s{H})}}{\binom{n}{|v(\s{H})|}\cdot \frac{|v(\s{H})|!}{|\s{Aut}(\s{H})|}}\\
    &=1+\sum_{\ell=1
}^k\sum_{m=1}^{\min(\ell,k+1-\ell)}\sum_{\substack{\s{H}\subseteq \binom{e(\Gamma)}{2}\\ (|e(\s{H})|,|P(\s{H})|)=(\ell,m)}}\lambda^{2\ell} \frac{(2k)^{\alpha^\star(\s{H})}2^m \prod_i p_{\s{H}}(i)}{\frac{n!}{(n-m-\ell)!}}\\
 &\leq 1+\sum_{\ell=1
}^k\sum_{m=1}^{\min(\ell,k+1-\ell)}\sum_{\substack{\s{H}\subseteq \binom{e(\Gamma)}{2}\\ (|e(\s{H})|,|P(\s{H})|)=(\ell,m)}}\lambda^{2\ell} (2k)^{\frac{\ell}{2}+m}2^m\frac{(n-m-\ell)!}{n!}  \prod_i p_{\s{H}}(i).
\end{align}
 The most inner sum in the above expression is the same for all subgraphs of the path $\s{H}$ with composed from a disjoint union of exactly $m$ paths of lengths of $\ell_1,\dots,\ell_m$. Now let us evaluate the inner sum. To that end, we should count over all equivalence classes (under graph isomorphism) representing subgraphs with $\ell$ edges and $m$ disjoint paths of lengths $\ell_1,\dots,\ell_m$. Next, we consider the size of any such equivalent class, which is exactly the size of $\calN(\s{H},\Gamma)$, for a representing element $\s{H}$. To that end, we begin by observing that a non-empty subgraph of a path is a disjoint union of paths. Let $\s{H}$ be such a subgraphs composed of $\ell$ edges, $m$ disjoint paths, such that for any $1\leq i\leq \ell$, the subgraph $\s{H}$ contains exactly $h_i$ paths of length $i$. We begin by analyzing $\calN(\s{H},\Gamma)$ and $\calS(\s{H})$.  As we know,
    \begin{align}
        \calS(\s{H})=\binom{n}{|v(\s{H})|}\frac{|v(\s{H})|!}{|\s{Aut}(\s{H})|}=\binom{n}{\ell+m}\frac{(\ell+m)!}{|\s{Aut}(\s{H})|}.
    \end{align}
    Thus, we aim to determine $\abs{\s{Aut}(\s{H})}$. We note  any two sub paths of $\s{H}$ are isomorphic if and only if they have the same length. We also note that a path as exactly two isomorphisms, the identity and the one that reverses the direction of the path. Using Lemma~\ref{lem:automorhisms}
    \begin{align}
        \abs{\s{Aut}(\s{H})}= \prod_{i=1}^\ell 2^{h_i}\cdot h_i!=2^{\sum_{i=1}^m h_i}\prod_{i=1}^\ell h_i!=2^m\prod_{i=1}^\ell h_i!.
    \end{align}
    Next, we turn to evaluate $\calN(\s{H},\Gamma)$, i.e., the number of copies of $\s{H}$ in $\Gamma$. We count this number of copies as follows. First, we go over all $m!$ orderings of the distinct paths composing $\s{H}$ in line, according to their order of appearance in $\Gamma$. We note that by doing so, generate any each order multiple times and we therefore over count since when counting copies of $\s{H}$ we do not distinguish between different sub-paths of the same length. Thus, we should normalize by 
    \begin{align}
        \frac{1}{\prod_{i=1}^\ell h_i!},
    \end{align} 
    which is the number of internal orderings of paths of the same length. In order to count the number of copies, it only remains to go over all possibilities of the number of edges separating between any two distinct sub-paths, which is at least $1$ (as the paths are disjoint). Thus, we count the number of ways distribute the renaming $k-(\ell+m-1)$ edges between the distinct paths. That is exactly the number ways to distribute $k-(\ell+m-1)$ elements in $m+1$ bins (we note that edges may be located before the first path and after the last path as well). The number of ways to distribute $q$ elements into $d$ cells  is known to be equal to $\binom{q+d-1}{d-1}$, and therefore we have
    \begin{align}
        \binom{k-(\ell+m-1)+(m+1)-1}{(m+1)-1}=\binom{k-\ell +1}{m}
    \end{align}
    possibilities of the remaining edges. Combining all together we have,
    \begin{align}\label{eq:NHGAMMAboundPath}
        \calN(\s{H},\Gamma)=\frac{ m!}{\prod_{i=1}^\ell h_i!}\cdot \binom{k-\ell +1}{m}=\frac{(k-\ell+1)!}{(k-\ell-m+1)!\prod_{i=1}^\ell h_i!}\leq \frac{\p{k-\ell+1}^m}{\prod_{i=1}^\ell p_{\s{H}}(i)!}.
    \end{align} 
    Note that different equivalence classes are distinguished by the number of parts of each size, which corresponds to integer partitions of $m$ to exactly $\ell$ parts. For an integer partition $p$, we denote $p(i)$ the number of elements of size $i$ in the partition. We combine it with our bound and obtain:
    \begin{align}
    \norm{\s{L(G)}}_{\calH_0}^2 &\leq 1+\sum_{\ell=1
}^k\sum_{m=1}^{\min(\ell,k+1-\ell)}\sum_{\substack{\s{H}\subseteq \binom{e(\Gamma)}{2}\\ (|e(\s{H})|,|P(\s{H})|)=(\ell,m)}}\lambda^{2\ell} (2k)^{\frac{\ell}{2}+m}2^m\frac{(n-m-\ell)!}{n!}  \prod_i p_{\s{H}}(i)\\
&\leq 1+\sum_{\ell=1
}^k\sum_{m=1}^{\min(\ell,k+1-\ell)}\sum_{\substack{\s{H}\subseteq \binom{e(\Gamma)}{2}\\ (|e(\s{H})|,|P(\s{H})|)=(\ell,m)}}\lambda^{2\ell} (2k)^{\frac{\ell}{2}+m}2^m (n-m-l)^{-(m+\ell)}\prod_i p_{\s{H}}(i)\\
&\leq 1+\sum_{\ell=1
}^k\sum_{m=1}^{\min(\ell,k+1-\ell)}\sum_{p\in\s{Par}(\ell,m)}\calN(\s{H}_p,\Gamma) \lambda^{2\ell} (2k)^{\frac{\ell}{2}+m}2^m (n-m-l)^{-(m+\ell)}\prod_i p_{\s{H}}(i)\\
&\overset{(a)}{\leq} 1+\sum_{\ell=1
}^k\sum_{m=1}^{\min(\ell,k+1-\ell)}\sum_{p\in\s{Par}(\ell,m)}(k-\ell+1)^m \lambda^{2\ell} (2k)^{\frac{\ell}{2}+m}2^m (n-k-1)^{-(m+\ell)}\\
&\overset{(b)}{\leq} 1+\sum_{\ell=1
}^k\sum_{m=1}^{\min(\ell,k+1-\ell)}e^{c\sqrt{\ell}}(k-\ell+1)^m \lambda^{2\ell} (2k)^{\frac{\ell}{2}+m}2^m (n-k-1)^{-(m+\ell)}\\
&\leq 1+\sum_{\ell=1
}^k \frac{e^{c\sqrt{\ell}}2^{\frac{3\ell}{2}}k^{\frac{\ell}{2}}}{(n-k-1)^\ell}
\sum_{m=1}^{\ell}\p{\frac{2k^2}{n-k-1}}^m
\end{align}
where 
\begin{itemize}
    \item[(a)] follows from \eqref{eq:NHGAMMAboundPath} and since $m+\ell\leq k+1$.
    \item[(b)] follows from \eqref{eq:Hardy-Ramanujan}.
\end{itemize}
Assuming that $k$ is in the interesting regime $\sqrt{n}\ll k $. we have $\frac{2k^2}{n-k-1}\to\infty$ and 
\begin{equation}
    \sum_{m=1}^{\ell}\p{\frac{2k^2}{n-k-1}}^m=(1+o(1))\p{\frac{2k^2}{n-k-1}}^\ell.
\end{equation}
Combining the above with our calculation, we get,
\begin{align}
     \norm{\s{L(G)}}_{\calH_0}^2 &\leq  1+(1+o(1))\sum_{\ell=1
}^k \frac{e^{c\sqrt{\ell}}2^{\frac{\ell}{2}}k^{\frac{\ell}{2}}}{(n-k-1)^\ell}
\p{\frac{2k^2}{n-k-1}}^\ell\\
&\leq  1+(1+o(1))\sum_{\ell=1
}^k  \p{ \frac{\tilde{\lambda^2 C}k^{\frac{5}{2}}}{(n-k-1)^2}}^\ell.
\end{align}
Thus, the above sum converge when $k\ll n^{\frac{4}{5}}$. This is of course loose, as we already proved that detection is impossible for $k\ll n^{\frac{4}{5}}$.
\begin{comment}
    
\end{comment}
\end{document}